\documentclass[toc=bibnumbered,ngerman,english,titlepage=false,twoside]{scrartcl}
\usepackage[lmargin=107.551428pt,rmargin=71.700952pt]{geometry}

\usepackage[utf8]{inputenc}
\usepackage{babel}
\usepackage[T1]{fontenc}
\usepackage{lmodern}
\usepackage{microtype}

\usepackage{amsmath,amsthm,amssymb}
\usepackage{bm}
\usepackage{textcomp}
\usepackage{csquotes}
\usepackage{enumitem}
\usepackage[color links=true,allcolors=black]{hyperref}

\usepackage{graphicx}
\usepackage{float}
\usepackage{times}

\usepackage[automark,footwidth=\textwidth,headwidth=\textwidth]{scrlayer-scrpage}

\numberwithin{equation}{section} 

\newcommand{\N}{\ensuremath{\mathbb{N}}}

\newcommand{\R}{\ensuremath{\mathbb{R}}}

\newcommand{\C}{\ensuremath{\mathbb{C}}}
\newcommand{\K}{\ensuremath{\mathbb{K}}}

\newcommand{\diverg}{\textup{div}}

\newcommand{\tr}{\textup{tr}}

\newcommand{\supp}{\textup{supp}}

\newcommand{\Hc}{\ensuremath{\mathcal{H}}}

\newcommand{\Lc}{\ensuremath{\mathcal{L}}}

\newcommand{\Oc}{\ensuremath{\mathcal{O}}}

\setkomafont{disposition}{\normalcolor\bfseries}
\setkomafont{section}{\LARGE} 
\setkomafont{subsection}{\Large} 
 
\begin{document}
\thispagestyle{plain}

\topmargin -18pt\headheight 12pt\headsep 25pt

\ifx\cs\documentclass \footheight 12pt \fi \footskip 30pt

\textheight 625pt\textwidth 431pt\columnsep 10pt\columnseprule 0pt 
                                                                      
\clearpairofpagestyles
\automark*[section]{section}
\automark*[subsection]{}
\ohead[]{\scshape\headmark}
\ofoot*{\pagemark}
\pagestyle{scrheadings}

 \renewcommand{\headfont}{\slshape}      
 \renewcommand{\pnumfont}{\upshape}      
\setcounter{secnumdepth}{5}             
\setcounter{tocdepth}{5}             

\newtheorem{Definition}{Definition}[section]
\newtheorem{Satz}[Definition]{Satz}
\newtheorem{Lemma}[Definition]{Lemma}
\newtheorem{Korollar}[Definition]{Korollar}
\newtheorem{Corollary}[Definition]{Corollary}
\newtheorem{Bemerkung}[Definition]{Bemerkung}
\newtheorem{Remark}[Definition]{Remark}
\newtheorem{Proposition}[Definition]{Proposition}
\newtheorem{Beispiel}[Definition]{Beispiel}
\newtheorem{Theorem}[Definition]{Theorem}

\thispagestyle{plain}

\pdfbookmark[1]{Titlepage}{title}
\begin{center}
	{\LARGE Convergence of the Scalar- and Vector-Valued Allen-Cahn Equation
	to Mean Curvature Flow with $90$°-Contact Angle\\[0.5ex]
	in Higher Dimensions}\\[2ex]
	\textsc{Maximilian Moser}\\[1ex]
	Fakultät für Mathematik, Universität Regensburg, Universitätsstraße 31,\\ D-93053 Regensburg, Germany\\
	maximilian1.moser@mathematik.uni-regensburg.de\\[1ex]
\end{center}

\begin{abstract}
	\textbf{Abstract.} We consider the sharp interface limit for the scalar-valued and vector-valued Allen-Cahn equation with homogeneous Neumann boundary condition in a bounded smooth domain $\Omega$ of arbitrary dimension $N\geq 2$ in the situation when a two-phase diffuse interface has developed and intersects the boundary $\partial\Omega$. The limit problem is mean curvature flow with $90$°-contact angle and we show convergence in strong norms for well-prepared initial data as long as a smooth solution to the limit problem exists. To this end we assume that the limit problem has a smooth solution on $[0,T]$ for some time $T>0$. Based on the latter we construct suitable curvilinear coordinates and set up an asymptotic expansion for the scalar-valued and the vector-valued Allen-Cahn equation. Finally, we prove a spectral estimate for the linearized Allen-Cahn operator in both cases in order to estimate the difference of the exact and approximate solutions with a Gronwall-type argument.\\
	
\end{abstract}

\section{Introduction}\label{sec_intro}
In the following we introduce the Allen-Cahn equation and the vector-valued variant considered in this paper. Moreover, we motivate and review results concerning the corresponding sharp interface limits. 

Let us begin with the scalar Allen-Cahn equation. Let $N\in\N$, $N\geq 2$ and $\Omega\subset\R^N$ be a bounded, smooth domain with outer unit normal $N_{\partial\Omega}$. Moreover, let $\varepsilon>0$ small. For $T>0$ and  $u_\varepsilon:\overline{\Omega}\times[0,T]\rightarrow\R$ we consider the Allen-Cahn equation with homogeneous Neumann boundary condition \hypertarget{AC}{(AC)} consisting of
\begin{alignat}{2}\label{eq_AC1}\tag{AC1}
\partial_tu_\varepsilon-\Delta u_\varepsilon+\frac{1}{\varepsilon^2}f'(u_\varepsilon)&=0&\qquad&\text{ in }\Omega\times(0,T),\\\label{eq_AC2}\tag{AC2}
\partial_{N_{\partial\Omega}}u_\varepsilon&=0&\qquad&\text{ on }\partial\Omega\times(0,T),\\
u_\varepsilon|_{t=0}&=u_{0,\varepsilon} &\qquad&\text{ in }\Omega,\label{eq_AC3}\tag{AC3}
\end{alignat} 
where $f:\R\rightarrow\R$ is a suitable smooth double well potential with wells of equal depth. A typical example is $f(u)=\frac{1}{2}(1-u^2)^2$, see Figure \ref{fig_double_well}. The precise conditions are
\begin{align}\label{eq_AC_fvor1}
f\in C^\infty(\R),\quad f'(\pm1)=0,\quad f''(\pm1)>0,\quad 
f(-1)=f(1),\quad f>f(1)\text{ in }(-1,1)
\end{align}
and we assume 
\begin{align}\label{eq_AC_fvor2}
uf'(u)\geq0\quad\text{ for all }|u|\geq R_0\text{ and some }R_0\geq 1.
\end{align} 
Note that \eqref{eq_AC_fvor2} is just a requirement for the sign of $f'$ outside a large ball. The condition \eqref{eq_AC_fvor2} is used later to obtain uniform a priori bounds for classical solutions $u_\varepsilon$, see Section \ref{sec_DC_prelim_bdd_scal} below. 
\begin{figure}[H]
	\centering
	\def\svgwidth{0.4\linewidth}
	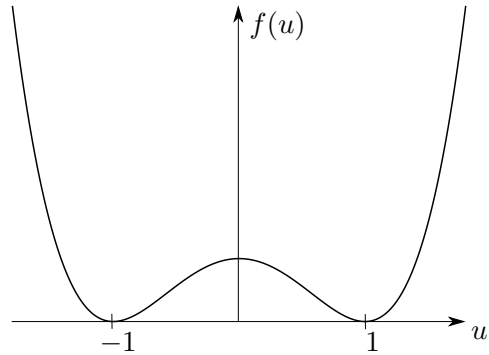
	\caption{Typical form of the double-well potential, $f(u)=\frac{1}{2}(1-u^2)^2$.}\label{fig_double_well}
\end{figure}

	The Allen-Cahn equation (similar to \eqref{eq_AC1}) was originally introduced by Allen and Cahn \cite{AC} to describe the evolution of antiphase boundaries in certain polycrystalline materials. Moreover, one can directly verify that equation \eqref{eq_AC1}-\eqref{eq_AC3} is the $L^2$-gradient flow to the energy\phantom{\qedhere}
	\begin{align}\label{eq_AC_energy}
	E_\varepsilon(u):=\int_\Omega\frac{1}{2}|\nabla u|^2+\frac{1}{\varepsilon^2}f(u)\,dx.
	\end{align}
	The latter is (up to a scaling in $\varepsilon$) the usual scalar Ginzburg-Landau energy or Modica-Mortola energy, cf.~Modica \cite{Modica}. For a summary of further motivations we refer to the introduction in Bronsard, Reitich \cite{BronsardReitich}. 
	
	The Allen-Cahn equation is a diffuse interface model: the $u_\varepsilon$ serves as an order parameter, where the values $\pm 1$ correspond to two distinct phases in applications. Typically after a short time $\Omega$ is partitioned into subdomains where the solution $u_\varepsilon$ of \eqref{eq_AC1}-\eqref{eq_AC3} is close to $\pm1$ and transition zones (diffuse interfaces; roughly $u_\varepsilon^{-1}([-1+\mu,1-\mu])$ for $\mu>0$ small) develop where $|\nabla u_\varepsilon|$ is large. See Figure \ref{fig_transition_zone} below for a typical situation. For a rigorous result in this direction (\enquote{generation of interfaces}) see Chen \cite{ChenGenPropInt}. Formally, one can see this in the equation since the \enquote{reaction term} $f'(u_\varepsilon)/\varepsilon^2$ should be large for small times away from the minima of $f$ compared to the diffusion term $\Delta u_\varepsilon$. One also speaks of fast reaction/slow diffusion, see Rubinstein, Sternberg, Keller \cite{RSK}. Neglecting $\Delta u_\varepsilon$, \eqref{eq_AC1} becomes an ODE in time for each space point. For this ODE the stationary points are $0$, $\pm 1$, where $0$ is unstable and $\pm 1$ are stable. Moreover, one can also have a look at the energy \eqref{eq_AC_energy}. It is well-known that solutions to the corresponding gradient flow behave in such a way that the energy is non-increasing (and decreases in some optimal sense) in time. In this perspective values of $u_\varepsilon$ away from $\pm 1$ are penalized strongly and values of $|\nabla u_\varepsilon|$ away from $0$ are penalized weakly. Heuristically (or in sufficiently smooth cases) one can argue that the thickness of the diffuse interfaces is proportional to $\varepsilon$. Hence for $\varepsilon\rightarrow0$ one should obtain a hypersurface $\Gamma=(\Gamma_t)_{t\in[0,T]}$ evolving in time, cf.~Figure \ref{fig_transition_zone}, and hence a sharp interface model. Such limits are therefore called \enquote{sharp interface limits}. 
	
	\begin{figure}[H]
		\centering
		\def\svgwidth{0.9\linewidth}
		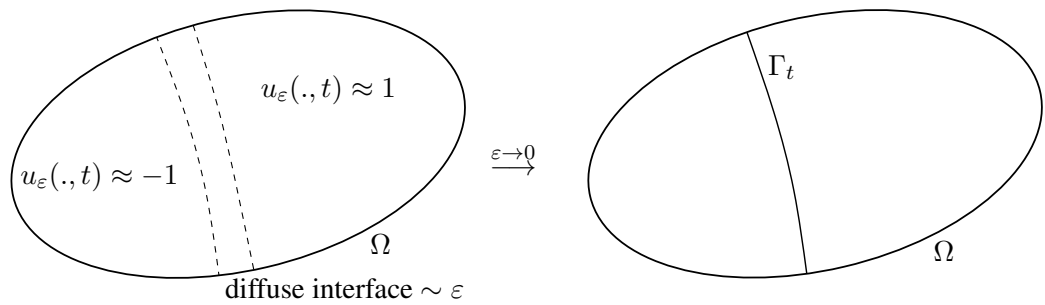
		\caption{Diffuse interface and sharp interface limit.}\label{fig_transition_zone}
	\end{figure}
	
	Next, let us introduce the vector-valued Allen-Cahn equation considered in the paper. Let $N,\Omega,N_{\partial\Omega},\varepsilon,T$ be as above. Moreover, let $m\in\N$ and $W:\R^m\rightarrow\R$ be a suitable potential which will be specified below. Let $\varepsilon>0$ be a small parameter. Then for $\vec{u}_\varepsilon:\overline{\Omega}\times[0,T]\rightarrow\R^m$ we consider the vector-valued Allen-Cahn equation with Neumann boundary condition \hypertarget{vAC}{(vAC)} consisting of
	\begin{alignat}{2}\label{eq_vAC1}\tag{vAC1}
	\partial_t\vec{u}_\varepsilon-\Delta \vec{u}_\varepsilon+\frac{1}{\varepsilon^2}\nabla W(\vec{u}_\varepsilon)&=0&\qquad&\text{ in }\Omega\times(0,T),\\\label{eq_vAC2}\tag{vAC2}
	\partial_{N_{\partial\Omega}}\vec{u}_\varepsilon&=0&\qquad&\text{ on }\partial\Omega\times(0,T),\\
	\vec{u}_\varepsilon|_{t=0}&=\vec{u}_{0,\varepsilon}&\qquad&\text{ in }\Omega.\label{eq_vAC3}\tag{vAC3}
	\end{alignat} 
	
	In analogy to the scalar case one can compute directly that equation \eqref{eq_vAC1}-\eqref{eq_vAC3} is the $L^2$-gradient flow to the vector-Ginzburg-Landau energy
	\begin{align}\label{eq_vAC_energy}
	\check{E}_\varepsilon(\vec{u}):=\int_\Omega\frac{1}{2}|\nabla \vec{u}|^2+\frac{1}{\varepsilon^2}W(\vec{u})\,dx.
	\end{align}
	See also Bronsard, Reitich \cite{BronsardReitich} for further motivations.

	\begin{proof}[The Potential $W$]
		Here we allow two types of potentials $W$. On the one hand, we consider $W$ with exactly two distinct minima and symmetry with respect to the hyperplane in the middle of these. On the other hand, we consider $m=2$ and triple-well potentials $W$ with symmetry. The first type is basically only interesting from a technical point of view, where with the second type one can describe e.g.~three distinct phases in a polycristalline material, cf.~\cite{BronsardReitich}. The precise requirements are as follows:\phantom{\qedhere}
		\begin{Definition}\upshape\label{th_vAC_W}
		Let $W:\R^m\rightarrow\R$ be smooth and one of the two assumptions hold:
		\begin{enumerate}
			\item $W$ has exactly two global minima $\vec{a}, \vec{b}$ with $W(\vec{a})=W(\vec{b})=0$ in which $D^2 W$ is positive definite and $W$ is symmetric with respect to the reflection $R_{\vec{a},\vec{b}}:\R^m\rightarrow\R^m$ at the hyperplane $\frac{1}{2} (\vec{a}+\vec{b})+\textup{span}\{\vec{a}-\vec{b}\}^\perp$. 
			\item $W$ is a symmetric triple well-potential for $m=2$, i.e.~$W$ has exactly three global minima $\vec{x}_i$, $i=1,3,5$ with $W(\vec{x}_i)=0$ for $i=1,3,5$ in which $D^2 W$ is positive definite and $W$ is symmetric with respect to the symmetry group $G$ of the equilateral triangle, cf.~Kusche \cite{Kusche}, Section 3.2 for the precise definition of $G$.
		\end{enumerate}
		Moreover, in both cases we require $\vec{u}\cdot \nabla W(\vec{u})\geq 0$ for all $\vec{u}\in\R^m$, $|\vec{u}|\geq \check{R}_0$ for some $\check{R}_0>0$. Furthermore, we assume that the kernel to a certain linear operator associated to $W$ is one-dimensional. For the precise condition see Remark \ref{th_ODE_vect_lin_op_rem} below.
		\end{Definition}
		
		\begin{Remark}\upshape\label{th_vAC_rem}
			\begin{enumerate}
				\item An example for a typical triple-well potential that fulfils the conditions in Definition \ref{th_vAC_W} can be found in \cite{Kusche}, Section 3.4. See Figure \ref{fig_triple_well} below. The example stems from Haas \cite{HaasDiss}, see also Garcke, Haas \cite{GarckeHaas}.
				\item Compared to the scalar case, we always require symmetry properties for the potential $W$. The assumption is used e.g.~in Section \ref{sec_ODE_vect_nonlin} below, but it might be possible to relax this.
			\end{enumerate}
		\end{Remark}
	\end{proof}
	
	\begin{figure}[htp]
		\centering
		\def\svgwidth{0.5\linewidth}
		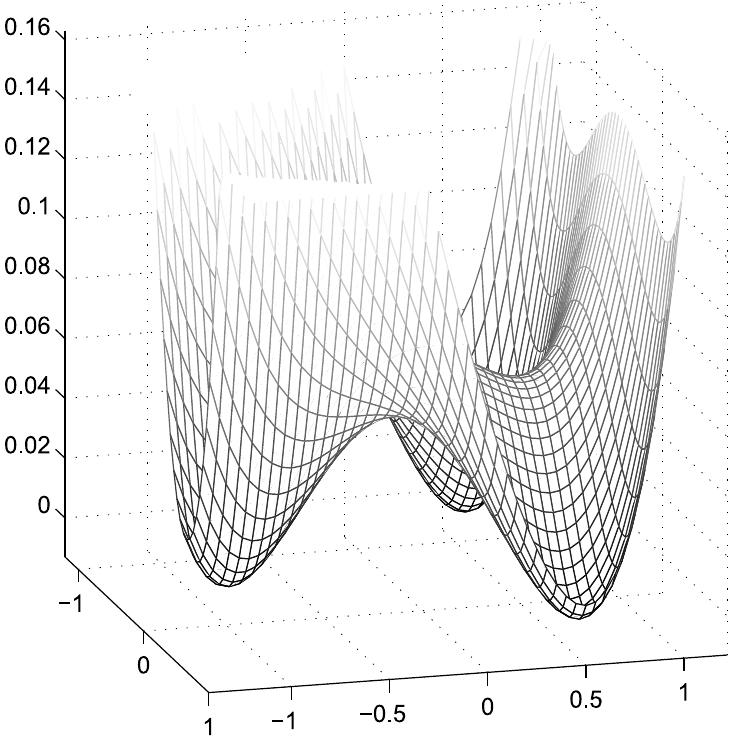
		\caption{Typical triple-well potential $W$. The image is taken from Kusche \cite{Kusche}.}\label{fig_triple_well}
	\end{figure}

	Analogously to the scalar case one can argue with formal arguments (system of fast reaction and slow diffusion; gradient flow to the energy \eqref{eq_vAC_energy}) that diffuse interfaces for solutions of \hyperlink{vAC}{(vAC)} should develop after short time. Note that the transition between minima of $W$ runs in $\R^m$. Moreover, in the case of a triple-well potential $W$ also three-fold diffuse interfaces between the three minima of $W$ are possible.

	\begin{proof}[Sharp/Diffuse Interface Models and Sharp Interface Limits]
	In general, sharp interface models and diffuse interface models are important model categories for the description of interfaces and moving boundaries in a large variety of applications. Some prominent examples are the melting of ice, the motion of an oil droplet in water, crystal growth, biological membranes, porous media, tumour evolution, spinodal decomposition of polymers and grain boundaries, see e.g.~\cite{Friedman}, \cite{PruessSimonett}, \cite{AC}, \cite{AndersonMcFaddenWheeler}, \cite{Novick_CH}, \cite{Miranville}, \cite{EckGarckeKnabner}, \cite{BarrettGarckeNuern}, \cite{BaiComMagPoz}, \cite{Ebenbeck} and the references therein.
	
	It is an important task to connect diffuse interface models and sharp interface models via their sharp interface limits for the following reasons (see also the partly universal introduction in Caginalp, Chen \cite{CaginalpChen} and general comments in Caginalp, Chen, Eck \cite{CaginalpChenEck}):\phantom{\qedhere}
	\begin{itemize}
		\item Modelling and Analysis: both types of models can usually be derived or motivated with physical principles, phenomenological observations or geometrical arguments etc., but one always incorporates some constitutive assumptions. Often the derivation for the sharp interface models is more transparent and these models appear simpler and more qualitative. On the other hand, diffuse interface models are usually advantageous in more complicated situations (see e.g.~\cite{AndersonMcFaddenWheeler}, p.141) and solutions typically have better analytical properties. Especially topology changes do not impose any difficulties. By identifying the sharp interface limit one confirms that the assumptions in the derivations are appropriate as well as that the models are compatible with each other and can be used to describe the same situation. Another motivation is the concept of using the diffuse interface model to extend solutions of the corresponding sharp interface model past singularities.
		\item Numerics: diffuse interface models are often simpler to solve numerically. By considering the sharp interface limit one justifies that the numerical solution to the diffuse interface model can be used to approximate the solution to the sharp interface model. 
	\end{itemize}
	Concerning results for sharp interface limits: in general there are formal results and rigorous proofs for convergence. The formal sharp interface limits are typically based on formal asymptotic expansions or numerical experiments. 
	However, see also Anderson, McFadden, Wheeler \cite{AndersonMcFaddenWheeler}, p.156ff for a \enquote{pillbox argument}, i.e.~reasoning with a small test volume. Regarding rigorous sharp interface limits, one can basically group such results into two types:
	\begin{itemize}
		\item Local time results that are applicable before singularities appear, i.e.~as long as the interface does not develop singularities and stays smooth. Relatively \enquote{strong} results are obtained, e.g.~norm estimates.
		\item Global time results using some kind of weak notion for the sharp interface system, e.g.~viscosity solutions for mean curvature flow, varifold solutions, distributional solutions, etc.
	\end{itemize}  
	\end{proof}
	
	\begin{proof}[Formal and Rigorous Sharp Interface Limit for the Allen-Cahn-Equation]
	In the situation of the scalar Allen-Cahn equation \hyperlink{AC}{(AC)}
	formal asymptotic analysis by Rubinstein, Sternberg, Keller \cite{RSK} yields that the limit sharp interface $(\Gamma_t)_{t\in[0,T]}$ should evolve according to the mean curvature flow
	\begin{align}\label{MCF}
	V_{\Gamma_t}=H_{\Gamma_t} \tag{MCF}
	\end{align}
	and, if there is boundary contact, there should be a $90$°-contact angle. 
	
	Moreover, the numerical experiments in Lee, Kim \cite{LeeKim} give another confirmation on a formal level for the convergence of \hyperlink{AC}{(AC)} to \eqref{MCF} with $90$°-contact angle. Finally, note that the 1D-case is not interesting for finite time in the $t$-scale because patterns persist for and evolve in exponentially slow time scales $\tau=e^{-c/\varepsilon}t$, cf.~Carr, Pego \cite{CarrPego}. More precisely, the sharp interface is a point and does not move in the time-scale $t$. This is consistent with \eqref{MCF} when the curvature of a point is defined as zero.
	
	For the vector-valued Allen-Cahn equation \hyperlink{vAC}{(vAC)} formal asymptotic calculations in Bronsard, Reitich \cite{BronsardReitich} yield that for a triple-well potential $W$ in the sharp interface limit $\varepsilon\rightarrow 0$ one should obtain \eqref{MCF} together with: 
	\begin{itemize}
		\item A $90$°-contact angle if a transition of two phases meets the boundary.\phantom{\qedhere}
		\item A $120$°-triple junction if the three phases meet at an interior point.
	\end{itemize}
	
	For the sake of completeness, note that the limit $\varepsilon\rightarrow 0$ in energies of the form \eqref{eq_AC_energy} (with similar potentials) has been considered in the context of $\Gamma$-convergence\footnote{~Of course this \enquote{$\Gamma$} has a different meaning than the $\Gamma$ in Figure \ref{fig_transition_zone}.}, see Modica \cite{Modica} (with mass constraint) and Sternberg \cite{SternbergSingPert} (with and without mass constraint). The $\Gamma$-limits are perimeter functionals which (at least formally) induce \eqref{MCF} with $90$°-contact angle via the $L^2$-gradient flow. This also motivates to study the dynamical problem \eqref{eq_AC1}-\eqref{eq_AC3} associated to the energy \eqref{eq_AC_energy} and its relation to \eqref{MCF} with $90$°-contact angle in the limit $\varepsilon\rightarrow 0$.
	
	For results in the direction of $\Gamma$-convergence with respect to $\varepsilon\rightarrow 0$ for energies like \eqref{eq_vAC_energy} (for several types of potentials and usually with mass constraint) see Baldo \cite{Baldo} and the references therein. The $\Gamma$-limits are (multiphase) perimeter functionals which (at least formally) induce (multiphase) mean curvature flow via the $L^2$-gradient flow. This gives another motivation to study the dynamical problem \eqref{eq_vAC1}-\eqref{eq_vAC3} and the connection to (multiphase) mean curvature flow in the limit $\varepsilon\rightarrow 0$.

	There are many rigorous results on the sharp interface limit for the Allen-Cahn equation (\eqref{eq_AC1} on $\R^N$ or \hyperlink{AC}{(AC)}) to \eqref{MCF} (in the case of \hyperlink{AC}{(AC)} with $90$°-contact angle).\phantom{\qedhere}
	
	We start with the local time results. Via a comparison principle and the construction of sub- and supersolutions, Chen \cite{ChenGenPropInt} proves local in time convergence as long as the interface stays smooth. Moreover, de Mottoni and Schatzman \cite{deMS} consider the $\R^N$-case and show convergence with strong norms for times when a smooth solution to \eqref{MCF} exists. This also works for \hyperlink{AC}{(AC)} when the interface is closed and strictly contained in $\Omega$. Note that the papers by Chen, Hilhorst, Logak \cite{CHL} and Abels, Liu \cite{ALiu} also yield results for \hyperlink{AC}{(AC)} with strictly contained interface by simple adjustments. The resulting proofs and results use ideas from \cite{deMS} but are more optimized. In Abels, Moser \cite{AbelsMoser} the methods in \cite{deMS}, \cite{CHL}, \cite{ALiu} are generalized to the case of boundary contact of the diffuse interface in two dimensions. Furthermore, there is the result by Fischer, Laux, Simon \cite{FischerLauxSimon}, where a relative entropy method is used. Finally, note that there is a paper by S\'aez \cite{SaezCurveShort}, but unfortunately there is a severe gap in the proof of the main theorem, cf.~\cite{AbelsMoser} for details.

	For global time results one has to use some weak formulation of \eqref{MCF}. There is the notion of viscosity solutions used by Evans, Soner, Souganidis \cite{ESS} for $\Omega=\R^N$ and by Katsoulakis, Kossioris, Reitich \cite{KKR} in the case of a convex, bounded domain. In the latter the maximum principle is used and sub- and supersolutions to the Allen-Cahn equation are constructed using the distance function from the level set of a viscosity solution for \eqref{MCF}. Moreover, varifold solutions to \eqref{MCF} are used by Ilmanen \cite{Ilmanen} in the $\R^N$-case, by Mizuno, Tonegawa \cite{MizunoTonegawa} for smooth, strictly convex, bounded domains and by Kagaya \cite{Kagaya} without the convexity assumption. For varifold solutions only convergence of a subsequence is achieved. Finally, there is the conditional result by Laux, Simon \cite{LauxSimon} where convergence of the (vector-valued; scalar case contained) Allen-Cahn equation to (multiphase) mean curvature flow in a BV-setting is obtained.

	Altogether, there is a large variety of results. However, to the authors knowledge \cite{ChenGenPropInt} and \cite{AbelsMoser} are the only results of local type that allow boundary contact for the diffuse interfaces. Moreover, to the authors knowledge in the vector-valued case there is only the conditional result \cite{LauxSimon} on the convergence of the vector-valued Allen-Cahn equation to multiphase mean curvature flow in a BV-setting. Note that there is a work by S\'aez \cite{SaezTriodFlow}, but unfortunately there is the same gap in the proof as in \cite{SaezCurveShort}, cf.~\cite{AbelsMoser} for details. It is difficult to generalize the methods in \cite{ChenGenPropInt} because comparison principles are used. On the other hand, the method by de Mottoni and Schatzman \cite{deMS} has proven to be versatile and was applied to many other diffuse interface models as well, see the comments below. 
	
	This strongly motivates to extend the result in \cite{AbelsMoser} to more complicated geometrical situations and equations. In this paper we generalize the latter in two directions. On the one hand, we look at the higher-dimensional setting and on the other hand, we also consider the vector-valued Allen-Cahn equation. The following results are obtained:
\begin{itemize}
	\item Convergence of (solutions to) the scalar-valued Allen-Cahn equation \hyperlink{AC}{(AC)} to \eqref{MCF} with $90$°-contact angle in any dimension $N\geq 2$. See Section \ref{sec_AC} below. For the case $N=2$ see also \cite{AbelsMoser}.
	\item Convergence of (solutions to) the vector-valued Allen-Cahn equation \hyperlink{vAC}{(vAC)} to \eqref{MCF} with $90$°-contact angle in any dimension $N\geq 2$, see Section \ref{sec_vAC} below. Here we only treat the case of two-phase transitions, i.e.~triple junctions are excluded. But the result is an important building block also for the triple junction case, since the arguments are localizable. 
\end{itemize}
The results are part of the PhD thesis of the author, cf.~Moser \cite{MoserDiss}. In a separate paper Abels, Moser \cite{AbelsMoserAlpha} the work \cite{AbelsMoser} is extended for the case of a non-linear Robin boundary condition which is designed in such a way that in the sharp interface limit \eqref{MCF} with $\alpha$-contact angle is attained, where $\alpha$ is fixed but close to $90$°. Note that also \cite{AbelsMoserAlpha} is part of \cite{MoserDiss}.
\end{proof}

\begin{proof}[The Method of de Mottoni and Schatzman] 
	One assumes that there exists a local smooth solution to the limit sharp interface problem. This can typically be shown for small times. Then 
	\begin{enumerate} 
		\item One \textit{rigorously} constructs an approximate solution to the diffuse interface model using asymptotic expansions based on the evolving surface that is (part of) the solution to the limit problem. To this end one has to solve model problems for the series coefficients.\phantom{\qedhere}
		\item Then one estimates the difference between exact and approximate solutions with a Gronwall-type argument. This typically involves a spectral estimate for a linear operator associated to the diffuse interface equation and the approximate solution. 
	\end{enumerate}
	This method also yields the typical profile of the solution and comparison principles are not needed in contrast to most of the other approaches. 
	
	Therefore the method was used for many other diffuse interface models as well. These results are based on general spectrum estimates in Chen \cite{ChenSpectrums} for Allen-Cahn, Cahn-Hilliard and phase-field-type operators. There are results for the Cahn-Hilliard equation by Alikakos, Bates, Chen \cite{ABC}, the phase-field equations by Caginalp, Chen \cite{CaginalpChen}, the mass-conserving Allen-Cahn equation by Chen, Hilhorst, Logak \cite{CHL}, the Cahn-Larché system by Abels, Schaubeck \cite{AbelsSchaubeck} and a Stokes/Allen-Cahn system by Abels, Liu \cite{ALiu}. See also Schaubeck \cite{Schaubeck} for a result on a convective Cahn-Hilliard equation. Moreover, Marquardt \cite{Marquardt} (see also \cite{AMa1}, \cite{AMa2}) studied the sharp interface limit for a Stokes/Cahn-Hilliard system. Furthermore, there is the result by Fei, Liu \cite{FeiLiu}, where a phase field approximation for the Willmore flow is considered. For the subtle variations in the rigorous asymptotic expansions and the spectral estimates used in applications of the method by de Mottoni and Schatzman see the inceptions to Sections 5 and 6 in \cite{MoserDiss}. Finally, \cite{AbelsMoser} is the first result obtained with the method of de Mottoni Schatzman that allows boundary contact for the diffuse interfaces.\phantom{qedhere}
\end{proof}

\begin{proof}[Mean Curvature Flow \eqref{MCF} with $90$°-Contact Angle, Coordinates and Notation]
	\enquote{Mean curvature flow} for evolving hypersurfaces means that the normal velocity equals mean curvature, where we define for convenience \enquote{mean curvature} as the sum of the principal curvatures. 
	
	For the convergence result below we will assume that \eqref{MCF} together with a $90$°-contact angle condition at $\partial\Omega$ has a smooth solution on a time interval $[0,T_0]$. This is a prerequisite for the method of de Mottoni and Schatzman \cite{deMS}. 
	
	The local well-posedness and existence of a smooth solution for small time starting from suitable initial sharp interfaces is basically well-known. At this point let us give some references in this direction. In Katsoulakis, Kossioris, Reitich \cite{KKR}, Section 2, a parametric approach is used to show local existence and uniqueness of classical solutions for \eqref{MCF} in arbitrary dimension and with fixed contact angle. In principle, it is also possible to reduce the evolution to a parabolic PDE by writing it over a reference hypersurface via suitable coordinates. For the typical procedure in the case of a closed interface see Prüss, Simonett \cite{PruessSimonett}. For curvilinear coordinates in the situation of boundary contact see Vogel \cite{Vogel} and Section \ref{sec_coord} below. Moreover, note that in Huisken \cite{Huisken} the special case of \eqref{MCF} with $90$°-contact angle in the graph case for cylindrical domains is considered and global existence and uniqueness of smooth solutions as well as convergence to a constant graph is obtained. 
	
	We need some notation in the context of the curvilinear coordinates in order to formulate the main theorems below.\phantom{\qedhere}

	\begin{Remark}[\textbf{Domain, Sharp Interface and Coordinates}]\label{th_intro_coord}\upshape
		For the details see Section \ref{sec_coord}.
		\begin{enumerate}
			\item \textit{Domain.} Let $N\in\N$, $N\geq 2$ and $\Omega\subset\R^N$ be a bounded, smooth domain with outer unit normal $N_{\partial\Omega}$. For $T>0$ we set $Q_T:=\Omega\times(0,T)$ and $\partial Q_T:=\partial\Omega\times[0,T]$. 
			\item \textit{Sharp Interface.} Consider $T_0>0$ and an evolving hypersurface $\Gamma=(\Gamma_t)_{t\in[0,T_0]}$ (with boundary, smooth, oriented, compact, connected) suitably parametrized over a reference hypersurface $\Sigma$ and such that $\partial\Gamma$ meets $\partial\Omega$ at contact angle $\frac{\pi}{2}$. For $\Gamma$ one can define the normal velocity $V_{\Gamma_t}$ and mean curvature $H_{\Gamma_t}$ at time $t\in[0,T_0]$ with respect to a unit normal $\vec{n}$ of $\Gamma$ in the classical sense. See Section \ref{sec_coord_surface_requ} for the precise assumptions and definitions.
			\item \textit{Coordinates.} We construct appropriate curvilinear coordinates $(r,s)$ with values in $[-2\delta,2\delta]\times\Sigma$ for some $\delta>0$ describing a neighbourhood of $\Gamma$ in $\overline{\Omega}\times[0,T_0]$. For the exact statements see in particular Theorem \ref{th_coordND}, with $2\delta$ instead of $\delta$ there. Here $r$ has the role of a signed distance function and $s$ works like a tangential projection. The set $\overline{Q_{T_0}}=\overline{\Omega}\times[0,T_0]$ is split by $\Gamma$ into two connected sets ($\Gamma$ excluded) according to the sign of $r$. We denote them with $Q_{T_0}^\pm$. Then we have the disjoint union
			\[
			\overline{Q_{T_0}}=\Gamma\cup Q_{T_0}^- \cup Q_{T_0}^+.
			\]
			Finally, we introduce tubular neighbourhoods $\Gamma(\eta):=r^{-1}((-\eta,\eta))$ for $\eta\in(0,2\delta]$ and define a suitable normal derivative $\partial_n$ and tangential gradient $\nabla_\tau$ on $\Gamma(\eta)$, see Remark \ref{th_coordND_rem}.	
		\end{enumerate} 
\end{Remark}\end{proof}

\begin{proof}[Structure of the paper.]
In Sections \ref{sec_AC}-\ref{sec_vAC} we formulate the main results for the scalar-valued and vector-valued Allen-Cahn equation, \hyperlink{AC}{(AC)} and \hyperlink{vAC}{(vAC)} respectively. In Section \ref{sec_fct} we fix some notation and introduce function spaces. The curvilinear coordinates are constructed in Section \ref{sec_coord}. In Section \ref{sec_model_problems} we solve the model problems appearing in the asymptotic expansions in Section \ref{sec_asym}. The spectral estimates are carried out in Section \ref{sec_SE}. The difference estimates and the proofs of the convergence theorems are carried out in Section \ref{sec_DC}.\phantom{\qedhere}
\end{proof}

\subsection[Result in Scalar Case]{Result in Scalar Case}\label{sec_AC}
Finally, we state the new convergence result for \hyperlink{AC}{(AC)} obtained with the method of de Mottoni and Schatzman \cite{deMS}.

\begin{Theorem}[\textbf{Convergence of (AC) to (MCF) with $90$°-Contact Angle}]\label{th_AC_conv}
	Let $N\geq2$, $\Omega$, $N_{\partial\Omega}$, $Q_T$ and $\partial Q_T$ for $T>0$ be as in Remark \ref{th_intro_coord},~1. Moreover, let $\Gamma=(\Gamma_t)_{t\in[0,T_0]}$ for some $T_0>0$ be a smooth evolving hypersurface with $\frac{\pi}{2}$-contact angle condition as in Remark \ref{th_intro_coord},~2.~and let $\Gamma$ satisfy \eqref{MCF}. Let $\delta>0$ small and the notation for $Q_{T_0}^\pm$, $\Gamma(\delta)$, $\nabla_\tau$, $\partial_n$ be as in Remark \ref{th_intro_coord},~3. Moreover, let $f$ satisfy \eqref{eq_AC_fvor1}-\eqref{eq_AC_fvor2}. Let $M\in\N$ with $M\geq k(N):=\max\{2,\frac{N}{2}\}$.
	
	Then there are $\varepsilon_0>0$ and $u^A_\varepsilon:\overline{\Omega}\times[0,T_0]\rightarrow\R$ smooth for $\varepsilon\in(0,\varepsilon_0]$ (depending on $M$) with $\lim_{\varepsilon\rightarrow 0}u^A_\varepsilon=\pm 1$ uniformly on compact subsets of $Q_{T_0}^\pm$ and such that the following holds: 
	\begin{enumerate}
	\item If $M>k(N)$, then let $u_{0,\varepsilon}\in C^2(\overline{\Omega})$ with $\partial_{N_{\partial\Omega}} u_{0,\varepsilon}=0$ on $\partial\Omega$ for $\varepsilon\in(0,\varepsilon_0]$ and
	\begin{align}\label{eq_AC_conv1}
	\sup_{\varepsilon\in(0,\varepsilon_0]}\|u_{0,\varepsilon}\|_{L^\infty(\Omega)}<\infty\quad\text{ and }\quad\|u_{0,\varepsilon}-u^A_\varepsilon|_{t=0}\|_{L^2(\Omega)}\leq R\varepsilon^{M+\frac{1}{2}}
	\end{align}
	for some $R>0$ and all $\varepsilon\in(0,\varepsilon_0]$.
	Then for any set of solutions $u_\varepsilon\in C^2(\overline{Q_{T_0}})$ of \eqref{eq_AC1}-\eqref{eq_AC3} for $\varepsilon\in(0,\varepsilon_0]$ with initial values $u_{0,\varepsilon}$ there are $\varepsilon_1\in(0,\varepsilon_0]$, $C>0$ such that
	\begin{align}\begin{split}\label{eq_AC_conv2}
	\sup_{t\in[0,T]}\|(u_\varepsilon-u^A_\varepsilon)(t)\|_{L^2(\Omega)}+\|\nabla(u_\varepsilon-u^A_\varepsilon)\|_{L^2(Q_T\setminus\Gamma(\delta))}&\leq C\varepsilon^{M+\frac{1}{2}},\\
	\|\nabla_\tau(u_\varepsilon-u^A_\varepsilon)\|_{L^2(Q_T\cap\Gamma(\delta))}+\varepsilon\|\partial_n(u_\varepsilon-u^A_\varepsilon)\|_{L^2(Q_T\cap\Gamma(\delta))}&\leq C\varepsilon^{M+\frac{1}{2}}\end{split}
	\end{align}
	for all $\varepsilon\in(0,\varepsilon_1]$ and $T\in(0,T_0]$.
	\item If $k(N)\in\N$ and $M\geq k(N)+1$, then there is a $\tilde{R}>0$ small such that the assertion in 1.~holds, when $R,M$ in \eqref{eq_AC_conv1}-\eqref{eq_AC_conv2} are replaced by $\tilde{R},k(N)$.
	\item If $N\in\{2,3\}$ and $M=2(=k(N))$, then there is $T_1\in(0,T_0]$ such that the assertion in 1.~is valid but only such that \eqref{eq_AC_conv2} holds for all $\varepsilon\in(0,\varepsilon_1]$ and $T\in(0,T_1]$. 
	\end{enumerate}
\end{Theorem}
\begin{Remark}\phantomsection{\label{th_AC_conv_rem}}\upshape\begin{enumerate}
	\item \textit{Interpretation of Theorem \ref{th_AC_conv}.}
	One can interpret $u^A_\varepsilon$ in the theorem as representation of a diffuse interface moving with $\Gamma$ since $u^A_\varepsilon$ is smooth but converges for $\varepsilon\rightarrow 0$ to a step function whose jump set is the solution $\Gamma$ to \eqref{MCF} with $\frac{\pi}{2}$-contact angle starting from $\Gamma_0$. The assumption on the initial values $u_{0,\varepsilon}$ in Theorem \ref{th_AC_conv} essentially means that a diffuse interface already has developed and is located at the initial sharp interface $\Gamma_0$ at time $t=0$, i.e.~the generation of diffuse interfaces in the evolution is skipped. One also speaks of \enquote{well-prepared initial data}, cf.~\cite{ALiu}. Hence Theorem \ref{th_AC_conv} basically shows that the qualitative behaviour of diffuse interfaces with boundary contact, generated by \hyperlink{AC}{(AC)}, is that of \eqref{MCF} with $90$°-contact angle, at least as long as the evolution of the latter stays smooth.
	\item \textit{Layout of the Proof.} 
	Required model problems, some ODEs on $\R$ and a linear elliptic equation on $\R\times(0,\infty)$ are considered in Section \ref{sec_ODE_scalar} and Section \ref{sec_hp_90} below, respectively. The asymptotic expansions are carried out in Section \ref{sec_asym_ACND}. The approximate solution $u^A_\varepsilon$ is defined in Section \ref{sec_asym_ACND_uA}. Note that $M$ corresponds to the number of terms in the expansion. The spectral estimate is proven in Section \ref{sec_SE_ACND} and the difference estimate in Section \ref{sec_DC_ACND_DE}. Finally, Theorem \ref{th_AC_conv} is obtained in Section \ref{sec_DC_ACND_conv}.
	\item \textit{Well-Posedness of \hyperlink{AC}{(AC)}.} In Theorem \ref{th_AC_conv} the existence of solutions $u_\varepsilon\in C^2(\overline{Q_{T_0}})$ of \hyperlink{AC}{(AC)} is assumed, but this is in principle well-known, see the references in Bellettini \cite{Bellettini}, Remark 15.1. An approach with weak solutions (obtained via time-discretization) can also be found in Bartels \cite{Bartels_book}, Chapter 6.1. Moreover, equation \eqref{eq_AC1}-\eqref{eq_AC3} fits in the general framework of Lunardi \cite{LunardiOptReg}, Section 7.3.1, where a semigroup approach and a Hölder-setting is used. Together with a priori boundedness of classical solutions (see Section \ref{sec_DC_prelim_bdd_scal} below) that can be obtained with maximum principle arguments, one can show global well-posedness for regular, bounded initial data. See e.g.~\cite{LunardiOptReg}, Proposition 7.3.2. Higher regularity then follows using linear theory, cf.~Lunardi, Sinestrari, von Wahl \cite{LunardiSvW}. Finally, note that well-posedness for \eqref{eq_AC1} on $\R^N$ is shown in \cite{deMS} for bounded initial data with estimates for the heat semigroup.
	\item The approximate solution $u^A_\varepsilon$ obtained from the explicit construction equals $\pm 1$ in the set $Q_{T_0}^\pm\setminus\Gamma(2\delta)$ and has a smooth, increasingly steep transition with a known \enquote{optimal profile} inbetween. Therefore Theorem \ref{th_AC_conv} also yields the typical profile of solutions to \hyperlink{AC}{(AC)} across diffuse interfaces.
	\item The level sets $\{u^A_\varepsilon=0\}$, $\{u_\varepsilon=0\}$ can be viewed as approximations for $\Gamma$. Note that in the explicit construction of $u^A_\varepsilon$ in Section \ref{sec_asym_ACND} below the error from $\{u^A_\varepsilon=0\}$ to $\Gamma$ is of order $\varepsilon$ and, in the case that $f$ is even, of order $\varepsilon^2$, see Remark \ref{th_asym_ACND_feven_rem2}. If one uses numerical computations for \hyperlink{AC}{(AC)} in order to approximate solutions to \eqref{MCF} with $90$°-contact angle condition this is of interest, cf.~also Caginalp, Chen, Eck \cite{CaginalpChenEck}.
	\item In principle also estimates of stronger norms are possible in the situation of Theorem \ref{th_AC_conv}, but better estimates for the initial values could be required. The basic idea is to interpolate the already controlled norms with stronger norms that can be estimated for exact solutions by some negative $\varepsilon$-orders. Cf.~Alikakos, Bates, Chen \cite{ABC}, Theorem 2.3 for a similar idea in the case of the Cahn-Hilliard equation. However, note that this does not improve the approximation of $\Gamma$ in the sense of 5.
	\item Theorem \ref{th_AC_conv} and the above comments hold analogously for closed $\Gamma$ moving by \eqref{MCF} and compactly contained in $\Omega$. The proof is basically contained since the constructions are localizable.
	\end{enumerate}
\end{Remark}

\subsection[Result in Vector-Valued Case]{Result in Vector-Valued Case} \label{sec_vAC} 

\begin{Theorem}[\textbf{Convergence of (vAC) to (MCF) with  $90$°-Contact Angle}]\label{th_vAC_conv}
	Let $N\geq2$, $\Omega$, $N_{\partial\Omega}$, $Q_T$ and $\partial Q_T$ be as in Remark \ref{th_intro_coord},~1. Moreover, let $\Gamma=(\Gamma_t)_{t\in[0,T_0]}$ for some $T_0>0$ be a smooth evolving hypersurface with $\frac{\pi}{2}$-contact angle condition as in Remark \ref{th_intro_coord},~2.~and let $\Gamma$ satisfy \eqref{MCF}. Let $\delta>0$ be small and the notation for $Q_{T_0}^\pm$, $\Gamma(\delta)$, $\nabla_\tau$, $\partial_n$ be as in Remark \ref{th_intro_coord},~3. Moreover, let $W:\R^m\rightarrow\R$ be as in Definition \ref{th_vAC_W} and $\vec{u}_\pm$ be any distinct pair of minimizers of $W$. Finally, let $M\in\N$ with $M\geq k(N):=\max\{2,\frac{N}{2}\}$.
	
	Then there are $\check{\varepsilon}_0>0$ and $\vec{u}^A_\varepsilon:\overline{\Omega}\times[0,T_0]\rightarrow\R^m$ smooth for $\varepsilon\in(0,\check{\varepsilon}_0]$ (depending on $M$) with $\lim_{\varepsilon\rightarrow 0}\vec{u}^A_\varepsilon=\vec{u}_\pm$ uniformly on compact subsets of $Q_{T_0}^\pm$ and such that the following holds: 
	\begin{enumerate}
		\item If $M>k(N)$, then let $\vec{u}_{0,\varepsilon}\in C^2(\overline{\Omega})^m$ with $\partial_{N_{\partial\Omega}} \vec{u}_{0,\varepsilon}=0$ on $\partial\Omega$ for $\varepsilon\in(0,\check{\varepsilon}_0]$ and
		\begin{align}\label{eq_vAC_conv1}
		\sup_{\varepsilon\in(0,\check{\varepsilon}_0]}\|\vec{u}_{0,\varepsilon}\|_{L^\infty(\Omega)^m}<\infty\quad\text{ and }\quad\|\vec{u}_{0,\varepsilon}-\vec{u}^A_\varepsilon|_{t=0}\|_{L^2(\Omega)^m}\leq R\varepsilon^{M+\frac{1}{2}}
		\end{align}
		for some $R>0$ and all $\varepsilon\in(0,\check{\varepsilon}_0]$.
		Then for any set of solutions $\vec{u}_\varepsilon\in C^2(\overline{Q_{T_0}})^m$ of \eqref{eq_vAC1}-\eqref{eq_vAC3} for $\varepsilon\in(0,\check{\varepsilon}_0]$ with initial values $\vec{u}_{0,\varepsilon}$ there are $\check{\varepsilon}_1\in(0,\check{\varepsilon}_0]$, $C>0$ with
		\begin{align}\begin{split}\label{eq_vAC_conv2}
		\sup_{t\in[0,T]}\|(\vec{u}_\varepsilon-\vec{u}^A_\varepsilon)(t)\|_{L^2(\Omega)^m}
		+\|\nabla(\vec{u}_\varepsilon-\vec{u}^A_\varepsilon)\|_{L^2(Q_T\setminus\Gamma(\delta))^{N\times m}}&\leq C\varepsilon^{M+\frac{1}{2}},\\
		\|\nabla_\tau(\vec{u}_\varepsilon-\vec{u}^A_\varepsilon)\|_{L^2(Q_T\cap\Gamma(\delta))^{N\times m}}
		+\varepsilon\|\partial_n(\vec{u}_\varepsilon-\vec{u}^A_\varepsilon)\|_{L^2(Q_T\cap\Gamma(\delta))^m}&\leq C\varepsilon^{M+\frac{1}{2}}\end{split}
		\end{align}
		for all $\varepsilon\in(0,\check{\varepsilon}_1]$ and $T\in(0,T_0]$.
		\item If $k(N)\in\N$ and $M\geq k(N)+1$, then there is a $\check{R}>0$ small such that the assertion in 1.~holds, when $R,M$ in \eqref{eq_vAC_conv1}-\eqref{eq_vAC_conv2} are replaced by $\check{R},k(N)$.
		\item If $N\in\{2,3\}$ and $M=2(=k(N))$, then there is $\check{T}_1\in(0,T_0]$ such that the assertion in 1.~is valid but only such that \eqref{eq_vAC_conv2} holds for all $\varepsilon\in(0,\check{\varepsilon}_1]$ and $T\in(0,\check{T}_1]$. 
	\end{enumerate}
\end{Theorem}
\begin{Remark}\phantomsection{\label{th_vAC_conv_rem}}\upshape\begin{enumerate}
		\item The interpretation of Theorem \ref{th_vAC_conv} is analogous to the one of Theorem \ref{th_AC_conv}, where convergence of \hyperlink{AC}{(AC)} to \eqref{MCF} with $90$°-contact angle is obtained, cf.~Remark \ref{th_AC_conv_rem},~1.
		\item \textit{Layout of the Proof.} 
		The new model problems, some vector-valued ODEs on $\R$ and a vector-valued linear elliptic equation on $\R\times(0,\infty)$ are solved in Section \ref{sec_ODE_vect} and Section \ref{sec_hp_vect} below, respectively. The asymptotic expansions are done in Section \ref{sec_asym_vAC} and the approximate solution $\vec{u}^A_\varepsilon$ is defined in Section \ref{sec_asym_vAC_uA}. Note that $M$ corresponds to the number of terms in the expansion. The spectral estimate is shown in Section \ref{sec_SE_vAC} and the difference estimate is proven in Section \ref{sec_DC_vAC_DE}. Finally, Theorem \ref{th_vAC_conv} is obtained in Section \ref{sec_DC_vAC_conv}. 
		\item \textit{Well-Posedness of \hyperlink{vAC}{(vAC)}.}
		In general the analysis of systems is more challenging than that of single equations. However, the derivatives in \eqref{eq_vAC1}-\eqref{eq_vAC3} are decoupled and hence some methods from the scalar case can also be used for (vAC), e.g.~regularity theory. Equation \eqref{eq_vAC1}-\eqref{eq_vAC3} matches the general setting of Lunardi \cite{LunardiOptReg}, Section 7.3.1, where a semigroup method and Hölder-spaces are used. Moreover, by reduction to a scalar equation and maximum principle arguments, one can obtain a priori boundedness of classical solutions, see Section \ref{sec_DC_prelim_bdd_vect} below. Hence global well-posedness for regular, bounded initial data follows. Higher regularity can be obtained with linear theory for scalar equations, cf.~Lunardi, Sinestrari, von Wahl \cite{LunardiSvW}.\phantom{\qedhere}
		\item The comments for Theorem \ref{th_AC_conv} in Remark \ref{th_AC_conv_rem}, 4.-7.~hold analogously. Here $\vec{u}_\pm$ has the role of $\pm 1$ and the order of the approximation of $\Gamma$ in the spirit of Remark \ref{th_AC_conv_rem}, 5.~is $\varepsilon^2$.
		\end{enumerate}
\end{Remark}

\section{Notation and Function Spaces}\label{sec_fct}
Let $\N$ be the natural numbers and $\N_0:=\N\cup\{0\}$. The symbol $\K$ stands for an element of $\{\R,\C\}$. Moreover, the Euclidean norm in $\R^m$, $m\in\N$ and the Frobenius norm in $\R^{m\times n}$, $m,n\in\N$ are for convenience denoted by $|.|$. The symbol \enquote{$\,\vec{\phantom{a}}\,$} indicates a vector or a vector-valued function. Moreover, objects (e.g.~vectors, operators and constants) that are associated to a vector-valued setting often get the addition \enquote{$\,\check{\phantom{a}}\,$}. Furthermore, a subset $\Omega\subseteq\R^n$, $n\in\N$ is called \enquote{domain}, if $\Omega$ is open, nonempty and connected. Additionally, restrictions or evaluations of functions are often indicated by \enquote{$|_.$}. The differential operators $\nabla$, $\diverg$ and $D^2$ are defined to act just on spatial variables. Let $X$ be a set and $Y$ a normed space. Then 
$B(X,Y):=\{ f:X\rightarrow Y\text{ bounded}\}$. Let $X,Y$ be normed spaces over $\K$. Then $\Lc(X,Y)$ denotes the set of bounded linear operators $T:X\rightarrow Y$. Finally, note that we use the usual constant convention.

\subsection{Unweighted Continuous/Continuously Differentiable Functions}\label{sec_fct_C}
\begin{Definition}\upshape\label{th_C_def}
Let $n,k\in\N$ and $\Omega\subseteq\R^n$ open and nonempty. Moreover, let $B$ be a Banach space over $\K=\R$ or $\C$. Then
\begin{enumerate}
\item $C^0(\Omega,B):=\{f:\Omega\rightarrow B\text{ continuous}\}$
and analogously we define $C^0(\overline{\Omega},B)$. Moreover,
\begin{align*}
C^k(\Omega,B)&:=\{f\in C^0(\Omega,B): f\text{ is }k\text{-times countinuously Fréchet-differentiable}\},\\
C^k(\overline{\Omega},B)&:=\{f\in C^k(\Omega,B): f\text{ and derivatives up to order }k\text{ have }C^0\text{-extensions to }\overline{\Omega}\}.
\end{align*}
\item The spaces including boundedness for the function and all appearing derivatives are denoted with $C_b^0(\Omega,B)$, $C_b^0(\overline{\Omega},B)$, $C_b^k(\Omega,B)$, $C_b^k(\overline{\Omega},B)$ and are equipped with the natural norms.  
\item Above spaces with \enquote{$\infty$} instead of \enquote{$k$} are defined via the intersection over all $k\in\N$. 
\item $C^{k,\gamma}(\overline{\Omega},B):=\{f\in C_b^k(\overline{\Omega},B):\|f\|_{C^{k,\gamma}(\overline{\Omega},B)}<\infty\}$ are the Hölder-spaces, where
\begin{align*}
\|f\|_{C^{0,\gamma}(\overline{\Omega},B)}&:=\|f\|_{C_b^0(\overline{\Omega},B)}+\sup_{x,y\in\overline{\Omega}, x\neq y}\frac{\|f(x)-f(y)\|_B}{|x-y|^\gamma},\\
\|f\|_{C^{k,\gamma}(\overline{\Omega},B)}&:=\|f\|_{C_b^k(\overline{\Omega},B)}+\sup_{\beta\in\N_0^n,|\beta|=k}\|\partial_x^\beta f\|_{C^{0,\gamma}(\overline{\Omega},B)}.
\end{align*}
\item Let $U$ be an open subset of the interior $M^\circ:=M\setminus\partial M$ of a smooth compact manifold $M$ with (or without) boundary, where $\partial M$ is defined via charts. Then $C^k(U,B)$, $C^k(\overline{U},B)$, $C_b^k(U,B)$ and $C_b^k(\overline{U},B)$ for every $k\in\N_0\cup\{\infty\}$ are defined via local coordinates and the respective spaces on domains.
\item If $B=\K$ and it is clear from the context if $\K=\R$ or $\C$, then we omit $B$ in the notation.
\item $C_0^\infty(\Omega)$ is the set of $f\in C^\infty(\Omega,\R)$ with compact support $\text{supp}\,f\subseteq\Omega$. Moreover, $C_0^\infty(\overline{\Omega})$ denotes $\{f|_{\overline{\Omega}}:f\in C_0^\infty(\R^n)\}$.\end{enumerate}
\end{Definition}

\begin{Lemma}\label{th_C_lemma}
Let $U$ be an open subset of $M^\circ$, where $M$ is a smooth compact manifold $M$ with (or without) boundary and dimension $l\in\N$. Let $x_j:U_j\rightarrow V_j\subseteq[0,\infty)\times\R^{l-1}$ for $j=1,...,L$ be charts of $M$ and $W_j$ open in $[0,\infty)\times\R^{l-1}$ with $\overline{W_j}\subset V_j$ compact for $j=1,...,L$ such that  $\bigcup_{j=1}^L x_j^{-1}(W_j)=M$. Then $C_b^k(U,B)$ for $k\in\N_0$ is a Banach space with
	\[
	\|f\|_{C_b^k(U,B)}:=\sum_{j=1}^L \|f\circ x_j^{-1}|_{x_j(U\cap U_j)\cap W_j}\|_{C_b^k(x_j(U\cap U_j)\cap W_j ,B)}\quad\text{ for all }f\in C_b^k(U,B).
	\]
Different choices of $x_j,W_j$ yield equivalent norms. The analogous assertion holds for $C_b^k(\overline{U},B)$.
\end{Lemma}
\begin{proof}
The Banach space property follows directly. Moreover, the assumptions ensure that for different choices of $x_j,W_j$ the relevant chart transformations have $C^k$-extensions to the closure of their domain. These induce bounded linear transformations of the associated $C_b^k$-spaces.
\end{proof}

\subsection{Unweighted Lebesgue- and Sobolev-Spaces}\label{sec_LebSob}
\subsubsection{Lebesgue-Spaces}\label{sec_Leb}
Let $(M,\mathcal{A},\mu)$ be a $\sigma$-finite, complete measure space and $B$ be a Banach space over $\K=\R$ or $\C$.
Then one can define the notions of ($\mu$- or strongly-) measurable and (Bochner-)integrable functions $f:M\rightarrow B$ and the \textit{Bochner(-Lebesgue)-Integral}, see Amann, Escher \cite{AmannEscherIII}, Chapter X for the definitions and properties. In particular the \textit{Lebesgue-spaces} $L^p(M,B)$ for $1\leq p\leq\infty$ are defined. If $B=\K$ and it is clear from the context if $\K=\R$ or $\C$, then we omit $B$ in the notation. We also use the Fubini Theorem for scalar-valued functions on $\sigma$-finite measure spaces, see Elstrodt \cite{Elstrodt},~Satz V.2.4. 

Later we need the notion of the support of a measurable function:

\begin{Remark}\upshape\label{th_Leb_supp}
	Let $\Omega\subseteq\R^n$, $n\in\N$ be open and $f:\Omega\rightarrow B$ measurable. Then the support of $f$
	is 
	\[
	\supp\,f:=\left[\bigcup_{U\subset\Omega\text{ open }:f=0\text{ a.e.~in }U}U\right]^c.
	\]
	With topological properties of $\R^n$ one can show that $\supp\,f$ is closed and $f=0$ a.e.~on $(\supp\,f)^c$. Moreover, $\supp(f+g)\subseteq\supp\,f\cup\supp\,g$ for all $f,g:\Omega\rightarrow B$ measurable. Finally, for continuous $f$ it holds $\supp\,f=\overline{\{x\in\Omega:f(x)\neq0\}}$.
\end{Remark}

Moreover, we need the following transformation theorem:
\begin{Theorem}[\textbf{Substitution Rule}]\label{th_Leb_trafo}
	Let $U,V\subseteq\R^n$ be open, nonempty and $\Phi:U\rightarrow V$ be a $C^1$-diffeomorphism. Moreover, let $B$ be a Banach space and $f:V\rightarrow B$. Then $f\in L^1(V,B)$ if and only if $(f\circ\Phi)|\det\,D\Phi|\in L^1(U,B)$. In this case it holds
	\[
	\int_V f\,dy=\int_U (f\circ\Phi)|\det\,D\Phi|\,dx.
	\]
\end{Theorem}

This is \cite{AmannEscherIII}, Theorem X.8.14.
Note that the corresponding assertion for measurable functions in general only holds if additionally $f$ is almost separable-valued, cf.~\cite{AmannEscherIII}, Theorem X.1.4.
 
\begin{Remark}\upshape\label{th_Leb_RiemRem}
Let $(M,g)$ be a $C^1$-Riemannian submanifold of $\R^n$ with (or without) boundary. We denote with $\partial M$ the boundary (defined via charts) and with $M^\circ:=M\setminus\partial M$ the interior.
\begin{enumerate}
	\item Let $\Lc_M$ be the \textit{Lebesgue $\sigma$-Algebra of $M$} and $\lambda_M$ the \textit{Riemann-Lebesgue Volume Measure of $M$}. See \cite{AmannEscherIII}, Chapter XI and Chapter XII.1 for the definitions and properties. In particular $(M,\Lc_M,\lambda_M)$ is a $\sigma$-finite complete measure space (and satisfies many more properties). Therefore the Bochner-Integral and the Lebesgue spaces are defined. 
	\item Let $g$ be the Euclidean Metric and let $M$ have dimension $m$. Then one can show that $\lambda_M$ coincides with the $m$-dimensional Hausdorff-measure $\Hc^m$ on $M$. Cf.~with Evans, Gariepy \cite{EvansGariepy}, Chapter 3.3.4. For the definitions and properties of Hausdorff-measures, in particular the connections to Lebesgue measure cf.~\cite{EvansGariepy}, Chapter 2. Therefore in the application later we write $\Hc^m$ instead of $\lambda_M$ for convenience.
\end{enumerate}
\end{Remark} 

Finally, we show a transformation theorem for Riemannian submanifolds. Note that later we will only need the Euclidean metric, but the proof for the general case is the same.

\begin{Theorem}[\textbf{Substitution Rule for Riemannian Submanifolds of $\R^n$}]\label{th_Leb_trafo_mfd}
	Let $(M,g)$ and $(N,h)$ be $C^1$-Riemannian submanifolds of $\R^n$ with (or without) boundary and dimension $m$. Moreover, let $U\subset M^\circ$, $V\subset N^\circ$ be open and $\Phi:U\rightarrow V$ be a $C^1$-diffeomorphism. Then
	\begin{enumerate}
		\item Define $|\det d\Phi|:U\rightarrow\R:p\mapsto |\det d_p\Phi|$, where the latter is defined as the modulus of the determinant of the representation matrix of $d_p\Phi$ with respect to arbitrary orthonormal bases of $T_pM$ and $T_{\Phi(p)}N$ for all $p\in U$. Then $|\det d\Phi|$ is well-defined and in $C^1(U,\R)$.
		\item Let $B$ be a Banach space and consider $f:N\rightarrow B$. Then $f\in L^1(V,B)$ if and only if $(f\circ\Phi)|\det\,d\Phi|\in L^1(U,B)$. In this case it holds
		\[
		\int_V f\,d\lambda_N=\int_U (f\circ\Phi)|\det\,d\Phi|\,d\lambda_M.
		\]
	\end{enumerate}
\end{Theorem}
\begin{proof}[Proof. Ad 1] 
	The definition is independent of the choice of the orthonormal bases since the representation matrix corresponding to the change of basis on each of the tangent spaces has determinant $\pm1$. Via local representations one can prove that $|\det d\Phi|\in C^1(U,\R)$.\qedhere$_{1.}$\end{proof}
	
	\begin{proof}[Ad 2]
	The assertion is compatible with restrictions on $U$ and $V$. Therefore we can assume w.l.o.g. $U=M$ and $V=N$. Moreover, it is enough to prove one direction. Let $f\in L^1(V,B)$ and let $(\psi,W)$ be a chart of $M$. Then $(\psi\circ\Phi^{-1},\Phi(W))$ is a chart of $N$. Let $(g_{ij})_{i,j=1}^m$, $(h_{ij})_{i,j=1}^m$ be the local representations of $g$ and $h$ corresponding to $(\psi,W)$ and $(\psi\circ\Phi^{-1},\Phi(W))$, respectively. Furthermore, we set $G:=\det[(g_{ij})_{i,j=1}^m]$ and $H:=\det[(h_{ij})_{i,j=1}^m]$. The latter are viewed as maps from $\psi(W)$ to $\R$. Then Amann, Escher \cite{AmannEscherIII}, Theorem XII.1.10 yields $(f\circ(\psi\circ\Phi^{-1})^{-1})\sqrt{H}\in L^1(\psi(W),B)$. Choosing orthonormal bases for the related tangent spaces, one can show with the chain rule that $\sqrt{H}=|\det d\Phi|\circ\psi^{-1}\sqrt{G}$. Therefore we obtain 
	\[
	[(f\circ\Phi) |\det d\Phi|]\circ\psi^{-1}\sqrt{G}\in L^1(\psi(W),B)
	\]
	and \cite{AmannEscherIII}, Theorem X.1.10 yields $(f\circ\Psi)|\det d\Phi|\in L^1(W,B)$ as well as
	\[	
	\int_{\Phi(W)} f\,d\lambda_N=\int_W(f\circ\Psi)|\det d\Phi|\,d\lambda_M. 
	\]
	In particular $(f\circ\Psi)|\det d\Phi|:M\rightarrow B$ is $\lambda_M$-measurable, cf.~e.g.~\cite{AmannEscherIII}, Proposition XII.1.8 and Theorem X.1.14. Finally, via $\Phi$ one can push forward a countable atlas for $M$ and a corresponding $C^1$-partition of unity. Hence analogous computations as above and \cite{AmannEscherIII}, Proposition XII.1.11 yield the claim.\qedhere$_{2.}$	
\end{proof}

\subsubsection{Sobolev-Spaces on Domains in $\R^n$}\label{sec_SobDom}
\begin{Definition}\upshape\phantomsection{\label{th_SobDom_def}}
\begin{enumerate} 
	\item Let $\Omega\subseteq\R^n$, $n\in\N$ be open and nonempty. Moreover, let $k\in\N_0$, $1\leq p\leq\infty$ and $B$ be a Banach space. Then $W^{k,p}(\Omega,B)$ are the usual \textit{Sobolev-spaces}, where $W^{0,p}(\Omega,B):=L^p(\Omega,B)$. We also write $H^k(\Omega,B)$ instead of $W^{k,2}(\Omega,B)$. If $B=\K$ and it is clear from the context if $\K=\R$ or $\C$, then we omit $B$ in the notation.
	\item Let $n\in\N$. Then $H^\beta(\R^n)$ for $\beta>0$ are the well-known $L^2$-\textit{Bessel-Potential spaces} and $W^{k+\mu,p}(\R^n)$ for $k\in\N_0$, $\mu\in(0,1)$ and $1\leq p<\infty$ the \textit{Sobolev-Slobodeckij spaces}.
\end{enumerate}
\end{Definition}

For the definitions and properties of scalar-valued function spaces, in particular embeddings, interpolation results and trace theorems see Adams, Fournier \cite{AdamsFournier}, Alt \cite{AltFA}, Leoni \cite{Leoni} and Triebel \cite{Triebel_Interpol_Theory}, \cite{Triebel_Fct_SpacesI}. Many properties can be generalized to vector-valued function spaces over domains, see e.g.~Kreuter \cite{Kreuter} and the references therein. In particular: 

\begin{Lemma}\label{th_SobDom_dense}
	Let $\Omega\subseteq\R^n$, $n\in\N$ be open and nonempty, $k\in\N_0$, $1\leq p<\infty$ and $B$ be a Banach space. Then $C^\infty(\Omega,B)\cap W^{k,p}(\Omega,B)$ is dense in $W^{k,p}(\Omega,B)$. 
\end{Lemma}
\begin{proof}
	This follows via convolution analogously to the scalar case, cf.~\cite{Kreuter}, Chapter 4.2.
\end{proof}

For transformations of Sobolev spaces we use
\begin{Theorem}\label{th_SobDom_trafo}
	Let $\Omega_1,\Omega_2\subseteq\R^n$ be open, nonempty and bounded. Moreover, let $l\in\N$, $l\geq 1$, $1\leq p<\infty$ and $B$ be a Banach space. Let $\Phi:\Omega_1\rightarrow\Omega_2$ be a $C^l$-diffeomorphism with $\Phi\in C^l(\overline{\Omega_1})^n$ and $\Phi^{-1}\in C^l(\overline{\Omega_2})^n$ such that
	\[
	|\det\,D(\Phi^{-1})|\leq R_1\quad\text{ and }\quad \|\Phi\|_{C_b^l(\overline{\Omega_1})^n}\leq R_2.
	\]  
	Then $T:W^{k,p}(\Omega_2,B)\rightarrow W^{k,p}(\Omega_1,B):f\mapsto f\circ\Phi$ is well-defined, continuous and linear for all $k\in\N_0$ with $0\leq k\leq l$ and the operator norm is bounded by some $C(R_1)>0$ if $k=0$ and bounded by $C(R_1,R_2,p,l)>0$ if $k$ is arbitrary.
\end{Theorem}
\begin{proof}
	For $B=\K$ this follows from the proof of Adams, Fournier \cite{AdamsFournier}, Theorem~3.41. One only needs density of $C^\infty(\Omega_2)\cap W^{k,p}(\Omega_2)$ in $W^{k,p}(\Omega_2)$, the substitution rule and the chain rule. Due to Lemma \ref{th_SobDom_dense} and Theorem \ref{th_Leb_trafo} analogous arguments apply for general $B$.
\end{proof}

For simplicity we only consider scalar-valued functions in the remainder of this section. In the following we need to know how Lebesgue and Sobolev spaces behave on product sets.

\begin{Lemma}\label{th_SobDom_prod_set} 
	Let $\Omega_1\subseteq\R^m, \Omega_2\subseteq \R^n$ for $m,n\in\N$ be measurable. Then
	\begin{enumerate}
		\item  Let $1\leq p<\infty$ and $f\in L^p(\Omega_1\times \Omega_2)$. Then $f(x_1,.)\in L^p(\Omega_2)$ for a.e.~$x_1\in \Omega_1$ and  $Tf:\Omega_1\rightarrow L^p(\Omega_2):x_1\mapsto f(x_1,.)$ is an element of $L^p(\Omega_1, L^p(\Omega_2))$. Moreover, the map $T:L^p(\Omega_1\times \Omega_2)\rightarrow L^p(\Omega_1,L^p(\Omega_2))$ is an isometric isomorphism.
		\item Let $\Omega_1,\Omega_2$ be open and $1<p<\infty$. Then by restriction of $T$ from 1.~it holds
		\[
		W^{1,p}(\Omega_1\times\Omega_2)\cong L^p(\Omega_1,W^{1,p}(\Omega_2))\cap W^{1,p}(\Omega_1,L^p(\Omega_2))
		\] 
		and derivatives are compatible via $T$. Analogous assertions hold for higher orders. 
		\item Let $\Omega_1,\Omega_2$ open, $k\in\N_0$ and $1\leq p\leq\infty$. Then for $f\in W^{k,p}(\Omega_1)$ and $g\in W^{k,p}(\Omega_2)$ the product $(f\otimes g)(x_1,x_2):=f(x_1)g(x_2)$ 
		is well-defined for a.e.~$(x_1,x_2)\in\Omega_1\times\Omega_2$. Moreover, it holds $f\otimes g\in W^{k,p}(\Omega_1\times\Omega_2)$, the derivatives are natural and 
		\[
		\|f\otimes g\|_{W^{k,p}(\Omega_1\times\Omega_2)}\leq C_{k,p}\|f\|_{W^{k,p}(\Omega_1)}\|g\|_{W^{k,p}(\Omega_2)}.
		\] 
		For $\Omega_1=\R$, $\Omega_2=\R_+$, $k=1$ and $1\leq p<\infty$ the trace is given by $\tr_{\partial\R^2_+}(f\otimes g)=g(0)f$.
	\end{enumerate}
\end{Lemma}

Note that we restricted $p$ in the second assertion due to a duality argument.
\begin{proof}
	The first two assertions follow with ideas from R\r{u}\v{z}i\v{c}ka \cite{Ruzicka}, Chapter 2.1.1 and the Paragraph \enquote{Zusammenhang mit elementaren Definitionen} in Schweizer \cite{Schweizer}, p.188f. The first part of the third claim can be proven directly with the definitions. The trace assertion follows with a density argument.
\end{proof}

Moreover, we need the notion of domains with Lipschitz-boundary. See Alt \cite{AltFA}, Section A8.2 for the precise definition of a bounded Lipschitz-domain. We generalize this definition for parts of the boundary as a preparation for the next section.

\begin{Definition}[\textbf{Lipschitz Condition}]\label{th_SobDom_LipDef}\upshape
	Let $\Omega\subset\R^n$, $n\in\N$ be open and nonempty. Then
	\begin{enumerate}
	\item Let $x\in\partial\Omega$. We say that $\Omega$ satisfies the \textit{local Lipschitz condition} in $x$ if $\partial\Omega$ is locally at $x$ the graph of a Lipschitz function in a suitable orthogonal coordinate system such that $\Omega$ lies on one side of the graph. Cf.~Alt \cite{AltFA}, Section A8.2 for more details.
	\item We say $\Omega$ has \textit{Lipschitz-boundary} if the local Lipschitz condition holds in $x$ for all $x\in\partial\Omega$.
	\item We call $\Omega$ a \textit{Lipschitz-domain} if $\Omega$ is a domain and has Lipschitz-boundary.
	\end{enumerate}
\end{Definition}

\begin{Remark}[\textbf{Integral on the Boundary of Bounded Lipschitz Domains in $\R^n$}]\label{th_SobDom_LipRem}\upshape
	Let $\Omega\subset\R^n$, $n\in\N$ be open and nonempty. Moreover, let $\Sigma$ be open in $\partial\Omega$ (for example $\Sigma=\partial\Omega$) and assume that $\Omega$ satisfies the local Lipschitz condition in every point in the compact set $\overline{\Sigma}$. For simplicity we only consider scalar-valued functions.
	\begin{enumerate}
	\item Due to the Rademacher-Theorem, cf.~\cite{EvansGariepy}, Chapter 3.1, one can define the notions of measurable and integrable functions $f:\Sigma\rightarrow\R$ and the integral over $\Sigma$ in a natural way, cf.~e.g.~Alt \cite{AltFA}, Section A8.5. Then for $1\leq p\leq\infty$ we denote the usual Lebesgue spaces by $L^p(\Sigma)$. The definitions are the same as the ones via $\Hc^{n-1}$ on $\Sigma$, cf.~\cite{EvansGariepy}, Chapter 3.3.4 and also the proof of \cite{AltFA}, A8.5(2).
	\item The outer unit normal $N_{\partial\Omega}$ to $\Omega$ can be defined a.e.~on $\Sigma$, cf.~\cite{AltFA}, A8.5(3).
	\item Let additionally $\Sigma=M\cup Z$ with a $C^1$-hypersurface $M$ of $\R^n$ and a null-set $Z$ with respect to $\Hc^{n-1}$ on $\R^n$. Then the notions in 1.~are equivalent to the ones coming from the measure space $(M,\Lc_M,\lambda_M)$ introduced in Remark \ref{th_Leb_RiemRem}, 1.~One can prove this by going into the constructions or via identification with $\Hc^{n-1}$ on $\Sigma$, cf.~1.~and Remark \ref{th_Leb_RiemRem}, 2.
	\end{enumerate}
\end{Remark}

We need some properties of Sobolev spaces on domains in $\R^n$, where parts of the boundary satisfy the Lipschitz condition:

\begin{Theorem}\label{th_SobDom_LipThm}
	Let $\Omega\subset\R^n$, $n\in\N$ be open, bounded and let $\Sigma$ be open in $\partial\Omega$ (e.g.~$\Sigma=\partial\Omega$) such that $\Omega$ satisfies the local Lipschitz condition in every point in $\overline{\Sigma}$. Let $1\leq p<\infty$. Then
	\begin{enumerate}
		\item $\{\psi\in C_0^\infty(\overline{\Omega}):\supp\,\psi\subset\Omega\cup\Sigma\}$ is dense in $\{f\in W^{k,p}(\Omega):\supp\, f\subset\Omega\cup\Sigma\}$ for all $k\in\N_0$, where $\supp\,f$ for measurable $f:\Omega\rightarrow\R$ is defined in Remark \ref{th_Leb_supp}.
		\item $C^0(\Omega\cup\overline{\Sigma})\cap W^{1,p}(\Omega)$ is dense in $W^{1,p}(\Omega)$.
		\item There is a unique bounded linear operator $\tr:W^{1,p}(\Omega)\rightarrow L^p(\Sigma)$ such that $\tr\,u=u|_{\Sigma}$ for all $u\in C^0(\Omega\cup\overline{\Sigma})\cap W^{1,p}(\Omega)$.
		\item Let $\Sigma=\partial\Omega$, i.e.~$\Omega$ has Lipschitz-boundary. Then the Gauß Theorem holds for $W^{1,1}$-functions in weak form.
	\end{enumerate}
\end{Theorem}
\begin{proof}[Proof. Ad 1]
	This can be shown via localization with a suitable partition of unity and convolution similar to the proof of Alt \cite{AltFA}, Lemma A8.7.\hfill$\square_{1.}$\\
	\textit{Ad 2.} Note that due to compactness of $\overline{\Sigma}$, there is another set $\tilde{\Sigma}$ open in $\partial\Omega$ such that $\overline{\Sigma}\subset\tilde{\Sigma}$ and $\tilde{\Sigma}$ satisfies analogous properties as $\Sigma$. Therefore one can combine 1.~and Lemma \ref{th_SobDom_dense} with a partition of unity to show density of $C^0(\Omega\cup\overline{\Sigma})\cap W^{1,p}(\Omega)$ in $W^{1,p}(\Omega)$.
	\hfill$\square_{2.}$\\
    \textit{Ad 3.} The proof of \cite{AltFA}, Theorem A8.6 can be directly adapted.\hfill$\square_{3.}$\\
    \textit{Ad 4.} See Alt \cite{AltFA}, Theorem A8.8.\qedhere$_{4.}$
\end{proof}

\subsubsection{Sobolev Spaces on Domains in Compact Submanifolds of $\R^n$}\label{sec_SobMfd}
In this section let $(M,g)$ be a $m$-dimensional compact Riemannian submanifold of $\R^n$ with (or without) boundary and class $C^l$, where $l\in\N\cup\{\infty\}$, $l\geq 1$. Let $U\subset M^\circ$ be open and nonempty. Moreover, let $B$ be a Banach space. Then $L^p(U,B)$ is defined due to Remark \ref{th_Leb_RiemRem}, 1.

\begin{Definition}\upshape\label{th_SobMfd_def}
Let $x_j:U_j\rightarrow V_j\subseteq[0,\infty)\times\R^{m-1}$ for $j=1,...,N$ be charts of $M$ and $W_j$ open in $[0,\infty)\times\R^{m-1}$ with $\overline{W_j}\subset V_j$ compact for $j=1,...,N$ and $\bigcup_{j=1}^N x_j^{-1}(W_j)=M$.
\begin{enumerate} 
\item Then for $k\in\N$, $k\leq l$ and $1\leq p<\infty$ we define the \textit{Sobolev spaces}
\[
W^{k,p}(U,B):=\{f\in L^p(U,B): f\circ x_j^{-1}|_{x_j(U\cap U_j)\cap W_j}\in W^{k,p}(x_j(U\cap U_j)\cap W_j,B)\,\forall j\}.
\]
\item For $f\in W^{1,p}(U,B)$, $1\leq p<\infty$ we define (in analogy to the scalar case) the \textit{surface gradient} 
\[
[\nabla_Uf]\circ x_j^{-1}|_{x_j(U\cap U_j)\cap W_j}
:=\sum_{r,s=1}^m g^{rs}\circ x_j^{-1}\partial_{y_r}[f\circ x_j^{-1}|_{x_j(U\cap U_j)\cap W_j}] \partial_{y_s}(x_j^{-1})
\]
for all $j=1,...,N$, where $(g^{rs})_{r,s=1}^m$ is the inverse of the representation matrix of $g$ with respect to $x_j$ and the product of $\partial_{y_r}[f\circ x_j^{-1}|_{x_j(U\cap U_j)\cap W_j}]\in L^p(x_j(U\cap U_j)\cap W_j,B)$ with $\partial_{y_s}(x_j^{-1})\in C^l(\overline{x_j(U\cap U_j)\cap W_j},\R^n)$ is understood component-wise.\end{enumerate}
\end{Definition}

\begin{Lemma}\label{th_SobMfd_def_lemma}
	Consider the situation of Definition \ref{th_SobMfd_def}. Let $k\in\N$, $k\leq l$ and $1\leq p<\infty$. Then
	\begin{enumerate}
	\item $W^{k,p}(U,B)$ is a Banach space with norm
	\[
	\|f\|_{W^{k,p}(U,B)}^\ast:=\sum_{j=1}^N \|f\circ x_j^{-1}|_{x_j(U\cap U_j)\cap W_j}\|_{W^{k,p}(x_j(U\cap U_j)\cap W_j,B)}\quad\text{ for all }f\in W^{k,p}(U,B).
	\]
	Different choices of $(x_j, W_j)$ yield the same spaces with equivalent norms. 
	\item $C^l(U,B)\cap W^{k,p}(U,B)$ is dense in $W^{k,p}(U,B)$.
	\item $\nabla_Uf$ is well-defined for all $f\in W^{1,p}(U,B)$ and independent of the choices of $(x_j,W_j)$. Moreover, $\nabla_Uf\in L^p(U,B^n)$ and
	\[
	\|f\|_{W^{1,p}(U,B)}:=\|f\|_{L^p(U,B)}+\|\nabla_Uf\|_{L^p(U,B^n)}
	\] 
	defines an equivalent norm on $W^{1,p}(U,B)$.
	\end{enumerate}
\end{Lemma}

Later on $W^{1,p}(U,B)$ we always take the norm in Lemma \ref{th_SobMfd_def_lemma}, 3. Note that for higher orders one can also define coordinate independent norms, cf.~with Hebey \cite{Hebey}, Chapter 2 in the scalar case. Nevertheless, later we only need $W^{1,p}(U,B)$.

	\begin{proof}[Proof. Ad 1]
	First, one can directly prove that $W^{k,p}(U,B)$ is a normed space with $\|.\|_{W^{k,p}(U,B)}^\ast$. Moreover, let $(f_i)_{i\in\N}$ be a Cauchy sequence in $W^{k,p}(U,B)$. Then because $L^p(U,B)$ and $W^{k,p}(x_j(U\cap U_j)\cap W_j,B)$ for $j=1,...,N$ are Banach spaces, there are $f\in L^p(U,B)$ and $h_j\in W^{k,p}(x_j(U\cap U_j)\cap W_j,B)$ such that for all $j=1,...,N$
	\[
	f_i\overset{i\rightarrow\infty}{\longrightarrow}f\text{ in }L^p(U,B)\quad\text{ and }\quad 
	f_i\circ x_j^{-1}|_{x_j(U\cap U_j)\cap W_j}\overset{i\rightarrow\infty}{\longrightarrow}h_j\text{ in }W^{k,p}(x_j(U\cap U_j)\cap W_j,B).
	\]
	Therefore $f\circ x_j^{-1}|_{x_j(U\cap U_j)\cap W_j}=h_j$ for all $j=1,...,N$ and $f\in W^{k,p}(U,B)$ with $f_i\overset{i\rightarrow\infty}{\longrightarrow}f$ in $W^{k,p}(U,B)$. Hence $W^{k,p}(U,B)$ is a Banach space.
		
	Now let $(\tilde{x}_j,\tilde{W}_j)$ for $j=1,...,\tilde{N}$ be another combination of coordinates and sets as in Definition \ref{th_SobMfd_def}. We denote with $\tilde{W}^{k,p}(U,B)$ and $\|.\|_{\tilde{W}^{k,p}(U,B)}^\ast$ the corresponding space and norm. We have to show $W^{k,p}(U,B)=\tilde{W}^{k,p}(U,B)$ and that the norms are equivalent. It is enough to prove one direction. Let $f\in W^{k,p}(U,B)$ and fix $i\in\{1,...,\tilde{N}\}$. It holds
	\[
	\tilde{x}_i(U\cap\tilde{U}_i)\cap\tilde{W}_i=\bigcup_{j=1}^N \tilde{x}_i(U\cap \tilde{U}_i\cap x_j^{-1}(W_j))\cap\tilde{W}_i.
	\]
	To obtain a suitable partition of unity for this note that $K_i:=\overline{\tilde{x}_i(U\cap\tilde{U}_i)\cap\tilde{W}_i}$ is compact and 
	$Y_{ij}:=\tilde{x}_i(x_j^{-1}(W_j)\cap \overline{U}\cap U_i)\cap\overline{\tilde{W}_i}$ is open in $K_i$ with $K_i\subseteq\bigcup_{j=1}^N Y_{ij}\subset\bigcup_{j=1}^N Y_{ij}\cup K_i^c$. Hence there are $\eta_{ij}\in C_0^\infty(\R^m)$, $j=1,...,N$ such that $0\leq \eta_{ij}\leq 1$, $\supp\,\eta_{ij}\subset Y_{ij}\cup K_i^c$ and $\sum_{j=1}^N\eta_{ij}\equiv 1$ on $K_i$. Therefore
	\[
	f\circ\tilde{x}_i^{-1}|_{\tilde{x}_i(U\cap\tilde{U}_i)\cap\tilde{W}_j}
	=\sum_{j=1}^N \eta_{ij}[f\circ x_j^{-1}|_{x_j(U\cap U_j\cap\tilde{x}_i^{-1}(\tilde{W}_i))\cap W_j}]\circ(x_j\circ\tilde{x}_i^{-1})|_{\tilde{x}_i(U\cap \tilde{U}_i\cap x_j^{-1}(W_j))\cap\tilde{W}_i}.
	\]
	Finally, due to Theorem \ref{th_SobDom_trafo} and since multiplication with smooth functions induce bounded linear operators on Sobolev spaces, we obtain $f\circ\tilde{x}_i^{-1}|_{\tilde{x}_i(U\cap\tilde{U}_i)\cap\tilde{W}_j}\in W^{k,p}(\tilde{x}_i(U\cap\tilde{U}_i)\cap\tilde{W}_j)$ and
	\[
	\|f\circ\tilde{x}_i^{-1}|_{\tilde{x}_i(U\cap\tilde{U}_i)\cap\tilde{W}_j}\|_{W^{k,p}(\tilde{x}_i(U\cap\tilde{U}_i)\cap\tilde{W}_j)}\leq C\|f\|_{W^{k,p}(U,B)}^\ast,
	\]
	where $C$ does not depend on $f$. Since $i\in\{1,...,\tilde{N}\}$ was arbitrary, it follows that $f\in\tilde{W}^{k,p}(U,B)$ with $\|f\|_{\tilde{W}^{k,p}(U,B)}^\ast\leq \overline{C}\|f\|_{W^{k,p}(U,B)}^\ast$ with $\tilde{C}$ independent of $f$. This yields the claim.
	\qedhere$_{1.}$
	\end{proof}

	\begin{proof}[Ad 2]
	Due to 1.~and Lemma \ref{th_SobDom_dense} there are 
	\[
	(f_i^j)_{i\in\N}\subset C^\infty(x_j(U\cap U_j)\cap W_j,B)\cap W^{k,p}(x_j(U\cap U_j)\cap W_j,B)
	\] 
	such that $f_i^j\overset{i\rightarrow\infty}{\longrightarrow}f\circ x_j^{-1}|_{x_j(U\cap U_j)\cap W_j}$ in $W^{k,p}(x_j(U\cap U_j)\cap W_j,B)$. In order to get a suitable partition of unity note that $\overline{U}=\bigcup_{j=1}^N\overline{U}\cap x_j^{-1}(W_j)$ and $\overline{U}\cap x_j^{-1}(W_j)$ is open in $\overline{U}$. Hence there are $\chi_j\in C^l(M)$, $j=1,...,N$ such that $0\leq \chi_j\leq 1$, $\supp\,\chi_j\subset(\overline{U}\cap x_j^{-1}(W_j))\cup\overline{U}^c$ and $\sum_{j=1}^N\chi_j\equiv 1$ on $M$. By construction and Theorem \ref{th_SobDom_trafo} it follows that
	\[
	f_i:=\sum_{j=1}^N \chi_j (f_i^j\circ x_j)\in C^l(U,B)\cap W^{k,p}(U,B)
	\]
	for all $i\in\N$ and $f_i\overset{i\rightarrow\infty}{\longrightarrow} f$ in $W^{k,p}(U,B)$.\qedhere$_{2.}$
	\end{proof}

	\begin{proof}[Ad 3]
	First let $f\in C^1(U,B)\cap W^{1,p}(U,B)$. Then the Hahn-Banach Theorem and the scalar case yield that $\nabla_Uf$ is well-defined and independent of the choice of $x_j,W_j$. Therefore it holds $\nabla_Uf\in C^0(U,B^n)$. Moreover, $\|[\nabla_Uf]\circ x_j^{-1}|_{x_j(U\cap U_j)\cap W_j}\|_{B^n}$ and $\|\nabla[f\circ x_j^{-1}|_{x_j(U\cap U_j)\cap W_j}]\|_{B^n}$ satisfy uniform equivalence estimates on $x_j(U\cap U_j)\cap W_j$ for all $j=1,...,N$ with constants independent of $f$ due to compactness. Hence the claim follows via density from 2.\qedhere$_{3.}$
	\end{proof}

Moreover, we need a transformation theorem.

\begin{Theorem}\label{th_SobMfd_trafo}
	Let $(M,g)$ and $(\tilde{M},\tilde{g})$ be $m$-dimensional compact Riemannian submanifolds of $\R^n$ with (or without) boundary and class $C^l$, where $l\in\N\cup\{\infty\}$, $l\geq 1$. Moreover, let $k\in\N_0$, $0\leq k\leq l$ and $1\leq p<\infty$ as well as $B$ be a Banach space. Let $U\subset M^\circ$, $V\subset \tilde{M}^\circ$ be open and $\Phi:U\rightarrow V$ be a $C^l$-diffeomorphism such that $\Phi\in C^l(\overline{U})^m$ and $\Phi^{-1}\in C^l(\overline{V})^m$. Then $T:W^{k,p}(V,B)\rightarrow W^{k,p}(U,B):f\mapsto f\circ\Phi$ is a well-defined bounded linear operator.
\end{Theorem}

Note that for convenience we did not attempt to obtain a uniform estimate for the operator norm. In order to get such estimates for $k=1$ in the application later, we use Theorem \ref{th_Leb_trafo_mfd} and uniform equivalence estimates for the surface gradient.

\begin{proof}
	The case $k=0$ directly follows from Theorem \ref{th_Leb_trafo_mfd}. Now let $(x_i,W_i)$ for $i=1,...,N$ and $(\tilde{x}_j,\tilde{W}_j)$ for $j=1,...,\tilde{N}$ be as in Definition \ref{th_SobMfd_def} for $M$ and $\tilde{M}$ respectively. W.l.o.g.~we can assume  $\Phi(U\cap x_i^{-1}(W_i))\subseteq\tilde{x}_j^{-1}(\tilde{W}_j)$ for some $j=j(i)\in\{1,...,\tilde{N}\}$ and all $i=1,...,N$. Otherwise one can simply refine the $W_i$. Then Theorem \ref{th_SobDom_trafo} yields the claim.
\end{proof}

From now on let $B=\K$ for convenience. We need a product lemma analogous to Lemma \ref{th_SobDom_prod_set}, 1.-2.~provided that one of the sets equals some $U$ as in the beginning of the section. Note that the product of $U$ with some open bounded set $\Omega\subseteq\R^q$ is again of the same type. Therefore the definitions and assertions in this section can also be applied for $\Omega\times U$ and $U\times \Omega$ instead of $U$.

\begin{Lemma}\label{th_SobMfd_prod_set}
	Let $(M,g)$ be an $m$-dimensional compact Riemannian submanifold of $\R^n$ of class $C^1$ and $U\subset M^\circ$ open. Moreover, let $\Omega\subset\R^q$, $q\in\N$ be open and bounded. Then
	\begin{enumerate}
	\item Let $1\leq p<\infty$ and $f\in L^p(U\times\Omega)$. Then $f(u,.)\in L^p(\Omega)$ for $\lambda_U$-a.e.~$u\in U$ and $Tf:U\rightarrow L^p(\Omega):u\mapsto f(u,.)$ is an element of $L^p(U,L^p(\Omega))$. Moreover, the map $T:L^p(U\times\Omega)\rightarrow L^p(U,L^p(\Omega))$ is an isometric isomorphism. 
	\item Let $1<p<\infty$. Then by restriction of $T$ from 1.~it holds
	\[
	W^{1,p}(U\times\Omega)\cong L^p(U,W^{1,p}(\Omega))\cap W^{1,p}(U,L^p(\Omega))
	\]
	and the derivatives $\nabla_U$ as well as $\nabla_\Omega=\nabla$ are compatible in both spaces via $T$. Here $\nabla_{U\times\Omega}=(\nabla_U,\nabla_\Omega)$ canonically on $W^{1,p}(U\times\Omega)$.
	\item Both assertions 1.~and 2.~also hold when we exchange $U$ and $\Omega$.
	\end{enumerate}
\end{Lemma}

\begin{proof}
	For the proof let $x_j:U_j\rightarrow V_j$ and $W_j$ for $j=1,...,N$ be as in Definition \ref{th_SobMfd_def} for $M$. Moreover, we need a partition of unity as in the proof of Lemma \ref{th_SobMfd_def_lemma},~3., i.e.~let $\chi_j\in C^1(M)$, $j=1,...,N$ such that $0\leq \chi_j\leq 1$, $\supp\,\chi_j\subset(\overline{U}\cap x_j^{-1}(W_j))\cup\overline{U}^c$ and $\sum_{j=1}^N\chi_j\equiv 1$ on $M$. Furthermore, in the following we often denote restrictions to some set by \enquote{$|$} without the set if there is no ambiguity. Finally, note that we often use the notation $u,y,z$ for points in $U,V_j,\Omega$, respectively. This convention also clarifies how some derivatives are understood. \phantom{\qedhere}
	
	\begin{proof}[Ad 1]
	Let $f\in L^p(U\times\Omega)$. Then $f\circ(x_j^{-1},\textup{id})|\in L^p([x_j(U\cap U_j)\cap W_j]\times\Omega)$ for all $j=1,...,N$ due to Theorem \ref{th_Leb_trafo_mfd}. Lemma \ref{th_SobDom_prod_set} yields $f(x_j^{-1}(y),.)\in L^p(\Omega)$ for a.e.~$y\in x_j(U\cap U_j)\cap W_j$ and the mapping $x_j(U\cap U_j)\cap W_j\rightarrow L^p(\Omega):y\mapsto f(x_j^{-1}(y),.)$ is strongly measurable and in $L^p$ for all $j=1,...,N$. Therefore $f(u,.)\in L^p(\Omega)$ for $\lambda_U$-a.e.~$u\in U$ and with the well-known characterization for measurability, see Amann, Escher \cite{AmannEscherIII}, Theorem X.1.4 we obtain that $Tf:U\rightarrow L^p(\Omega):u\mapsto f(u,.)$ is strongly measurable. Moreover, the Fubini Theorem implies that $u\mapsto\|f(u,.)\|_{L^p(\Omega)}^p$ is an element of $L^1(U)$ and
	\[
	\int_{U\times\Omega}|f|^p\,d\lambda_{U\times\Omega}(u,z)=
	\int_U\|f(u,.)\|_{L^p(\Omega)}^p\,d\lambda_U(u).
	\]   
	Hence $Tf$ is Bochner-integrable due to the Bochner Theorem, see \cite{Ruzicka}, Satz 1.12. Therefore $Tf$ is contained in $L^p(U,L^p(\Omega))$ with norm equal to $\|f\|_{L^p(U\times\Omega)}$. In particular the mapping $T:L^p(U\times\Omega)\rightarrow L^p(U,L^p(\Omega))$ is well-defined and isometric.
	
	It remains to prove that $T$ is surjective. To this end consider $h\in L^p(U,L^p(\Omega))$. Then Theorem \ref{th_Leb_trafo_mfd} yields $h\circ x_j^{-1}|\in L^p([x_j(U\cap U_j)\cap W_j],L^p(\Omega))$ for all $j=1,...,N$. Due to Lemma \ref{th_SobDom_prod_set} there exist $h_j\in L^p([x_j(U\cap U_j)\cap W_j]\times\Omega)$ such that $[y\mapsto h_j(y,.)]=h\circ x_j^{-1}|_{x_j(U\cap U_j)\cap W_j}$ for all $j=1,...,N$. Then $h_j\circ(x_j,\textup{id})\in L^p([U\cap x_j^{-1}(W_j)]\times\Omega)$ due to Theorem \ref{th_Leb_trafo_mfd} and
	\begin{align}\label{eq_SobMfd_prod_set_1}
    f_h:=\sum_{j=1}^N\chi_j [h_j\circ(x_j,\textup{id})]\in L^p(U\times\Omega).
	\end{align}
	By construction it holds $f_h(u,.)=h(u)$ for $\lambda_U$-a.e.~$u\in U$, i.e.~$Tf_h=h$. \qedhere$_{1.}$
	\end{proof}

	\begin{proof}[Ad 2]
	Let $1<p<\infty$ and $f\in W^{1,p}(U\times\Omega)$. We build up on the proof of 1. By definition $f\circ(x_j^{-1},\textup{id})|\in W^{1,p}([x_j(U\cap U_j)\cap W_j]\times\Omega)$ for all $j=1,...,N$. Hence Lemma \ref{th_SobDom_prod_set} yields
	\[
	[y\mapsto f(x_j^{-1}(y),.)]\in W^{1,p}([x_j(U\cap U_j)\cap W_j],L^p(\Omega))\cap L^p([x_j(U\cap U_j)\cap W_j],W^{1,p}(\Omega)),
	\]
	$\partial_{y_i}[f(x_j^{-1},\textup{id})|](y,.)=
	\partial_{y_i}[y\mapsto f(x_j^{-1}(y),.)]|_y$ and $\partial_{z_k}[f(x_j^{-1},\textup{id})|](y,.)=\partial_{z_k}[f(x_j^{-1}(y),.)]$ for all $i=1,...,m$, $k=1,...,q$ and a.e.~$y\in x_j(U\cap U_j)\cap W_j$, $j=1,...,N$. With 1., Definition \ref{th_SobMfd_def} and Lemma \ref{th_SobMfd_def_lemma} we obtain $Tf\in W^{1,p}(U,L^p(\Omega))\cap L^p(U,W^{1,p}(\Omega))$ and since 
	\[
	[\nabla_{U\times\Omega}f]|_{(x_j^{-1},\textup{id})}
	=[\nabla_Uf,\nabla_\Omega f]|_{(x_j^{-1},\textup{id})}=\left(\sum_{r,s=1}^m g^{rs}\partial_{y_r}(f|_{(x_j^{-1},\textup{id})})\partial_{y_s}(x_j^{-1}),\nabla_z(f|_{(x_j^{-1},\textup{id})})
	\right)
	\]
	it follows that
	$[\nabla_U f](u,.)=\nabla_U[Tf]|_u$ and $[\nabla_\Omega f](u,.)=\nabla_z[Tf]|_u$ for $\lambda_U$-a.e.~$u\in U$. Therefore the derivatives are compatible under $T$ and Lemma \ref{th_SobMfd_def_lemma},~3.~yields the norm equivalence.
	
	It is left to show that $T$ on $W^{1,p}(U\times\Omega)$ with values in $W^{1,p}(U,L^p(\Omega))\cap L^p(U,W^{1,p}(\Omega))$ is surjective. Therefore let $h$ be in the target space and $h_j$ for $j=1,...,N$ be as in 1.~for $h$. Then due to Theorem \ref{th_SobMfd_trafo} and Lemma \ref{th_SobDom_prod_set} it holds $h_j\in W^{1,p}([x_j(U\cap U_j)\cap W_j]\times\Omega)$. Therefore Lemma \ref{th_SobMfd_def_lemma} and Theorem \ref{th_SobMfd_trafo} yield that $f_h$ defined in 	
	\eqref{eq_SobMfd_prod_set_1} is an element of $W^{1,p}(U\times\Omega)$.\qedhere$_{2.}$
	\end{proof}

	\begin{proof}[Ad 3]
	Now we exchange the order of $U$ and $\Omega$. The proof is divided into four parts in accordance with the proofs of 1.-2. For the proof we denote the corresponding map in 1.~with $\tilde{T}$.
	
	Let $1\leq p<\infty$ and $\tilde{f}\in L^p(\Omega\times U)$. Then $\tilde{f}\circ(\textup{id},x_j^{-1})|\in L^p(\Omega\times[x_j(U\cap U_j)\cap W_j])$ for all $j=1,...,N$ due to Theorem \ref{th_Leb_trafo_mfd}. Lemma \ref{th_SobDom_prod_set} yields $\tilde{f}(z,x_j^{-1}|)\in L^p(x_j(U\cap U_j)\cap W_j)$ for a.e.~$z\in\Omega$ and the map $\Omega\rightarrow L^p(x_j(U\cap U_j)\cap W_j):z\mapsto \tilde{f}(z,x_j^{-1}|)$ is strongly measurable and contained in $L^p$ for all $j=1,...,N$. Therefore because of Theorem \ref{th_Leb_trafo_mfd}
	\[
	\tilde{T}\tilde{f}:\Omega\rightarrow L^p(U):z\mapsto\tilde{f}(z,.)=\sum_{j=1}^N\chi_j[\tilde{f}(z,x_j^{-1})]\circ x_j|_{U\cap x_j^{-1}(W_j)}
	\]
	is strongly measurable. Now analogously to the proof of 1.~we obtain with the Fubini Theorem and the Bochner Theorem that $\tilde{T}\tilde{f}$ is Bochner-integrable and  $\tilde{T}\tilde{f}\in L^p(\Omega,L^p(U))$ with norm equal $\|\tilde{f}\|_{L^p(\Omega\times U)}$. In particular $\tilde{T}:L^p(\Omega\times U)\rightarrow L^p(\Omega,L^p(U))$ is well-defined and isometric.
	
	Next we prove that $\tilde{T}$ is surjective. Therefore let $\tilde{h}\in L^p(\Omega,L^p(U))$. First note that due to Theorem \ref{th_Leb_trafo_mfd} the map $L^p(U)\rightarrow L^p(x_j(U\cap U_j)\cap W_j):\phi\mapsto \phi\circ x_j^{-1}|$ is bounded and linear for all $j=1,...,N$. Hence because of Lemma \ref{th_SobDom_prod_set} there are $\tilde{h}_j\in L^p(\Omega\times[x_j(U\cap U_j)\cap W_j])$ such that $\tilde{h}_j(z,.)=\tilde{h}(z)\circ x_j^{-1}|$ for a.e.~$z\in\Omega$ and $j=1,...,N$. Therefore Theorem \ref{th_Leb_trafo_mfd} yields
	\begin{align}\label{eq_SobMfd_prod_set_2}
	\tilde{f}_{\tilde{h}}:=\sum_{j=1}^N\chi_j[\tilde{h}_j\circ(\textup{id},x_j)]\in L^p(\Omega\times U)
	\end{align}
	and by construction $\tilde{f}_{\tilde{h}}(z,.)=\tilde{h}(z)$ for a.e.~$z\in\Omega$, i.e.~$\tilde{T}\tilde{f}_{\tilde{h}}=\tilde{h}$. Hence $\tilde{T}$ is an isomorphism.
	
	Now let $1<p<\infty$ and $\tilde{f}\in W^{1,p}(\Omega\times U)$. Then $\tilde{f}\circ(\textup{id},x_j^{-1})|\in W^{1,p}(\Omega\times[x_j(U\cap U_j)]\cap W_j)$ for $j=1,...,N$ by definition. Therefore Lemma \ref{th_SobDom_prod_set} yields
	\[
	[z\mapsto\tilde{f}(z,x_j^{-1}|)]\in 
	W^{1,p}(\Omega, L^p([x_j(U\cap U_j)\cap W_j]))\cap L^p(\Omega,W^{1,p}([x_j(U\cap U_j)\cap W_j])),
	\]
	$\partial_{y_i}[\tilde{f}(\textup{id},x_j^{-1})|](z,.)=\partial_{y_i}[\tilde{f}(z,x_j^{-1}|)]$ and $\partial_{z_k}[\tilde{f}(\textup{id},x_j^{-1})|](z,.)=\partial_{z_k}[z\mapsto\tilde{f}(z,x_j^{-1}|)]|_z$ for all $i=1,...,m$, $k=1,...,q$ and a.e.~$z\in\Omega$, $j=1,...,N$. Using the isomorphism property of $\tilde{T}$ on $L^p$-spaces, Lemma \ref{th_SobMfd_def_lemma} and Theorem \ref{th_SobMfd_trafo} it follows that 
	\[
	T\tilde{f}=\sum_{j=1}^N\chi_j[z\mapsto\tilde{f}(z,x_j^{-1}|)]\circ x_j|_{U\cap x_j^{-1}(W_j)}\in W^{1,p}(\Omega,L^p(U))\cap L^p(\Omega,W^{1,p}(U)),
	\]
	$[\nabla_\Omega f](z,.)=\nabla_z[\tilde{T}\tilde{f}]|_z$ and $[\nabla_U f](z,.)=\nabla_U[\tilde{T}\tilde{f}]|_z$ for a.e.~$z\in\Omega$. Hence the derivatives are compatible under $\tilde{T}$ and Lemma \ref{th_SobMfd_def_lemma},~3.~yields the norm equivalence.
	
	Finally, we prove that $\tilde{T}$ on $W^{1,p}(\Omega\times U)$ with values in $W^{1,p}(\Omega,L^p(U))\cap L^p(\Omega,W^{1,p}(U))$ is surjective. To this end let $\tilde{h}$ be in the target space and $\tilde{h}_j$ for $j=1,...,N$ be as above for $\tilde{h}$. Then due to Theorem \ref{th_SobMfd_trafo} and Lemma \ref{th_SobDom_prod_set} it holds $h_j\in W^{1,p}(\Omega\times[x_j(U\cap U_j)\cap W_j])$ for all $j=1,...,N$. Therefore Lemma \ref{th_SobMfd_def_lemma} and Theorem \ref{th_SobMfd_trafo} imply that $\tilde{f}_{\tilde{h}}$ defined in \eqref{eq_SobMfd_prod_set_2} is contained in $W^{1,p}(\Omega\times U)$.\qedhere$_{3.}$
	\end{proof}
\end{proof}

Finally, we need the notion of domains with Lipschitz-boundary in compact Riemannian submanifolds of $\R^n$ and the analogue of Theorem \ref{th_SobDom_LipThm} in the case $\Sigma=\partial\Omega$.

\begin{Definition}\label{th_SobMfd_LipDef}\upshape
	Let $(M,g)$ be a compact $m$-dimensional Riemannian submanifold of $\R^n$ with (or without) boundary and class $C^1$. Let $U\subset M^\circ$ be open and nonempty. Then 
	\begin{enumerate}
	\item Let $u\in\partial U$. Then we say that $U$ satisfies the \textit{local Lipschitz condition} in $u$ if this holds in local coordinates, i.e.~for any chart $x:\tilde{U}\rightarrow V\subseteq[0,\infty)\times\R^{m-1}$ with $u\in\tilde{U}$ it follows that the domain $x(U\cap\tilde{U})$ satisfies the local Lipschitz condition in $x(u)$.
	\item We say $U$ has \textit{Lipschitz-boundary} if the local Lipschitz-condition holds in $u$ for all $u\in \partial U$.
	\item We call $U$ a \textit{Lipschitz-domain} in $M$ if $U$ is connected and has Lipschitz-boundary.
	\end{enumerate}
\end{Definition}

By definition $M^\circ$ has Lipschitz-boundary.

\begin{Remark}\label{th_SobMfd_LipRem}\upshape
	The local Lipschitz condition from Definition \ref{th_SobDom_LipDef} for domains in $\R^n$ is invariant under $C^1$-diffeomorphisms (between open subsets of $\R^n$) defined on an open neighbourhood of the closure of the domain, cf.~Hofmann, Mitrea, Taylor \cite{HMitreaTaylor}, Theorem 4.1. Therefore
	\begin{enumerate}
	\item The invariance under $C^1$-diffeomorphisms carries over to Definition \ref{th_SobMfd_LipDef}.
	\item It is enough to prove the condition in Definition \ref{th_SobMfd_LipDef},~1.~for one admissible chart.
	\item Definition \ref{th_SobMfd_LipDef} is consistent with Definition \ref{th_SobDom_LipDef} in the case $m=n$.
	\end{enumerate} 
\end{Remark}

\begin{Theorem}\label{th_SobMfd_LipThm}
	Let $(M,g)$ be a compact $m$-dimensional Riemannian submanifold of $\R^n$ with (or without) boundary and class $C^l$, $l\in\N_0\cup\{\infty\}$. Let $U\subset M^\circ$ be a Lipschitz-domain and $1\leq p<\infty$, $k\in\N_0$. Then $C^l(\overline{U})$ is dense in $W^{k,p}(U)$.
	
	Let (for convenience) additionally $\partial U=\Sigma\cup Z$ with an $(m-1)$-dimensional $C^1$-submanifold $\Sigma$ of $M$ and a null set $Z$ with respect to $\Hc^{m-1}$. Then $L^p(\partial U):=L^p(\Sigma)$ is defined in Remark \ref{th_Leb_RiemRem} and there is a unique bounded linear operator $\tr:W^{1,p}(U)\rightarrow L^p(\partial U)$ such that $\tr\,u=u|_{\partial U}$ for all $u\in C^0(\overline{U})\cap W^{1,p}(U)$.
\end{Theorem}

\begin{proof}
	This follows from Theorem \ref{th_SobDom_LipThm} via localization and a suitable partition of unity.
\end{proof}

\subsection{Exponentially Weighted Spaces}
\label{sec_fct_exp} We define all used spaces with exponential weight.
\begin{Definition}\label{th_exp_def1}\upshape
	Let $1\leq p\leq\infty, k\in\N_0, \mu\in(0,1)$ and $\beta, \beta_1,\beta_2\geq 0$.
	\begin{enumerate}
		\item Then we introduce with canonical norms
		\begin{align*}
		L^p_{(\beta_1,\beta_2)}(\R^2_+)&:=\{u\in L^1_\textup{loc}(\R^2_+):e^{\beta_1|R|+\beta_2 H}u\in L^p(\R^2_+)\},\\
		W^{k,p}_{(\beta_1,\beta_2)}(\R^2_+)&:=\{u\in L^1_\textup{loc}(\R^2_+):D^\gamma u\in L^p_{(\beta_1,\beta_2)}(\R^2_+)\text{ for all }|\gamma|\leq k\}.
		\end{align*}
		We also write $H^k$ instead of $W^{k,2}$. Moreover, $C^k_{(\beta_1,\beta_2)}(\overline{\R^2_+}):=C_b^k(\overline{\R^2_+})\cap W^{k,\infty}_{(\beta_1,\beta_2)}(\R^2_+)$.
		\item In a similar way we define $L^p_{(\beta)}(\R)$, $L^p_{(\beta)}(\R_+)$, $W^{k,p}_{(\beta)}(\R)$, $W^{k,p}_{(\beta)}(\R_+)$, $C^k_{(\beta)}(\R)$, $C^k_{(\beta)}(\overline{\R_+})$. 
		\item
		Let $\eta:\R\rightarrow\R$ be smooth with $\eta(R)=|R|$ for all $|R|\geq\overline{R}$ and some $\overline{R}>0$. Then we define
		\[
		W^{k+\mu,p}_{(\beta)}(\R):=\{u\in L^1_\textup{loc}(\R): e^{\beta\eta(R)}u\in W^{k+\mu,p}(\R)\}
		\]
		for $1\leq p<\infty$ with natural norm.
	\end{enumerate}
\end{Definition}
The following lemma summarizes all the needed properties for these spaces: 
\begin{Lemma} \phantomsection{\label{th_exp1}}
	\begin{enumerate}
		\item The spaces in Definition \ref{th_exp_def1} are Banach spaces.
		\item \textup{Equivalent norms:} Let $\eta:\R\rightarrow\R$ be as in Definition \ref{th_exp_def1}, 3., $k\in\N_0$, $1\leq p\leq\infty$, $\beta_1,\beta_2\geq 0$. Then  $W^{k,p}_{(\beta_1,\beta_2)}(\R^2_+)=\{u\in L^1_\textup{loc}(\R^2_+):e^{\beta_1\eta(R)+\beta_2H}u\in W^{k,p}(\R^2_+)\}$ and
		\begin{align*}
		&\sum_{\gamma\in\N_0^2,|\gamma|\leq k}\|e^{\beta_1|R|+\beta_2H} D^\gamma u\|_{L^p(\R^2_+)},\quad 
		\sum_{\gamma\in\N_0^2,|\gamma|\leq k}\|e^{\beta_1\eta(R)+\beta_2H} D^\gamma u\|_{L^p(\R^2_+)}\\
		&\text{ and }\quad \|e^{\beta_1\eta(R)+\beta_2 H} u\|_{W^{k,p}(\R^2_+)}
		\end{align*}
		are equivalent uniformly in $u\in W^{k,p}_{(\beta_1,\beta_2)}(\R^2_+)$. For fixed $B>0$ the constants in the equivalence estimates can be taken uniformly in $\beta_1, \beta_2\in[0,B]$. Analogous assertions hold for $\R$ instead of $\R^2_+$.
		\item \textup{Density of smooth functions with compact support:} For all $k\in\N_0$, $1\leq p<\infty$ and $\beta_1,\beta_2\geq 0$ it holds: $C_0^\infty(\overline{\R^2_+})$ is dense in $W^{k,p}_{(\beta_1,\beta_2)}(\R^2_+)$,
		$C_0^\infty(\overline{\R_+})$ is dense in $W^{k,p}_{(\beta_2)}(\R_+)$ and 
		$C_0^\infty(\R)$ is dense in $W^{k,p}_{(\beta_1)}(\R)$.
		\item \textup{Embeddings:} It holds $W^{k,p}_{(\beta_1,\beta_2)}(\R^2_+)\hookrightarrow W^{k,p}_{(\gamma_1,\gamma_2)}(\R^2_+)$ for all $k\in\N_0$, $1\leq p\leq\infty$ and $0\leq \gamma_1\leq\beta_1$, $0\leq \gamma_2\leq\beta_2$, as well as
		\begin{align*}
		L^p_{(\beta_1,\beta_2)}(\R^2_+)\hookrightarrow L^q_{(\beta_1-\varepsilon,\beta_2-\varepsilon)}(\R^2_+)\quad\text{ for all }\min\{\beta_1,\beta_2\}>\varepsilon>0, 1\leq q\leq p. 
		\end{align*} 
		Analogous embeddings hold for spaces on $\R$ and $\R_+$. 
		\item \textup{Traces of weighted functions on $\R^2_+$:} For all $k\in\N$, $1\leq p<\infty$ and $\beta\geq 0$ the trace operator
		\[
		\tr:W^{k,p}_{(\beta,0)}(\R^2_+)\subset W^{k,p}(\R^2_+)\rightarrow W^{k-\frac{1}{p},p}_{(\beta)}(\R)
		\]
		is well-defined, bounded and there is a coretract operator $R_\beta$ (independent of $k,p$), i.e.
		\[
		R_\beta\in\Lc(W^{k-\frac{1}{p},p}_{(\beta)}(\R),W^{k,p}_{(\beta,0)}(\R^2_+))\quad\text{ with }\tr\circ R_\beta=\textup{id}.
		\]
		Finally, all operator norms for fixed $k,p$ are bounded uniformly in $\beta\geq 0$ if we take the third norm in Lemma \ref{th_exp1},~2. 
		\item \textup{$L^2$-Poincaré Inequality for weighted functions on $\R_+$:} For $\beta>0$ and all $u\in H^1_{(\beta)}(\R_+)$ it holds $\|u\|_{L^2_{(\beta)}(\R_+)}\leq\frac{1}{\beta}\|\partial_Hu\|_{L^2_{(\beta)}(\R_+)}$.
		\item \textup{Reverse Fundamental Theorem for weighted $L^2$-functions on $\R_+$:} For $\beta>0$ and $v$ in $L^2_{(\beta)}(\R_+)$ it holds $-\int_.^\infty v\,ds=:w\in H^1_{(\beta)}(\R_+)$ with $\partial_Hw=v$. In particular 6.~is applicable. 
	\end{enumerate}
\end{Lemma}
\begin{Remark}\phantomsection{\label{th_exp_rem}}\upshape
\begin{enumerate}
\item Note that the choice $\int_0^. v\,ds$ in Lemma \ref{th_exp1}, 7.~would not be appropriate for integrability on $\R_+$. 
\item From now on we will always use the third norm in Lemma \ref{th_exp1},~2.~for the weighted spaces.\end{enumerate}
\end{Remark}

\begin{proof}[Proof of Lemma \ref{th_exp1}. Ad 1]
	That all spaces are normed ones directly follows from the unweighted case. It is left to prove the completeness. For $L^p_{(\beta_1,\beta_2)}(\R^2_+)$ let $(u_l)_{l\in\N}$ be a Cauchy sequence. Then $(u_l)_{l\in\N}$ and $(e^{\beta_1|R|+\beta_2H}u_l)_{l\in\N}$ are Cauchy sequences in $L^p(\R^2_+)$ and therefore converge to some $u$ and $v$, respectively, in $L^p(\R^2_+)$. Since one finds a.e.~convergent subsequences it follows that $e^{\beta_1|R|+\beta_2H}u=v$ and hence $u_l\rightarrow u$ in $L^p_{(\beta_1,\beta_2)}(\R^2_+)$ for $l\rightarrow\infty$. For $W^{k,p}_{(\beta_1,\beta_2)}(\R^2_+)$ one shows the completeness with the case $k=0$ and embeddings into the unweighted spaces. For the fractional Sobolev spaces, the completeness follows directly with the definition and the unweighted case.\qedhere$_{1.}$
\end{proof}

\begin{proof}[Ad 2] For $k=0$ this directly follows from $c_\eta e^{|R|}\leq e^{\eta(R)}\leq C_\eta e^{|R|}$ for all $R\in\R$ and some constants $c_\eta, C_\eta>0$. In the case $k\geq 1$ one uses the product rule for distributions with smooth functions. \qedhere$_{2.}$\end{proof}

\begin{proof}[Ad 3] The density properties directly carry over from the unweighted case since smooth functions with compact support stay in this class when multiplied with a smooth function.\qedhere$_{3.}$\end{proof}

\begin{proof}[Ad 4] The first embedding is clear. For the second one we use Hölder's inequality.\qedhere$_{4.}$\end{proof}

\begin{proof}[Ad 5] The trace operator $\tr$ is a bounded operator from $W^{k,p}(\R^2_+)$ onto $W^{k-\frac{1}{p},p}(\R)$ if $k\in\N_0$, $1\leq p<\infty$ and there is a coretract operator $R$ independent of $k,p$, cf.~Triebel \cite{Triebel_Fct_SpacesI}, Theorem 2.7.2 and the construction therein. For $u\in W^{k,p}_{(\beta,0)}(\R^2_+)$ we write 
	\[
	u=e^{-\beta\eta(.)}\cdot e^{\beta\eta(.)}u\in C_b^\infty(\overline{\R^2_+})\cdot W^{k,p}(\R^2_+).
	\] 
	Then $W^{k-\frac{1}{p},p}(\R)\ni\tr\,u=e^{-\beta\eta(.)}\tr(e^{\beta\eta(.)}u)\in  W^{k-\frac{1}{p},p}_{(\beta)}(\R)$. Moreover, we have the estimate
	\[
	\|\tr\,u\|_{W^{k-\frac{1}{p},p}_{(\beta)}(\R)}=\|\tr(e^{\beta\eta(.)}u)\|_{W^{k-\frac{1}{p},p}(\R)}\leq C_{k,p}\|e^{\beta\eta(.)}u\|_{W^{k,p}(\R^2_+)}=C_{k,p}\|u\|_{W^{k,p}_{(\beta,0)}(\R^2_+)}
	\]
	for all $u\in W^{k,p}_{(\beta,0)}(\R^2_+)$. The coretract operator can be taken as $R_\beta v:=e^{-\beta\eta(.)} R(e^{\beta\eta(.)}v)$ for all $v\in W^{k-\frac{1}{p},p}_{(\beta)}(\R)$. One can directly verify the claimed properties.\qedhere$_{5.}$\end{proof}

\begin{proof}[Ad 6] By density it is enough to prove the estimate for $u\in C_0^\infty(\overline{\R_+})$. With the Fundamental Theorem of Calculus, Fubini's Theorem and the Hölder inequality we obtain
	\begin{align*}
	&\|u\|_{L^2_{(\beta)}(\R_+)}^2=\int_0^\infty e^{2\beta H} u^2(H)\,dH\leq 2\int_0^\infty e^{2\beta H}\int_H^\infty|u \partial_su|\,ds\,dH\leq\\
	&\leq 2\int_0^\infty \int_0^s e^{2\beta H}\,dH |u \partial_su|\,ds\leq 
	\frac{1}{\beta}\int_0^\infty e^{2\beta s}|u\partial_su|\,ds
	\leq \frac{1}{\beta}\|\partial_Hu\|_{L^2_{(\beta)}(\R_+)}\|u\|_{L^2_{(\beta)}(\R_+)},
	\end{align*}
	where we used $\int_0^s e^{2\beta H}\,dH=\frac{1}{2\beta}[e^{2\beta s}-1]\leq\frac{1}{2\beta}e^{2\beta s}$. This shows the estimate.\qedhere$_{6.}$\end{proof}

\begin{proof}[Ad 7] Let $v\in L^2_{(\beta)}(\R_+)$ for a $\beta>0$ and $v_l\in C_0^\infty(\overline{\R_+})$ with $v_l\rightarrow v$ for $l\rightarrow \infty$. Then $u_l:=-\int_.^\infty v_l(s)\,ds\in C_0^\infty(\overline{\R_+})$ with $\frac{d}{dH}u_l=v_l$. From 6.~we obtain that $(u_l)_{l\in\N}$ is a Cauchy sequence in $H^1_{(\beta)}(\R_+)$, hence by 1.~there is a limit $u$ in $H^1_{(\beta)}(\R_+)$ and $\frac{d}{dH}u=v$. Because of $u_l=-\int_.^\infty v_l(s)\,ds\rightarrow -\int_.^\infty v(s)\,ds$ for $l\rightarrow\infty$ pointwise, we get $u=-\int_.^\infty v(s)\,ds$.
	\qedhere$_{7.}$\end{proof}

\section{Curvilinear Coordinates}\label{sec_coord}
Let $N\geq2$ and $\Omega\subseteq\R^N$ be a bounded, smooth domain (i.e.~nonempty, open and connected\footnote{\label{foot_connected}~For convenience. The considerations can be adapted for the case of finitely many connected components.}) with outer unit normal $N_{\partial\Omega}$. 
In this section we show the existence of a curvilinear coordinate system describing a neighbourhood of a suitable evolving surface\footnote{~For the definition of an evolving hypersurface cf.~Depner \cite{Depner}, Definition 2.31.} in $\overline{\Omega}$ that meets the boundary $\partial\Omega$ at a $90$°-contact angle. We adapt the ideas from the $2$-dimensional case in \cite{AbelsMoser} to the $N$-dimensional case.

\subsection{Requirements for the Evolving Surface}\label{sec_coord_surface_requ}
Let $\Sigma\subset\R^N$ be a smooth, orientable, compact and connected$^\text{\ref{foot_connected}}$ hypersurface with boundary $\partial\Sigma$ and let $X_0:\Sigma\times[0,T]\rightarrow\overline{\Omega}$ be smooth such that $X_0(.,t)$ is an injective immersion for all $t\in[0,T]$. For technical reasons, we assume that there is a smooth, orientable and connected hypersurface $\Sigma_0\subset\R^N$ without boundary such that $\Sigma\subsetneq\Sigma_0$ and a smooth extension of $X_0$ to $\tilde{X}_0:\Sigma_0\times(-\tau_0,T+\tau_0)\rightarrow\R^N$ for some $\tau_0>0$ such that $\tilde{X}_0(.,t)$ is an injective immersion for all $t\in(-\tau_0,T+\tau_0)$. Finally, we choose a smooth, orientable, compact and connected hypersurface $\tilde{\Sigma}$ with boundary such that $\Sigma\subsetneq\tilde{\Sigma}^\circ$ and $\tilde{\Sigma}\subsetneq\Sigma_0$. 

\begin{Remark}\upshape\label{th_coord_surface_requ_rem}
	Such $\Sigma_0, \tau_0, \tilde{X}_0$ should exist for any $\Sigma$, $X_0$ as above. For $N=2$ this is clear, but for $N\geq 3$ this is more difficult to show. First, it should be possible to extend any $\Sigma$ as above to a smooth orientable hypersurface $\hat{\Sigma}\subset\R^N$ without boundary by merging together local extensions in a suitable way. However, this is quite technical since one has to deal with fraying.

	Then $X_0$ can be extended to a smooth immersion $\hat{X}$ on an open neighbourhood of $\Sigma\times[0,T]$ in $\hat{\Sigma}\times\R$. Because immersions are locally injective (cf.~O'Neill \cite{ONeill}, Lemma 1.33), one can prove injectivity of $\hat{X}$ on a possibly smaller open neighbourhood of $\Sigma\times[0,T]$ in $\hat{\Sigma}\times\R$ with a contradiction argument and compactness.
\end{Remark} 

Since continuous bijections of compact topological spaces into Hausdorff spaces are homeo-morphisms, we know that $X_0(.,t)$ is an embedding and $\Gamma_t:=X_0(\Sigma,t)\subset\R^N$ is a smooth, orientable, compact and connected hypersurface with boundary for all $t\in[0,T]$. Moreover,
\[
\Gamma:=\bigcup_{t\in[0,T]}\Gamma_t\times\{t\}
\]
is a smooth evolving hypersurface and 
\[
\overline{X}_0:=(X_0,\textup{pr}_t):\Sigma\times[0,T]\rightarrow\Gamma:(s,t)\mapsto(X_0(s,t),t)
\] 
is a homeomorphism. We choose a smooth normal field $\vec{n}:\Sigma\times[0,T]\rightarrow\R^N$ meaning that $\vec{n}$ is smooth and $\vec{n}(.,t)$ describes a normal field on $\Gamma_t$. Due to Depner \cite{Depner}, Lemma 2.40 the corresponding normal velocity is given by 
\[
V(s,t):=V_{\Gamma_t}(s):=\vec{n}(s,t)\cdot\partial_tX_0(s,t)\quad\text{ for }\quad(s,t)\in \Sigma\times[0,T].
\]
Moreover, let $H(s,t):=H_{\Gamma_t}(s)$ for $(s,t)\in \Sigma\times[0,T]$ be the mean curvature which we choose to be the sum of the principal curvatures. The above definitions applied to $\tilde{X}_0$ on $\tilde{\Sigma}\times[-\frac{\tau_0}{2},T+\frac{\tau_0}{2}]$ yield suitable extensions of $\Gamma_t, \Gamma, \vec{n}, V$ and $H$. For convenience, we use the same notation for $\vec{n}$.

Additionally, we require $(\Gamma_t)^\circ\subseteq\Omega$ and $\partial\Gamma_t\subseteq\partial\Omega$. Then the \textit{contact angle} of $\Gamma_t$ with $\partial\Omega$ in any boundary point $X_0(s,t), (s,t)\in\partial\Sigma\times[0,T]$ with respect to $\vec{n}(s,t)$ is defined by
\[
|\measuredangle(N_{\partial\Omega}|_{X_0(s,t)},\vec{n}(s,t))|\in(0,\pi),
\]
where $\measuredangle(N_{\partial\Omega}|_{X_0(s,t)},\vec{n}(s,t))$ is taken in $(-\pi,\pi)$.

\subsection{Existence of Curvilinear Coordinates}\label{sec_coordND}
Let the assumptions in Section \ref{sec_coord_surface_requ} hold for dimension $N\geq2$ and constant contact angle $\frac{\pi}{2}$ for times $t\in[0,T]$. We consider the \textit{outer unit conormal} $\vec{n}_{\partial\Sigma}:\partial\Sigma\rightarrow\R^N$, cf.~\cite{Depner}, Definition 2.28 on p.22. Moreover, we introduce the outer unit conormal for the evolving hypersurface $\Gamma$, $\vec{n}_{\partial\Gamma}:\partial\Sigma\times[0,T]\rightarrow\R^N$, where $\vec{n}_{\partial\Gamma}(\sigma,t):=\vec{n}_{\partial\Gamma_t}(\sigma)$ is the outer unit conormal with respect to $\partial\Gamma_t$ at $X_0(\sigma,t)$ for all $(\sigma,t)\in\partial\Sigma\times[0,T]$.  One can show smoothness and
\[ \vec{n}_{\partial\Gamma_t}(\sigma)=N_{\partial\Omega}|_{X_0(\sigma,t)}\quad\text{ for all }(\sigma,t)\in\partial\Sigma\times[0,T]
\] 
with the considerations in \cite{Depner}. Furthermore, we use the tubular neighbourhood coordinate system of $\partial\Sigma$ in $\tilde{\Sigma}$: for $\mu_1>0$ small there is a smooth diffeomorphism
\[
\tilde{Y}:\partial\Sigma\times[-2\mu_1,2\mu_1]\rightarrow R(\tilde{Y})\subset\tilde{\Sigma},\quad(\sigma,b)\mapsto\tilde{Y}(\sigma,b)
\]
onto a neighbourhood $R(\tilde{Y})$ of $\partial\Sigma$ in $\tilde{\Sigma}$ such that $\tilde{Y}|_{b=0}=\textup{id}_{\partial\Sigma}$ and $Y:=\tilde{Y}|_{\partial\Sigma\times[0,2\mu_1]}$ is a diffeomorphism onto a neighbourhood $R(Y)$ of $\partial\Sigma$ in $\Sigma$. We use the notation $(\tilde{\sigma},\tilde{b}):=\tilde{Y}^{-1}$. We define $\tilde{Y}$ via the exponential map on the normal bundle of $\partial\Sigma$ in $\tilde{\Sigma}$, cf.~Proposition 7.26 in O'Neill \cite{ONeill}. Then 
\begin{align}\label{eq_coordND_DbY}
\partial_bY(\sigma,0)=-\vec{n}_{\partial\Sigma}(\sigma)\quad\text{ for all }\sigma\in\partial\Sigma.
\end{align}

\begin{Theorem}[\textbf{Curvilinear Coordinates}]\label{th_coordND}
	Let the above assumptions hold. There exist $\delta>0$ and a smooth map $[-\delta,\delta]\times\Sigma\times[0,T]\ni(r,s,t)\mapsto X(r,s,t)\in\overline{\Omega}$ with the following properties: 
	\begin{enumerate}
		\item $\overline{X}:=(X,\textup{pr}_t)$ is a homeomorphism onto a neighbourhood of $\Gamma$ in $\overline{\Omega}\times[0,T]$. Moreover, $\overline{X}$ can be extended to a smooth diffeomorphism defined on an open neighbourhood of $[-\delta,\delta]\times\Sigma\times[0,T]$ in $\R\times\tilde{\Sigma}\times\R$ mapping onto an open set in $\R^{N+1}$. The set
		\[
		\Gamma(\tilde{\delta})
		:=\overline{X}((-\tilde{\delta},\tilde{\delta})\times\Sigma\times[0,T])
		\]
		is an open neighbourhood of $\Gamma$ in $\overline{\Omega}\times[0,T]$ for $\tilde{\delta}\in(0,\delta]$. \item $X|_{r=0}=X_0$ and $X$ coincides with the well-known tubular neighbourhood coordinate system for $s\in\Sigma\textbackslash Y(\partial\Sigma\times[0,\mu_0])$ for some $\mu_0\in(0,\mu_1]$ small. Additionally, for points $(r,s,t)\in[-\delta,\delta]\times\Sigma\times[0,T]$ it holds $X(r,s,t)\in\partial\Omega$ if and only if $s\in\partial\Sigma$.
		\item Let $(r,s,\textup{pr}_t)$ be the inverse of $\overline{X}$. Then $(\partial_{x_j}s|_{(x,t)})_{j=1}^N$ generate the tangent space $T_{s(x,t)}\Sigma$, $|\nabla r|_{(x,t)}|\geq c>0$ for some $c>0$ independent of $(x,t)$ and $D_xs(D_xs)^\top|_{(x,t)}$ is uniformly positive definite as a linear map in $\Lc(T_{s(x,t)}\Sigma)$ for all $(x,t)\in\overline{\Gamma(\delta)}$. Furthermore, we have
		\[
		|\nabla r|^2|_\Gamma=1,\quad \partial_r(|\nabla r|^2\circ\overline{X})|_{r=0}=0\quad\text{ and }\quad D_xs\nabla r|_\Gamma=0
		\] 
		and for all $(r,s,t)\in[-\delta,\delta]\times\left[\Sigma\textbackslash Y(\partial\Sigma\times[0,\mu_0])\right]\times[0,T]$ it holds
		\[
		\nabla r|_{\overline{X}(r,s,t)}=\vec{n}(s,t)\quad\text{ and }\quad D_xs|_{\overline{X}(r,s,t)}\vec{n}(s,t)=0.
		\]
		Moreover, we can choose $\nabla r\circ\overline{X}_0=\vec{n}$. Then it holds $V=-\partial_tr\circ\overline{X}_0$ and $H=-\Delta r\circ\overline{X}_0$.
		\item Let $(\sigma,b):=Y^{-1}\circ s:
		\overline{X}([-\delta,\delta]\times R(Y)\times[0,T])\rightarrow
		\partial\Sigma\times[0,2\mu_1]$. 
		Then 
		\begin{alignat*}{2}
		N_{\partial\Omega}\cdot\nabla b|_{\overline{X}_0(\sigma,t)}=-D_xsN_{\partial\Omega}|_{\overline{X}_0(\sigma,t)}\cdot \vec{n}_{\partial\Sigma}|_{\sigma}, \quad |N_{\partial\Omega}\cdot\nabla b|_{\overline{X}_0(\sigma,t)}|\geq c>0
		\end{alignat*}
		and $\nabla b\cdot\nabla r|_{\overline{X}_0(\sigma,t)}=0$, $|\nabla b|_{\overline{X}_0(\sigma,t)}|\geq c>0$ for all $(\sigma,t)\in\partial\Sigma\times[0,T]$.
	\end{enumerate}
\end{Theorem}
\begin{Remark}\phantomsection{\label{th_coordND_rem}}\upshape
	\begin{enumerate}
		\item Let $Q_T:=\Omega\times(0,T)$. There are unique connected $Q_T^\pm\subseteq\overline{Q_T}=\overline{\Omega}\times[0,T]$ such that $\overline{Q_T}=Q_T^-\cup Q_T^+\cup\Gamma$ (disjoint) and $\textup{sign}\,r=\pm 1$ on $Q_T^\pm\cap\Gamma(\delta)$. Moreover, we set
		\[
		\Gamma^C(\tilde{\delta},\mu):=\overline{X}((-\tilde{\delta},\tilde{\delta})\times Y(\partial\Sigma\times(0,\mu))\times[0,T]),\quad
		\Gamma(\tilde{\delta},\mu):=\Gamma(\tilde{\delta})\textbackslash\overline{\Gamma^C(\tilde{\delta},\mu)}
		\]
		for $\tilde{\delta}\in(0,\delta]$ and $\mu\in(0,2\mu_1]$. For $t\in[0,T]$ fixed let $\Gamma_t(\tilde{\delta}), \Gamma_t^C(\tilde{\delta},\mu)$ and $\Gamma_t(\tilde{\delta},\mu)$ be the respective sets intersected with $\R^N\times\{t\}$ and then projected to $\R^N$. Here $\Gamma(\tilde{\delta})$ is defined in Theorem \ref{th_coordND}.
		
		\item Let $\tilde{\delta}\in(0,\delta]$. For a sufficiently smooth $\psi:\Gamma(\tilde{\delta})\rightarrow\R$ we define the \emph{tangential} and \emph{normal derivative} by 
		\[
		\nabla_\tau\psi:=(D_xs)^\top[\nabla_\Sigma(\psi\circ\overline{X})\circ\overline{X}^{-1}]\quad\text{ and }\quad\partial_n\psi:=\partial_r(\psi\circ\overline{X})\circ\overline{X}^{-1},
		\]
		respectively. For $t\in[0,T]$ fixed and $\psi:\Gamma_t(\tilde{\delta})\rightarrow\R$ smooth enough, we define $\nabla_\tau\psi$ and $\partial_n\psi$ analogously. The same notation applies if $\psi$ is only defined on an open subset of $\Gamma(\tilde{\delta})$ or $\Gamma_t(\tilde{\delta})$ for some $\mu\in(0,2\mu_1]$ and $t\in[0,T]$. Note that in the case $N=2$ and $\Sigma=[-1,1]$ the definitions here coincide with the ones in \cite{AbelsMoser}, Remark 2.2. Important properties of $\nabla_\tau$ and $\partial_n$ will be shown in Corollary \ref{th_coordND_nabla_tau_n}.
		
		\item For transformation arguments we define 
		\[
		J(r,s,t):=J_t(r,s):=|\det d_{(r,s)}X(r,s,t)|\quad\text{ for }(r,s,t)\in[-\delta,\delta]\times\Sigma\times[0,T],
		\]
		where the determinant is taken with respect to an arbitrary orthonormal basis of $T_s\Sigma$. The latter is well-defined, cf.~Theorem \ref{th_Leb_trafo_mfd}, 1. Via local coordinates it follows that $J$ is smooth and with a compactness argument we obtain that $0<c\leq J\leq C$ for some $c,C>0$.
	\end{enumerate}
\end{Remark}

The proof of Theorem \ref{th_coordND} is similar to \cite{AbelsMoser}, Theorem 2.1, where the case $N=2$ was proven. The first step for the proof of Theorem \ref{th_coordND} is to show an analogue of \cite{AbelsMoser}, Lemma 2.3.

\begin{Lemma}\label{th_coordND_lemma}
	There is an $\eta>0$ such that $\partial\Omega\cap R_\eta(\sigma,t)$ admits a graph parametrization over $X_0(\sigma,t)+[B_\eta(0)\cap T_{X_0(\sigma,t)}\partial\Omega]$ for all $(\sigma,t)\in\partial\Sigma\times[0,T]$, where
	\[
	R_\eta(\sigma,t):=X_0(\sigma,t)+(-\eta,\eta)\vec{n}_{\partial\Gamma}(\sigma,t)+[B_\eta(0)\cap T_{X_0(\sigma,t)}\partial\Omega].
	\]
	Moreover, for $\eta>0$ small there exists $w:(-\eta,\eta)\times\partial\Sigma\times[0,T]\rightarrow\R$ smooth such that $w|_{r=0}=\partial_rw|_{r=0}=0$ and
	\[
	(-\eta,\eta)\ni r\mapsto X_0(\sigma,t)+r\vec{n}(\sigma,t)+w(r,\sigma,t)\vec{n}_{\partial\Gamma}(\sigma,t)
	\]
	describes $\partial\Omega$ in $X_0(\sigma,t)+(-\eta,\eta)\vec{n}(\sigma,t)+(-\eta,\eta)\vec{n}_{\partial\Gamma}(\sigma,t)$ for all $(\sigma,t)\in\partial\Sigma\times[0,T]$.
\end{Lemma} 

The assertions are compatible with shrinking $\eta$ for small $\eta>0$ which follows from the contact angle assumption and Taylor's Theorem.

\begin{proof}
	Let $(\sigma_0,t_0)\in\partial\Sigma\times[0,T]$ be arbitrary. We choose a basis $\vec{v}_1,...,\vec{v}_{N-2}$ of $T_{X_0(\sigma_0,t_0)}\partial\Gamma_{t_0}$ and extend it to smooth tangential vector fields $\vec{\tau}_1,...,\vec{\tau}_{N-2}$ on $T\partial\Gamma$ such that locally in $\partial\Gamma$ around $X_0(\sigma_0,t_0)$ these are again bases in the corresponding tangential spaces over $\partial\Gamma_{t}$ for all $t\in B_\varepsilon(t_0)\cap[0,T]$, $\varepsilon>0$ small. In coordinates one can apply similar arguments as in the proof of \cite{AbelsMoser}, Lemma 2.3., to obtain smooth graph parametrizations of $\partial\Omega\cap R_\eta(\sigma,t)$ over $X_0(\sigma,t)+[B_\eta(0)\cap T_{X_0(\sigma,t)}\partial\Omega]$ for some $\eta>0$ and $(\sigma,t)\in\partial\Sigma\times[0,T]$ close to $(\sigma_0,t_0)$. More precisely, there is an $\eta>0$ and a smooth
	\[
	w_0:(-\eta,\eta)\times B_\eta(0)\times U\times V\subseteq\R\times\R^{N-2}\times\Sigma\times[0,T]\rightarrow\R,
	\]
	where $U,V$ are open neighbourhoods of $\sigma_0, t_0$ in $\partial\Sigma, [0,T]$, respectively, such that
	\begin{align*}
	(-\eta,\eta)\!\times\! B_{\eta}(0)\ni (r,r_1,...,r_{N-2})\mapsto &X_0(\sigma,t)+r\vec{n}(\sigma,t)+r_1\vec{\tau}_1(\sigma,t)+...+r_{N-2}\vec{\tau}_{N-2}(\sigma,t)\\
	&+w_0(r,r_2,...,r_{N-2},\sigma,t)\vec{n}_{\partial\Gamma}(\sigma,t)
	\end{align*}
	describes $\partial\Omega\cap R_\eta(\sigma,t)$ in  $R_\eta(\sigma,t)$. Moreover, $w_0|_{r=0}=\partial_rw_0|_{r=0}=0$. We set
	\[
	w:(-\eta,\eta)\times U\times V\rightarrow\R: (r,\sigma,t)\mapsto w_0(r,0,...,0,\sigma,t)
	\]
	and we observe that this definition is independent of the choice of $\vec{v}_1,...,\vec{v}_{N-2}$ and $\vec{\tau}_1,...,\vec{\tau}_{N-2}$ as well as $(\sigma_0,t_0)$. By compactness $\eta>0$ can be taken uniformly in $(\sigma_0,t_0)\in\partial\Sigma\times[0,T]$.  
\end{proof}

\begin{proof}[Proof of Theorem \ref{th_coordND}]
	Let $\Sigma_1$ be a compact hypersurface with boundary such that $\Sigma\subsetneq\Sigma_1^\circ$ and $\Sigma_1\subsetneq\tilde{\Sigma}^\circ$.
	Similar as in the proof of \cite{AbelsMoser}, Theorem 2.1 there is a $\delta_0>0$ such that for all $\delta\in(0,\delta_0]$ it holds that \phantom{\qedhere}
	\[
	(-\delta,\delta)\times\Sigma_1\ni (r,s)\mapsto \tilde{X}_0(s,t)+r\vec{n}(s,t)\in\R^N
	\] 
	is a diffeomorphism onto its image $U_\delta(t)$ and $U_\delta(t)\cap\overline{\Omega}=B_\delta(\tilde{\Gamma}_t)\cap\overline{\Omega}$ for all $t\in[0,T]$, where we have set $\tilde{\Gamma}_t:=\tilde{X}_0(\tilde{\Sigma},t)$.
	
	We choose $\eta>0$ small such that $R_\eta(\partial\Sigma,t)$ is contained in $U_{\delta_0}(t)$, the assertions of Lemma~\ref{th_coordND_lemma} are fulfilled and such that the angles between the tangent planes of $R_{\eta}(\sigma,t)\cap\partial\Omega$ 
	are smaller than a fixed $\beta>0$ (which will be chosen later).
	
	Now we define $X$. Let $\vec{\tau}:\tilde{\Sigma}\times[0,T]\rightarrow\R^N$ be a smooth vector field with the property that $\vec{\tau}(s,t)\in T_{\tilde{X}_0(s,t)}\tilde{\Gamma}_t$ for all $(s,t)\in\tilde{\Sigma}\times[0,T]$ and $\vec{\tau}|_{\partial\Sigma\times[0,T]}=\vec{n}_{\partial\Gamma}$. Existence of such a $\vec{\tau}$ follows via local extensions of $\vec{n}_{\partial\Gamma}$ in submanifold charts of $\partial\Sigma$ with respect to $\tilde{\Sigma}$ and compactness arguments.  Moreover, by uniform continuity there is an $\varepsilon\in(0,\mu_1]$ such that $\tilde{X}_0(\tilde{Y}(\sigma,\tilde{b}),t)\in R_{\frac{\eta}{2}}(\sigma,t)$ for all $(\sigma,\tilde{b},t)\in\partial\Sigma\times[-\varepsilon,\varepsilon]\times[0,T]$ as well as
	\begin{align}\label{eq_coordND_eps}
	|\partial_b\tilde{Y}(\sigma,\tilde{b})+\vec{n}_{\partial\Sigma}(\sigma)|+\left|\frac{d}{d\tilde{b}}\tilde{X}_0(\tilde{Y}(\sigma,\tilde{b}))-\left.\frac{d}{d\tilde{b}}\right|_{\tilde{b}=0}\tilde{X}_0(\tilde{Y}(\sigma,.))\right|\leq c_0
	\end{align}
	for a fixed $c_0>0$ small (to be determined later), where we used \eqref{eq_coordND_DbY}. Let $\chi:\R\rightarrow[0,1]$ be a smooth cutoff-function with $\chi=1$ for $|b|\leq \frac{\varepsilon}{2}$ and $\chi=0$ for $|b|\geq \varepsilon$. We define
	$\vec{T}(s,t):=\chi(\tilde{b}(s))\vec{\tau}(s,t)$ for $(s,t)\in\Sigma\times[0,T]$ and 
	\[
	X(r,s,t):=\tilde{X}_0(s,t)+r\vec{n}(s,t)+w(r,\tilde{\sigma}(s),t)\vec{T}(s,t)\in\R^N\quad\text{ for }(r,s,t)\in[-\delta,\delta]\times\tilde{\Sigma}\times[0,T]
	\] 
	and $\delta>0$ small. In the following we show that the properties in the theorem are satisfied if $\delta>0$ is small and $\beta>0$ as well as $c_0>0$ above were chosen properly. 
	
	\begin{proof}[Ad 1.-2]
		First of all, $X$ is well-defined due to the cutoff-function. Moreover, $X$ is smooth and
		\begin{align}\label{eq_coordND_Xabl_r}
		\partial_rX(r,s,t)&=\vec{n}(s,t)+\partial_rw(r,\tilde{\sigma}(s),t)\vec{T}(s,t)\in\R^N,\\
		d_s[X(r,.,t)]&=d_s[\tilde{X}_0(.,t)]+rd_s[\vec{n}(.,t)]+d_s[w(r,\tilde{\sigma}(.),t)]\vec{T}+w|_{(r,\tilde{\sigma}(s),t)}d_s[\vec{T}(.,t)]\label{eq_coordND_Xabl_s}
		\end{align}
		for all $(r,s,t)\in[-\delta,\delta]\times\tilde{\Sigma}\times[0,T]$, where $d_s[X(r,.,t)]$, $d_s[\tilde{X}_0(.,t)]$, $d_s[\vec{n}(.,t)]$ and $d_s[\vec{T}(.,t)]$ map from $T_s\tilde{\Sigma}$ to $\R^N$. Hence
		\begin{align}\label{eq_coordND_Xabl1}
		d_{(r,s)}[X(.,t)]&:\R\times T_s\tilde{\Sigma}\rightarrow\R^N, (v_1,v_2)\mapsto \partial_rX|_{(r,s,t)}v_1+d_s[X(r,.,t)](v_2),\\
		d_{(0,s)}[X(.,t)]&:\R\times T_s\tilde{\Sigma}\rightarrow\R^N, (v_1,v_2)\mapsto\vec{n}(s,t)v_1+d_s[\tilde{X}_0(.,t)](v_2),\label{eq_coordND_Xabl2}
		\end{align}
		where we used $w|_{r=0}=\partial_rw|_{r=0}=0$. Since $d_s[\tilde{X}_0(.,t)]:T_s\tilde{\Sigma}\rightarrow T_{\tilde{X}_0(s,t)}\tilde{\Gamma}_t$ is an isomorphism and $\R^N=T_{\tilde{X}_0(s,t)}\tilde{\Gamma}_t\oplus N_{\tilde{X}_0(s,t)}\tilde{\Gamma}_t$, we obtain that $d_{(0,s)}[X(.,t)]:\R\times T_s\tilde{\Sigma}\rightarrow\R^N$ is invertible for every $(s,t)\in\tilde{\Sigma}\times[0,T]$. By compactness this is also valid for $d_{(r,s)}[X(.,t)]$ for every $(r,s,t)\in[-\delta,\delta]\times\tilde{\Sigma}\times[0,T]$ if $\delta>0$ is small, cf.~the similar argument in the proof of Theorem 2.1 in \cite{AbelsMoser}. The Inverse Function Theorem yields that $\overline{X}$ is locally injective and together with injectivity on $\{0\}\times\tilde{\Sigma}\times[0,T]$ we get similarly as in the $2$-dimensional case in \cite{AbelsMoser} by contradiction and compactness that $\overline{X}$ is injective on $[-\delta,\delta]\times\tilde{\Sigma}\times[0,T]$ for $\delta>0$ small. Moreover, due to the Inverse Function Theorem, $\overline{X}$ can locally in $\R\times\Sigma_0\times\R$ be extended to a smooth diffeomorphism. With a similar contradiction and compactness argument as before, it follows that $\overline{X}$ can be extended to a smooth diffeomorphism on an open neighbourhood of $[-\delta,\delta]\times\Sigma\times[0,T]$ in $\R\times\Sigma_0\times\R$ mapping onto an open set in $\R^{N+1}$.
		
		Next we prove that $X([-\delta,\delta]\times\Sigma\times[0,T])\subset\overline{\Omega}$ if $\delta>0$ is small and related properties. First, note that the set $\Gamma\textbackslash\overline{X}_0(Y(\partial\Sigma\times[0,\varepsilon]\times[0,T]))$ has a positive distance to $\partial\Omega\times[0,T]$ by compactness. Moreover,
		\[
		X(r,s,t)=X_0(s,t)+r\vec{n}(s,t)\quad\text{ for }(r,s,t)\in[-\delta,\delta]\times\left[\Sigma\textbackslash Y(\partial\Sigma\times[0,\varepsilon])\right]\times[0,T].
		\]
		Therefore $X(r,s,t)$ stays in $\Omega$ for such $(r,s,t)$ if $\delta>0$ is small. For the remaining points we use geometric arguments with angles and Lemma \ref{th_coordND_lemma}. For $s\in Y(\partial\Sigma\times[0,\varepsilon])$ we observe that $Y(\tilde{\sigma}(s),.):[0,\tilde{b}(s)]\rightarrow\Sigma$ is a curve from $\tilde{\sigma}(s)$ to $s$. Hence $X(r,Y(\tilde{\sigma}(s),.),t):[0,b(s)]\rightarrow\R^N$ is a curve from $X(r,\tilde{\sigma}(s),t)$ to $X(r,s,t)$, where
		\[
		\frac{d}{db}[X(r,Y(\tilde{\sigma}(s),b),t)]=d_{Y(\tilde{\sigma}(s),b)}[X(r,.,t)](\partial_bY(\tilde{\sigma}(s),b)).
		\]
		Because of \eqref{eq_coordND_Xabl_s} and $w|_{r=0}=\partial_rw|_{r=0}=0$ we have 
		\[
		\left.\frac{d}{db}\right|_{b=0}[X(0,Y(\sigma,.),t)]=d_\sigma[X_0(.,t)](-\vec{n}_{\partial\Sigma}(\sigma))\quad\text{ for all }(\sigma,t)\in\partial\Sigma\times[0,T].
		\]
		Here $d_\sigma[X_0(.,t)]$ is invertible from $T_\sigma\Sigma$ to $T_{X_0(\sigma,t)}\Gamma_t$ as well as from $T_\sigma\partial\Sigma$ to $T_{X_0(\sigma,t)}\partial\Gamma_t$. Therefore $d_{\sigma}[X_0(.,t)](\vec{n}_{\partial\Sigma}(\sigma))\cdot\vec{n}_{\partial\Gamma}(\sigma,t)>0$ and by compactness
		\[
		d_{\sigma}[X_0(.,t)](\vec{n}_{\partial\Sigma}(\sigma))\cdot\vec{n}_{\partial\Gamma}(\sigma,t)\geq c_1>0\quad\text{ for all }(\sigma,t)\in\partial\Sigma\times[0,T].
		\]
		Due to \eqref{eq_coordND_eps} and \eqref{eq_coordND_Xabl_s} we obtain for all $(r,b,s,t)\in[-\delta,\delta]\times[0,\varepsilon]\times Y(\partial\Sigma\times[0,\varepsilon])\times[0,T]$ that
		\begin{align}\label{eq_coordND_cone}
		-\vec{n}_{\partial\Gamma}|_{(\tilde{\sigma}(s),t)}\cdot\frac{d}{db}\left[X(r,Y(\tilde{\sigma}(s),b),t)\right]\geq\frac{c_1}{2}>0
		\end{align}
		provided that $\delta>0$ is small and $c_0>0$ was chosen sufficiently small before. Moreover, it holds $X_0(Y(\sigma,b),t)\in R_{\frac{\eta}{2}}(\sigma,t)$ for $(\sigma,b,t)\in\partial\Sigma\times[0,\varepsilon]\times[0,T]$ by the choice of $\varepsilon$. This yields
		\begin{align}\label{eq_coordND_cylinder}
		X(r,Y(\sigma,b),t)\in R_{\frac{3\eta}{4}}(\sigma,t)\quad\text{ for all }(r,\sigma,b,t)\in[-\delta,\delta]\times\partial\Sigma\times[0,\varepsilon]\times[0,T]
		\end{align}
		if $\delta>0$ is small. Altogether we can determine the location of $X(r,Y(\sigma,b),t)$ geometrically: By \eqref{eq_coordND_cylinder} we know that $X(r,Y(\sigma,b),t)$ is contained in a cylinder where we have a suitable graph parametrization of $\partial\Omega$ due to Lemma \ref{th_coordND_lemma}. Moreover, \eqref{eq_coordND_cone} and the Fundamental Theorem of Calculus yield that $X(r,Y(\sigma,b),t)$ lies in a cone (where $c_1$ determines how close it can be to a half space) viewed from $X(r,Y(\sigma,0),t)$. Therefore if $\beta>0$ in the beginning of the proof was chosen sufficiently small, the cone without the tip lies inside of $\Omega$. Note that $c_1$ above is independent of $\beta,c_0,\eta,\varepsilon,\delta$. Therefore we can choose $\beta,c_0>0$ small first (both only depending on $c_1$), then $\eta>0$, then $\varepsilon >0$ and finally $\delta>0$. This proves $X([-\delta,\delta]\times\Sigma\times[0,T])\subset\overline{\Omega}$ and $X(r,s,t)\in\partial\Omega$ if and only if $s\in\partial\Sigma$. Moreover, with an analogous argument it follows that $X$ maps
		\[
		[-\delta,\delta]\times\tilde{Y}(\partial\Sigma\times[-\varepsilon,0))\times[0,T]
		\]
		outside $\overline{\Omega}$ for possibly smaller $\beta,c_0,\eta,\varepsilon,\delta$. Finally, the extension property of $\overline{X}$ yields that $\Gamma(\tilde{\delta})$
		is an open neighbourhood of $\Gamma$ in $\overline{\Omega}\times[0,T]$ for $\tilde{\delta}\in(0,\delta]$ if $\delta>0$ is small.\qedhere$_{\text{1.-2.}}$\end{proof}
	
	\begin{proof}[Ad 3]
		Consider $(r,s,\textup{pr}_t):=\overline{X}^{-1}:\overline{\Gamma(\delta)}\rightarrow[-\delta,\delta]\times\Sigma\times[0,T]$. Then
		\[
		d_x[X^{-1}(.,t)]:\R^N\rightarrow\R\times T_{s(x,t)}\Sigma:
		\vec{v}\mapsto(d_x[r(.,t)],d_x[s(.,t)])\vec{v}
		=(\nabla r|_{(x,t)}\cdot\vec{v},D_xs|_{(x,t)}\vec{v})
		\]
		is invertible for all $(x,t)\in\overline{\Gamma(\delta)}$. Therefore $|\nabla r|_{(x,t)}|>0$ and by compactness $|\nabla r|_{(x,t)}|\geq c>0$ for all $(x,t)\in\overline{\Gamma(\delta)}$. Moreover, $(\partial_{x_j}s|_{(x,t)})_{j=1}^N$ generate $T_{s(x,t)}\Sigma$ for all such $(x,t)$. In particular
		\[
		D_xs(D_xs)^\top|_{(x,t)}
		=(\nabla s_i\cdot\nabla s_j|_{(x,t)})_{i,j=1}^N
		=\sum_{q=1}^N\partial_qs(\partial_qs)^\top|_{(x,t)}
		\]
		is injective as a linear map in $\Lc(T_{s(x,t)}\Sigma)$ for all $(x,t)\in\overline{\Gamma(\delta)}$. The latter follows directly since $D_xs(D_xs)^\top|_{(x,t)}\vec{v}=0$ for some $\vec{v}\in T_{s(x,t)}\Sigma$ implies $\sum_{q=1}^N|(\partial_qs|_{(x,t)})^\top\vec{v}|^2=0$ and hence $\vec{v}=0$. Therefore $D_xs(D_xs)^\top|_{(x,t)}$ is positive definite on $T_{s(x,t)}\Sigma$ for all $(x,t)\in\overline{\Gamma(\delta)}$. In local coordinates the latter transforms to a linear map on $\R^{N-1}$. Note that by scaling it is equivalent to consider vectors in the sphere in $\R^{N-1}$ in order to verify the definition of positive definiteness. Therefore by compactness we obtain that $D_xs(D_xs)^\top|_{(x,t)}$ is uniformly positive definite as a linear map in $\Lc(T_{s(x,t)}\Sigma)$ for all $(x,t)\in\overline{\Gamma(\delta)}$.

		Let $(r,s,t)\in[-\delta,\delta]\times\Sigma\times[0,T]$.
		Then $d_{(r,s)}[X(.,t)]\circ d_{X(r,s,t)}[X^{-1}(.,t)]=\textup{Id}_{\R^N}$. Hence
		\[
		(d_{(r,s)}[X(.,t)])^{-1}=(d_{X(r,s,t)}[r(.,t)],d_{X(r,s,t)}[s(.,t)]):\R^N\rightarrow\R\times T_s\Sigma.
		\]
		On the other hand \eqref{eq_coordND_Xabl2} implies
		\begin{align}\label{eq_eq_coordND_abl_r_s_0}
		d_{X(0,s,t)}[r(.,t)]=\vec{n}(s,t)^\top\quad\text{ and }\quad d_{X(0,s,t)}[s(.,t)]=d_s[X_0(.,t)]^{-1}P_{T_{X_0(s,t)}\Gamma_t}.
		\end{align}
		We can also write 
		\begin{align}\label{eq_coordND_abl_r_s}
		d_{X(r,s,t)}[r(.,t)](\vec{v})=D_xr|_{\overline{X}(r,s,t)} \vec{v}\quad\text{ and }\quad d_{X(r,s,t)}[s(.,t)](\vec{v})=D_xs|_{\overline{X}(r,s,t)}\vec{v}
		\end{align} 
		for all $\vec{v}\in\R^N$. Altogether this yields $\nabla r|_{\overline{X}_0(s,t)}=\vec{n}(s,t)$ and $D_xs|_{\overline{X}_0(s,t)}\vec{n}(s,t)=0$ for all $(s,t)\in\Sigma\times[0,T]$. With similar arguments we obtain
		\[
		\nabla r|_{\overline{X}(r,s,t)}=\vec{n}(s,t)\quad\text{ and }\quad D_xs|_{\overline{X}(r,s,t)}\vec{n}(s,t)=0
		\]
		for all $(r,s,t)\in[-\delta,\delta]\times\left[\Sigma\textbackslash Y(\partial\Sigma\times[0,\mu_0])\right]\times[0,T]$. Moreover, it holds
		\[
		\partial_r(|\nabla r|^2\circ\overline{X})|_{(0,s,t)}=2\frac{d}{dr}[D_xr|_{X(.,s,t)}]|_{r=0}\cdot \nabla r|_{\overline{X}_0(s,t)}.
		\] 
		In order to use \eqref{eq_coordND_abl_r_s} we compute 
		\[
		\frac{d}{dr}[(d_{(r,s)}[X(.,t)])^{-1}]|_{r=0}=-(d_{(0,s)}X(.,t))^{-1}\circ\frac{d}{dr}[d_{(r,s)}X(.,t)]|_{r=0}\circ(d_{(0,s)}X(.,t))^{-1}
		\]
		with the formula for the Fréchet derivative of the inverse of a differentiable family of invertible, linear operators. Here $(d_{(0,s)}X(.,t))^{-1}:\R^N\rightarrow\R\times T_s\Sigma$ is explicitly determined by \eqref{eq_eq_coordND_abl_r_s_0}. Furthermore, differentiating \eqref{eq_coordND_Xabl_r}-\eqref{eq_coordND_Xabl_s} with respect to $r$ we obtain for all $(v_1,v_2)\in\R\times T_s\Sigma$
		\[
		\frac{d}{dr}[d_{(r,s)}X(.,t)]|_{r=0}(v_1,v_2)=\partial_r^2w(0,\sigma(s),t)\vec{T}(s,t)v_1+d_s[\vec{n}(.,t)](v_2)\in\R^N.
		\]
		Therefore it holds
		\begin{align*}
		&\frac{d}{dr}[(d_{(r,s)}[X(.,t)])^{-1}]|_{r=0}(\vec{n}(s,t))=
		-(d_{(0,s)}X(.,t))^{-1}\left(\frac{d}{dr}[d_{(r,s)}X(.,t)]|_{r=0}(1,0)\right)\\
		&=-(d_{(0,s)}X(.,t))^{-1}\left(\partial_r^2w(0,\sigma(s),t)\vec{T}(s,t)\right)=
		-(0,\partial_r^2w(0,\sigma(s),t)d_s[X_0(.,t)]^{-1}\vec{T}(s,t)),
		\end{align*}
		where $\partial_r(|\nabla r|^2\circ\overline{X})|_{(0,s,t)}$ equals twice the first component, thus equals zero. Moreover, one can prove the identities for the normal velocity $V$ and mean curvature $H$ in an analogous way as in the case $N=2$, cf.~the proof of Theorem 2.1 in \cite{AbelsMoser}.\qedhere$_{3.}$\end{proof}
	
	\begin{proof}[Ad 4]
		Finally, we show the properties of $b$. Let $(\overline{\sigma},\overline{b}):=Y^{-1}:R(Y)\rightarrow\partial\Sigma\times[0,2\mu_1]$. Then $b=\overline{b}\circ s$ on $\overline{X}([-\delta,\delta]\times R(Y)\times[0,T])$ and by chain rule 
		\[
		d_x[b(.,t)]=d_{s(x,t)}\overline{b}\circ d_x[s(.,t)]:\R^N\rightarrow\R,
		\] 
		where we are interested in boundary points $x=X_0(\sigma,t)$ for any $(\sigma,t)\in\partial\Sigma\times[0,T]$. For such $x$ it holds $s(x,t)=\sigma$. Because of $Y(.,0)=\textup{id}_{\partial\Sigma}$ and \eqref{eq_coordND_DbY} we have
		\[
		d_{(\sigma,0)}Y:T_\sigma\partial\Sigma\times\R\rightarrow T_\sigma\Sigma:
		(v_1,v_2)\mapsto v_1-\vec{n}_{\partial\Sigma}(\sigma)v_2.
		\]
		Therefore $(d_\sigma\overline{\sigma},d_\sigma\overline{b})=(d_{(\sigma,0)}Y)^{-1}=(\textup{pr}_{T_\sigma\partial\Sigma},-\vec{n}_{\partial\Sigma}(\sigma)^\top)$. Together with \eqref{eq_eq_coordND_abl_r_s_0} we obtain
		\[
		\nabla b|_{(x,t)}\cdot v=d_x[b(.,t)](v)=-\vec{n}_{\partial\Sigma}(\sigma)\cdot((d_\sigma[X_0(.,t)])^{-1}\circ P_{T_x\Gamma_t}(v))\quad\text{ for all }v\in\R^N.
		\]
		Hence $\nabla b|_{\overline{X}_0(\sigma,t)}=\nabla b|_{(x,t)}\in T_{X_0(\sigma,t)}\Gamma_t$, in particular $\nabla b\cdot\nabla r|_{\overline{X}_0(\sigma,t)}=0$, and
		\[
		\nabla b|_{\overline{X}_0(\sigma,b)}\cdot N_{\partial\Omega}|_{X_0(\sigma,b)}=-\vec{n}_{\partial\Sigma}(\sigma)\cdot (d_\sigma[X_0(.,t)])^{-1}\vec{n}_{\partial\Gamma}(\sigma,t)\quad\text{ for all }(\sigma,t)\in\partial\Sigma\times[0,T].
		\]
		Note that due to \eqref{eq_eq_coordND_abl_r_s_0}-\eqref{eq_coordND_abl_r_s} it holds 
		\[
		(d_\sigma[X_0(.,t)])^{-1}\vec{n}_{\partial\Gamma}(\sigma,t)=D_xs N_{\partial\Omega}|_{\overline{X}_0(\sigma,t)}
		\] 
		for all $(\sigma,t)\in\partial\Sigma\times[0,T]$. Hence we obtain the identity in the theorem. Moreover, we know that $d_\sigma[X_0(.,t)]$ is invertible from $T_\sigma\Sigma$ to $T_{X_0(\sigma,t)}\Gamma_t$ as well as from $T_\sigma\partial\Sigma$ to $T_{X_0(\sigma,t)}\partial\Gamma_t$. This yields that $|\nabla b|_{\overline{X}_0(\sigma,t)}\cdot N_{\partial\Omega}|_{X_0(\sigma,t)}|>0$ for all $(\sigma,t)\in\partial\Sigma\times[0,T]$ and by smoothness and compactness the latter is bounded from below by a uniform positive constant. Because of the Cauchy-Schwarz-Inequality, this estimate carries over to $|\nabla b|_{\overline{X}_0(\sigma,t)}|$.\qedhere$_{\text{4.}}$\end{proof}
\end{proof}

Finally, we show relations of $\partial_n, \nabla_\tau$ defined in Remark \ref{th_coordND_rem},~2.~to $\nabla,\nabla_\Sigma,\nabla_{\partial\Sigma},\partial_b$.

\begin{Corollary}\label{th_coordND_nabla_tau_n}
	Let $\psi:\Gamma(\tilde{\delta})\rightarrow\R$ for some $\tilde{\delta}\in(0,\delta]$ be sufficiently smooth. Then 
	\begin{enumerate}
		\item $\nabla\psi=\partial_n\psi\nabla r+\nabla_\tau\psi$ on $\Gamma(\tilde{\delta})$ and there are $c,C>0$ independent of $\psi$ and $\tilde{\delta}$ such that  
		\begin{alignat*}{2}
		c(|\partial_n\psi|+|\nabla_\tau\psi|)
		&\leq|\nabla\psi|
		\leq C(|\partial_n\psi|+|\nabla_\tau\psi|)&\quad&\text{ on }\Gamma(\tilde{\delta}),\\
		c|\partial_n\psi|&\leq|\nabla r\partial_r(\psi\circ\overline{X})\circ\overline{X}^{-1}|\leq C|\partial_n\psi|&\quad&\text{ on }\Gamma(\tilde{\delta}),\\
		c|\nabla_\tau\psi|&\leq|\nabla_\Sigma(\psi\circ\overline{X})\circ\overline{X}^{-1}|\leq C|\nabla_\tau\psi|&\quad&\text{ on }\Gamma(\tilde{\delta}).
		\end{alignat*}
		\item It holds $|\nabla\psi|^2
		=|\partial_n\psi|^2+|\nabla_\tau\psi|^2$ on $\Gamma(\tilde{\delta},\mu_0)$.
		\item Set 
		$\overline{Y}:[-\delta,\delta]\times\partial\Sigma\times[0,2\mu_1]\times[0,T]\rightarrow[-\delta,\delta]\times\Sigma\times[0,T]:(r,\sigma,b,t)\mapsto(r,Y(\sigma,b),t)$ and $\overline{\psi}:=\psi\circ\overline{X}\circ\overline{Y}|_{(-\tilde{\delta},\tilde{\delta})\times\partial\Sigma\times[0,2\mu_1]\times[0,T]}$.
		Then 
		\[
		\tilde{c}(|\nabla_{\partial\Sigma}\overline{\psi}|+|\partial_b\overline{\psi}|)
		\leq|\nabla_\Sigma[\psi\circ\overline{X}]\circ\overline{Y}|
		\leq \tilde{C}(|\nabla_{\partial\Sigma}\overline{\psi}|+|\partial_b\overline{\psi}|)
		\]
		on $(-\tilde{\delta},\tilde{\delta})\times\partial\Sigma\times[0,2\mu_1]\times[0,T]$ for some $\tilde{c},\tilde{C}>0$ independent of $\psi, \tilde{\delta}$.
	\end{enumerate}
	Analogous assertions hold for $\psi$ defined on $\Gamma_t(\tilde{\delta})$, $t\in[0,T]$ and for $\psi$ defined on open subsets of $\Gamma(\tilde{\delta})$ or $\Gamma_t(\tilde{\delta})$, $t\in[0,T]$ with natural adjustments and uniform constants (w.r.t.~$\psi$, $t$ and the sets).
\end{Corollary}

\begin{proof}
	We only consider $\psi:\Gamma(\tilde{\delta})\rightarrow\R$. The case of sufficiently smooth $\psi:\Gamma_t(\tilde{\delta})\rightarrow\R$ for any $t\in[0,T]$ and the case of other open subsets can be shown with analogous arguments.\phantom{\qedhere}
	
	\begin{proof}[Ad 1]
		The second equivalence estimate is evident since $0<\tilde{c}\leq|\nabla r|\leq\tilde{C}$ due to Theorem \ref{th_coordND}. Moreover, it holds $\psi=(\psi\circ\overline{X})\circ\overline{X}^{-1}|_{(-\tilde{\delta},\tilde{\delta})\times\Sigma\times[0,T]}$. The chain rule yields 
		\[
		\nabla\psi|_{(x,t)}\cdot.=d_x[\psi(.,t)]=d_{(r,s)}[\psi\circ \overline{X}(.,t)]\circ d_x[X(.,t)^{-1}]:\R^N\rightarrow\R
		\]
		for all $(x,t)=\overline{X}(r,s,t)\in\Gamma(\tilde{\delta})$.
		Here 
		\begin{align*}
		d_{(r,s)}[\psi\circ\overline{X}(.,t)]:\R\times T_s\Sigma\rightarrow\R:(w,\vec{v})&\mapsto  
		d_r[\psi\circ\overline{X}|_{(.,s,t)}](w)
		+d_s[\psi\circ\overline{X}|_{(r,.,t)}](\vec{v})\\
		&=\partial_{\tilde{r}}[\psi\circ\overline{X}|_{(.,s,t)}]|_r\,w
		+\nabla_\Sigma[\psi\circ\overline{X}|_{(r,.,t)}]|_s\cdot\vec{v}.
		\end{align*}
		Furthermore,
		$d_x[X(.,t)^{-1}]:\R^N\rightarrow\R\times T_s\Sigma:\vec{u}\mapsto(\nabla r\cdot\vec{u},D_xs\,\vec{u})$ is invertible for all $(x,t)=\overline{X}(r,s,t)\in\overline{\Gamma(\delta)}$ and the operator norm and the one of the inverse is uniformly bounded due to compactness. This yields $\nabla\psi=\partial_n\psi\nabla r+\nabla_\tau\psi$ on $\Gamma(\tilde{\delta})$ and 
		\[
		c(|\partial_n\psi|+|\nabla_\Sigma(\psi\circ\overline{X})\circ\overline{X}^{-1}|)
		\leq |\nabla\psi|
		\leq C(|\partial_n\psi|+|\nabla_\Sigma(\psi\circ\overline{X})\circ\overline{X}^{-1}|)\quad\text{ on }\Gamma(\tilde{\delta})
		\]
		with $c,C>0$ independent of $\psi,\tilde{\delta}$. In order to show 1.~it remains to prove the last equivalence estimate in the claim. The latter is valid since $D_xs$ is uniformly bounded and $D_xs(D_xs)^\top$ is uniformly positive definite on $T_{s(x,t)}\Sigma$ for all $(x,t)\in\overline{\Gamma(\delta)}$ due to Theorem \ref{th_coordND}.\qedhere$_{1.}$\end{proof}
	
	\begin{proof}[Ad 2]
		Consequence of 1.~and $|\nabla r|=1$, $D_xs\nabla r=0$ on $\Gamma(\tilde{\delta},\mu_0)$ due to Theorem \ref{th_coordND}, 3.\qedhere$_{2.}$\end{proof}
	
	\begin{proof}[Ad 3]
		The chain rule yields $d_{(\sigma,b)}[\overline{\psi}(r,.,t)]=d_{Y(\sigma,b)}[\psi\circ\overline{X}(r,.,t)]\circ d_{(\sigma,b)}Y:T_\sigma\partial\Sigma\times\R\rightarrow\R$ for all $(r,\sigma,b,t)\in(-\tilde{\delta},\tilde{\delta})\times\partial\Sigma\times[0,2\mu_1]\times[0,T]$. Here $d_{(\sigma,b)}Y:T_\sigma\partial\Sigma\times\R\rightarrow T_{Y(\sigma,b)}\Sigma$ is invertible for all $(\sigma,b)\in\partial\Sigma\times[0,2\mu_1]$ and the operator norm and the one of the inverse is uniformly bounded by compactness. Moreover, it holds
		\[
		d_{(\sigma,b)}[\overline{\psi}(r,.,t)](\vec{v},w)
		=\nabla_{\partial\Sigma}[\overline{\psi}(r,.,b,t)]|_\sigma\cdot\vec{v}+\partial_b[\overline{\psi}(r,\sigma,.,t)]|_b\,w
		\]
		for all $(\vec{v},w)\in T_\sigma\partial\Sigma\times\R$
		and $d_{Y(\sigma,b)}[\psi\circ\overline{X}(r,.,t)](\vec{u})=\nabla_\Sigma[\psi\circ\overline{X}(r,.,t)]|_{Y(\sigma,b)}\cdot\vec{u}$ for all $\vec{u}\in T_{Y(\sigma,b)}\Sigma$. This proves the claim. 
		\qedhere$_{3.}$\end{proof}
\end{proof}

\section{Model Problems}\label{sec_model_problems}
Unless otherwise stated we use real-valued function spaces in this section.

\subsection{Some Scalar-valued ODE Problems on $\R$}\label{sec_ODE_scalar}
In this section we recall existence and regularity results needed for ODEs appearing in the inner asymptotic expansion for \hyperlink{AC}{(AC)}. For the potential $f:\R\rightarrow\R$ in this section we assume \eqref{eq_AC_fvor1}. 
\subsubsection{The ODE for the Optimal Profile}\label{sec_ODE_scalar_nonlin}
The ODE system for the lowest order is
\begin{align}\label{eq_theta_0_ODE}
-w''+f'(w)=0,\quad w(0)=0,\quad \lim_{z\rightarrow\pm\infty}w(z)=\pm1.
\end{align}
\begin{Theorem}\label{th_theta_0}
	Let $f$ be as in \eqref{eq_AC_fvor1}. Then \eqref{eq_theta_0_ODE} has a unique solution $\theta_0\in C^2(\R)$. Moreover, $\theta_0$ is smooth, $\theta_0'=\sqrt{2(f(\theta_0)-f(-1))}>0$ and
	\[
	D_z^k(\theta_0\mp 1)(z)=\Oc(e^{-\beta|z|})\quad\text{ for }z\rightarrow\pm\infty\text{ and all }k\in\N_0, \beta\in\left(0,\sqrt{\min\{f''(\pm1)\}}\right).
	\]
\end{Theorem}
\begin{proof}
	See Schaubeck \cite{Schaubeck}, Lemma 2.6.1 and its proof. The idea is to solve the equivalent first order ODE
	\[
	w'=\sqrt{\int_{-1}^w 2f'(s)\,ds},\quad w(0)=0.
	\]
	Note that only ODE-methods and elementary arguments are used.
\end{proof}
We call $\theta_0$ the \emph{optimal profile}. A rescaled version will be the typical profile of the solutions for the scalar-valued Allen-Cahn equation with Neumann boundary condition \eqref{eq_AC1}-\eqref{eq_AC3} across the interface. If $f$ is even, then $\theta_0$ is even, $\theta_0'$ is odd and $\theta_0''$ even etc. For the typical double-well potential $f(u)=\frac{1}{2}(1-u^2)^2$ shown in Figure \ref{fig_double_well} one can directly compute that the optimal profile is $\theta_0=\tanh$, cf.~Figure \ref{fig_opt_profile}.

\begin{figure}[H]
	\centering
	\def\svgwidth{\linewidth}
	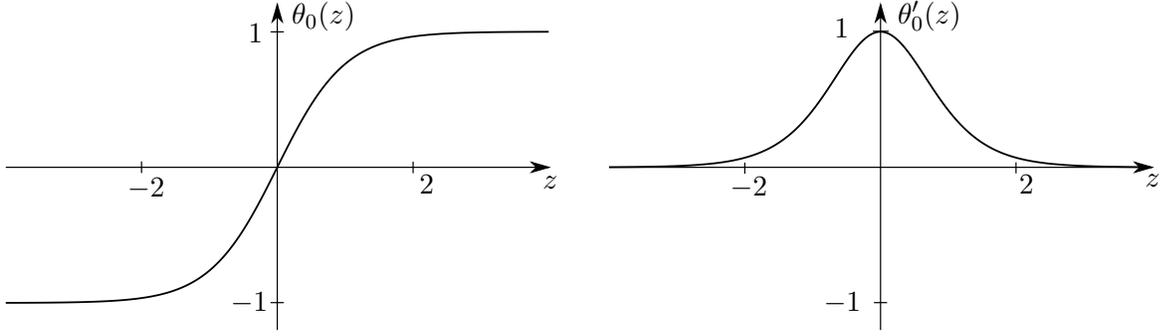
	\caption{Typical optimal profile $\theta_0=\tanh$ and the derivative $\theta_0'$.}\label{fig_opt_profile}
\end{figure}

\subsubsection{The Linearized ODE}\label{sec_ODE_scalar_lin}
The following theorem is concerned with the solvability of the equation which is obtained by linearization of \eqref{eq_theta_0_ODE} at $\theta_0'$, i.e.
\begin{align}\label{eq_ODE_lin}
\Lc_0w=A\quad\text{ in }\R,\quad w(0)=0,
\end{align}
where we have set $\Lc_0:=-\frac{d^2}{dz^2}+f''(\theta_0)$.
\begin{Theorem}\phantomsection{\label{th_ODE_lin}}
	\begin{enumerate}
		\item Let $A\in C_b^0(\R)$. Then \eqref{eq_ODE_lin} has a solution $w\in C^2(\R)\cap C_b^0(\R)$ if and only if $\int_{\R}A\theta_0'\,dz=0$. In that case $w$ is unique. Moreover, if $A(z)-A^\pm=\Oc(e^{-\beta|z|})$ for $z\rightarrow\pm\infty$ for some $\beta\in(0,\sqrt{\min\{f''(\pm1)\}})$, then 
		\[
		D_z^{l}\left[w-\frac{A^\pm}{f''(\pm 1)}\right]=\Oc(e^{-\beta|z|})\quad\text{ for }z\rightarrow\pm\infty,\quad l=0,1,2.
		\]
		\item Let $U\subseteq\R^d$ (any set $U$ is allowed, e.g.~a point) and $A:\R\times U\rightarrow\R$, $A^\pm:U\rightarrow\R$ be smooth (i.e.~locally smooth extendible) and the following hold uniformly in $U$:
		\[
		D_x^k D_z^l\left[A(z,.)-A^\pm\right]=\Oc(e^{-\beta|z|})\quad\text{ for }z\rightarrow\pm\infty,\quad k=0,...,K, l=0,...,L,
		\] 
		for some $\beta\in(0,\sqrt{\min\{f''(\pm1)\}})$ and $K,L\in\N_0$. Then $w:\R\times U\rightarrow\R$, where $w(.,x)$ is the solution of \eqref{eq_ODE_lin} for $A(.,x)$ for all $x\in U$, is also smooth and uniformly in $U$ it holds
		\[
		D_x^k D_z^l\left[w(z,.)-\frac{A^\pm}{f''(\pm1)}\right]=\Oc(e^{-\beta|z|})\quad\text{ for }z\rightarrow\pm\infty, m=0,...,K, l=0,...,L+2.
		\] 
	\end{enumerate}
\end{Theorem}
For our purpose $A^\pm =0$ will be enough. 

\begin{proof}
	The result follows from the proof of \cite{Schaubeck}, Lemma 2.6.2. The idea is to reduce to a first order ODE for the derivative of $w/\theta_0'$. In order to show boundedness of the $w\in C^2(\R)$ in \cite{Schaubeck} for $A\in C_b^0(\R)$ provided $\int_\R A\theta_0'\,dz=0$, one can use $\theta_0'>0$, estimate $A$ roughly in the formula for $w$ in \cite{Schaubeck} and apply the convergence proof in \cite{Schaubeck} for the case of constant $A$ there. Note that only ODE-methods and elementary arguments are used.
\end{proof}

\subsection{An Elliptic Problem on $\R^2_+$ with Neumann Boundary Condition}\label{sec_hp_90}
Let $f:\R\rightarrow\R$ be as in \eqref{eq_AC_fvor1} and $\theta_0$ be as in Theorem \ref{th_theta_0}. For the contact point expansion for \hyperlink{AC}{(AC)} in any dimension $N\geq 2$ we have to solve the following model problem on $\R^2_+$: For data $G:\overline{\R^2_+}\rightarrow\R$, $g:\R\rightarrow\R$ with suitable regularity and exponential decay find a solution $u:\overline{\R^2_+}\rightarrow\R$ with similar decay to  
\begin{alignat}{2}\label{eq_hp1}
\left[-\Delta+f''(\theta_0(R))\right]u(R,H)&=G(R,H)& \quad &\text{ for }(R,H)\in\R^2_+,\\
-\partial_H u|_{H=0}(R)&=g(R)&\quad &\text{ for }R\in\R.
\label{eq_hp2}
\end{alignat}  

In Section \ref{sec_hp_weak_sol_reg} we recall existence and uniqueness assertions for weak solutions from \cite{AbelsMoser}. For the proofs see \cite{AbelsMoser}, Section 2.4.1. Moreover, in \cite{AbelsMoser}, Section 2.4.2 exponential decay estimates were proven via ordinary differential inequality arguments.  Here we proceed differently. In Section \ref{sec_hp_exp_sol} we introduce a functional analytic setting with several types of exponentially weighted Sobolev spaces (defined in Section \ref{sec_fct_exp}) in order to have solution operators for \eqref{eq_hp1}-\eqref{eq_hp2}. The rough idea is always to multiply the equation with the weights, use the product rule and known isomorphisms.

The framework with exponentially weighted Sobolev spaces helps with additional independent variables and simplifies the induction procedure needed in the asymptotic expansion in Section \ref{sec_asym_ACND} below. Note that in the 2D-case an induction is not necessary and hence was not carried out in \cite{AbelsMoser}. Moreover, the isomorphism property can be used to solve via the Implicit Function Theorem a corresponding model problem appearing in the sharp interface limit for an Allen-Cahn equation with a nonlinear Robin boundary condition designed to approximate \eqref{MCF} with a constant $\alpha$-contact angle for $\alpha$ close to $\frac{\pi}{2}$. See \cite{MoserDiss} or \cite{AbelsMoserAlpha} for details.

\subsubsection{Weak Solutions and Regularity}\label{sec_hp_weak_sol_reg}
The definition of a weak solution is canonical:
\begin{Definition}\label{th_hp_weak_sol_def}\upshape
	Let $G\in L^2(\R^2_+)$ and $g\in L^2(\R)$. Then $u\in H^1(\R^2_+)$ is called \emph{weak solution} of \eqref{eq_hp1}-\eqref{eq_hp2} if for all $\varphi\in H^1(\R^2_+)$ it holds that
	\[
	a(u,\varphi):=\int_{\R^2_+}\nabla u\cdot\nabla\varphi+f''(\theta_0(R))u\varphi\,d(R,H)=\int_{\R^2_+}G\varphi\,d(R,H)+\int_{\R}g(R)\varphi|_{H=0}(R)\,dR. 
	\]
\end{Definition}

Regarding weak solutions we have the following theorem:
\begin{Theorem}\label{th_hp_weak_sol} 
	Let $G\in L^2(\R^2_+)$ and $g\in L^2(\R)$. Then it holds:
	\begin{enumerate}
		\item $a:H^1(\R^2_+)\times H^1(\R^2_+)\rightarrow\R$ is not coercive.
		\item If $G(.,H),g\perp\theta_0'$ for a.e.~$H>0$ in $L^2(\R)$, then there is a weak solution $u$ such that $u(.,H)\perp\theta_0'$ for a.e.~$H>0$ and it holds $\|u\|_{H^1(\R^2_+)}\leq C(\|G\|_{L^2(\R^2_+)}+\|g\|_{L^2(\R)})$.
		\item Weak solutions are unique.
		\item If $G\theta_0'\in L^1(\R^2_+)$ and $u$ is a weak solution with $\partial_Hu\,\theta_0'\in L^1(\R^2_+)$, then the following compatibility condition holds:
		\begin{align}\label{eq_hp_comp}
		\int_{\R^2_+}G(R,H)\theta_0'(R)\,d(R,H)+\int_{\R}g(R)\theta_0'(R)\,dR=0.
		\end{align}
		\item If $G\theta_0'\in L^1(\R^2_+)$, then $\tilde{G}(H):=(G(.,H),\theta_0')_{L^2(\R)}$ is well-defined for a.e.~$H>0$ and $\tilde{G}\in L^1(\R_+)\cap L^2(\R_+)$. Moreover, we have the decomposition
		\begin{align}\label{eq_hp_orth_decomp}
		G=\tilde{G}(H)\frac{\theta_0'(R)}{\|\theta_0'\|_{L^2(\R)}^{2}}+G^\perp(R,H),\quad
		g=(g,\theta_0')_{L^2(\R)}\frac{\theta_0'(R)}{\|\theta_0'\|_{L^2(\R)}^{2}}+g^\perp(R)
		\end{align}
		for some $G^\perp\in L^2(\R^2_+), g^\perp\in L^2(\R)$ with $G^\perp(.,H),g\perp \theta_0'$ in $L^2(\R)$ for a.e.~$H>0$.
		\item If $\|G(.,H)\|_{L^2(\R)}\leq C e^{-\nu H}$ for a.e.~$H>0$ and a constant $\nu>0$, then $G\theta_0'\in L^1(\R^2_+)$. Let $\tilde{G}$ be defined as in 5.~and the compatibility condition \eqref{eq_hp_comp} hold. Then 
		\begin{align}\label{eq_hp_sol_formula}
		u_1(R,H):=-\int_H^\infty\int_{\tilde{H}}^{\infty}\tilde{G}(\hat{H})\,d\hat{H}\,d\tilde{H}\,\frac{\theta_0'(R)}{\|\theta_0'\|_{L^2(\R)}^{2}}
		\end{align}
		is well-defined for a.e.~$(R,H)\in\R^2_+$, it holds $u_1\in W^2_1(\R^2_+)\cap H^2(\R^2_+)$ and $u_1$ is a weak solution of \eqref{eq_hp1}-\eqref{eq_hp2} for $G-G^\perp,g-g^\perp$ from \eqref{eq_hp_orth_decomp} instead of $G,g$.
	\end{enumerate}
\end{Theorem}

\begin{proof}
	See \cite{AbelsMoser}, Theorem 2.9. Here note that the spectral properties of the linear operator $L_0:H^2(\R)\subseteq L^2(\R)\rightarrow L^2(\R):u\mapsto[-\frac{d^2}{dR^2}+f''(\theta_0)]u$ are important. See \cite{MoserDiss}, Lemma 4.2 or \cite{AbelsMoser}, Lemma 2.5 for the latter.
\end{proof}

In Theorem \ref{th_hp_weak_sol},~6.~weaker conditions on $G$ are enough, cf.~Section \ref{sec_hp_exp_sol}. The point is included for aesthetic reasons. Altogether we obtain an existence theorem for weak solutions: 

\begin{Corollary}\phantomsection{\label{th_hp_weak_sol_2}}
	\begin{enumerate}
		\item Let $g\in L^2(\R)$, $G\in L^2(\R^2_+)$ with $\|G(.,H)\|_{L^2(\R)}\leq Ce^{-\nu H}$ f.a.e.~$H>0$ and some $\nu>0$. Let \eqref{eq_hp_comp} hold. Then there is a unique weak solution of \eqref{eq_hp1}-\eqref{eq_hp2}. 
		\item Let $k\in\N_0$ and $u\in H^1(\R^2_+)$ be a weak solution of \eqref{eq_hp1}-\eqref{eq_hp2} for $G\in H^k(\R^2_+)$ and $g\in H^{k+\frac{1}{2}}(\R)$. Then $u\in H^{k+2}(\R^2_+)\hookrightarrow C^{k,\gamma}(\overline{\R^2_+})$ for all $\gamma\in(0,1)$ and it holds 
		\[
		\|u\|_{H^{k+2}(\R^2_+)}\leq C_k(\|G\|_{H^k(\R^2_+)}+\|g\|_{H^{k+\frac{1}{2}}(\R)}+\|u\|_{H^1(\R^2_+)}).
		\]
	\end{enumerate}
\end{Corollary}

\begin{proof}
	See \cite{AbelsMoser}, Corollary 2.10.
\end{proof}

\subsubsection{Solution Operators in Exponentially Weighted Spaces}\label{sec_hp_exp_sol} 
In the following the superscript \enquote{$\perp$} always means $u(.,H)\perp \theta_0'$ in $L^2(\R)$ for a.e.~$H\in\R_+$, if $u\in L^2(\R^2_+)$, and $u\perp\theta_0'$ if $u\in L^2(\R)$. The symbol \enquote{$\parallel$} has the same meaning with \enquote{$\perp$} replaced by \enquote{$\parallel$}.

\begin{Theorem}[\textbf{Solution Operators for Decay in $H$}]\phantomsection{\label{th_hp_exp1}}
	\begin{enumerate}
		\item For $\gamma\geq 0$ small the operator
		\[
		L_{\frac{\pi}{2}}:=(-\Delta+f''(\theta_0(R)),-\partial_H|_{H=0}):H^{2,\perp}_{(0,\gamma)}(\R^2_+)\rightarrow L^{2,\perp}_{(0,\gamma)}(\R^2_+)\times H^{\frac{1}{2},\perp}(\R)
		\] 
		is well-defined and invertible. Moreover, for small $\gamma\geq0$ the operator norm of $L_{\frac{\pi}{2}}^{-1}$ is uniformly bounded in the corresponding spaces.  
		\item For all $\gamma\in(0,\overline{\gamma}]$ and any $\overline{\gamma}>0$ the operator 
		\[
		L_\frac{\pi}{2}:H^{2,\parallel}_{(0,\gamma)}(\R^2_+)\rightarrow \left\{(G,g)\in L^{2,\parallel}_{(0,\gamma)}(\R^2_+)\times H^{\frac{1}{2},\parallel}(\R):\int_{\R^2_+}G\theta_0'+\int_\R g\theta_0'=0\right\}
		\]
		is well-defined, invertible and the norm of $L_{\frac{\pi}{2}}^{-1}$ is bounded by $C_{\overline{\gamma}}(1+\frac{1}{\gamma^2})$ for all $\gamma\in(0,\overline{\gamma}]$.
		\item For $\gamma>0$ small 
		\[
		L_\frac{\pi}{2}:H^2_{(0,\gamma)}(\R^2_+)\rightarrow \left\{(G,g)\in L^2_{(0,\gamma)}(\R^2_+)\times H^\frac{1}{2}(\R):\int_{\R^2_+}G\theta_0'+\int_\R g\theta_0'=0\right\}
		\]
		is well-defined, invertible and the norm of $L_{\frac{\pi}{2}}^{-1}$ is bounded by $C(1+\frac{1}{\gamma^2})$ for small $\gamma>0$.
	\end{enumerate}
\end{Theorem}
\begin{proof}[Proof. Ad 1]
	$L_{\frac{\pi}{2}}$ is well-defined in the spaces because of Lemma \ref{th_exp1},~2.,~5.~and since the orthogonality property can be shown via integration by parts as well as by differentiating the orthogonality condition for $u$ with respect to $H$. In the case $\gamma=0$ invertibility follows from Theorem \ref{th_hp_weak_sol},~2.-3.~and Corollary \ref{th_hp_weak_sol_2},~2. Now let $\gamma>0$. In order to solve $L_{\frac{\pi}{2}}u=(G,g)$ we make the ansatz $u=e^{-\gamma H}v$ with $v\in H^{2,\perp}(\R^2_+)$. By computing derivatives of $u$ we obtain equations we want to solve for $v$. Note that the exponential factor does not destroy the orthogonality property. It holds
	\[
	\partial_Hu=-\gamma e^{-\gamma H}v+e^{-\gamma H}\partial_Hv\quad\text{ and }\quad 
	\partial_H^2u=\gamma^2 e^{-\gamma H}v - 2\gamma e^{-\gamma H}\partial_Hv + e^{-\gamma H}\partial_H^2v.
	\]
	Therefore we consider
	\begin{align}\label{eq_hp_exp_sol1}
	L_\frac{\pi}{2}v+N_\gamma v=(Ge^{\gamma H},g),\quad\text{ where }N_\gamma v:=(-\gamma^2 v+2\gamma\partial_Hv, \gamma v|_{H=0}).
	\end{align}
	Here $N_\gamma$ is a bounded linear operator from  $H^{2,\perp}(\R^2_+)$ to  $L^{2,\perp}(\R^2_+)\times H^{\frac{1}{2},\perp}(\R)$ and the operator norm is estimated by $C(\gamma+\gamma^2)$. Hence a Neumann series argument yields that $L_{\frac{\pi}{2}}+N_\gamma$ is invertible in those spaces for small $\gamma$ and the norm of the inverse is bounded uniformly. Let $(G,g)$ be in $L^{2,\perp}_{(0,\gamma)}(\R^2_+)\times H^{\frac{1}{2}}(\R)$. Then we obtain for small $\gamma>0$ a unique $v\in H^{2,\perp}(\R^2_+)$ that solves \eqref{eq_hp_exp_sol1}. The above computations yield that $u:=e^{-\gamma H}v\in H^{2,\perp}_{(0,\gamma)}(\R^2_+)$ is a solution of $L_\frac{\pi}{2}u=(G,g)$ and that it is unique. Finally, we have the estimate
	\[
	\|u\|_{H^{2,\perp}_{(0,\gamma)}(\R^2_+)}=\|v\|_{H^{2,\perp}(\R^2_+)}\leq C\|(Ge^{\gamma H},g)\|_{L^{2,\perp}(\R^2_+)\times H^{\frac{1}{2}}(\R)}=C\|(G,g)\|_{L^{2,\perp}_{(0,\gamma)}(\R^2_+)\times H^{\frac{1}{2}}(\R)},
	\]
	where $C>0$ is independent of $\gamma>0$ small.\qedhere$_{1.}$\end{proof}

\begin{proof}[Ad 2] Let $\gamma>0$. Then $u\in H^{2,\parallel}_{(0,\gamma)}(\R^2_+)$ if and only if 
	\[
	e^{\gamma H}u\in H^2(\R^2_+)\quad\text{ and }\quad u(R,H)=(u(.,H),\theta_0')_{L^2(\R)}\frac{\theta_0'(R)}{\|\theta_0'\|^2_{L^2(\R)}}\text{ almost everywhere}.
	\]
	Since $H^2(\R^2_+)\hookrightarrow H^2(\R_+,L^2(\R))$ by Lemma \ref{th_SobDom_prod_set} and because multiplication with $\theta_0'$ is a bounded linear operator from $L^2(\R)$ to $\R$, it follows that $u\in H^{2,\parallel}_{(0,\gamma)}(\R^2_+)$ is equivalent to $u(R,H)=\tilde{u}(H)\theta_0'(R)$ f.a.a. $(R,H)\in\R^2_+$ for some $\tilde{u}\in H^2_{(\gamma)}(\R_+)$. The operator $L_\frac{\pi}{2}$ acts as
	\[
	L_\frac{\pi}{2}u=(-\partial_H^2\tilde{u}(H)\theta_0'(R),-\partial_H\tilde{u}(0)\theta_0'(R))\in L^{2,\parallel}_{(0,\gamma)}(\R^2_+)\times H^{\frac{1}{2},\parallel}(\R). 
	\]
	Additionally, the compatibility condition  \eqref{eq_hp_comp} holds because of Theorem \ref{th_hp_weak_sol},~4. The latter could also be directly computed here. Altogether, $L_\frac{\pi}{2}$ is well-defined in the spaces. Moreover, let $(G,g)$ be in the space $L^{2,\parallel}_{(0,\gamma)}(\R^2_+)\times H^{\frac{1}{2},\parallel}(\R)$ and let the compatibility condition \eqref{eq_hp_comp} hold. Then $G=\tilde{G}(H)\theta_0'(R)$, $g=\tilde{g}\theta_0'$ for some $\tilde{G}\in L^2_{(\gamma)}(\R_+)$ and $\tilde{g}\in\R$. By Lemma \ref{th_exp1},~6.-7.~it holds
	\[
	\tilde{u}:=-\int_.^\infty\int_{\hat{H}}^\infty\tilde{G}(\overline{H})\,d\overline{H}\,d\hat{H}\in H^2_{(\gamma)}(\R_+)
	\]
	with $\partial_H\tilde{u}=\int_.^\infty\tilde{G}(\hat{H})\,d\hat{H}$ and $\partial_H^2\tilde{u}=-\tilde{G}$ as well as $\|\tilde{u}\|_{H^2_{(\gamma)}(\R_+)}\leq c_{\overline{\gamma}}(1+\frac{1}{\gamma^2})\|\tilde{G}\|_{L^2_{(\gamma)}(\R_+)}$ for all $\gamma\in(0,\overline{\gamma}]$, where $\overline{\gamma}>0$ is fixed. Therefore $u_1:=\tilde{u}(H)\theta_0'(R)\in H^2_{(0,\gamma)}(\R^2_+)$ solves
	\begin{align*}
	[-\Delta+f''(\theta_0)]u_1 &=-\partial_H^2\tilde{u}\theta_0'=G,\\
	-\partial_Hu_1|_{H=0} &=-\int_0^\infty \tilde{G}(\hat{H})\,d\hat{H} \theta_0' =\tilde{g}\theta_0'=g,
	\end{align*}
	where the last equality follows from the compatibility condition \eqref{eq_hp_comp}. Hence $u_1$ is a solution of $L_\frac{\pi}{2}u_1=(G,g)$ and it is unique because of Theorem \ref{th_hp_weak_sol},~3. Moreover, we have
	\[
	\|u_1\|_{H^2_{(0,\gamma)}(\R^2_+)}\leq C\|\tilde{u}\|_{H^2_{(\gamma)}(\R_+)}\leq Cc_{\overline{\gamma}}(1+\frac{1}{\gamma^2})\|\tilde{G}\|_{L^2_{(\gamma)}(\R_+)}\leq C_{\overline{\gamma}}(1+\frac{1}{\gamma^2})\|G\|_{L^2_{(0,\gamma)}(\R^2_+)}.
	\]
	Altogether this proves the claim.\qedhere$_{2.}$\end{proof}

\begin{proof}[Ad 3] Via \eqref{eq_hp_orth_decomp} we have isomorphic splitting operators from $H^k_{(0,\gamma)}(\R^2_+)$ onto the direct sum $H^{k,\perp}_{(0,\gamma)}(\R^2_+)\oplus H^{k,\parallel}_{(0,\gamma)}(\R^2_+)$ for all $k\in\N_0$ and the operator norms for fixed $k$ are estimated by a constant independent of $\gamma\in(0,\overline{\gamma}]$. Therefore the claim follows from 1.~and 2.\qedhere$_{3.}$
\end{proof}

\begin{Theorem}[\textbf{Solution Operators for Decay in $(R,H)$}]\label{th_hp_exp2}
	Let $\overline{\gamma}>0$ be such that Theorem \ref{th_hp_exp1},~3.~holds for $\gamma\in(0,\overline{\gamma}]$. Then there is a non-decreasing $\overline{\beta}:(0,\overline{\gamma}]\rightarrow(0,\infty)$ such that 
	\[
	L_\frac{\pi}{2}:H^2_{(\beta,\gamma)}(\R^2_+)\rightarrow Y_{(\beta,\gamma)}:=\left\{(G,g)\in L^2_{(\beta,\gamma)}(\R^2_+)\times H^\frac{1}{2}_{(\beta)}(\R):\int_{\R^2_+}G\theta_0'+\int_\R g\theta_0'=0\right\}
	\]
	is an isomorphism for all $\beta\in[0,\overline{\beta}(\gamma)]$ and the operator norm of the inverse is bounded by $\tilde{C}(1+\frac{1}{\gamma^2})$ with $\tilde{C}$ independent of $(\beta,\gamma)$. 
\end{Theorem}

\begin{proof}
	The idea is similar as in the proof of Theorem \ref{th_hp_exp1},~1. $L_\frac{\pi}{2}$ is well-defined in the spaces due to Theorem \ref{th_hp_weak_sol},~4. In order to solve $L_\frac{\pi}{2}u=(G,g)\in Y_{(\beta,\gamma)}$, we make the ansatz $u=e^{-\beta\eta(R)}v$ for $v\in H^2_{(0,\gamma)}(\R^2_+)$, where $\eta:\R\rightarrow\R$ is as in Definition \ref{th_exp_def1}, 3. We compute
	\begin{align*}
	\partial_Ru&=e^{-\beta\eta(R)}[\partial_Rv-\beta\eta'(R)v],\\ \partial_R^2u&=e^{-\beta\eta(R)}[\partial_R^2v-2\beta \eta'(R)\partial_Rv+v(\beta^2\eta'(R)^2-\beta \eta''(R))].
	\end{align*}
	Therefore for $v\in H^2_{(0,\gamma)}(\R^2_+)$ we consider the equation 
	\begin{align}\label{eq_hp_exp_sol2}
	L_\frac{\pi}{2}v+(N_{(\beta,\gamma)}v,0)=e^{\beta\eta(R)}(G,g),\quad N_{(\beta,\gamma)}v:=2\beta\eta'\partial_Rv-v(\beta^2(\eta')^2-\beta\eta'').
	\end{align}
	There is a problem with the compatibility condition \eqref{eq_hp_comp} here. For $L_{\frac{\pi}{2}}v$ the latter is valid by Theorem \ref{th_hp_weak_sol},~4., but for $(N_{(\beta,\gamma)}v,0)$ and $e^{\beta\eta(R)}(G,g)$ it does not hold necessarily. Therefore we enforce the condition on both sides artificially and look at the adjusted equation
	\begin{align}\label{eq_hp_exp_sol3}
	L_\frac{\pi}{2}v+\tilde{N}_{(\beta,\gamma)}v&=(\overline{G},\overline{g}),\\\notag
	\tilde{N}_{(\beta,\gamma)}v&:=\left(N_{(\beta,\gamma)}v, -\left[\int_{\R^2_+}N_{(\beta,\gamma)}v\theta_0'\right]\frac{\theta_0'}{\|\theta_0'\|_{L^2(\R)}^2}\right),
	\\\notag (\overline{G},\overline{g})&:=e^{\beta\eta(R)}(G,g)-\left(0,\left[\int_{\R^2_+}e^{\beta\eta(R)}G\theta_0'+\int_{\R}e^{\beta\eta(R)}g\theta_0'\right]\frac{\theta_0'}{\|\theta_0'\|_{L^2(\R)}^2}\right).
	\end{align}
	
	In order to solve \eqref{eq_hp_exp_sol3}, we observe that $\tilde{N}_{(\beta,\gamma)}$ is a bounded linear operator from $H^2_{(0,\gamma)}(\R^2_+)$ to $Y_{(0,\gamma)}$ with norm estimated by $C(\beta+\beta^2)$ for all $\beta\geq 0$ and $\gamma\in(0,\overline{\gamma}]$. Moreover, Theorem \ref{th_hp_exp1},~3.~yields that $L_\frac{\pi}{2}$ is an isomorphism in these spaces and that the inverse is bounded by $\overline{C}(1+\frac{1}{\gamma^2})$ for all $\gamma\in(0,\overline{\gamma}]$. We choose $\overline{\beta}=\overline{\beta}(\gamma)$ such that $C(\overline{\beta}+\overline{\beta}^2)\leq 1/[2\overline{C}(1+\frac{1}{\gamma^2})]$ and such that $\overline{\beta}:(0,\overline{\gamma}]\rightarrow(0,\infty)$ is non-decreasing. Then a Neumann series argument yields that $L_\frac{\pi}{2}+\tilde{N}_{(\beta,\gamma)}$ is invertible from $H^2_{(0,\gamma)}(\R^2_+)$ onto $Y_{(0,\gamma)}$ for all $\beta\in[0,\overline{\beta}(\gamma)]$, $\gamma\in(0,\overline{\gamma}]$ and the norm of the inverse is bounded by $2\overline{C}(1+\frac{1}{\gamma^2})$.
	
	Now let $\beta\in[0,\overline{\beta}(\gamma)]$, $\gamma\in(0,\overline{\gamma}]$ and $v\in H^2_{(0,\gamma)}(\R^2_+)$ solve \eqref{eq_hp_exp_sol3} for $(G,g)\in Y_{(\beta,\gamma)}$. Then $u:=e^{-\beta\eta(R)}v\in H^2_{(\beta,\gamma)}(\R^2_+)$ is a solution of
	\[
	L_\frac{\pi}{2}u=(G,g)+\left(0,\frac{e^{-\beta\eta(R)}\theta_0'}{\|\theta_0'\|_{L^2(\R)}^2}\left[\int_{\R^2_+}(N_{(\beta,\gamma)}v+e^{\beta\eta(R)}G)\theta_0'+\int_\R e^{\beta\eta(R)}g\theta_0'\right]\right).
	\]
	The compatibility condition \eqref{eq_hp_comp} holds for $L_\frac{\pi}{2}u$ and $(G,g)$. Since $\int_{\R}e^{-\beta\eta(R)}\theta_0'(R)^2\,dR$ is positive, it follows that the second term is zero for the solution, i.e.~$u$ is a solution of $L_\frac{\pi}{2}u=(G,g)$. By construction or alternatively by Theorem \ref{th_hp_weak_sol},~3.~the solution is unique and we have the estimate
	\[
	\|u\|_{H^2_{(\beta,\gamma)}(\R^2_+)}=\|v\|_{H^2_{(0,\gamma)}(\R^2_+)}\leq 2\overline{C}(1+\frac{1}{\gamma^2})\|(\overline{G},\overline{g})\|_{Y_{(0,\gamma)}}\leq\tilde{C}(1+\frac{1}{\gamma^2})\|(G,g)\|_{Y_{(\beta,\gamma)}}
	\] 
	with $\tilde{C}>0$ independent of $\beta, \gamma$ and the functions. This proves the theorem.
\end{proof}

\begin{Theorem}[\textbf{Solution Operators for Higher Regularity}]\label{th_hp_exp3}
	Let $\overline{\beta}:(0,\overline{\gamma}]\rightarrow(0,\infty)$ and $\overline{\gamma}>0$ be as in Theorem \ref{th_hp_exp2}. Then for all $k\in\N_0$, $\gamma\in(0,\overline{\gamma}]$ and $\beta\in[0,\overline{\beta}(\gamma)]$ it follows that
	\[
	L_\frac{\pi}{2}:H^{k+2}_{(\beta,\gamma)}(\R^2_+)\rightarrow Y_{(\beta,\gamma)}^k:=
	\left\{(G,g)\in H^k_{(\beta,\gamma)}(\R^2_+)\times H^{k+\frac{1}{2}}_{(\beta)}(\R):
	\int_{\R^2_+}G\theta_0'+\int_\R g\theta_0'=0\right\}
	\]
	is invertible and the operator norm of the inverse is bounded by $C(k)(1+\frac{1}{\gamma^2})^{k+1}$.
\end{Theorem}

\begin{proof}
	$L_\frac{\pi}{2}$ is a well-defined, bounded linear operator in the above spaces. Let $(G,g)\in Y_{(\beta,\gamma)}^k$. Then by Theorem \ref{th_hp_exp2} there is a unique $u\in H^2_{(\beta,\gamma)}(\R^2_+)$ that solves $L_\frac{\pi}{2}u=(G,g)$. By regularity theory, cf.~Corollary \ref{th_hp_weak_sol_2},~2., it follows that $u\in H^{k+2}(\R^2_+)$. 
	
	Now we show $\partial_R^lu\in H^2_{(\beta,\gamma)}(\R^2_+)$ for all $l=1,...,k$ and suitable estimates. To this end we apply $\partial_R^l$ to the equation and get
	\[
	L_\frac{\pi}{2}\partial_R^l u =\left(\partial_R^l G-\sum_{j=1}^l
	\begin{pmatrix}l \\j\end{pmatrix}
	\partial_R^j(f''(\theta_0))\partial_R^{l-j}u,\partial_R^lg\right)=:(G_l,g_l).
	\]
	First we consider $l=1$. It holds $G_1=\partial_RG-\partial_R(f''(\theta_0))u\in L^2_{(\beta,\gamma)}(\R^2_+)$ and $g_1\in H^{\frac{1}{2}}_{(\beta)}(\R)$. Moreover, due to Theorem \ref{th_hp_weak_sol},~4.~the compatibility condition \eqref{eq_hp_comp} holds for $(G_1,g_1)$. Therefore Theorem \ref{th_hp_exp2} and the uniqueness of solutions in Theorem \ref{th_hp_weak_sol},~3.~implies $\partial_Ru\in H^2_{(\beta,\gamma)}(\R^2_+)$ and
	\begin{align*}
	\|\partial_Ru\|_{H^2_{(\beta,\gamma)}(\R^2_+)}&\leq 
	\tilde{C}(1+\frac{1}{\gamma^2})\|(G_1,g_1)\|_{Y_{(\beta,\gamma)}}\\ 
	&\leq C(1+\frac{1}{\gamma^2})(\|G\|_{H^1_{(\beta,\gamma)}(\R^2_+)}+\|g\|_{H^{\frac{3}{2}}_{(\beta)}(\R)}+\|u\|_{L^2_{(\beta,\gamma)}(\R^2_+)}),
	\end{align*}
	where $\|u\|_{L^2_{(\beta,\gamma)}(\R^2_+)}\leq\tilde{C}(1+\frac{1}{\gamma^2})\|(G,g)\|_{Y_{(\beta,\gamma)}}$ and we used the product rule to rewrite and estimate $e^{\beta\eta(R)}\partial_Rg$.  This shows the case $l=1$. By mathematical induction on $l$ it follows that $\partial_R^l u\in H^2_{(\beta,\gamma)}(\R^2_+)$ for $l=1,...,k$ and $\|\partial_R^l u\|_{H^2_{(\beta,\gamma)}(\R^2_+)}\leq C(k)(1+\frac{1}{\gamma^2})^{k+1}$.
	
	The remaining assertions and estimates will be shown by differentiating and rearranging the first equation in $L_\frac{\pi}{2}u=(G,g)$. For $l=0,...,k$ we have
	\[
	\partial_R^l\partial_H^2u = -\partial_R^l[-\partial_R^2u+f''(\theta_0(R))u-G].
	\]
	For $k=1$ and $l=1$ this implies $\partial_H^2u\in H^1_{(\beta,\gamma)}(\R^2_+)$ together with a suitable estimate. Hence in the case $k=1$ we are done. Now let $k\geq 2$. Then we obtain $\partial_R^l\partial_H^2u\in H^2_{(\beta,\gamma)}(\R^2_+)$ for all $l=0,...,k-2$ with appropriate estimates. This also shows the case $k=2$. Now let $k\geq 3$. Applying $\partial_H^2$ to the above equation and similar arguments as before yield $\partial_H^4u\in H^1_{(\beta,\gamma)}(\R^2_+)$ in the case $k=3$ and $\partial_R^l\partial_H^4u\in H^2_{(\beta,\gamma)}(\R^2_+)$ for all $k\geq 4$ and $l=0,...,k-4$. Additionally, one also obtains suitable estimates. One can complete the argument with an induction proof.
\end{proof}

\begin{Remark}[\textbf{Dependence on Parameters}]\upshape \label{th_hp_exp_sol_rem}
	When the right hand sides $(G,g)$ depend on independent variables, e.g.~time $t\in[0,T]$, one directly obtains a solution $u$ with the same regularity with respect to those variables because we have linear, bounded solution operators in Theorems \ref{th_hp_exp1}-\ref{th_hp_exp3}. E.g.~if for $(\beta,\gamma)$ as in Theorem \ref{th_hp_exp3} and $n,k\in\N_0$ we have
	\[
	(G,g)\in C^n([0,T],H^k_{(\beta,\gamma)}(\R^2_+)\times H^{k+\frac{1}{2}}_{(\beta)}(\R))\quad\text{ with }\int_{\R^2_+}G\theta_0'+\int_\R g\theta_0'=0,
	\] 
	then there is exactly one solution $u\in C^n([0,T],H^{k+2}_{(\beta,\gamma)}(\R^2_+))$ of $L_\frac{\pi}{2}u=(G,g)$. By embeddings, this can e.g.~be applied for sufficiently smooth right hand sides with pointwise exponential decay for the functions and enough derivatives.
\end{Remark}

\subsection{Some Vector-valued ODE Problems on $\R$}\label{sec_ODE_vect}
The structure of this section is similar to Section \ref{sec_ODE_scalar} which is the analogue in the scalar case. We consider vector-valued ODEs appearing in the inner asymptotic expansion of \hyperlink{vAC}{(vAC)} and also the linear operator belonging to a linearized ODE. In order to solve the linearized ODE, we will assume that the kernel of the linearization is $1$-dimensional. The latter is fulfilled for a typical potential, cf.~the example in Remark \ref{th_ODE_vect_lin_op_rem} below. 

Let $W:\R^m\rightarrow\R$ be as in Definition \ref{th_vAC_W} and $\vec{u}_\pm$ be any distinct pair in $\{ \vec{a},\vec{b} \}$ or $\{\vec{x}_1,\vec{x}_3,\vec{x}_5\}$, respectively. From now on, we fix $\vec{u}_\pm$. 

\subsubsection{The Nonlinear ODE}\label{sec_ODE_vect_nonlin}
The nonlinear ODE problem in the lowest order is the following: Find $\vec{u}:\R\rightarrow\R^m$ smooth with suitable decay such that
\begin{align}\label{eq_ODE_vect}
-\vec{u}'' + \nabla W(\vec{u}) = 0\quad\text{ on }\R, \quad \lim_{z\rightarrow\pm\infty}\vec{u}(z)=\vec{u}_\pm.
\end{align}

\begin{Theorem}\label{th_ODE_vect}
	Let $m$, $W$, $\vec{u}_\pm$ be as above and let $\lambda>0$ be such that $D^2 W(\vec{u}_\pm)\geq\lambda I$. Then there is a smooth solution $\vec{u}:\R\rightarrow\R^m$ to \eqref{eq_ODE_vect} such that 
	\[
	\partial_z^k[\vec{u}-\vec{u}_\pm](z)=\Oc(e^{-\beta|z|})\quad\text{ for }z\rightarrow\pm\infty\text{ and all }k\in\N_0, \beta\in\left(0,\sqrt{\lambda/2}\right).
	\] 
	Moreover, $\vec{u}$ can be chosen $R_{\vec{u}_-,\vec{u}_+}$-odd, i.e.~$\vec{u}(-.)=R_{\vec{u}_-,\vec{u}_+}\vec{u}$ with $R_{\vec{u}_-,\vec{u}_+}$ as in Definition \ref{th_vAC_W}. In this case it holds $R_{\vec{u}_-,\vec{u}_+}\vec{u}'|_{z=0}\neq \vec{u}'|_{z=0}$.
\end{Theorem}

\begin{Remark}\upshape\phantomsection{\label{th_ODE_vect_rem}}
	\begin{enumerate}
	\item From now on, we fix a $R_{\vec{u}_-,\vec{u}_+}$-odd solution and simply denote it by $\vec{\theta}_0$.
	\item The proof relies on minimizing an energy over an approporiate set (see below). Similar to Bronsard, Gui, Schatzman \cite{BGS}, Section 2, where the triple-well case is considered, it should be possible to determine the qualitative behaviour of the set of minimizers for both types of $W$ in Definition \ref{th_vAC_W} precisely. E.g.~in the triple-well case the minimizers are trapped in the smaller sector between $\vec{u}_-$ and $\vec{u}_+$. But this is not needed here.
	\end{enumerate}
\end{Remark}

\begin{proof}[Proof of Theorem \ref{th_ODE_vect}]
	Let $\xi:\R\rightarrow\R$ be smooth and odd such that $\xi(z)=\textup{sign}(z)$ for $|z|\geq 1$. Moreover, we define
	\begin{align}\label{eq_ODE_vect_Xi}
	\Xi(\vec{u}_-,\vec{u}_+):=\frac{1}{2}(\vec{u}_- +\vec{u}_+) + \xi\frac{1}{2}(\vec{u}_+ -\vec{u}_-).
	\end{align}
	Then \cite{Kusche}, Section 2.1 for potentials $W$ as in Definition \ref{th_vAC_W},~1.~and \cite{BGS}, Section 2 for triple-well potentials $W$ in Definition \ref{th_vAC_W},~2., respectively, yield that
	\[
	E: \Xi(\vec{u}_-,\vec{u}_+) + H^1(\R)^m \rightarrow\R : \vec{u} \mapsto \int_\R \frac{1}{2}|\vec{u}'|^2 + W(\vec{u})\,dz
	\]
	admits a global minimizer $\vec{u}$ that satisfies $\vec{u}\in C^3(\R)^m \cap (H^2(\R)^m + \Xi(\vec{u}_-,\vec{u}_+))$ and is  $R_{\vec{u}_-,\vec{u}_+}$-odd. Moreover, $\vec{u}$ satisfies \eqref{eq_ODE_vect} and 
	\[
	\partial_z^k[\vec{u}-\vec{u}_\pm](z)=\Oc(e^{-\beta|z|})\quad\text{ for }z\rightarrow\pm\infty\text{ and }k=0,1,2,3, \beta\in(0,\sqrt{\lambda/2}).
	\]
	Furthermore, one obtains $\vec{u}\in C^{k+1}(\R)^m \cap (H^k(\R)^m + \Xi(\vec{u}_-,\vec{u}_+))$ for all $k\in\N$, $k\geq 2$ and the decay properties by induction and differentiating the equation. 
	
	Finally, we show $R_{\vec{u}_-,\vec{u}_+}\vec{u}'|_{z=0}\neq \vec{u}'|_{z=0}$ for any smooth $\vec{u}:\R\rightarrow\R^m$ that solves \eqref{eq_ODE_vect} and is $R_{\vec{u}_-,\vec{u}_+}$-odd. This will be shown by contradiction with uniqueness for the initial value ODE problem for the part of $\vec{u}$ orthogonal to the hypersurface inbetween $\vec{u}_-$ and $\vec{u}_+$. Therefore let
	\[
	\vec{v}^\perp := 
	\frac{1}{2}P_{\textup{span}\{\vec{u}_+-\vec{u}_-\}}[	\vec{v}-R_{\vec{u}_-,\vec{u}_+}	\vec{v}]\quad\text{ for every }\vec{v}\in \R^m.
	\]
	Then $(\vec{u}')^\perp: \R\rightarrow\R$ is smooth and solves $-[(\vec{u}')^\perp]'' = [D^2W(\vec{u})\vec{u}']^\perp$. Since $\vec{u}$ is $R_{\vec{u}_-,\vec{u}_+}$-odd, this also holds for $\vec{u}''$ and hence $[(\vec{u}')^\perp]'(0)=0$. Due to the boundary condition in \eqref{eq_ODE_vect} we obtain $(\vec{u}')^\perp\not\equiv 0$ . Therefore $(\vec{u}')^\perp(0)\neq 0$, otherwise we get a contradiction to $(\vec{u}')^\perp\equiv 0$ due to ODE-theory. This proves $R_{\vec{u}_-,\vec{u}_+}\vec{u}'|_{z=0}\neq \vec{u}'|_{z=0}$.
\end{proof}

\subsubsection{The Linearized Operator}\label{sec_ODE_vect_lin_op}
	We look at the operator obtained by linearization of the left hand side of the ODE \eqref{eq_ODE_vect} at $\vec{\theta}_0$, i.e.
	\begin{align}\label{eq_ODE_vect_L0}
	\check{L}_0: H^2(\R,\K)^m\subseteq L^2(\R,\K)^m\rightarrow L^2(\R,\K)^m:\vec{u}\mapsto \check{\Lc}_0\vec{u}:=\!\left[-\frac{d^2}{dz^2} + D^2W(\vec{\theta}_0)\right]\!\vec{u}.
	\end{align}
\begin{Remark}[\textbf{Assumption 				$\dim\ker\check{L}_0=1$}]\upshape\label{th_ODE_vect_lin_op_rem}
	$\vec{\theta}_0 '$ is an element of $\ker \check{L}_0$ and $\vec{\theta}_0'\not\equiv 0$ due to Theorem \ref{th_ODE_vect}. In order to have a spectral gap property that is needed for solving the vector-valued linearized ODE and the vector-valued $\R^2_+$-model problem in the next sections, we \textit{assume} $\dim \ker\check{L}_0 =1$ (this is independent of $\K$ since $D^2W(\vec{\theta}_0)$ is real-valued). This is reasonable, cf.~\cite{Kusche}, Section 3.4 for a typical triple-well potential that fulfils this. Note that the assumption should be stable under suitable \enquote{small} perturbations of the potential $W$ due to the upper continuity of the nullity index for (semi-)Fredholm operators, cf.~Kato \cite{Kato}, Theorem 5.22.
\end{Remark}
		
\begin{Lemma}\label{th_ODE_vect_lin_op}
	Let $\check{L}_0$ be as above and $\K=\R$ or $\C$. Then $\check{L}_0$ is self-adjoint, $\check{L}_0\geq 0$ and 
	\[
	\sigma_e(\check{L}_0)=[\min\{\sigma(D^2 W(\vec{u}_\pm))\},\infty).	
	\]
	In particular $\sigma_d(\check{L}_0)\subset [0, \min\{\sigma(D^2 W(\vec{u}_\pm))\})$. Moreover, if $\dim \ker\check{L}_0 =1$, then with $(\ker\check{L}_0)^\perp:=
	\{\vec{w}\in L^2(\R,\C)^m:(\vec{w},\vec{\theta}_0')_{L^2}=0\}$ it holds
	\begin{align*} 
	0<\check{\nu}_0
	&:=\inf_{\vec{w}\in H^2(\R,\C)^m\cap\textup{span}\{\vec{\theta}_0'\} ^\perp,\|\vec{w}\|_{L^2}=1}
	(\check{L}_0 \vec{w},\vec{w})_{L^2(\R,\C)^m}\\
	&\phantom{:}=\inf_{\vec{w}\in H^1(\R,\C)^m\cap\textup{span}\{\vec{\theta}_0'\} ^\perp,\|\vec{w}\|_{L^2}=1}\int_{\R}|\vec{w}'|^2+(D^2W(\vec{\theta}_0)\vec{w},\vec{w})_{\C^m}\,dz.
\end{align*}
\end{Lemma}

\begin{proof}
	That $\check{L}_0$ is self-adjoint e.g.~follows from \cite{Kusche}, Proposition 1.1 or with a typical argument with the symmetry of $\check{L}_0$ and $\rho(\check{L}_0)\cap\R\neq\emptyset$ by the Lax-Milgram Theorem. The property $\check{L}_0\geq 0$ is an outcome of the energetic approach in the proof of Theorem \ref{th_ODE_vect} (first for $\K=\R$, then it follows for $\K=\C$). Moreover, one can use Persson's Theorem and Weyl sequences to show the identity for $\sigma_e(\check{L}_0)$, cf.~the proof of Proposition 2.1 in \cite{Kusche}. The remaining assertions can be deduced in the analogous way as in the scalar case, cf.~the proof of \cite{AbelsMoser}, Lemma 2.5.
\end{proof}
	   
\subsubsection{The Linearized ODE}\label{sec_ODE_vect_lin}
We have to consider the ODE that arises from the linearization of \eqref{eq_ODE_vect} at $\vec{\theta}_0$. More precisely, for $\vec{A}:\R\rightarrow\R^m$ with suitable regularity and decay we seek a function $\vec{u}:\R\rightarrow\R^m$ such that
\begin{align}\label{eq_ODE_vect_lin}
\check{L}_0\vec{u}=\vec{A}\quad\text{ and }\quad\check{B}\vec{u}=0,
\end{align}
where $\check{B}\in\Lc(H^l(\R)^m,\R)$ for some $l\in\{0,1,2\}$ with $\check{B}\vec{\theta}_0'\neq 0$. As before we make the assumption $\dim\ker\check{L}_0 =1$, cf.~Remark \ref{th_ODE_vect_lin_op_rem}.

\begin{Remark}\upshape\label{th_ODE_vect_lin_rem1}
	The additional condition with $\check{B}$ is imposed in order to get uniqueness below. The natural choice from a functional analytic point of view is $\check{B}:=(.,\vec{\theta}_0')_{L^2(\R)^m}:L^2(\R)^m\rightarrow\R$. However, the canonical choice for the application later is  
	\[
	\check{B}:=
	(\vec{u}_- -\vec{u}_+)^\top
	[R_{\vec{u}_-,\vec{u}_+}-I](.)|_{z=0}:H^1(\R)^m\rightarrow\R,
	\] 
	where $R_{\vec{u}_-,\vec{u}_+}$ is defined analogously to Definition \ref{th_vAC_W}. The latter fulfils $\check{B}\vec{\theta}_0'\neq 0$ due to Theorem \ref{th_ODE_vect} and heuristically the condition $\check{B}\vec{u}=0$  means that $\vec{u}|_{z=0}$ is precisely in the middle of the two phases. E.g.~for $\vec{u}_-=(-1,1)^\top$ and $\vec{u}_+=(1,1)^\top$ the latter reduces to $\vec{u}(0)_1=0$.
\end{Remark}

\begin{Theorem}\label{th_ODE_vect_lin}
	Let $\dim\ker\check{L}_0 =1$, cf.~Remark \ref{th_ODE_vect_lin_op_rem}. Then it holds
	\begin{enumerate}
	\item Let $\vec{A}\in L^2(\R)^m$. Then there is a $\vec{u}\in H^2(\R)^m$ such that $\check{L}_0 \vec{u}=\vec{A}$ if and only if $\int_\R\vec{A}\cdot \vec{\theta}_0'=0$. In this case $\vec{u}$ is unique up to multiples of $\vec{\theta}_0'$. In particular, \eqref{eq_ODE_vect_lin} admits a unique solution $\vec{u}\in H^2(\R)^m$ if and only if $\int_\R\vec{A}\cdot \vec{\theta}_0'=0$. Moreover, for all $k\in\N_0$
	\[
	\check{L}_0:\left\{\vec{u}\in H^{k+2}(\R)^m:\check{B}\vec{u}=0\right\}\rightarrow\left\{\vec{A}\in H^k(\R)^m:\int_{\R}\vec{A}\cdot\vec{\theta}_0'=0\right\}
	\]
	is an isomorphism and the inverse is bounded by some $c(\check{B},k)>0$.
	\item There is a $\check{\beta}_0>0$ small such that for all $\beta\in(0,\check{\beta}_0)$ and $k\in\N_0$
	\[
	\check{L}_0:\left\{\vec{u}\in H^{k+2}_{(\beta)}(\R)^m:\check{B}\vec{u}=0\right\}\rightarrow \left\{ \vec{A}\in H^k_{(\beta)}(\R)^m : \int_\R\vec{A}\cdot\vec{\theta}_0'=0 \right\}
	\]
	is an isomorphism and the norm of the inverse is bounded by a constant $C(\check{B},k)>0$. 
	\end{enumerate}
\end{Theorem}

\begin{Remark}\upshape\phantomsection{\label{th_ODE_vect_lin_rem2}}
	\begin{enumerate}
	\item \textit{Dependence on parameteres.} In Theorem \ref{th_ODE_vect_lin} we obtained linear solution operators in suitable spaces with exponential decay. Therefore, if the right hand side in \eqref{eq_ODE_vect_lin} depends on additional parameters and satisfies such exponential decay estimates, the regularity and decay carries over to the solution.
	\item Using Theorem \ref{th_ODE_vect_lin} one can directly obtain a result for right hand sides $\vec{A}$ that converge with appropriate rate to non-zero vectors $\vec{A}^\pm\in\R^m$ at $\pm\infty$. The idea is as follows: set
	\[
	\vec{U}_\pm := [D^2W(\vec{u}_\pm)]^{-1}(\vec{A}_\pm)
	\quad\text{ and }\quad \vec{U}:=\Xi(\vec{U}_-,\vec{U}_+),
	\] 
	where the latter is analogous to \eqref{eq_ODE_vect_Xi}. Then formally it holds $\check{\Lc}_0\vec{u}=\vec{A}$ if and only if $\check{\Lc}_0(\vec{u}-\vec{U})=\vec{A}-\vec{A}_0$, where $\vec{A}_0:=\check{L}_0\vec{U}$. To this equation one can apply the results in Theorem \ref{th_ODE_vect_lin} for suitable $\vec{A}$. Note that the compatibility condition is the same as before since $\int_\R \vec{A}_0\cdot\vec{\theta}_0'=0$ due to integration by parts. The case $\vec{A}^\pm\neq0$ is not needed here but may be interesting for more sophisticated equations, e.g.~a vector-valued Cahn-Hilliard equation. Moreover, the idea to reduce to $\vec{A}_\pm=0$ could also be helpful for triple junction cases. Finally, note the analogy in the limits at infinity to the ones in Theorem \ref{th_ODE_lin}.
	\item Kusche \cite{Kusche}, Proposition 1.6 and Corollary 1.2 yield pointwise exponential decay estimates. The latter would be enough for our purpose. But the downside is that the exponent shrinks. Moreover, the proof of Theorem \ref{th_ODE_vect_lin} is self-contained and simpler. However, in \cite{Kusche} there are also uniform estimates for finite large intervals which are important for the spectral estimates, cf.~Section \ref{sec_SE_1Dvect} below.
	\end{enumerate}
\end{Remark}

\begin{proof}[Proof of Theorem \ref{th_ODE_vect_lin}. Ad 1]
		Consider $\K=\R$ in Lemma \ref{th_ODE_vect_lin_op}. With the latter one can show that
		\[
		\check{L}_0^\perp:=\check{L}_0|_{(\ker\check{L}_0)^\perp}: H^2(\R)^m\cap (\ker\check{L}_0)^\perp \rightarrow (\ker\check{L}_0)^\perp
		\] 
		is well-defined, self-adjoint and $\sigma(\check{L}_0)=\sigma(\check{L}_0^\perp)\cup\{0\}$. As in the proof of Lemma 2.5 in \cite{AbelsMoser} it follows that $\sigma(\check{L}_0^\perp)=\sigma(\check{L}_0)\textbackslash\{0\}$ and hence $0\in\rho(\check{L}_0^\perp)$. Moreover, the graph norm is equivalent to the $H^2(\R)^m$-norm. This can e.g.~be seen via direct estimates.
		
		Now let $\vec{v}\in H^2(\R)^m$ solve $\check{L}_0\vec{v}=\vec{A}$ for some $\vec{A}\in L^2(\R)^m$. Then due to integration by parts it holds $(\vec{A},\vec{\theta}_0')_{L^2(\R)^m}=(\vec{v},\check{L}_0\vec{\theta}_0')_{L^2(\R)^m}=0$. On the other hand, let $\vec{A}\in L^2(\R)^m$ be such that $\int_\R\vec{A}\cdot \vec{\theta}_0'=0$. The latter is equivalent to $\vec{A}\in (\ker\check{L}_0)^\perp \cap L^2(\R)^m$. The above considerations yield a unique solution $\vec{v}\in H^2(\R)^m\cap (\ker\check{L}_0)^\perp$ to $\check{L}_0^\perp\vec{v}=\vec{A}$ and $\|\vec{v}\|_{H^2(\R)^m}\leq C\|\vec{A}\|_{L^2(\R)^m}$. Because of the assumption $\dim\ker\check{L}_0 =1$ it holds $\ker \check{L}_0=\textup{span} \{\vec{\theta}_0'\}$. This implies the uniqueness in $H^2(\R)^m$ up to multiples of $\vec{\theta}_0'$. Due to $\check{B}\in\Lc(H^l(\R)^2,\R)$ for some $l\in\{0,1,2\}$ and $\check{B}\vec{\theta}_0'\neq0$, we obtain that
		\[
		\vec{u}:=\vec{v}-\frac{\check{B}\vec{v}}{\vec{B}\vec{\theta}_0'} \vec{\theta}_0' \in H^2(\R)^m
		\]
		is well-defined, the unique solution to \eqref{eq_ODE_vect_lin} and that the estimate $\|\vec{u}\|_{H^2(\R)^m}\leq C(\check{B})\|\vec{A}\|_{L^2(\R)^m}$ holds. In particular the claim follows for $k=0$. For arbitrary $k\in\N$ let $\vec{A}\in H^k(\R)^m$ and $\vec{u}\in H^2(\R)^m$ solve $\check{L}_0\vec{u}=\vec{A}$. Then it follows iteratively from the equation that $\vec{u}\in H^{k+2}(\R)^m$. In particular $\check{L}_0$ is an isomorphism with respect to the spaces in the theorem for all $k\in\N$. The estimate for the inverse can also be shown iteratively using the equation.\qedhere$_{1.}$\end{proof}
		
		\begin{proof}[Ad 2] We prove this with similar ideas as in Section \ref{sec_hp_exp_sol}, i.e.~for $\vec{A}\in L^2_{(\beta)}(\R)^m$ such that $\int_\R\vec{A}\cdot\vec{\theta}_0'=0$ we make the ansatz $\vec{u}=e^{-\beta\eta}\vec{v}$ with $\vec{v}\in H^2(\R)^m$, where $\eta:\R\rightarrow\R$ is as in Definition \ref{th_exp_def1}, 3. It holds
		\[
		\vec{u}''=e^{-\beta\eta}[\vec{v}''-2\beta \eta'\vec{v}'+\vec{v}(\beta^2(\eta')^2-\beta \eta'')].
		\]
		Therefore we consider the equation
		\[
		\check{L}_0\vec{v} + \check{N}_{\beta}\vec{v} = e^{\beta\eta}\vec{A}, \quad \check{N}_{\beta}\vec{v}:=2\beta\eta'\vec{v}' - \vec{v}(\beta^2(\eta')^2-\beta\eta'').
		\]
		In order to solve this, we want to use the spaces in 1. One problem is the compatibility condition for $\check{N}_{\beta}\vec{v}$ and the right hand side. Therefore we solve a different equation, where suitable terms are subtracted that enforce the compatibility condition. Moreover, $\vec{v}$ should satisfy $\check{B}_{(\beta)}\vec{v}=0$, where $\check{B}_{(\beta)}:=\check{B}(e^{-\beta\eta}.)$. This is a problem since then the space would depend on $\beta$. Therefore note that it is enough to show the assertion e.g.~for $\check{B}=\check{B}_0:=(.)_1|_{z=0}$ since in the end one can subtract $\vec{\theta}_0' \check{B}\vec{u}/\check{B}\vec{\theta}_0'$ to get a solution of \eqref{eq_ODE_vect_lin}. If $\check{\beta}_0>0$ is small enough, in particular $\check{\beta}_0\leq\sqrt{\lambda/2}$ with $\lambda$ as in Theorem \ref{th_ODE_vect}, the assertion carries over to general $\check{B}$. The advantage of $\check{B}_0$ is that $\check{B}_0(e^{-\beta\eta}\vec{v})=0$ if and only if $\check{B}_0\vec{v}=0$. Therefore the space for $\vec{v}$ can be chosen uniformly in $\beta$. Hence we solve $\check{B}_0\vec{v}=0$ together with
		\begin{align}\label{eq_ODE_vect_lin1}
		\check{L}_0\vec{v}+\check{M}_{\beta}\vec{v}&=
		e^{\beta\eta}\vec{A} - \left[\int_\R e^{\beta\eta}\vec{A}\cdot\vec{\theta}_0' \right] \frac{\vec{\theta}_0'}{\|\vec{\theta}_0'\|_{L^2(\R)^m}},\\
		\check{M}_{\beta}\vec{v}&:=
		\check{N}_{\beta}\vec{v}-\left[\int_\R \check{N}_{\beta}\vec{v}\cdot\vec{\theta}_0' \right] \frac{\vec{\theta}_0'}{\|\vec{\theta}_0'\|_{L^2(\R)^m}}.\notag
		\end{align}
		Using the properties of exponentially weighted Sobolev spaces in Lemma \ref{th_exp1} one can directly show that
		\[
		\check{N}_{\beta}, \check{M}_{\beta} \in \Lc(H^2_{(\beta)}(\R)^m,L^2_{(\beta)}(\R)^m)
		\]
		with norm bounded by $C\beta$ for all $\beta\in[0,\check{\beta}_0)$ and any fixed $\check{\beta}_0>0$. Therefore the first part of the theorem and a Neumann series argument imply that for all $\beta\in[0,\check{\beta}_0)$ and $\check{\beta}_0>0$ small
		\[
		\check{L}_0+\check{M}_{\beta}:
		\left\{\vec{v}\in H^2(\R)^m:\check{B}_0\vec{v}=0\right\}
		\rightarrow
		\left\{\vec{A}\in L^2(\R)^m:\int_{\R}\vec{A}\cdot\vec{\theta}_0'=0
		\right\}
		\]
		is an isomorphism and the norm of the inverse is bounded by a constant independent of $\beta$. Hence \eqref{eq_ODE_vect_lin1} admits a unique solution $\vec{v}\in H^2(\R)^m$ with $\check{B}_0\vec{v}=0$ for all $\beta\in[0,\check{\beta}_0)$. The computations above yield that $\vec{u}:=e^{-\beta\eta}\vec{v}\in H^2_{(\beta)}(\R)^m$ solves
		\[
		\check{L}_0\vec{u} =\vec{A} + \left[-\int_\R e^{\beta\eta}\vec{A}\cdot\vec{\theta}_0' + \int_\R \check{N}_{\beta}\vec{v}\cdot\vec{\theta}_0'\right]\frac{e^{-\beta\eta}\vec{\theta}_0'}{\|\vec{\theta}_0'\| _{L^2(\R)^m} ^2}.
		\]
		By assumption it holds 
		$\int_\R \vec{A}\cdot\vec{\theta}_0'=0$ and because of the first part of the theorem we have $\int_\R \check{L}_0\vec{u}\cdot\vec{\theta}_0'=0$. Due to Theorem \ref{th_ODE_vect} it holds $\vec{\theta}_0'\neq 0$ and hence $\int_\R e^{-\beta\eta}|\vec{\theta}_0'|^2>0$. Therefore 
		\[
		\check{L}_0\vec{u}=\vec{A}\quad\text{ and }\quad\check{B}_0\vec{u}=0.
		\] 
		Finally, the first part in the theorem yields uniqueness and the estimate follows with the above considerations and Lemma \ref{th_exp1}.\qedhere$_{2.}$
\end{proof}

\subsection[A Vector-valued Elliptic Problem on $\R^2_+$ with Neumann Bdry. Cond.]{A Vector-valued Elliptic Problem on $\R^2_+$ with Neumann Boundary Condition}\label{sec_hp_vect}
This section is analogous to Section \ref{sec_hp_90}, where the scalar case was done. Let $W:\R^m\rightarrow\R$ be as in Definition \ref{th_vAC_W} and $\vec{\theta}_0$ be as in Remark \ref{th_ODE_vect_rem},~1. In the contact point expansion for \hyperlink{vAC}{(vAC)} in any dimension $N\geq 2$ the following model problem appears: For suitable data $\vec{G}:\overline{\R^2_+}\rightarrow\R^m$, $\vec{g}:\R\rightarrow\R^m$ find a solution $\vec{u}:\overline{\R^2_+}\rightarrow\R^m$ to the system
\begin{alignat}{2}\label{eq_hp_vect1}
\left[-\Delta+D^2W(\vec{\theta}_0(R))\right]\vec{u}(R,H)&=\vec{G}(R,H)& \quad &\text{ for }(R,H)\in\R^2_+,\\
-\partial_H\vec{u}|_{H=0}(R)&=\vec{g}(R)&\quad &\text{ for }R\in\R.
\label{eq_hp_vect2}
\end{alignat}  
As often in the last Section \ref{sec_ODE_vect} we make the \textit{assumption} that $\dim\ker\check{L}_0=1$, where $\check{L}_0$ is defined in \eqref{eq_ODE_vect_L0}, cf.~also Remark \ref{th_ODE_vect_lin_op_rem}. This implies a useful estimate for functions orthogonal to $\vec{\theta}_0'$, cf.~Lemma \ref{th_ODE_vect_lin_op}. The solution strategy for \eqref{eq_hp_vect1}-\eqref{eq_hp_vect2} is completely analogous to Section \ref{sec_hp_90}. First of all, we show some assertions for weak solutions of the problem in Section \ref{sec_hp_vect_weak_sol}. Then in Section \ref{sec_hp_vect_exp_sol} we obtain solution operators in exponentially weighted Sobolev spaces.

\subsubsection{Weak Solutions and Regularity}\label{sec_hp_vect_weak_sol}
In the vector-valued case weak solutions are defined as follows:
\begin{Definition}\label{th_hp_vect_weak_def}\upshape
	Let $\vec{G}\in L^2(\R^2_+)^m$ and $\vec{g}\in L^2(\R)^m$. Then $\vec{u}\in H^1(\R^2_+)^m$ is called \textit{weak solution} of \eqref{eq_hp_vect1}-\eqref{eq_hp_vect2} if for all $\vec{\varphi}\in H^1(\R^2_+)^m$ it holds that
	\begin{align*}
	\check{a}(\vec{u},\vec{\varphi})&:=\int_{\R^2_+}\nabla \vec{u} :\nabla\vec{\varphi}+(D^2W(\vec{\theta}_0(R))\vec{u},\vec{\varphi})_{\R^m}\,d(R,H)\\
	&\phantom{:}=\int_{\R^2_+}\vec{G}\cdot\vec{\varphi}\,d(R,H)+\int_{\R}\vec{g}(R)\cdot\vec{\varphi}|_{H=0}(R)\,dR.
	\end{align*}
\end{Definition}

We obtain the analogue of Theorem \ref{th_hp_weak_sol}:
\begin{Theorem}\label{th_hp_vect_weak} 
	Let $\dim\ker\check{L}_0=1$ and consider $\vec{G}\in L^2(\R^2_+)^m$ and $\vec{g}\in L^2(\R)^m$. Then
	\begin{enumerate}
		\item $\check{a}:H^1(\R^2_+)^m\times H^1(\R^2_+)^m\rightarrow\R$ is not coercive.
		\item If $\vec{G}(.,H),\vec{g}\perp\vec{\theta}_0'$ for a.e.~$H>0$ in $L^2(\R)^m$, then there exists a weak solution $\vec{u}$ with $\vec{u}(.,H)\perp\vec{\theta}_0'$ for a.e.~$H>0$ and it holds $\|\vec{u}\|_{H^1(\R^2_+)^m}\leq C(\|\vec{G}\|_{L^2(\R^2_+)^m} +\|\vec{g}\|_{L^2(\R)^m})$.
		\item Weak solutions are unique.
		\item If $\vec{G}\cdot\vec{\theta}_0'\in L^1(\R^2_+)$ and $\vec{u}$ is a weak solution with $\partial_H\vec{u}\cdot\vec{\theta}_0'\in L^1(\R^2_+)$, then the following compatibility condition holds:
		\begin{align}\label{eq_hp_vect_comp}
		\int_{\R^2_+}\vec{G}(R,H)\cdot\vec{\theta}_0'(R)\,d(R,H)+\int_{\R}\vec{g}(R)\cdot\vec{\theta}_0'(R)\,dR=0.
		\end{align}
		\item If $\vec{G}\cdot\vec{\theta}_0'\in L^1(\R^2_+)$, then $\tilde{G}(H):=(\vec{G}(.,H),\vec{\theta}_0')_{L^2(\R)^m}$ is well-defined for a.e.~$H>0$ and $\tilde{G}\in L^1(\R_+)\cap L^2(\R_+)$. Moreover, we have the decomposition
		\begin{align}\label{eq_hp_vect_orth_dec}
		\vec{G}=\tilde{G}(H)\frac{\vec{\theta}_0'(R)}{\|\vec{\theta}_0'\|_{L^2(\R)^m}^{2}}+\vec{G}^\perp(R,H),\:
		\vec{g}=(\vec{g},\vec{\theta}_0')_{L^2(\R)^m}\frac{\vec{\theta}_0'(R)}{\|\vec{\theta}_0'\|_{L^2(\R)^m}^{2}}+\vec{g}^\perp(R),
		\end{align}
		where $\vec{G}^\perp\in L^2(\R^2_+)^m, \vec{g}^\perp\in L^2(\R)^m$  with $\vec{G}^\perp(.,H), \vec{g}^\perp\perp\vec{\theta}_0'$ in $L^2(\R)^m$ f.a.e.~$H>0$.
		\item If $\|\vec{G}(.,H)\|_{L^2(\R)^m}\leq C e^{-\nu H}$ for a.e.~$H>0$ and a constant $\nu>0$, then $\vec{G}\cdot\vec{\theta}_0'\in L^1(\R^2_+)$. Let $\tilde{G}$ be defined as in 5.~and the compatibility condition \eqref{eq_hp_vect_comp} hold. Then 
		\begin{align}\label{eq_hp_vect_formula}
		\vec{u}_1(R,H):=-\int_H^\infty\int_{\tilde{H}}^{\infty}\tilde{G}(\hat{H})\,d\hat{H}\,d\tilde{H}\,\frac{\vec{\theta}_0'(R)}{\|\vec{\theta}_0'\|_{L^2(\R)^m}^{2}}
		\end{align}
		is well-defined for a.e.~$(R,H)\in\R^2_+$, $\vec{u}_1\in W^2_1(\R^2_+)^m\cap H^2(\R^2_+)^m$ and $\vec{u}_1$ is a weak solution of \eqref{eq_hp_vect1}-\eqref{eq_hp_vect2} for $\vec{G}-\vec{G}^\perp,\vec{g}-\vec{g}^\perp$ from \eqref{eq_hp_vect_orth_dec} instead of $\vec{G},\vec{g}$.
	\end{enumerate}
\end{Theorem}

\begin{proof}
	One can essentially copy the proof of Theorem \ref{th_hp_weak_sol}. All multiplications of functions that are now vector-valued have to be interpreted as scalar products and $f''(\theta_0)$ has to be replaced by $D^2W(\vec{\theta}_0)$. Moreover, all spaces except for products and for $\tilde{G}$ change to vector-valued ones. Finally, one uses Lemma \ref{th_ODE_vect_lin_op} instead of Lemma 2.5 in \cite{AbelsMoser} to get coercivity of $\check{a}$ on the orthogonal parts and uniqueness of weak solutions.
\end{proof}

Again this yields an existence theorem for weak solutions analogously to Corollary \ref{th_hp_weak_sol_2}:

\begin{Corollary}\label{th_hp_vect_weak2} Let $\dim\ker\check{L}_0=1$. Then it holds
	\begin{enumerate}
		\item Let $\vec{g}\in L^2(\R)^m$, $\vec{G}\in L^2(\R^2_+)^m$ with $\|\vec{G}(.,H)\|_{L^2(\R)^m}\leq Ce^{-\nu H}$ for a.e.~$H>0$ and some $\nu>0$. Let \eqref{eq_hp_vect_comp} hold. Then there is a unique weak solution of \eqref{eq_hp_vect1}-\eqref{eq_hp_vect2}. 
		\item Let $k\in\N_0$ and $\vec{u}\in H^1(\R^2_+)^m$ be a weak solution of \eqref{eq_hp_vect1}-\eqref{eq_hp_vect2} for $\vec{G}\in H^k(\R^2_+)^m$ and $\vec{g}\in H^{k+\frac{1}{2}}(\R)^m$. Then $\vec{u}\in H^{k+2}(\R^2_+)^m\hookrightarrow C^{k,\gamma}(\overline{\R^2_+})^m$ for all $\gamma\in(0,1)$ and 
		\[
		\|\vec{u}\|_{H^{k+2}(\R^2_+)^m}\leq C_k(\|\vec{G}\|_{H^k(\R^2_+)^m}+\|\vec{g}\|_{H^{k+\frac{1}{2}}(\R)^m}+\|\vec{u}\|_{H^1(\R^2_+)^m}).
		\]
	\end{enumerate}
\end{Corollary}

\begin{proof}
	The first part is a direct consequence of Theorem \ref{th_hp_vect_weak}. The second assertion can be shown similar to the proof of Corollary 2.10 in \cite{AbelsMoser},~2.~using iteratively \textit{scalar} regularity theory for every component of the elliptic equation, where $D^2W(\vec{\theta}_0)\vec{u}$ is viewed as part of the right hand side.
\end{proof}

\subsubsection{Solution Operators in Exponentially Weighted Spaces}\label{sec_hp_vect_exp_sol}
With analogous adjustments as above one obtains solution operators in exponentially weighted Sobolev spaces similar to Section \ref{sec_hp_exp_sol}, cf.~Theorems \ref{th_hp_exp1}-\ref{th_hp_exp3}. We will just need the analogue of Theorem \ref{th_hp_exp3}, hence we only formulate the latter in the vector-valued setting:

\begin{Theorem}[\textbf{Solution Operators for the Vector-valued Case}]\label{th_hp_vect_exp_sol}
	Let $\dim\ker\check{L}_0=1$. There exist $\check{\gamma}>0$ and $\check{\beta}:(0,\check{\gamma}] \rightarrow (0,\infty)$ non-decreasing such that 
	\begin{align*}
	\check{L}_\frac{\pi}{2}
	&:= 
	(-\Delta + D^2W(\vec{\theta}_0(R)) , -\partial_H|_{H=0})
	:H^{k+2}_{(\beta,\gamma)}(\R^2_+)^m\rightarrow \check{Y}_{(\beta,\gamma)}^k,\\
	\check{Y}_{(\beta,\gamma)}^k
	&:=
	\left\{
		(\vec{G},\vec{g})\in H^k_{(\beta,\gamma)}(\R^2_+)^m\times H^{k+\frac{1}{2}}_{(\beta)}(\R)^m :
		\int_{\R^2_+} \vec{G}\cdot\vec{\theta}_0'+\int_\R \vec{g}\cdot\vec{\theta}_0'=0
	\right\}
	\end{align*}
	is well-defined and invertible for all $k\in\N_0$, $\gamma\in(0,\check{\gamma}]$ and $\beta\in[0,\check{\beta}(\gamma)]$ and the operator norm of the inverse is bounded by $\check{C}(k)(1+\frac{1}{\gamma^2})^{k+1}$.
\end{Theorem}

\begin{Remark}[\textbf{Dependence on Parameters}]\upshape \label{th_hp_vect_exp_rem}
	For data that depend on other indepenent variables the regularity directly carries over to the solution since we have linear and bounded solution operators. This is analogous to the scalar case, cf.~Remark \ref{th_hp_exp_sol_rem}.
\end{Remark}

\section{Asymptotic Expansions}\label{sec_asym}
In this section we carry out the rigorous asymptotic expansions for \hyperlink{AC}{(AC)} and \hyperlink{vAC}{(vAC)} in the situations mentioned in the introduction. The expansions are based on the curvilinear coordinates from Section \ref{sec_coord} and use the solutions for the model problems in Section \ref{sec_model_problems}.

\subsection{Asymptotic Expansion of (AC) in ND}\label{sec_asym_ACND}
Let $N\geq 2$, $\Omega\subseteq\R^N$ be as in Remark \ref{th_intro_coord},~1.~and $\Gamma:=(\Gamma_t)_{t\in[0,T]}$ be as in Section \ref{sec_coord_surface_requ} with contact angle $\alpha=\frac{\pi}{2}$. We use the notation from Section \ref{sec_coord_surface_requ} and \ref{sec_coordND}. Moreover, let $\delta>0$ be such that the assertions of Theorem \ref{th_coordND} hold for $2\delta$ instead of $\delta$. Finally, we assume that $\Gamma$ evolves according to \eqref{MCF}. Moreover, let $f$ be as in \eqref{eq_AC_fvor1}. Based on $\Gamma$ we construct a smooth approximate solution $u_\varepsilon^A$ to \eqref{eq_AC1}-\eqref{eq_AC3} with $u_\varepsilon^A=\pm 1$ in $Q_T^\pm\setminus\Gamma(2\delta)$, increasingly steep \enquote{transition} from $-1$ to $1$ for $\varepsilon\rightarrow 0$ and such that $\{ u^A_\varepsilon=0 \}$ approaches $\Gamma$ for $\varepsilon\rightarrow 0$. This is the analogous qualitative behaviour as in the 2D-case, cf.~\cite{AbelsMoser}, Section 3. Compared to the latter case, the computations are similar, but more technical and we also iterate the construction. The most striking insight is that in the contact point expansion we also end up with model problems on the half space $\R^2_+$, where elements of $\partial\Sigma$ enter as independent variables. 

Let $M\in\N$ with $M\geq 2$. Then we consider height functions $h_j:\Sigma\times[0,T]\rightarrow\R$ for $j=1,...,M$ and we set $h_\varepsilon:=\sum_{j=1}^M \varepsilon^{j-1} h_j$. Analogously as in the $2$-dimensional case we define $h_{M+1}:=h_{M+2}:=0$ and we introduce the scaled variable
\begin{align}\label{eq_asym_ACND_rho}
\rho_\varepsilon(x,t):=\frac{r(x,t)}{\varepsilon}-h_\varepsilon(s(x,t),t)\quad\text{ for }(x,t)\in\overline{\Gamma(2\delta)}.
\end{align}

In Section \ref{sec_asym_ACND_in} we construct the inner expansion and in Section \ref{sec_asym_ACND_cp} the contact point expansion. Finally, in Section \ref{sec_asym_ACND_uA} we show that the construction yields a suitable approximate solution $u^A_\varepsilon$ to \eqref{eq_AC1}-\eqref{eq_AC3}.

\subsubsection{Inner Expansion of (AC) in ND}\label{sec_asym_ACND_in}
For the inner expansion we consider the following ansatz: Let $\varepsilon>0$ be small and
\[
u^I_\varepsilon:=\sum_{j=0}^{M+1}\varepsilon^ju_j^I,\quad u_j^I(x,t):=\hat{u}_j^I(\rho_\varepsilon(x,t),s(x,t),t)\quad\text{ for }(x,t)\in\overline{\Gamma(2\delta)},
\]
where 
\[
\hat{u}_j^I:\R\times\Sigma\times[0,T]\rightarrow\R: (\rho,s,t)\mapsto \hat{u}_j^I(\rho,s,t)
\] 
for $j=0,...,M+1$. Moreover, we set $u^I_{M+2}:=0$ and $\hat{u}^I_\varepsilon:=\sum_{j=0}^{M+1}\varepsilon^j\hat{u}_j^I$. We will expand \eqref{eq_AC1} for $u_\varepsilon=u^I_\varepsilon$ into $\varepsilon$-series with coefficients in $(\rho_\varepsilon,s,t)$ up to $\Oc(\varepsilon^{M-1})$. This yields equations of analogous form as in \cite{AbelsMoser}, Section 3.1, where $I$ is replaced by $\Sigma$. Therefore we have to compute the action of the differential operators on $u^I_\varepsilon$. 

In the following the surface gradient $\nabla_\Sigma$ (see \cite{Depner}, Definition 2.21) for functions $g:\Sigma\rightarrow\R$ is viewed as a map $\nabla_\Sigma g:\Sigma\rightarrow \R^N$. Then we set $(\nabla_\Sigma)_i g:=(\nabla_\Sigma g)_i$ for $i=1,...,N$. If $g$ depends on other variables as well, the analogous definition applies. 

\begin{Lemma}\label{th_asym_ACND_in_trafo}
	Let $\varepsilon>0$, $\hat{w}:\R\times \Sigma\times[0,T]\rightarrow\R$ be sufficiently smooth and $w:\overline{\Gamma(2\delta)}\rightarrow\R$ be defined by $w(x,t):=\hat{w}(\rho_\varepsilon(x,t),s(x,t),t)$ for all $(x,t)\in\overline{\Gamma(2\delta)}$. Then
	\begin{align*}
	\partial_tw&=\partial_\rho\hat{w}\left[\frac{\partial_tr}{\varepsilon}-(\partial_th_\varepsilon+\partial_ts\cdot \nabla_\Sigma h_\varepsilon)\right]+\partial_ts\cdot \nabla_\Sigma\hat{w}+\partial_t\hat{w},\\
	\nabla w&=\partial_\rho\hat{w}\left[\frac{\nabla r}{\varepsilon}-(D_x s)^\top \nabla_\Sigma h_\varepsilon\right]+(D_x s)^\top \nabla_\Sigma\hat{w},\\
	\Delta w&=\partial_\rho\hat{w}\left[\frac{\Delta r}{\varepsilon}-\left(\Delta s\cdot \nabla_\Sigma h_\varepsilon+\sum_{i,l=1}^N\nabla s_i\cdot\nabla s_l(\nabla_\Sigma)_i(\nabla_\Sigma)_l h_\varepsilon\right)\right]\\
	&+\Delta s\cdot\nabla_\Sigma\hat{w}+\sum_{i,l=1}^N\nabla s_i\cdot\nabla s_l(\nabla_\Sigma)_i(\nabla_\Sigma)_l\hat{w}\\
	&+2\left((D_x s)^\top\nabla_\Sigma\partial_\rho\hat{w}\right)\cdot\left[\frac{\nabla r}{\varepsilon}-(D_xs)^\top\nabla_\Sigma h_\varepsilon\right]+\partial_\rho^2\hat{w}\left|\frac{\nabla r}{\varepsilon}-(D_xs)^\top\nabla_\Sigma h_\varepsilon\right|^2,
	\end{align*}
	where the $w$-terms on the left hand side and derivatives of $r$ or $s$ are evaluated at $(x,t)\in\overline{\Gamma(2\delta)}$, the $h_\varepsilon$-terms at $(s(x,t),t)$ and the $\hat{w}$-terms at $(\rho_\varepsilon(x,t),s(x,t),t)$. 
\end{Lemma}

\begin{Remark}\phantomsection{\label{th_asym_ACND_in_rem}} \upshape
	\begin{enumerate}
		\item The differential operator $\nabla_\Sigma$ commutes with other ones acting on different variables. This can be shown directly with the definitions or suitable extension arguments. For the latter there are similar ideas in the proof of Lemma \ref{th_asym_ACND_in_trafo}.
		\item Note the similarity to Lemma 3.1 in \cite{AbelsMoser}. The terms on the right hand side are more sophisticated, but are always of the same type as in the 2D-case. Therefore in the expansion they will contribute in the analogous way.
	\end{enumerate}
\end{Remark}

\begin{proof}[Proof of Lemma \ref{th_asym_ACND_in_trafo}]
	Basically, this follows from the chain and product rules as well as from the properties of $\nabla_\Sigma$. Let $g:\Sigma\rightarrow\R$ be $C^1$. Then $g\circ s:\overline{\Gamma(2\delta)}\rightarrow\R$. In order to compute $\nabla_x[g(s)]$, let $\overline{g}:\R^N\rightarrow\R$ and $\overline{s}:\R^{N+1}\rightarrow\R^N$ be smooth extensions of $g$ and $s$, respectively. Such a $\overline{g}$ can be constructed via local extensions in submanifold charts and the existence of $\overline{s}$ follows from Theorem \ref{th_coordND}, 1. Then the chain rule yields
	\[
	D_x(g(s))|_{(x,t)}=D\overline{g}|_{s(x,t)}D_x\overline{s}|_{(x,t)}=(\nabla_\Sigma g)^\top|_{s(x,t)} D_xs|_{(x,t)},
	\]
	where we have used $\nabla_\Sigma g|_s=P_{T_s\Sigma}\nabla\overline{g}|_s$ for all $s\in\Sigma$ due to \cite{Depner}, Remark 2.22 as well as $\partial_{x_j}s|_{(x,t)}\in T_{s(x,t)}\Sigma$ for all $(x,t)\in\overline{\Gamma(2\delta)}$, cf.~Theorem \ref{th_coordND}. Alternatively, one can also use the chain rule for differentials and the definition of the surface gradient. For the derivative in time this works analogously. Therefore it holds for all $(x,t)\in \overline{\Gamma(2\delta)}$:
	\[
	\nabla_x[g(s)]|_{(x,t)}=(D_xs)^\top|_{(x,t)} \nabla_\Sigma g|_{s(x,t)}\quad\text{ and }\quad \partial_t[g(s)]|_{(x,t)}=\partial_ts|_{(x,t)}\cdot \nabla_\Sigma g|_{s(x,t)}.
	\]
	With similar arguments and the chain rule one can derive formulas for the first derivatives of functions of type $g(s(x,t),t)$ for $(x,t)\in\overline{\Gamma(2\delta)}$, where $g:\Sigma\times[0,T]\rightarrow\R$. In this case it holds
	\begin{align*}
	\frac{d}{dt}[g(s(x,.),.)]|_t&=\partial_tg|_{(s(x,t),t)} + \partial_ts|_{(x,t)}\cdot \nabla_\Sigma g|_{(s(x,t),t)},\\ \nabla_x[g(s(.,t),t)]|_x&=(D_xs)|_{(x,t)}^\top\nabla_\Sigma g|_{(s(x,t),t)}.
	\end{align*}
	Using this and similar arguments as before, one can derive formulas for the first derivatives of functions of the form $g(\rho_\varepsilon(x,t),s(x,t),t)$ for $(x,t)\in\overline{\Gamma(2\delta)}$ with $g:\R\times\Sigma\times[0,T]\rightarrow\R$. The latter are written in the lemma for $\hat{w}$ instead of $g$. Putting all those identities together and using the product rule in $\partial_{x_j}(\nabla w)_j$ for $j=1,...,N$, one obtains the formula for $\Delta w$.
\end{proof}

To expand the Allen-Cahn equation $\partial_tu^I_\varepsilon-\Delta u^I_\varepsilon+\frac{1}{\varepsilon^2}f'(u^I_\varepsilon)=0$ into $\varepsilon$-series, we again use Taylor expansions. For the $f'$-part this is identically to the 2D-case: If the $u_j^I$ are bounded, then 
\begin{align}\label{eq_asym_ACND_in_taylor_f}
f'(u^I_\varepsilon)=f'(u_0^I)+\sum_{k=1}^{M+2}\frac{f^{(k+1)}(u_0^I)}{k!}\left[\sum_{j=1}^{M+1} u_j^I\varepsilon^j\right]^k+\Oc(\varepsilon^{M+3})\quad\text{ on }\overline{\Gamma(2\delta)}.
\end{align}
The terms in the $\varepsilon$-expansion that are needed explicitly are
\begin{align*}
\Oc(1)&:\quad f'(u_0^I),\\
\Oc(\varepsilon)&:\quad f''(u_0^I)u_1^I,\\
\Oc(\varepsilon^2)&:\quad f''(u_0^I)u_2^I+\frac{f^{(3)}(u_0^I)}{2!}(u_1^I)^2.
\end{align*}
For $k=3,...,M+1$ the order $\Oc(\varepsilon^k)$ is given by
\begin{alignat*}{2}
\Oc(\varepsilon^k)&:\quad f''(u_0^I)u_k^I +\quad & [\text{some polynomial in }(u_1^I,...,u_{k-1}^I)\text{ of order }\leq k,\text{ where the}\\
& &\text{ coefficients are multiples of }f^{(3)}(u_0^I),...,f^{(k+1)}(u_0^I)\\
& &\text{and every term contains a }u^I_j\text{-factor}].
\end{alignat*}
Let $u_{M+2}^I:=0$. Then the latter also holds for $k=M+2$. The other explicit terms in \eqref{eq_asym_ACND_in_taylor_f} are of order  $\Oc(\varepsilon^{M+3})$.

Moreover, we expand functions of $(x,t)\in\overline{\Gamma(2\delta)}$ into $\varepsilon$-series with a Taylor expansion via $r(x,t)=\varepsilon(\rho_\varepsilon(x,t)+h_\varepsilon(s(x,t),t))$ for $(x,t)\in\overline{\Gamma(2\delta)}$. Then again $\rho_\varepsilon$ is replaced by $\rho\in\R$. For a smooth  $g:\overline{\Gamma(2\delta)}\rightarrow\R$ the Taylor expansion yields for $r\in[-2\delta,2\delta]$ uniformly in $(s,t)$: 
\begin{align}\label{eq_asym_ACND_in_taylor2}
\tilde{g}(r,s,t):=g(\overline{X}(r,s,t))=\sum_{k=0}^{M+2}\frac{\partial_r^k\tilde{g}|_{(0,s,t)}}{k!}r^k+\Oc(|r|^{M+3}).
\end{align}
Only the first few terms in the $\varepsilon$-expansion are needed explicitly. These are
\begin{align*}
\Oc(1)&:\quad g|_{\overline{X}_0(s,t)},\\
\Oc(\varepsilon)&:\quad (\rho+h_1(s,t))\partial_r\tilde{g}|_{(0,s,t)},\\
\Oc(\varepsilon^2)&:\quad h_2(s,t)\partial_r\tilde{g}|_{(0,s,t)}+(\rho+h_1(s,t))^2\frac{\partial_r^2\tilde{g}|_{(0,s,t)}}{2}.
\end{align*} 
For $k=3,...,M$ the order $\Oc(\varepsilon^k)$ is
\begin{align*}
\Oc(\varepsilon^k):\quad h_k \partial_r\tilde{g}|_{(0,s,t)} &+ \frac{\partial_r^2\tilde{g}|_{(0,s,t)}}{2}2(\rho+h_1(s,t))h_{k-1}(s,t)\\ & +[\text{some polynomial in }(\rho,h_1(s,t),...,h_{k-2}(s,t))\text{ of order }\leq k,\\
& \phantom{+[}\text{ where the coefficients are multiples of }\partial_r^2\tilde{g}|_{(0,s,t)},...,\partial_r^k\tilde{g}|_{(0,s,t)}].
\end{align*}
Let $h_{M+1}:=h_{M+2}:=0$. Then the latter also holds for $k=M+1, M+2$. The other explicit terms in \eqref{eq_asym_ACND_in_taylor2} are bounded by $\varepsilon^{M+3}$ times some polynomial in $|\rho|$ if the $h_j$ are bounded. Later, these terms and the $\Oc(|r|^{M+3})$-term in \eqref{eq_asym_ACND_in_taylor2} for each choice of $g$ will be multiplied with terms that decay exponentially in $|\rho|$. Then these remainder terms will become $\Oc(\varepsilon^{M+3})$.

For the higher orders in the expansion the following definition is useful:
\begin{Definition}[\textbf{Notation for Inner Expansion of (AC) in ND}]\upshape\phantomsection{\label{th_asym_ACND_in_def}}
	\begin{enumerate}
		\item We call $(\theta_0,u^I_1)$ the \textit{zero-th inner order} and $(h_j,u^I_{j+1})$ the \textit{$j$-th inner order} for $j=1,...,M$. 
		\item Let $k\in\{-1,...,M+2\}$. We denote with $P^I_k$ the set of polynomials in $\rho$ with smooth coefficients in $(s,t)\in \Sigma\times[0,T]$ depending only on the $h_j$ for $1\leq j\leq \min\{k,M\}$.    
		\item Let $k\in\{-1,...,M+2\}$ and $\beta>0$. We denote with $R^I_{k,(\beta)}$ the set of smooth functions $R:\R\times\Sigma\times[0,T]\rightarrow\R$ that depend only on the $j$-th inner orders for $0\leq j\leq \min\{k,M\}$ and satisfy uniformly in $(\rho,s,t)$:
		\[
		|\partial_\rho^i(\nabla_\Sigma)_{n_1}...(\nabla_\Sigma)_{n_d}\partial_t^n R(\rho,s,t)|=\Oc(e^{-\beta|\rho|})
		\]
		for all $n_1,...,n_d\in\{1,...,N\}$ and $d,i,n\in\N_0$.
		\item For $k\in\{-1,...,M+2\}$ and $\beta>0$ the set $\hat{R}^I_{k,(\beta)}$ is defined analogously to $R^I_{k,(\beta)}$ with functions of type $R:\R\times\partial\Sigma\times[0,T]\rightarrow\R$ instead.\footnote{~Note that this set only appears in the contact point expansion later.}
	\end{enumerate}
\end{Definition}

Now we expand \eqref{eq_AC1} for $u_\varepsilon=u^I_\varepsilon$ into $\varepsilon$-series. This works analogously to the 2D-case, cf.~Remark \ref{th_asym_ACND_in_rem},~2. In the following $(\rho,s,t)$ are always in $\R\times \Sigma\times[0,T]$ and sometimes omitted.

\paragraph{Inner Expansion: $\Oc(\varepsilon^{-2})$}\label{sec_asym_ACND_in_m2}
We obtain that the $\Oc(\frac{1}{\varepsilon^2})$-order is zero if
\[
\quad -|\nabla r|^2|_{\overline{X}_0(s,t)}\partial_\rho^2\hat{u}_0^I(\rho,s,t)+f'(\hat{u}_0^I(\rho,s,t))=0.
\]
Because of Theorem \ref{th_coordND} we have $|\nabla r|^2|_{\overline{X}_0(s,t)}=1$. Moreover, $\{\rho_\varepsilon=0\}$ should approximate the zero level set of $u^I_\varepsilon$. Hence we require $\hat{u}_0^I(0,s,t)=0$. Finally, $\hat{u}_0^I$ should connect the values $\pm1$, i.e.~$\lim_{\rho\rightarrow\pm\infty}\hat{u}_0^I(\rho,s,t)=\pm1$. Altogether due to Theorem \ref{th_theta_0} we have to define
\[
\hat{u}_0^I(\rho,s,t):=\theta_0(\rho).
\] 

\paragraph{Inner Expansion: $\Oc(\varepsilon^{-1})$}\label{sec_asym_ACND_in_m1}
The $\partial_tu$-part yields $\frac{1}{\varepsilon}\partial_tr|_{\overline{X}_0(s,t)}\theta_0'(\rho)$ and $\Delta u$ yields
\begin{align*}
&\frac{1}{\varepsilon^2}\left[\partial_r(|\nabla r|^2\circ \overline{X})|_{(0,s,t)}\varepsilon(\rho+h_1(s,t))\theta_0''(\rho)+|\nabla r|^2|_{\overline{X}_0(s,t)}\varepsilon\partial_\rho^2\hat{u}_1^I(\rho,s,t)\right]\\
&+\frac{1}{\varepsilon}\left[\theta_0'(\rho)\Delta r|_{\overline{X}_0(s,t)}+2(D_xs\nabla r)^\top|_{\overline{X}_0(s,t)}(\nabla_\Sigma\theta_0'(\rho)-\nabla_\Sigma h_1(s,t)\theta_0''(\rho))\right]\\
&=\frac{1}{\varepsilon}\left[\partial_\rho^2\hat{u}_1^I(\rho,s,t)+\theta_0'(\rho)\Delta r|_{\overline{X}_0(s,t)}\right],
\end{align*}
where we used Theorem \ref{th_coordND}. Therefore the $\Oc(\frac{1}{\varepsilon})$-order cancels if  
\[
\Lc_0\hat{u}_1^I(\rho,s,t)+\theta_0'(\rho)(\partial_tr-\Delta r )|_{\overline{X}_0(s,t)}=0,\quad\text{ where }\Lc_0:=-\partial_\rho^2+f''(\theta_0).
\]
Due to Theorem \ref{th_ODE_lin} this parameter-dependent ODE together with $\hat{u}_1^I(0,s,t)=0$ and boundedness in $\rho$ has a (unique) solution $\hat{u}_1^I$ if and only if $(\partial_tr-\Delta r)|_{\overline{X}_0(s,t)}=0$. The latter is valid because it is equivalent to \eqref{MCF} for $\Gamma$ by Theorem \ref{th_coordND}. Therefore we define $\hat{u}_1^I:=0$.

\paragraph{Inner Expansion: $\Oc(\varepsilon^0)$}\label{sec_asym_ACND_in_0}
From $\partial_tu$ we obtain
\begin{align*}
&\frac{1}{\varepsilon}\left[\partial_tr|_{\overline{X}_0(s,t)}\varepsilon\partial_\rho\hat{u}_1^I+\partial_r(\partial_tr\circ\overline{X})|_{(0,s,t)}\varepsilon(\rho+h_1(s,t))\theta_0'(\rho)\right]\\
&+\theta_0'(\rho)\left[-\partial_th_1(s,t)-\partial_ts|_{\overline{X}_0(s,t)}\cdot \nabla_\Sigma h_1(s,t)\right]+\partial_ts|_{\overline{X}_0(s,t)}\cdot\nabla_\Sigma\theta_0+\partial_t\theta_0(\rho)\\
&=\theta_0'(\rho)\left[(\rho+h_1(s,t))\partial_r(\partial_tr\circ \overline{X})|_{(0,s,t)}-\partial_th_1(s,t)-\partial_ts|_{\overline{X}_0(s,t)}\cdot \nabla_\Sigma h_1(s,t)\right],
\end{align*}
and from $\Delta u$:
\begin{align*}
&\frac{1}{\varepsilon^2}\theta_0''(\rho)\left[
\varepsilon^2\frac{1}{2}(\rho+h_1)^2\partial_r^2(|\nabla r|^2\circ\overline{X})|_{(0,s,t)}+\varepsilon^2 h_2\partial_r(|\nabla r|^2\circ\overline{X})|_{(0,s,t)}\right]\\
&+\frac{1}{\varepsilon^2}\partial_\rho^2\hat{u}_1^I\varepsilon^2(\rho+h_1)\partial_r(|\nabla r|^2\circ\overline{X})|_{(0,s,t)}+\frac{1}{\varepsilon^2}|\nabla r|^2|_{\overline{X}_0(s,t)}\varepsilon^2\partial_\rho^2\hat{u}_2^I\\
&+\frac{1}{\varepsilon}\left[\theta_0'(\rho)\varepsilon(\rho+h_1)\partial_r(\Delta r\circ\overline{X})|_{(0,s,t)}+\varepsilon\partial_\rho\hat{u}_1^I\Delta r|_{\overline{X}_0(s,t)}\right]\\
&+\frac{1}{\varepsilon}2(D_xs\nabla r)^\top|_{\overline{X}_0(s,t)}\left[\nabla_\Sigma\partial_\rho\hat{u}_1^I\varepsilon-\nabla_\Sigma h_1\varepsilon\partial_\rho^2\hat{u}_1^I-\varepsilon\nabla_\Sigma h_2\theta_0''(\rho)\right]\\
&+\frac{1}{\varepsilon}2\partial_r((D_xs\nabla r)^\top\circ\overline{X})|_{(0,s,t)}\varepsilon(\rho+h_1)\left[\nabla_\Sigma\theta_0'(\rho)-\nabla_\Sigma h_1\theta_0''(\rho)\right]\\
&+\Delta s|_{\overline{X}_0(s,t)}\cdot\nabla_\Sigma\theta_0(\rho)+\sum_{i,l=1}^N\nabla s_i\cdot\nabla s_l|_{\overline{X}_0(s,t)}(\nabla_\Sigma)_i(\nabla_\Sigma)_l\theta_0(\rho)\\
&-2\nabla_\Sigma\theta_0'(\rho)^\top D_xs(D_xs)^\top|_{\overline{X}_0(s,t)}\nabla_\Sigma h_1+\left|(D_xs)^\top|_{\overline{X}_0(s,t)}\nabla_\Sigma h_1\right|^2\theta_0''(\rho)\\
&-\theta_0'(\rho)\left[\Delta s|_{\overline{X}_0(s,t)}\cdot\nabla_\Sigma h_1+\sum_{i,l=1}^N\nabla s_i\cdot\nabla s_l|_{\overline{X}_0(s,t)}(\nabla_\Sigma)_i(\nabla_\Sigma)_l h_1\right].
\end{align*}
Because of Theorem \ref{th_coordND} the latter is the same as
\begin{align*}
&\theta_0''(\rho)\left[\frac{1}{2}(\rho+h_1)^2\partial_r^2(|\nabla r|^2\circ\overline{X})|_{(0,s,t)}+\left|(D_xs)^\top|_{\overline{X}_0(s,t)}\nabla_\Sigma h_1\right|^2\right]+\partial_\rho^2\hat{u}_2^I\\
&+\theta_0''(\rho)2\partial_r((D_xs\nabla r)^\top\circ\overline{X})|_{(0,s,t)}(\rho+h_1)(-\nabla_\Sigma h_1)+\theta_0'(\rho)(\rho+h_1)\partial_r(\Delta r\circ\overline{X})|_{(0,s,t)}\\
&-\theta_0'(\rho)\left[\Delta s|_{\overline{X}_0(s,t)}\cdot\nabla_\Sigma h_1+\sum_{i,l=1}^N\nabla s_i\cdot\nabla s_l(\nabla_\Sigma)_i(\nabla_\Sigma)_l h_1\right].
\end{align*}
Due to $\hat{u}_1^I=0$, the $f'$-part contributes $f''(\theta_0)\hat{u}_2^I$. Hence for the cancellation of the $\Oc(1)$-term in the expansion for the Allen-Cahn equation \eqref{eq_AC1} for $u_\varepsilon=u^I_\varepsilon$ we require
\begin{align}\label{eq_asym_ACND_in_u2}
-\Lc_0\hat{u}_2^I(\rho,s,t)&=R_1(\rho, s,t),\\ \notag
R_1(\rho,s,t):=\theta_0'(\rho)&\left[-\partial_th_1+\sum_{i,l=1}^N\nabla s_i\cdot\nabla s_l(\nabla_\Sigma)_i(\nabla_\Sigma)_l h_1 \right.\\\notag
&\left. +(\rho+h_1)\partial_r((\partial_tr-\Delta r)\circ\overline{X})|_{(0,s,t)}-(\partial_ts-\Delta s)|_{\overline{X}_0(s,t)}\cdot\nabla_\Sigma h_1\right]\\ \notag
+\theta_0''(\rho)&\left[-\frac{1}{2}(\rho+h_1)^2\partial_r^2(|\nabla r|^2\circ\overline{X})|_{(0,s,t)}\right.\\
&\left.+2(\rho+h_1)\partial_r((D_xs\nabla r)^\top\circ\overline{X})|_{(0,s,t)}\nabla_\Sigma h_1-\left|(D_xs)^\top|_{\overline{X}_0(s,t)}\nabla_\Sigma h_1\right|^2\right].\notag
\end{align}
If $h_1$ is smooth, then $R_1$ is smooth and together with all derivatives decays exponentially in $|\rho|$ uniformly in $(s,t)$ with rate $\beta$ for every $\beta\in(0,\min\{\sqrt{f''(\pm 1)}\})$ because of Theorem \ref{th_theta_0}. Therefore Theorem \ref{th_ODE_lin} applied in local coordinates for $\Sigma$ yields that there is a unique bounded solution $\hat{u}_2^I$ to \eqref{eq_asym_ACND_in_u2} together with $\hat{u}_2^I(0,s,t)=0$ if and only if the compatibility condition $\int_\R R_1(\rho,s,t)\theta_0'(\rho)\,d\rho=0$ holds. Since $\int_\R\theta_0'(\rho)\theta_0''(\rho)\,d\rho=0$ due to integration by parts, the nonlinearities in $h_1$ drop out and we obtain a linear non-autonomous parabolic equation for $h_1$ on $\Sigma$ with principal part $\partial_t-\sum_{i,l=1}^N\nabla s_i\cdot\nabla s_l|_{\overline{X}_0(s,t)}(\nabla_\Sigma)_i(\nabla_\Sigma)_l$:
\begin{align}\label{eq_asym_ACND_in_h1}
\partial_th_1-\sum_{i,l=1}^N\nabla s_i\cdot\nabla s_l|_{\overline{X}_0(s,t)}(\nabla_\Sigma)_i(\nabla_\Sigma)_l h_1+a_1\cdot \nabla_\Sigma h_1+a_0h_1=f_0
\end{align} 
in $\Sigma\times[0,T]$. Here, with $d_1,...,d_5$ defined by
\begin{alignat*}{3}
d_1&:=\int_\R\theta_0'(\rho)^2\,d\rho,&\quad d_2&:=\int_\R\theta_0'(\rho)^2\rho\,d\rho,&\quad
d_3&:=\int_\R\theta_0'(\rho)^2\rho^2\,d\rho,\\
d_4&:=\int_\R\theta_0'(\rho)\theta_0''(\rho)\rho\,d\rho,&\quad d_5&:=\int_\R\theta_0'(\rho)\theta_0''(\rho)\rho^2\,d\rho,&\quad
d_6&:=\int_\R\theta_0'(\rho)\theta_0''(\rho)\rho^3\,d\rho,
\end{alignat*}
we have set for all $(s,t)\in\Sigma\times[0,T]$:
\begin{align}
a_1(s,t)&:=(\partial_ts-\Delta s)|_{\overline{X}_0(s,t)}-2\frac{d_4}{d_1}\partial_r((D_xs\nabla r)^\top \circ\overline{X})|_{(0,s,t)}\in \R^N,\label{eq_asym_ACND_in_a1}\\
a_0(s,t)&:=-\partial_r((\partial_tr-\Delta r)\circ\overline{X})|_{(0,s,t)}+\frac{d_4}{d_1}\partial_r^2(|\nabla r|^2\circ\overline{X})|_{(0,s,t)}\in\R,\label{eq_asym_ACND_in_a0}\\
f_0(s,t)&:=\frac{d_2}{d_1}\partial_r((\partial_tr-\Delta r)\circ\overline{X})|_{(0,s,t)}-\frac{d_5}{2d_1}\partial_r^2(|\nabla r|^2\circ\overline{X})|_{(0,s,t)}\in\R.
\end{align}
If $h_1$ is smooth and solves \eqref{eq_asym_ACND_in_h1}, then Theorem \ref{th_ODE_lin} (applied in local coordinates for $\Sigma$) yields a smooth solution $\hat{u}_2^I$ to \eqref{eq_asym_ACND_in_u2} and we also get decay estimates. By compactness, we obtain $\hat{u}_2^I\in R^I_{1,(\beta)}$ for any $\beta\in(0,\min\{\sqrt{f''(\pm 1)}\})$ because of the following remark:
\begin{Remark}\label{th_nabla_sigma_equiv}\upshape
	The norm of the entirety of derivatives in local coordinates up to any fixed order $d\in\N$ is equivalent to the norm of the collection of all $\nabla_\Sigma$-derivatives up to order $d$ on any compact subset of a chart domain. This can be shown inductively via local representations. 
\end{Remark}

\begin{Remark}\label{th_asym_ACND_feven_rem1}\upshape
	If additionally $f$ is even, then $\theta_0'$ is even and $\theta_0''$ is odd. Hence $d_2=d_5=0$ and $f_0=0$. Therefore the equation \eqref{eq_asym_ACND_in_h1} for $h_1$ is homogeneous in this case.
\end{Remark}

\paragraph{Inner Expansion: $\Oc(\varepsilon^k)$}\label{sec_asym_ACND_in_k}
Let $k\in\{1,...,M-1\}$ and suppose that the $j$-th inner order has already been constructed for $j=0,...,k$, that it is smooth and $\hat{u}_{j+1}^I\in R^I_{j,(\beta)}$ for every $\beta\in(0,\min\{\sqrt{f''(\pm1)}\})$. This assumption will be fulfilled later, when we apply an induction argument.

Then with the notation in Definition \ref{th_asym_ACND_in_def} it holds for all $\beta\in(0,\min\{\sqrt{f''(\pm1)}\})$:
\begin{align*}
\text{For }j=1,...,k+2:\,[\Oc(\varepsilon^j)\text{ in }\eqref{eq_asym_ACND_in_taylor_f}] &\in \, f''(u_0^I)u^I_j + R^I_{j-2,(\beta)}\quad [\subseteq R^I_{j-1,(\beta)},\text{ if }j\leq k+1],\\
\text{For }j=1,...,k+1:\,[\Oc(\varepsilon^j)\text{ in }\eqref{eq_asym_ACND_in_taylor2}] &\in \,
\partial_r\tilde{g}|_{(0,s,t)} h_j + P_{j-1}^I\quad[\subseteq P_j^I,\text{ if }j\leq k],\\
\text{For }j=3,...,k+1:\,[\Oc(\varepsilon^j)\text{ in }\eqref{eq_asym_ACND_in_taylor2}] &\in \,
\partial_r\tilde{g}|_{(0,s,t)} h_j + \partial_r^2\tilde{g}|_{(0,s,t)}(\rho+h_1)h_{j-1} + P_{j-2}^I.
\end{align*}

With these identities one can compute the $\Oc(\varepsilon^k)$-order in \eqref{eq_AC1} for $u_\varepsilon=u_\varepsilon^I$. The calculation is straightforward but lengthy, cf.~\cite{MoserDiss}, Section 5.1.1.4.~for the 2D-case. Therefore we do not go into details. However, note that the explicit identities from Theorem \ref{th_coordND} enter as well as that $\Gamma$ evolves according to \eqref{MCF} and $\hat{u}^I_1=0$. This yields that the $\Oc(\varepsilon^k)$-order in \eqref{eq_AC1} for $u_\varepsilon=u_\varepsilon^I$ vanishes if 
\begin{align}\label{eq_asym_ACND_in_uk}
-\Lc_0\hat{u}_{k+2}^{I}(\rho,s,t)&=R_{k+1}(\rho,s,t),\\
R_{k+1}(\rho,s,t):=\theta_0'(\rho)\notag
&\left[
-\partial_th_{k+1}
+\sum_{i,l=1}^N\nabla s_i\cdot\nabla s_l|_{\overline{X}_0(s,t)}(\nabla_\Sigma)_i(\nabla_\Sigma)_l h_{k+1}\right. \\ \notag
&\left.-(\partial_ts-\Delta s)|_{\overline{X}_0(s,t)}\cdot \nabla_\Sigma h_{k+1}
+h_{k+1}\partial_r((\partial_tr-\Delta r)\circ\overline{X})|_{(0,s,t)}
\right]\\ \notag
+\theta_0''(\rho)
&\left[
-(\rho+h_1)h_{k+1} \partial_r^2(|\nabla r|^2\circ\overline{X})|_{(0,s,t)}\right.\\ \notag
&-2(\nabla_\Sigma h_1)^\top D_xs(D_xs)^\top|_{\overline{X}_0(s,t)}\nabla_\Sigma h_{k+1}\\ \notag
&\left.+2\partial_r((D_xs\nabla r)^\top\circ\overline{X})|_{(0,s,t)}[(\rho+h_1)\nabla_\Sigma h_{k+1}+h_{k+1}\nabla_\Sigma h_1]
\right]\\ 
+ &\tilde{R}_{k}(\rho,s,t), \notag
\end{align}
where $\tilde{R}_{k}\in R^I_{k,(\beta)}$. If $h_{k+1}$ is smooth, then due to Theorem \ref{th_ODE_lin} (applied in local coordinates for $\Sigma$) equation \eqref{eq_asym_ACND_in_uk} has a unique bounded solution $\hat{u}_{k+2}^I$ with $\hat{u}_{k+2}^I(0,s,t)=0$ if and only if $\int_\R R_{k+1}(\rho,s,t)\theta_0'(\rho)\,d\rho=0$. Because of $\int_\R\theta_0''\theta_0'=0$ the latter is equivalent to 
\begin{align}\label{eq_asym_ACND_in_hk}
\partial_th_{k+1}-\sum_{i,l=1}^N\nabla s_i\cdot\nabla s_l|_{\overline{X}_0(s,t)}(\nabla_\Sigma)_i(\nabla_\Sigma)_l h_{k+1}+a_1\cdot\nabla_\Sigma h_{k+1} + a_0 h_{k+1} = f_k,
\end{align}
where 
\[
f_k(s,t):=\int_\R \tilde{R}_k(\rho,s,t)\theta_0'(\rho)\,d\rho\frac{1}{\|\theta_0'\|_{L^2(\R)}^2}
\] 
is a smooth function of $(s,t)$ and depends only on the $j$-th inner orders for $0\leq j\leq k$. Here $a_0, a_1$ are defined in \eqref{eq_asym_ACND_in_a1}-\eqref{eq_asym_ACND_in_a0}. If $h_{k+1}$ is smooth and solves \eqref{eq_asym_ACND_in_hk}, then Theorem \ref{th_ODE_lin} yields as in the last Section \ref{sec_asym_ACND_in_0} a smooth solution $\hat{u}_{k+2}^I$ to \eqref{eq_asym_ACND_in_uk} such that $\hat{u}_{k+2}^I\in R^I_{k+1,(\beta)}$ for all $\beta\in(0,\min\{\sqrt{f''(\pm1)}\})$.

\subsubsection{Contact Point Expansion of (AC) in ND}\label{sec_asym_ACND_cp}
In the contact point expansion we proceed similarly as in the 2D-case, cf.~\cite{AbelsMoser}, Section 3.2, but the computations are more technical. Here we have more contact points, namely all $X_0(\sigma,t)$, where $(\sigma,t)\in\partial\Sigma\times[0,T]$. We make the ansatz $u_\varepsilon=u^I_\varepsilon+u^C_\varepsilon$ in $\Gamma(2\delta)$ close to $X_0(\partial\Sigma\times[0,T])$. Therefore we use the mappings $Y:\partial\Sigma\times[0,2\mu_1]\rightarrow R(Y)\subset\Sigma$ as well as
\[
(\sigma,b)=Y^{-1}\circ s:\overline{\Gamma^C(2\delta,2\mu_1)}\rightarrow\partial\Sigma\times[0,2\mu_1]
\] 
from Section \ref{sec_coordND}. Besides $r$ we only scale $b$ with $\varepsilon$. Let $H_\varepsilon:=\frac{b}{\varepsilon}$ and
\[
u^C_\varepsilon:=\sum_{j=1}^M\varepsilon^j u_j^C,\quad u_j^C(x,t):=\hat{u}_j^C(\rho_\varepsilon(x,t),H_\varepsilon(x,t),\sigma(x,t),t)\quad\text{ for }(x,t)\in\overline{\Gamma^C(2\delta,2\mu_1)},
\]
where
\[
\hat{u}_j^C:\overline{\R^2_+}\times\partial\Sigma\times[0,T]\rightarrow\R:(\rho,H,\sigma,t)\mapsto \hat{u}_j^C(\rho,H,\sigma,t)
\] 
for $j=1,...,M$. Moreover, we set $u^C_{M+1}:=u^C_{M+2}:=0$ and $\hat{u}^C_\varepsilon:=\sum_{j=1}^M\varepsilon^j \hat{u}_j^C$. Instead of \eqref{eq_AC1} for $u_\varepsilon=u^I_\varepsilon+u^C_\varepsilon$, we will expand 
\begin{align}\label{eq_asym_ACND_cp} 
\partial_tu^C_\varepsilon-\Delta u^C_\varepsilon+\frac{1}{\varepsilon^2}\left[f'(u^I_\varepsilon+u^C_\varepsilon)-f'(u^I_\varepsilon)\right]=0
\end{align}
into $\varepsilon$-series with coefficients in $(\rho_\varepsilon,H_\varepsilon,\sigma,t)$. We call \eqref{eq_asym_ACND_cp} the \enquote{bulk equation} and expand it up to $\Oc(\varepsilon^{M-2})$. Moreover, we will expand \eqref{eq_AC2} for $u_\varepsilon=u^I_\varepsilon+u^C_\varepsilon$ into $\varepsilon$-series with coefficients in $(\rho_\varepsilon,\sigma,t)$ up to $\Oc(\varepsilon^{M-1})$. Altogether we end up with similar equations as in \cite{AbelsMoser}, Section 3.2. Here besides $t\in[0,T]$ also points on $\partial\Sigma$ enter as independent variables in the model problems on $\R^2_+$. The solvability condition \eqref{eq_hp_comp} yields the boundary conditions on $\partial\Sigma\times[0,T]$ for the height functions $h_j$.

For the expansion we calculate the action of the differential operators on $u^C_\varepsilon$ in the next lemma. Here for $\nabla_{\partial\Sigma}$ and $\nabla_\Sigma$ we use similar conventions as in Lemma \ref{th_asym_ACND_in_trafo}.

\begin{Lemma}\label{th_asym_ACND_cp_trafo}
	Let $\overline{\R^2_+}\times\partial\Sigma\times[0,T]\ni(\rho,H,\sigma,t)\mapsto\hat{w}(\rho,H,\sigma,t)\in\R$ be sufficiently smooth and let $w:\overline{\Gamma^C(2\delta,2\mu_1)}\rightarrow\R:(x,t)\mapsto\hat{w}(\rho_\varepsilon(x,t),H_\varepsilon(x,t),\sigma(x,t),t)$. Then
	\begin{align*}
	\partial_tw&=\partial_\rho\hat{w}\left[\frac{\partial_tr}{\varepsilon}-\left(\partial_th_\varepsilon+\partial_ts\cdot \nabla_\Sigma h_\varepsilon\right)\right]+\partial_H\hat{w}\frac{\partial_tb}{\varepsilon}+\partial_t\sigma\cdot\nabla_{\partial\Sigma}\hat{w}+\partial_t\hat{w},\\
	\nabla w&=\partial_\rho\hat{w}\left[\frac{\nabla r}{\varepsilon}-(D_xs)^\top\nabla_\Sigma h_\varepsilon\right]+\partial_H\hat{w}\frac{\nabla b}{\varepsilon}+(D_x\sigma)^\top\nabla_{\partial\Sigma}\hat{w},\\
	\Delta w&=\partial_\rho\hat{w}\left[\frac{\Delta r}{\varepsilon}-\left(\Delta s\cdot\nabla_\Sigma h_\varepsilon+\sum_{i,l=1}^N\nabla s_i\cdot\nabla s_l(\nabla_\Sigma)_i(\nabla_\Sigma)_lh_\varepsilon\right)\right]+\partial_H\hat{w}\frac{\Delta b}{\varepsilon}\\
	&+\partial_H^2\hat{w}\frac{|\nabla b|^2}{\varepsilon^2}
	+\partial_\rho^2\hat{w}\left|\frac{\nabla r}{\varepsilon}-(D_xs)^\top\nabla_\Sigma h_\varepsilon\right|^2+ 2\partial_\rho\partial_H\hat{w}\,\frac{\nabla b}{\varepsilon}\cdot\left[\frac{\nabla r}{\varepsilon}-(D_xs)^\top\nabla_\Sigma h_\varepsilon\right]\\
	&+2\left( (D_x\sigma)^\top\nabla_{\partial\Sigma}\partial_\rho \hat{w} \right)
	\cdot\left[\frac{\nabla r}{\varepsilon}
	-(D_xs)^\top \nabla_\Sigma h_\varepsilon\right]
	+2\left((D_x\sigma)^\top\nabla_{\partial\Sigma}\partial_H\hat{w}\right)\cdot\frac{\nabla b}{\varepsilon}\\
	&+\Delta\sigma\cdot\nabla_{\partial\Sigma}\hat{w}+\sum_{i,l=1}^N \nabla\sigma_i\cdot\nabla\sigma_l(\nabla_{\partial\Sigma})_i(\nabla_{\partial\Sigma})_l\hat{w},
	\end{align*}
	where the $w$-terms on the left hand side and derivatives of $r$ or $s$ are evaluated at $(x,t)$, the $h_\varepsilon$-terms at $(s(x,t),t)$ and the $\hat{w}$-terms at $(\rho_\varepsilon(x,t),H_\varepsilon(x,t),\sigma(x,t),t)$.
\end{Lemma}

\begin{proof}
	This can be shown in a similar manner as in the proof of Lemma \ref{th_asym_ACND_in_trafo}.
\end{proof}

\begin{Remark}\label{th_asym_ACND_cp_rem}\upshape
	The formulas in Lemma \ref{th_asym_ACND_cp_trafo} without the $\nabla_{\partial\Sigma}$-terms directly correspond to Lemma 3.3 in \cite{AbelsMoser}. The structure of the new terms is similar whereas their $\varepsilon$-order is the same or higher. Hence later these will only contribute to lower order remainder terms in the expansion.
\end{Remark}

\paragraph{Contact Point Expansion: The Bulk Equation}\label{sec_asym_ACND_cp_bulk}
We expand the $f'$-part in \eqref{eq_asym_ACND_cp}: If the $u_j^I, u_j^C$ are bounded, the Taylor expansion yields on $\overline{\Gamma(2\delta)}$
\begin{align}\label{eq_asym_ACND_cp_taylor_f}
f'(u^I_\varepsilon+u^C_\varepsilon)=f'(\theta_0)+\sum_{k=1}^{M+2}\frac{1}{k!}f^{(k+1)}(\theta_0)\left[\sum_{j=1}^{M+1}\varepsilon^j(u_j^I+u_j^C)\right]^k+\Oc(\varepsilon^{M+3}).
\end{align}
Combining this with the expansion for $f'(u^I_\varepsilon)$ in \eqref{eq_asym_ACND_in_taylor_f} and using $u^I_1=0$, the terms in the asymptotic expansion for $f'(u^I_\varepsilon+u^C_\varepsilon)-f'(u^I_\varepsilon)$ are for $k=1,...,M+1$:
\begin{alignat*}{2}
\Oc(1)&:\quad 0,&\\
\Oc(\varepsilon)&:\quad f''(\theta_0)u^C_1,& \\
\Oc(\varepsilon^k)&:\quad f''(\theta_0)u_k^C +\quad & [\text{some polynomial in }(u_1^I,...,u_{k-1}^I, u^C_1,...,u_{k-1}^C)\text{ of order }\leq k,\\
& &\text{ where the coefficients are multiples of }f^{(3)}(\theta_0),...,f^{(k+1)}(\theta_0)\\
& &\text{ and every term contains a }u^C_j\text{-factor}].
\end{alignat*}
Let $u_{M+2}^C:=0$. Then the latter is also valid for $k=M+2$. The other explicit terms in $f'(u^I_\varepsilon+u^C_\varepsilon)-f'(u^I_\varepsilon)$ are of order $\Oc(\varepsilon^{M+3})$.

Moreover, we expand terms arising from Lemma \ref{th_asym_ACND_cp_trafo} in \eqref{eq_asym_ACND_cp} that are functions of $(s,t)$ or $(\rho,s,t)$, i.e.~all the $h_j$-terms and the $u_j^I$-terms from the $f'$-expansion, respectively, as well as the terms depending on $(x,t)$, i.e.~all the derivatives of $r,b, s, \sigma$. Therefore we consider smooth $g_1:\Sigma\times[0,T]\rightarrow\R$ or $g_1:\R\times \Sigma\times[0,T]\rightarrow\R$  such that $\tilde{g}_1:=g_1|_{s=Y}$ admits bounded derivatives in $b$, where $Y:\partial\Sigma\times[0,2\mu_1]\rightarrow\Sigma:(\sigma,b)\mapsto Y(\sigma,b)$. Due to $b=\varepsilon H_\varepsilon$ we apply a Taylor expansion to obtain uniformly
\begin{align}\label{eq_asym_ACND_cp_taylor2}
g_1|_{b=\varepsilon H}=\tilde{g}_1|_{b=0}+\sum_{k=1}^{M+2}(\varepsilon H)^k\frac{\partial_b^k\tilde{g}_1|_{b=0}}{k!}+\Oc((\varepsilon H)^{M+3})\quad\text{ for }H\in[0,\frac{2\mu_1}{\varepsilon}].
\end{align}
Furthermore, let $g_2:\overline{\Gamma^C(2\delta,2\mu_1)}\rightarrow\R$ be smooth. For convenience we define
\[
\overline{X}_1:[-2\delta,2\delta]\times[0,2\mu_1]\times\partial\Sigma\times[0,T]\rightarrow\overline{\Gamma^C(2\delta,2\mu_1)}:(r,b,\sigma,t)\mapsto \overline{X}(r,Y(\sigma,b),t).
\] 
Then a Taylor expansion yields
\begin{align}\label{eq_asym_ACND_cp_taylor3}
\tilde{g}_2(r,b,\sigma,t):=g_2(\overline{X}_1(r,b,\sigma,t))=\sum_{j+k=0}^{M+2}\frac{\partial_r^j\partial_s^k\tilde{g}_2|_{(0,\sigma,t)}}{j!\,k!}r^j b^k+\Oc(|(r,b)|^{M+3})
\end{align}
uniformly in $(r,b,\sigma,t)\in[-2\delta,2\delta]\times[0,2\mu_1]\times\partial\Sigma\times[0,T]$. Later we evaluate at
\[
r=\varepsilon(\rho_\varepsilon(x,t)+h_\varepsilon(s(x,t),t)),\quad b=\varepsilon H_\varepsilon(x,t),\quad \sigma=\sigma(x,t)\quad\text{ for }(x,t)\in\overline{\Gamma^C(2\delta,2\mu_1)}
\] 
and expand $h_\varepsilon$ with \eqref{eq_asym_ACND_cp_taylor2}. Then we replace $(\rho_\varepsilon,H_\varepsilon)$ by arbitrary $(\rho,H)\in\overline{\R^2_+}$. Hence for $k=1,...,M+2$ we obtain
\begin{alignat*}{2}
\Oc(1)&:\quad & &g_2|_{\overline{X}_0(\sigma,t)} \\
\Oc(\varepsilon)&:\quad & &\partial_r\tilde{g}_2|_{(0,\sigma,t)}(\rho+h_1|_{(\sigma,t)})+\partial_b\tilde{g}_2|_{(0,\sigma,t)} H.\\
\Oc(\varepsilon^k)&:\quad & &[\text{some polynomial in }(\rho,H,\partial_b^l\tilde{h}_j|_{(\sigma,t)}), l=0,...,k-1, j=1,...,k\text{ of order}\leq k,\\
& & &\phantom{[}\text{where the coefficients are multiples of }\partial_r^{l_1}\partial_b^{l_2}\tilde{g}_2|_{(0,\sigma,t)}, l_1,l_2\in\N_0, l_1+l_2\leq k],
\end{alignat*}
where $h_{M+1}=h_{M+2}=0$ by definition. The order  $\Oc(\varepsilon)$ is not needed explicitly and just included for clarity. The other explicit terms in \eqref{eq_asym_ACND_cp_taylor2} are estimated by $\varepsilon^{M+3}$ times some polynomial in $(|\rho_\varepsilon|,H_\varepsilon)$. In the end these terms and the remainder in \eqref{eq_asym_ACND_cp_taylor2} are multiplied with exponentially decaying terms and hence become $\Oc(\varepsilon^{M+3})$.

Let us introduce some notation for the higher orders in the expansion:
\begin{Definition}[\textbf{Notation for Contact Point Expansion of (AC) in ND}]\upshape\phantomsection{\label{th_asym_ACND_cp_def}}
	\begin{enumerate}
		\item We denote with $(\theta_0,u^I_1)$ the \textit{zero-th order} and with $(h_j,u^I_{j+1},u^C_j)$ the \textit{$j$-th order} for $j=1,...,M$. 
		\item Let $k\in\{-1,...,M+2\}$. Let $P^C_k(\rho,H)$ be the set of polynomials in $(\rho,H)$ with smooth coefficients in $(\sigma,t)\in\partial\Sigma\times[0,T]$ that depend only on the $h_j$ for $1\leq j\leq \min\{k,M\}$. The sets $P^C_k(\rho)$ and $P^C_k(H)$ are defined with $(\rho,H)$ replaced by $\rho$ and $H$, respectively.
		\item Let $k\in\{-1,...,M+2\}$ and $\beta,\gamma>0$. Then $R^C_{k,(\beta,\gamma)}$ denotes the set of smooth functions $R:\overline{\R^2_+}\times\partial\Sigma\times[0,T]\rightarrow\R$ depending only on the $j$-th orders for $0\leq j\leq \min\{k,M\}$ and such that uniformly in $(\rho,H,\sigma,t)$:
		\[
		|\partial_\rho^i\partial_H^l(\nabla_{\partial\Sigma})_{n_1}...(\nabla_{\partial\Sigma})_{n_d}\partial_t^n R(\rho,H,\sigma,t)|=\Oc(e^{-(\beta|\rho|+\gamma H)})
		\]
		for all $n_1,...,n_d\in\{1,...,N\}$ and $d,i,l,n\in\N_0$.
		\item The set $R^C_{k,(\beta)}$ is defined analogously to $R^C_{k,(\beta,\gamma)}$ but without the $H$-dependence.
	\end{enumerate}
\end{Definition}

In the following we expand the bulk equation \eqref{eq_asym_ACND_cp} with the above formulas into $\varepsilon$-series with coefficients in $(\rho_\varepsilon,H_\varepsilon,\sigma,t)$.

\subparagraph{Bulk Equation: $\Oc(\varepsilon^{-1})$}\label{sec_asym_ACND_cp_bulk_m1}
The lowest order $\Oc(\frac{1}{\varepsilon})$ in  \eqref{eq_asym_ACND_cp} vanishes if 
\begin{align}\label{eq_asym_ACND_cp_bulk1}
[-\Delta^{\sigma,t}+f''(\theta_0(\rho))]\hat{u}^C_1(\rho,H,\sigma,t)=0,
\end{align}
where $\Delta^{\sigma,t}:=\partial_\rho^2+|\nabla b|^2|_{\overline{X}_0(\sigma,t)}\partial_H^2$ and we used $\nabla r\cdot\nabla b|_{\overline{X}_0(\sigma,t)}=0$ for all $(\sigma,t)\in\partial\Sigma\times[0,T]$.

\subparagraph{Bulk Equation: $\Oc(\varepsilon^{k-1})$}\label{sec_asym_ACND_cp_bulk_km1}
For $k=1,...,M-1$ we compute $\Oc(\varepsilon^{k-1})$ in  \eqref{eq_asym_ACND_cp} and derive an equation for $\hat{u}^C_{k+1}$. Therefore we assume that the $j$-th order is already constructed for $j=0,...,k$, that it is smooth and it holds $\hat{u}^I_{j+1}\in R^I_{j,(\beta_1)}$ for every $\beta_1\in(0,\min\{\sqrt{f''(\pm1)}\})$ (bounded and all derivatives bounded is enough here) and we assume $\hat{u}^C_j\in R^C_{j,(\beta,\gamma)}$ for every $\beta\in(0,\min\{\overline{\beta}(\gamma),\sqrt{f''(\pm1)}\})$, $\gamma\in(0,\overline{\gamma})$, where $\overline{\beta}, \overline{\gamma}$ are as in Theorem \ref{th_hp_exp3}. 

Then with the notation from Definition \ref{th_asym_ACND_cp_def} it holds for all those $(\beta,\gamma)$:
\begin{alignat*}{2}
\text{For }j=1,...,k+1:&\quad[\Oc(\varepsilon^j)\text{ in }\eqref{eq_asym_ACND_cp_taylor_f}\text{ minus }\eqref{eq_asym_ACND_in_taylor_f}]\quad& \in\quad & f''(\theta_0(\rho))\hat{u}^C_j + R^C_{j-1,(\beta,\gamma)},\\
\text{For }i,j=1,...,k:&\quad[\Oc(\varepsilon^j)\text{ in }\eqref{eq_asym_ACND_cp_taylor2}\text{ for }g_1=g_1(h_i)]\quad& \in\quad & P^C_i(H),\\
\text{For }j=0,...,k:&\quad[\Oc(\varepsilon^j)\text{ in }\eqref{eq_asym_ACND_cp_taylor3}]\quad& \in\quad & P^C_j(\rho,H).
\end{alignat*}

With these identities, one can compute that the $\Oc(\varepsilon^{k-1})$-order in the expansion for the bulk equation \eqref{eq_asym_ACND_cp} is zero if 
\begin{align}\label{eq_asym_ACND_cp_bulkk}
[-\Delta^{\sigma,t}+f''(\theta_0)]\hat{u}^C_{k+1}=G_k(\rho,H,\sigma,t),
\end{align}
where $G_k \in R^C_{k,(\beta,\gamma)}$. See \cite{MoserDiss}, Section 5.1.2.1.2 for the computation in the 2D-case.

\paragraph{Contact Point Expansion: The Neumann Boundary Condition}\label{sec_asym_ACND_cp_neum}
The boundary conditions complementing \eqref{eq_asym_ACND_cp_bulk1} and \eqref{eq_asym_ACND_cp_bulkk} will be obtained from the expansion of the Neumann boundary condition \eqref{eq_AC2} for $u_\varepsilon=u^I_\varepsilon+u^C_\varepsilon$, i.e.~$N_{\partial\Omega}\cdot\nabla(u^I_\varepsilon+u^C_\varepsilon)|_{\partial Q_T}=0$. Lemma \ref{th_asym_ACND_in_trafo} and Lemma \ref{th_asym_ACND_cp_trafo} yield on $\overline{\Gamma^C(2\delta,2\mu_1)}$
\begin{align*}
\nabla u^I_\varepsilon|_{(x,t)}&=\partial_\rho\hat{u}^I_\varepsilon|_{(\rho,s,t)}\left[\frac{\nabla r|_{(x,t)}}{\varepsilon}-(D_xs)^\top|_{(x,t)}\nabla_\Sigma h_\varepsilon|_{(s,t)}\right]+(D_xs)^\top|_{(x,t)}\nabla_\Sigma\hat{u}^I_\varepsilon|_{(\rho,s,t)},\\
\nabla u^C_\varepsilon|_{(x,t)}&=\partial_\rho\hat{u}^C_\varepsilon
|_{(\rho,H,\sigma,t)}\left[\frac{\nabla r|_{(x,t)}}{\varepsilon}-(D_xs)^\top|_{(x,t)}\nabla_\Sigma h_\varepsilon|_{(s,t)}\right]+\frac{\nabla b|_{(x,t)}}{\varepsilon}\partial_H\hat{u}^C_\varepsilon|_{(\rho,H,\sigma,t)}\\
&\phantom{=}+(D_x\sigma)^\top|_{(x,t)}\nabla_{\partial\Sigma}\hat{u}^C_\varepsilon|_{(\rho,H,\sigma,t)},
\end{align*}
where $\rho=\rho_\varepsilon(x,t), H=H_\varepsilon(x,t)$, $s=s(x,t)$ and $\sigma=\sigma(x,t)$. We consider the points $x=X(r,\sigma,t)$ for $(r,\sigma,t)\in[-2\delta,2\delta]\times\partial\Sigma\times[0,T]$, in particular $H=0$ and $s=\sigma$. 

For $g:\overline{\Gamma(2\delta)}\cap\partial Q_T\rightarrow\R$ smooth we use a Taylor expansion similar to \eqref{eq_asym_ACND_in_taylor2}:
\begin{align}\label{eq_asym_ACND_cp_taylor4}
\tilde{g}(r,\sigma,t):=g(\overline{X}(r,\sigma,t))=\sum_{k=0}^{M+2}\frac{\partial_r^k\tilde{g}|_{(0,\sigma,t)}}{k!}r^k+\Oc(|r|^{M+3}).
\end{align}
Then we insert $r=\varepsilon(\rho_\varepsilon+h_\varepsilon|_{(\sigma,t)})$ and replace $\rho_\varepsilon$ by an arbitrary $\rho\in\R$. Analogous to the inner expansion, the terms in the $\varepsilon$-expansion of \eqref{eq_asym_ACND_cp_taylor4} are for $k=2,...,M$:
\begin{align*}
\Oc(1)&:\quad g|_{\overline{X}_0(\sigma,t)},\\
\Oc(\varepsilon)&:\quad (\rho+h_1|_{(\sigma,t)})\partial_r\tilde{g}|_{(0,\sigma,t)},\\
\Oc(\varepsilon^k)&:\quad h_k|_{(\sigma,t)}\partial_r\tilde{g}|_{(0,\sigma,t)} +[\text{a polynomial in }(\rho,h_1|_{(\sigma,t)},...,h_{k-1}|_{(\sigma,t)})\text{ of order }\leq k,\\
&\phantom{:\quad h_k|_{(\sigma,t)}\partial_r\tilde{g}|_{(0,\sigma,t)} +[} \text{ where coefficients are multiples of }(\partial_r^2\tilde{g},...,\partial_r^k\tilde{g})|_{(0,\sigma,t)}].
\end{align*}
The latter also holds for $k=M+1, M+2$ since $h_{M+1}=h_{M+2}=0$ by definition. The other explicit terms in \eqref{eq_asym_ACND_cp_taylor4} are bounded by $\varepsilon^{M+3}$ times some polynomial in $|\rho|$ if the $h_j$ are bounded. Later, these terms and the $\Oc(|r|^{M+3})$-term in \eqref{eq_asym_ACND_cp_taylor4} for each choice of $g$ will be multiplied with terms that decay exponentially in $|\rho|$. Then these remainder terms will become $\Oc(\varepsilon^{M+3})$.

In the following we expand the Neumann boundary condition into $\varepsilon$-series with coefficients in $(\rho_\varepsilon,\sigma,t)$ up to the order $\Oc(\varepsilon^{M-1})$.

\subparagraph{Neumann Boundary Condition: $\Oc(\varepsilon^{-1})$}\label{sec_asym_ACND_cp_neum_m1}
For the lowest order $\Oc(\frac{1}{\varepsilon})$ we obtain the equation $(N_{\partial\Omega}\cdot\nabla r)|_{\overline{X}_0(\sigma,t)}\theta_0'(\rho)=0$. This is valid due to the $90$°-contact angle condition.

\subparagraph{Neumann Boundary Condition: $\Oc(\varepsilon^0)$}\label{sec_asym_ACND_cp_neum_0}
The order $\Oc(1)$ is zero if we require
\begin{align}\label{eq_asym_ACND_cp_bc1}
&(N_{\partial\Omega}\cdot\nabla b)|_{\overline{X}_0(\sigma,t)}\partial_H\hat{u}^C_1|_{H=0}(\rho,\sigma,t)=g_1(\rho,\sigma,t),\\ \notag
g_1|_{(\rho,\sigma,t)}&:=\theta_0'[(D_xsN_{\partial\Omega})^\top|_{\overline{X}_0(\sigma,t)}\nabla_\Sigma h_1|_{(\sigma,t)}-\partial_r((N_{\partial\Omega}\!\cdot\!\nabla r)\circ\overline{X})|_{(0,\sigma,t)}h_1|_{(\sigma,t)}]+\tilde{g}_0|_{(\rho,\sigma,t)},
\end{align}
where $\tilde{g}_0(\rho,\sigma,t):=-\rho \theta_0'(\rho)\partial_r((N_{\partial\Omega}\cdot\nabla r)\circ\overline{X})|_{(0,\sigma,t)}$. For $j=1,...,M$ let
\begin{align}\label{eq_asym_ACND_cp_ubar}
\overline{u}^C_j:\overline{\R^2_+}\times\partial\Sigma\times[0,T]\rightarrow\R:(\rho,H,\sigma,t)\mapsto\hat{u}^C_j(\rho,|\nabla b|(\overline{X}_0(\sigma,t))H,\sigma,t).
\end{align}
Note that $|\nabla b|_{\overline{X}_0(\sigma,t)}|\geq c>0$ and $|N_{\partial\Omega}\cdot\nabla b|_{\overline{X}_0(\sigma,t)}|\geq c>0$ for all $(\sigma,t)\in\partial\Sigma\times[0,T]$ because of Theorem \ref{th_coordND}. Hence equations \eqref{eq_asym_ACND_cp_bulk1} and \eqref{eq_asym_ACND_cp_bc1} for $\hat{u}^C_1$ are equivalent to
\begin{align}\label{eq_asym_ACND_cp_uC1}
[-\Delta+f''(\theta_0(\rho))]\overline{u}^C_1&=0,\\
-\partial_{H}\overline{u}^C_1|_{H=0}&=(|\nabla b|/N_{\partial\Omega}\cdot\nabla b)|_{\overline{X}_0(\sigma,t)} g_1(\rho,\sigma,t).\label{eq_asym_ACND_cp_uC2}
\end{align}
The solvability condition \eqref{eq_hp_comp} belonging to \eqref{eq_asym_ACND_cp_uC1}-\eqref{eq_asym_ACND_cp_uC2} is
\[
(|\nabla b|/N_{\partial\Omega}\cdot\nabla b)|_{\overline{X}_0(\sigma,t)}
\int_\R g_1(\rho,\sigma,t)\theta_0'(\rho)\,d\rho=0.
\]
This yields a linear boundary condition for $h_1$:
\begin{align}\label{eq_asym_ACND_cp_h1}
b_1(\sigma,t)\cdot\nabla_\Sigma h_1|_{(\sigma,t)}+ b_0(\sigma,t)h_1|_{(\sigma,t)}=f_0(\sigma,t)\quad\text{ for }(\sigma,t)\in\partial\Sigma\times[0,T],
\end{align}
where 
\begin{align*}
b_1(\sigma,t) &:=
(|\nabla b|/N_{\partial\Omega}\cdot\nabla b)|_{\overline{X}_0(\sigma,t)} (D_xsN_{\partial\Omega})|_{\overline{X}_0(\sigma,t)}\in\R^N,\\
b_0(\sigma,t) &:= 
-(|\nabla b|/N_{\partial\Omega}\cdot\nabla b)|_{\overline{X}_0(\sigma,t)} 
\partial_r((N_{\partial\Omega}\cdot\nabla r)\circ\overline{X})|_{(0,\sigma,t)}\in\R,\\
f_0(\sigma,t) &:= (|\nabla b|/N_{\partial\Omega}\cdot\nabla b)|_{\overline{X}_0(\sigma,t)} \int_\R\theta_0'(\rho)\tilde{g}_0(\rho,\sigma,t)\,d\rho/\|\theta_0'\|_{L^2(\R)}^2\in\R
\end{align*}
are smooth in $(\sigma,t)\in\partial\Sigma\times[0,T]$. Together with the linear parabolic equation \eqref{eq_asym_ACND_in_h1} for $h_1$ from Subsection \ref{sec_asym_ACND_in_0}, we obtain a time-dependent linear parabolic boundary value problem for $h_1$, where the initial value $h_1|_{t=0}$ is not prescribed yet. 

\begin{Remark}\label{th_asym_ACND_feven_rem2}\upshape
	If $f$ is even, then so is $\theta_0'$ and thus $f_0=0$. Therefore the boundary condition \eqref{eq_asym_ACND_cp_h1} for $h_1$ is homogeneous and because of Remark \ref{th_asym_ACND_feven_rem1} we can choose $h_1=0$ in this case.
\end{Remark}

Now we solve \eqref{eq_asym_ACND_in_h1} together with \eqref{eq_asym_ACND_cp_h1}. We show that the principal part in \eqref{eq_asym_ACND_in_h1} satisfies a suitable ellipticity condition and that \eqref{eq_asym_ACND_cp_h1} fulfils a non-tangentiality condition in local coordinates. Based on this one can show maximal regularity results in Hölder spaces (similar to Lunardi, Sinestrari, von Wahl \cite{LunardiSvW}) and Sobolev spaces (similar to Prüss, Simonett \cite{PruessSimonett}, Chapter 6.1-6.4 and Denk, Hieber, Prüss \cite{DenkHieberPrüss}) with typical localization procedures. This always involves compatibility conditions for the initial value. In our case these can be avoided via extension arguments similar to \cite{AbelsMoser}, Section 3.2.2. All these arguments involve many technical computations, but are in principle well-known. Therefore we refrain from going into details.

\begin{proof}[The Ellipticity Condition]
	Let $y:U\subseteq\Sigma\rightarrow V\subseteq\overline{\R^{N-1}_+}$ be a chart. Moreover, we denote with $(g_{ij})_{i,j=1}^{N-1}:V\rightarrow\R^{(N-1)\times(N-1)}$ the local representation of the Euclidean metric on $\Sigma$ and let $(g^{kl})_{k,l=1}^{N-1}$ denote its pointwise inverse. Then one can show with local representations that for $h:U\rightarrow\R$ sufficiently smooth and $i,j=1,...,N$ it holds \phantom{\qedhere}
	\begin{align*}
	[(\nabla_\Sigma)_i(\nabla_\Sigma)_jh]\circ y^{-1} &= \sum_{p,q,k,l=1}^{N-1} g^{pq}g^{kl}\partial_{v_q}(y^{-1})_i \partial_{v_l}(y^{-1})_j \partial_{v_p}\partial_{v_k}(h\circ y^{-1})\\
	&+g^{pq}\partial_{v_q}(y^{-1})_i \left[\partial_{v_p}g^{kl}\partial_{v_l}(y^{-1})_j+g^{kl}\partial_{v_p}\partial_{v_k}(y^{-1})_j\right] \partial_{v_k}(h\circ y^{-1}).
	\end{align*}
	Therefore the principal part of 
	$\sum_{i,j=1}^N\nabla s_i\cdot\nabla s_j|_{\overline{X}_0(.,t)}(\nabla_\Sigma)_i(\nabla_\Sigma)_j$ for fixed $t\in[0,T]$ in the local coordinates with respect to $y$ is given by
	\[
	\sum_{p,k=1}^{N-1} A_{p,k} \partial_{v_p}\partial_{v_k}, \quad A_{p,k}:=\sum_{i,j=1}^N \nabla s_i\cdot\nabla s_j|_{\overline{X}_0(y^{-1},t)} \sum_{q,l=1}^{N-1} g^{pq} g^{kl} \partial_{v_q}(y^{-1})_i \partial_{v_l}(y^{-1})_j.
	\]
	For any appropriate ellipticity notion, it is enough to prove that $A:=(A_{p,k})_{p,k=1}^{N-1}$ is uniformly positive definite on compact subsets of $V$. First we show that $A$ is pointwise positive definite. Therefore we represent $\partial_{x_n}s|_{\overline{X}_0(y^{-1}(.),t)}=\sum_{\mu=1}^{N-1}\tilde{A}_{n,\mu}|_{(.,t)}\partial_{v_\mu}(y^{-1})$ on $V$ for $n=1,...,N$ and we denote $\tilde{A}:=(\tilde{A}_{n,\mu}|_{(.,t)})_{n,\mu=1}^{N,N-1}:V\rightarrow\R^{N\times(N-1)}$. Then it holds for all $p,k=1,...,N-1$
	\begin{align*}
	A_{p,k}
	&=\sum_{i,j,n=1}^N\sum_{\mu,\nu,q,l=1}^{N-1} \tilde{A}_{n,\mu}\tilde{A}_{n,\nu} g^{pq} g^{kl} \partial_{v_\mu}(y^{-1})_i\partial_{v_q}(y^{-1})_i\partial_{v_\nu}(y^{-1})_j\partial_{v_l}(y^{-1})_j\\
	&=\sum_{n=1}^N\sum_{\mu,\nu,q,l=1}^{N-1}
	\tilde{A}_{n,\mu}\tilde{A}_{n,\nu} g^{pq} g^{kl} g_{\mu q}g_{\nu l}=\sum_{n=1}^N\sum_{\mu,\nu=1}^{N-1}\tilde{A}_{n,\mu}\tilde{A}_{n,\nu}\delta^p_\mu\delta^k_\nu=(\tilde{A}^\top\tilde{A})_{p,k}.
	\end{align*}
	Moreover, Theorem \ref{th_coordND} yields that $(\partial_{x_n}s|_{\overline{X}_0(s,t)})_{n=1}^{N}$ generate $T_s\Sigma$ for all $s\in\Sigma$. Hence the matrix $\tilde{A}|_v\in\R^{N\times(N-1)}$ is injective for all $v\in V$. Finally, this yields
	\[
	w^\top Aw=|\tilde{A}w|^2>0\quad\text{ for all }\quad w\in\R^{N-1}.
	\]
	Therefore $A$ is pointwise positive definite. Since it is equivalent to prove the estimate for vectors on the sphere in $\R^{N-1}$, by compactness $A$ is uniform positive definite on compact subsets of $V$.
	Altogether, the principal part in \eqref{eq_asym_ACND_in_h1} satisfies a suitable ellipticity condition.\end{proof}

\begin{proof}[The Non-Tangentiality Condition] 
	Let $y:U\subseteq\Sigma\rightarrow V\subseteq\overline{\R^{N-1}_+}$ be a chart with $U\cap \partial\Sigma\neq\emptyset$. Then for $h:U\rightarrow\R$ sufficiently smooth and fixed $t\in[0,T]$ it holds:\phantom{\qedhere}
	\[
	b_1(y^{-1},t)\cdot(\nabla_\Sigma h\circ y^{-1}) = \nabla_v(h\circ y^{-1})\cdot\left[\sum_{l=1}^{N-1}g^{kl}b_1(y^{-1},t)\cdot\partial_{v_l}(y^{-1})\right]_{k=1}^{N-1}
	\]
	on $V\cap(\R^{N-2}\times\{0\})$, where $b_1$ is defined below \eqref{eq_asym_ACND_cp_h1}. Therefore the transformed boundary condition in the local coordinates with respect to $y$ satisfies a non-tangentiality condition if $\sum_{l=1}^{N-1}g^{N-1,l}b_1(y^{-1},t)\cdot\partial_{v_l}(y^{-1})\neq 0$ on the set $V\cap(\R^{N-2}\times\{0\})$. On the latter set we use the representation $b_1(y^{-1},t)=\sum_{n=1}^{N-1} B_n(y^{-1},t)\partial_{v_n}(y^{-1})$. Then the condition reads as $B_{N-1}(\sigma,t)\neq 0$ for  all $(\sigma,t)\in(U\cap\partial\Sigma)\times[0,T]$. However, Theorem \ref{th_coordND} yields $b_1(\sigma,t)\cdot\vec{n}_{\partial\Sigma}(\sigma)\neq 0$ for all $(\sigma,t)\in \partial\Sigma\times[0,T]$. Due to the properties of $y$, we know that $(\partial_{v_l}(y^{-1})|_{y(\sigma)})_{l=1}^{N-1}$ form a basis of $T_\sigma\Sigma$ and the first $N-2$ components are a basis of $T_\sigma\partial\Sigma$ for all $\sigma\in U\cap\partial\Sigma$. Because $\vec{n}_{\partial\Sigma}$ is orthogonal to $T_\sigma\partial\Sigma$, it necessarily holds $B_{N-1}(\sigma,t)\neq 0$ for all $(\sigma,t)\in(U\cap\partial\Sigma)\times[0,T]$. Therefore the boundary condition \eqref{eq_asym_ACND_cp_h1} satisfies a non-tangentiality condition in local coordinates.\end{proof}

Finally, we obtain a smooth solution $h_1$ to \eqref{eq_asym_ACND_in_h1} and \eqref{eq_asym_ACND_cp_h1}. Therefore $\hat{u}^I_2$ (solving \eqref{eq_asym_ACND_in_u2}) is determined from Section \ref{sec_asym_ACND_in_0} and it holds $\hat{u}^I_2\in R^I_{1,(\beta_1)}$ for every $\beta_1\in(0,\min\{\sqrt{f''(\pm 1)}\})$. In particular the first inner order is computed. Moreover, it holds $g_1\in \hat{R}^I_{1,(\beta_1)}$ for all $\beta_1$ as above because of Theorem \ref{th_theta_0}. Hence with Theorem \ref{th_hp_exp3} (applied in local coordinates for $\partial\Sigma$) there is a unique smooth solution $\overline{u}^C_1$ to \eqref{eq_asym_ACND_cp_uC1}-\eqref{eq_asym_ACND_cp_uC2} and we get decay properties. By compactness and Remark \ref{th_nabla_sigma_equiv} with $\partial\Sigma$ instead of $\Sigma$ we obtain the decay property $\overline{u}^C_1\in R^C_{1,(\beta,\gamma)}$ for all $\beta\in(0,\min\{\overline{\beta}(\gamma),\sqrt{f''(\pm1)}\})$, $\gamma\in(0,\overline{\gamma})$, where $\overline{\beta}, \overline{\gamma}$ are as in Theorem \ref{th_hp_exp3}. Altogether the first order is determined.

\subparagraph{Neumann Boundary Condition: $\Oc(\varepsilon^k)$ and Induction}\label{sec_asym_ACND_cp_neum_k}
For $k=1,...,M-1$ we consider $\Oc(\varepsilon^k)$ in \eqref{eq_AC2} for $u_\varepsilon=u^I_\varepsilon+u^C_\varepsilon$ and derive equations for the ($k+1$)-th order. We assume the following induction hypothesis: suppose that the $j$-th order already has been constructed for all $j=0,...,k$, that it is smooth and admits the decay $\hat{u}^I_{j+1}\in R^I_{j,(\beta_1)}$ for all $\beta_1\in(0,\min\{\sqrt{f''(\pm1)}\})$ as well as $\hat{u}^C_j\in R^C_{j,(\beta,\gamma)}$ for every $\beta\in(0,\min\{\overline{\beta}(\gamma),\sqrt{f''(\pm1)}\})$, $\gamma\in(0,\overline{\gamma})$, where $\overline{\beta}, \overline{\gamma}$ are as in Theorem \ref{th_hp_exp3}. The assumption holds for $k=1$ by Section \ref{sec_asym_ACND_cp_neum_0}.

With the notation as in Definition \ref{th_asym_ACND_cp_def} it holds for $j=1,...,k+1$:
\begin{align*}
[\Oc(\varepsilon^j)\text{ in }\eqref{eq_asym_ACND_cp_taylor4}]\quad\in 
\partial_r\tilde{g}|_{(0,\sigma,t)} h_j|_{(\sigma,t)} + P^C_{j-1}(\rho)\quad 
[\subseteq P^C_j(\rho)\text{, if }j\leq k].
\end{align*}
With this identity one can show (see \cite{MoserDiss}, Section 5.1.2.2.3 for the calculation in the 2D-case) that the $\Oc(\varepsilon^k)$-order in \eqref{eq_AC2} for $u_\varepsilon=u^I_\varepsilon + u^C_\varepsilon$ vanishes if
\begin{align}\label{eq_asym_ACND_cp_bck}
&(N_{\partial\Omega}\cdot\nabla b)|_{\overline{X}_0(\sigma,t)}\partial_H\hat{u}^C_{k+1}|_{H=0}(\rho,\sigma,t)=g_{k+1}(\rho,\sigma,t),\\ \notag
g_{k+1}|_{(\rho,\sigma,t)}&:=\theta_0'(\rho)[(D_xsN_{\partial\Omega})^\top|_{\overline{X}_0(\sigma,t)}\nabla_\Sigma h_{k+1}|_{(\sigma,t)}-\partial_r((N_{\partial\Omega}\cdot\nabla r)\circ\overline{X})|_{(0,\sigma,t)}h_{k+1}|_{(\sigma,t)}]\\ \notag
&+\tilde{g}_k(\rho,\sigma,t),\notag
\end{align}
where $\tilde{g}_k\in R^C_{k,(\beta)}$ and therefore $g_{k+1}\in\hat{R}^I_{k+1,(\beta)}+R^C_{k,(\beta)}$, if $h_{k+1}$ is smooth. 

As in the last Section \ref{sec_asym_ACND_cp_neum_0}, the equations \eqref{eq_asym_ACND_cp_bulkk}, \eqref{eq_asym_ACND_cp_bck} are equivalent to
\begin{align}\label{eq_asym_ACND_cp_uCk_1}
[-\Delta+f''(\theta_0(\rho))]\overline{u}^C_{k+1}&=\overline{G}_k(\rho,H,\sigma,t),\\
-\partial_H\overline{u}^C_{k+1}|_{H=0}&=(|\nabla b|/N_{\partial\Omega}\cdot\nabla b)|_{\overline{X}_0(\sigma,t)} g_{k+1}(\rho,H,\sigma,t),\label{eq_asym_ACND_cp_uCk_2}
\end{align}
where $\overline{u}^C_{k+1}$ was defined in \eqref{eq_asym_ACND_cp_ubar} and $\overline{G}_k$ is defined analogously with the $G_k\in R^C_{k,(\beta,\gamma)}$ from Section \ref{sec_asym_ACND_cp_bulk_km1}. The corresponding compatibility condition \eqref{eq_hp_comp}, namely
\[
\int_{\R^2_+}\overline{G}_k(\rho,H,\sigma,t)\theta_0'(\rho)\,d(\rho,H)
+(|\nabla b|/N_{\partial\Omega}\cdot\nabla b)|_{\overline{X}_0(\sigma,t)}\int_\R g_{k+1}(\rho,\sigma,t)\theta_0'(\rho)\,d\rho=0,
\]
yields a linear boundary condition for $h_{k+1}$:
\begin{align}\label{eq_asym_ACND_cp_hk}
b_1(\sigma,t)\cdot\nabla_\Sigma h_{k+1}|_{(\sigma,t)}+ b_0(\sigma,t)h_{k+1}|_{(\sigma,t)}=f^B_k(\sigma,t)\quad\text{ for }(\sigma,t)\in\partial\Sigma\times[0,T],
\end{align}
where $b_0,b_1$ are defined below \eqref{eq_asym_ACND_cp_h1} and 
\[
f^B_k(\sigma,t) := -\frac{1}{\|\theta_0'\|_{L^2(\R)}^2}
\left[
\int_{\R^2_+}\overline{G}_k|_{(\rho,H,\sigma,t)}\theta_0'(\rho)\,d\rho +\left.\frac{|\nabla b|}{N_{\partial\Omega}\cdot\nabla b}\right|_{\overline{X}_0(\sigma,t)}\int_\R \tilde{g}_k|_{(\rho,\sigma,t)}\theta_0'(\rho)\,d\rho
\right]
\] 
is smooth in $(\sigma,t)\in\partial\Sigma\times[0,T]$.

Because of the computations in the last Section \ref{sec_asym_ACND_cp_neum_0} one can solve \eqref{eq_asym_ACND_in_hk} from Section \ref{sec_asym_ACND_in_k} together with \eqref{eq_asym_ACND_cp_hk} and get a smooth solution $h_{k+1}$. Therefore Section \ref{sec_asym_ACND_in_k} yields $\hat{u}^I_{k+2}$ (solving \eqref{eq_asym_ACND_in_uk}) with $\hat{u}^I_{k+2}\in R^I_{k+1,(\beta_1)}$ for all $\beta_1\in(0,\min\{\sqrt{f''(\pm 1)}\})$. In particular the $(k+1)$-th inner order is computed and it holds $G_k\in R^C_{k,(\beta,\gamma)}$ as well as $g_{k+1}\in \hat{R}^I_{k+1,(\beta)}+R^C_{k,(\beta)}$ for all $\beta\in(0,\min\{\overline{\beta}(\gamma),\sqrt{f''(\pm1)}\})$, $\gamma\in(0,\overline{\gamma})$, where $\overline{\beta}, \overline{\gamma}$ are as in Theorem \ref{th_hp_exp3}. As in the last Section \ref{sec_asym_ACND_cp_neum_0} we obtain a unique smooth solution $\overline{u}^C_{k+1}$ to \eqref{eq_asym_ACND_cp_uCk_1}-\eqref{eq_asym_ACND_cp_uCk_2} with the decay $\hat{u}^C_{k+1}\in R^C_{k,(\beta,\gamma)}$ for all $(\beta,\gamma)$ as above. Altogether, the $(k+1)$-th order is constructed.

Finally, by induction the $j$-th order is determined for all $j=0,...,M$, the $h_j$ are smooth and $\hat{u}^I_{j+1}\in R^I_{j,(\beta_1)}$ for all $\beta_1\in(0,\min\{\sqrt{f''(\pm1)}\})$ as well as $\hat{u}^C_j\in R^C_{j,(\beta,\gamma)}$ for every $\beta\in(0,\min\{\overline{\beta}(\gamma),\sqrt{f''(\pm1)}\})$, $\gamma\in(0,\overline{\gamma})$, where $\overline{\beta}, \overline{\gamma}$ are as in Theorem \ref{th_hp_exp3}.

\subsubsection{The Approximate Solution for (AC) in ND}\label{sec_asym_ACND_uA}
Let $N\geq2$ and $\Gamma:=(\Gamma_t)_{t\in[0,T]}$ be as in Section \ref{sec_coord_surface_requ} with contact angle $\alpha=\frac{\pi}{2}$ and a solution to \eqref{MCF} in $\Omega$. Moreover, let $\delta>0$ be such that the assertions of Theorem \ref{th_coordND} hold for $2\delta$ instead of $\delta$ and let $r,s,b,\sigma, \mu_1$ be as in the theorem. Furthermore, let $M\in\N$, $M\geq 2$ be as in the beginning of Section \ref{sec_asym_ACND}. Let $\eta:\R\rightarrow[0,1]$ be smooth with $\eta(r)=1$ for $|r|\leq 1$ and $\eta(r)=0$ for $|r|\geq 2$. Then for $\varepsilon>0$ we set
\begin{align*}
u^A_\varepsilon:=
\begin{cases}
\eta(\frac{r}{\delta})\left[u^I_\varepsilon+u^C_\varepsilon \eta(\frac{b}{\mu_1})\right]+(1-\eta(\frac{r}{\delta}))\textup{sign}(r)&\quad\text{ in }\overline{\Gamma(2\delta)},\\
\pm 1&\quad\text{ in }Q_T^\pm\setminus\Gamma(2\delta),
\end{cases}
\end{align*}
where $u^I_\varepsilon$ and $u^C_\varepsilon$ were constructed in Sections \ref{sec_asym_ACND_in} and \ref{sec_asym_ACND_cp}. This yields an approximate solution for \eqref{eq_AC1}-\eqref{eq_AC3} in the following sense:
\begin{Lemma}\label{th_asym_ACND_uA}
	The function $u^A_\varepsilon$ is smooth, uniformly bounded with respect to $x,t,\varepsilon$ and for the remainder 
	$r^A_\varepsilon := 
	\partial_t u^A_\varepsilon-\Delta u^A_\varepsilon +\frac{1}{\varepsilon^2}f'(u^A_\varepsilon)$ in \eqref{eq_AC1} and $s^A_\varepsilon:=\partial_{N_{\partial\Omega}} u^A_\varepsilon$ in \eqref{eq_AC2} it holds
	\begin{alignat*}{2}
	|r^A_\varepsilon|&\leq C(\varepsilon^M e^{-c|\rho_\varepsilon|}+\varepsilon^{M+1})&\quad &\text{ in }\Gamma(2\delta,\mu_1),\\
	|r^A_\varepsilon|&
	\leq 
	C(\varepsilon^{M-1} e^{-c(|\rho_\varepsilon|+H_\varepsilon)}
	+\varepsilon^M e^{-c|\rho_\varepsilon|}+\varepsilon^{M+1})
	&\quad &\text{ in }\Gamma^C(2\delta,2\mu_1),\\
	r^A_\varepsilon &
	=0&\quad&\text{ in }Q_T\setminus\Gamma(2\delta),\\
	|s^A_\varepsilon|&
	\leq C\varepsilon^M e^{-c|\rho_\varepsilon|}&\quad&\text{ on }\partial Q_T\cap\Gamma(2\delta),\\
	s^A_\varepsilon &
	=0&\quad&\text{ on }\partial Q_T\setminus\Gamma(2\delta)
	\end{alignat*}
	for $\varepsilon>0$ small and some $c,C>0$. Here $\rho_\varepsilon$ is defined in \eqref{eq_asym_ACND_rho} and $H_\varepsilon=\frac{b}{\varepsilon}$.
\end{Lemma}

\begin{Remark}\upshape \label{th_asym_ACND_uA_rem}
	The estimate also holds without the $\varepsilon^{M+1}$-term. This follows from a more precise consideration of the remainder terms in the Taylor expansions in Sections \ref{sec_asym_ACND_in}-\ref{sec_asym_ACND_cp} and below. Moreover, one could also lower the number of terms needed in the Taylor expansions a little bit by looking closely at the construction in the previous sections. This would only be of interest if one considers hypersurfaces of class $C^l$ for some large but finite $l$.
\end{Remark}

\begin{proof}[Proof of Lemma \ref{th_asym_ACND_uA}.]
	The third and the last equation are evident from the construction. Moreover, the rigorous Taylor expansions \eqref{eq_asym_ACND_in_taylor_f}-\eqref{eq_asym_ACND_in_taylor2}, \eqref{eq_asym_ACND_cp_taylor_f}-\eqref{eq_asym_ACND_cp_taylor3} and \eqref{eq_asym_ACND_cp_taylor4} together with the remarks for the remainders and Sections \ref{sec_asym_ACND_in}-\ref{sec_asym_ACND_cp} yield
	\begin{alignat*}{2}
	|\partial_tu^I_\varepsilon-\Delta u^I_\varepsilon+\frac{1}{\varepsilon^2}f'(u^I_\varepsilon)|&\leq C(\varepsilon^M e^{-c|\rho_\varepsilon|}+\varepsilon^{M+1}) 
	&\text{ in }\Gamma(2\delta),\\
	\left|\partial_tu^C_\varepsilon-\Delta u^C_\varepsilon+\frac{f'(u^I_\varepsilon+u^C_\varepsilon)-f'(u^I_\varepsilon)}{\varepsilon^2}\right|&\leq C(\varepsilon^{M-1} e^{-c(|\rho_\varepsilon|+H_\varepsilon)}+\varepsilon^{M+1})\hspace{-5.5ex}
	&\text{ in }\Gamma^C(2\delta,2\mu_1),\\
	|\partial_{N_{\partial\Omega}}(u^I_\varepsilon+u^C_\varepsilon)|&\leq C\varepsilon^M e^{-c|\rho_\varepsilon|} &\text{ on }\Gamma^C(2\delta,2\mu_1)\cap\partial Q_T.
	\end{alignat*}
	Therefore the first estimate in the lemma holds for the remainder of $u^I_\varepsilon$ in \eqref{eq_AC1} on $\Gamma(2\delta)$ and the second estimate in the lemma is true for the remainder of $u^I_\varepsilon+u^C_\varepsilon$ in \eqref{eq_AC1} on $\Gamma^C(2\delta,2\mu_1)$. In order to use this for $u^A_\varepsilon$, we have to deal with the mixed terms due to the cutoff-functions. First, we prove that the remainder of 
	$\tilde{u}^A_\varepsilon :=u^I_\varepsilon+u^C_\varepsilon \eta(b/\mu_1)$ in \eqref{eq_AC1} satisfies the first two estimates in the lemma. Note that on $\Gamma(2\delta,1-2\mu_1)$ and $\Gamma^C(2\delta,\mu_1)$ this follows from the above estimates. Moreover, due to Taylor expansions it holds
	\begin{align*}
	f'(u^I_\varepsilon+\eta(\tfrac{b}{\mu_1})u^C_\varepsilon) 
	&= f'(u^I_\varepsilon) +\Oc\left(\eta(\tfrac{b}{\mu_1})|u^C_\varepsilon|\right)\\
	&= (1-\eta(\tfrac{b}{\mu_1})) f'(u^I_\varepsilon) 
	+ \eta(\tfrac{b}{\mu_1}) f'(u^I_\varepsilon + u^C_\varepsilon) + \Oc\left(\eta(\tfrac{b}{\mu_1})|u^C_\varepsilon|\right).
	\end{align*}
	Hence with the product rule for $\partial_t$ and $\Delta$ we obtain
	\begin{align*}
	&\left[\text{l.h.s. in \eqref{eq_AC1} for } u^I_\varepsilon+\eta(\tfrac{b}{\mu_1})u^C_\varepsilon\right]
	\\ 
	&= 
	(1-\eta(\tfrac{b}{\mu_1})) 
	\left[\text{l.h.s. in \eqref{eq_AC1} for } u^I_\varepsilon\right] 
	+
	\eta(\tfrac{b}{\mu_1}) 
	\left[\text{l.h.s. in \eqref{eq_AC1} for } u^I_\varepsilon+u^C_\varepsilon\right]\\
	&+u^C_\varepsilon(\partial_t-\Delta)\left[\eta(\tfrac{b}{\mu_1})\right] -2\nabla\left[\eta(\tfrac{b}{\mu_1})\right]\cdot \nabla u^C_\varepsilon
	+\frac{1}{\varepsilon^2}\Oc\left(\eta(\tfrac{b}{\mu_1})|u^C_\varepsilon|\right),
	\end{align*}
	where \enquote{l.h.s.} stands for \enquote{left hand side}. Hence the above estimates and the asymptotics of $\hat{u}^C_j$ for $j=0,...,M$ yield the desired estimates for the remainder of $\tilde{u}^A_\varepsilon$ in \eqref{eq_AC1} on $\Gamma(2\delta)$. Altogether the first two estimates hold for $\delta$ instead of $2\delta$. In $\Gamma(2\delta)\textbackslash\Gamma(\delta)$ we have again similar mixed terms as above due to the cutoff functions. With Taylor expansions we obtain in $Q_T^\pm\textbackslash\Gamma(\delta)$:
	\begin{align*}
	f'(u^A_\varepsilon) 
	= \Oc\left(\eta(\tfrac{r}{\delta})|\tilde{u}^A_\varepsilon\mp 1|\right),\quad 
	\eta(\tfrac{r}{\delta}) f'(\tilde{u}^A_\varepsilon) 
	= \Oc\left(\eta(\tfrac{r}{\delta})|\tilde{u}^A_\varepsilon\mp 1| \right),
	\end{align*}
	where we used $f'(\pm 1)=0$. Hence the product rule for $\partial_t$ and $\Delta$ yields in $\Gamma(2\delta)\textbackslash\Gamma(\delta)$
	\begin{align*}
	\left[\text{l.h.s. in \eqref{eq_AC1} for }u^A_\varepsilon\right] 
	&=\eta(\tfrac{r}{\delta})\left[\text{l.h.s. in \eqref{eq_AC1} for }\tilde{u}^A_\varepsilon\right]
	+ (\tilde{u}^A_\varepsilon\mp 1)(\partial_t-\Delta)\left[\eta(\tfrac{r}{\delta})\right]\\
	&-2\nabla\left[\eta(\tfrac{r}{\delta})\right]\cdot\nabla\tilde{u}^A_\varepsilon  
	+\frac{1}{\varepsilon^2}\Oc\left(\eta(\tfrac{r}{\delta})|\tilde{u}^A_\varepsilon\mp 1|\right).
	\end{align*}
	Finally, the asymptotics of $u^I_j$ and $u^C_j$ for $j=0,...,M$ imply the estimates for $r^A_\varepsilon$. 
	
	It is left to prove the remaining assertion for $s^A_\varepsilon$. By definition it holds $u^A_\varepsilon=u^I_\varepsilon+u^C_\varepsilon$ on $\Gamma^C(\delta,\mu_1)\cap\partial Q_T$. For the latter we already have an estimate of the Neumann derivative on $\Gamma^C(2\delta,\mu_1)\cap\partial Q_T$, see above. Again we have mixed terms for $u^A_\varepsilon$ on $[\Gamma^C(2\delta,\mu_1)\cap\partial Q_T]\textbackslash \Gamma(\delta)$ because of the cutoff-functions. On the latter set it holds 
	$\tilde{u}^A_\varepsilon=u^I_\varepsilon + u^C_\varepsilon$ and
	\[
	\partial_{N_{\partial\Omega}} u^A_\varepsilon 
	= (\tilde{u}^A_\varepsilon\mp 1) \partial_{N_{\partial\Omega}} \left[\eta(\tfrac{r}{\varepsilon})\right] + \eta(\tfrac{r}{\varepsilon}) \partial_{N_{\partial\Omega}}\tilde{u}^A_\varepsilon.
	\]
	Therefore the claim follows with the asymptotics of $u^I_j$ and $u^C_j$ for $j=0,...,M$.
\end{proof}

\subsection{Asymptotic Expansion of (vAC) in ND}\label{sec_asym_vAC}
Let $N\geq 2$, $\Omega\subseteq\R^N$, $\Gamma:=(\Gamma_t)_{t\in[0,T]}$ and $\delta>0$ be as in the beginning of Section \ref{sec_asym_ACND}, in particular $\Gamma$ is a smooth solution to \eqref{MCF} with $90$°-contact angle condition in $\Omega$. Moreover, let $W:\R^m\rightarrow\R$ be as in Definition \ref{th_vAC_W} and $\vec{u}_\pm$ be any distinct pair of minimizers of $W$. In this section we construct a smooth approximate solution $\vec{u}_\varepsilon^A$ to \eqref{eq_vAC1}-\eqref{eq_vAC3} with $\vec{u}_\varepsilon^A=\vec{u}_\pm$ in $Q_T^\pm\setminus\Gamma(2\delta)$, increasingly \enquote{steep} transition from $\vec{u}_-$ to $\vec{u}_+$ for $\varepsilon\rightarrow 0$ and such that $\{ \vec{u}^A_\varepsilon=0 \}$ converges to $\Gamma$ for $\varepsilon\rightarrow 0$. All computations are very similar to the ones in Section \ref{sec_asym_ACND}. We just have to incorporate vector-valued functions and for the appearing vector-valued model problems we use the corresponding solution theorems in Sections \ref{sec_ODE_vect}-\ref{sec_hp_vect}. For the latter we make the assumption $\dim\ker\check{L}_0=1$, where $\check{L}_0$ is as in Remark \ref{th_ODE_vect_lin_op_rem} for a solution $\vec{\theta}_0$ as in Theorem \ref{th_ODE_vect}.

Let $M\in\N$ with $M\geq 2$. Again we introduce height functions 
\[
\check{h}_j:\Sigma\times[0,T]\rightarrow\R\quad\text{ for }j=1,...,M\quad\text{ and }\quad\check{h}_\varepsilon:=\sum_{j=1}^M \varepsilon^{j-1}\check{h}_j.
\] 
Moreover, we set $\check{h}_{M+1}:=\check{h}_{M+2}:=0$ and we define the scaled variable
\begin{align}\label{eq_asym_vAC_rho}
\check{\rho}_\varepsilon(x,t):=\frac{r(x,t)}{\varepsilon}-\check{h}_\varepsilon(s(x,t),t)\quad\text{ for }(x,t)\in\overline{\Gamma(2\delta)}.
\end{align}

In Section \ref{sec_asym_vAC_in} we construct the inner expansion and in Section \ref{sec_asym_vAC_cp} the contact point expansion. Finally, in Section \ref{sec_asym_vAC_uA} the result on the approximation error of $\vec{u}^A_\varepsilon$ can be found.

\subsubsection{Inner Expansion of (vAC) in ND}\label{sec_asym_vAC_in}
For the inner expansion we consider the following ansatz: Let $\varepsilon>0$ be small and
\[
\vec{u}^I_\varepsilon:=\sum_{j=0}^{M+1}\varepsilon^j\vec{u}_j^I,\quad \vec{u}_j^I(x,t):=\check{u}_j^I(\rho_\varepsilon(x,t),s(x,t),t)\quad\text{ for }(x,t)\in\overline{\Gamma(2\delta)},
\]
where 
\[
\check{u}_j^I:\R\times\Sigma\times[0,T]\rightarrow\R^m: (\rho,s,t)\mapsto \check{u}_j^I(\rho,s,t)
\] 
for $j=0,...,M+1$. Moreover, we set $\vec{u}^I_{M+2}:=0$ and $\check{u}^I_\varepsilon:=\sum_{j=0}^{M+1}\varepsilon^j\check{u}_j^I$. We will expand \eqref{eq_vAC1} for $\vec{u}_\varepsilon=\vec{u}^I_\varepsilon$ into $\varepsilon$-series with coefficients in $(\check{\rho}_\varepsilon,s,t)$ up to $\Oc(\varepsilon^{M-1})$. This yields equations of analogous form as in the scalar case in Section \ref{sec_asym_ACND_in}. Therefore we have to compute the action of the differential operators on $\vec{u}^I_\varepsilon$. 

In the following we use the same conventions as in Lemma \ref{th_asym_ACND_in_trafo}. Moreover, for a sufficiently smooth $\vec{g}:\Sigma\rightarrow\R^m$ we set 
$D_\Sigma \vec{g}:=(\nabla_\Sigma g_1,...,\nabla_\Sigma g_m)^\top:\Sigma\rightarrow\R^{m\times N}$.

\begin{Lemma}\label{th_asym_vAC_in_trafo}
	Let $\varepsilon>0$, $\check{w}:\R\times \Sigma\times[0,T]\rightarrow\R^m$ be sufficiently smooth and $\vec{w}:\overline{\Gamma(2\delta)}\rightarrow\R^m$ be defined by $\vec{w}(x,t):=\check{w}(\check{\rho}_\varepsilon(x,t),s(x,t),t)$ for all $(x,t)\in\overline{\Gamma(2\delta)}$. Then it holds
	\begin{align*}
	\partial_t\vec{w}&=\partial_\rho\check{w}\left[\frac{\partial_tr}{\varepsilon}-(\partial_t\check{h}_\varepsilon+\partial_ts\cdot \nabla_\Sigma \check{h}_\varepsilon)\right]+D_\Sigma\check{w}\,\partial_ts +\partial_t\check{w},\\
	D_x\vec{w}&=\partial_\rho\check{w}\left[\frac{\nabla r}{\varepsilon}-(D_x s)^\top \nabla_\Sigma \check{h}_\varepsilon\right]^\top +D_\Sigma\check{w} D_x s,\\
	\Delta\vec{w}&=\partial_\rho\check{w}\left[\frac{\Delta r}{\varepsilon}-\left(\Delta s\cdot \nabla_\Sigma \check{h}_\varepsilon+\sum_{i,l=1}^N\nabla s_i\cdot\nabla s_l(\nabla_\Sigma)_i(\nabla_\Sigma)_l \check{h}_\varepsilon\right)\right]\\
	&+D_\Sigma\check{w}\,\Delta s+\sum_{i,l=1}^N\nabla s_i\cdot\nabla s_l(\nabla_\Sigma)_i(\nabla_\Sigma)_l\check{w}\\
	&+2D_\Sigma\partial_\rho\check{w}\,D_x s\,\left[\frac{\nabla r}{\varepsilon}-(D_xs)^\top\nabla_\Sigma \check{h}_\varepsilon\right]+\partial_\rho^2\check{w}\left|\frac{\nabla r}{\varepsilon}-(D_xs)^\top\nabla_\Sigma\check{h}_\varepsilon\right|^2,
	\end{align*}
	where the $\vec{w}$-terms on the left hand side and derivatives of $r$ or $s$ are evaluated at $(x,t)\in\overline{\Gamma(2\delta)}$, the $\check{h}_\varepsilon$-terms at $(s(x,t),t)$ and the $\check{w}$-terms at $(\check{\rho}_\varepsilon(x,t),s(x,t),t)$. 
\end{Lemma}

\begin{proof}[Proof of Lemma \ref{th_asym_vAC_in_trafo}]
	This follows from Lemma \ref{th_asym_ACND_in_trafo} applied to every component.
\end{proof}

For the expansion of \eqref{eq_vAC1} for $\vec{u}_\varepsilon=\vec{u}_\varepsilon^I$ we use Taylor expansions again. For the $\nabla W$-part this yields: If the $\vec{u}_j^I$ are bounded, then 
\begin{align}\label{eq_asym_vAC_in_taylor_W}
\nabla W(\vec{u}^I_\varepsilon)=\nabla W(\vec{u}_0^I)+\sum_{\nu\in\N_0^m,|\nu|=1}^{M+2}\frac{\partial_y^{\nu}\nabla W(\vec{u}_0^I)}{\nu!}\left[\sum_{j=1}^{M+1}\vec{u}_j^I\varepsilon^j\right]^\nu+\Oc(\varepsilon^{M+3})\,\,\text{ on }\overline{\Gamma(2\delta)}.
\end{align}
The terms in the $\varepsilon$-expansion that are needed explicitly are
\begin{align*}
\Oc(1)&:\quad \nabla W(\vec{u}_0^I),\\
\Oc(\varepsilon)&:\quad D^2W(\vec{u}_0^I)\vec{u}_1^I,\\
\Oc(\varepsilon^2)&:\quad D^2W(\vec{u}_0^I)\vec{u}_2^I
+\sum_{\nu\in\N_0^m,|\nu|=2}\frac{\partial_y^{\nu}\nabla W(\vec{u}_0^I)}{\nu!}\left[\vec{u}_1^I\right]^\nu.
\end{align*}
For $k=3,...,M+2$ the order $\Oc(\varepsilon^k)$ is given by
\begin{alignat*}{2}
\Oc(\varepsilon^k)&:\quad D^2W(\vec{u}_0^I)\vec{u}_k^I + & [\text{some polynomial in entries of }(\vec{u}_1^I,...,\vec{u}_{k-1}^I)\text{ of order }\leq k,\\
& &\text{where the coefficients are multiples of }\partial_y^\nu\nabla W(\vec{u}^I_0),\\
& & \nu\in\N_0^m,|\nu|=2,...,k\text{ and every term contains a }(\vec{u}^I_j)_n\text{-factor}].
\end{alignat*}
The other explicit terms in \eqref{eq_asym_vAC_in_taylor_W} are of order $\Oc(\varepsilon^{M+3})$.

Moreover, we expand functions of $(x,t)\in\overline{\Gamma(2\delta)}$ into $\varepsilon$-series analogously to the scalar case, cf.~the Taylor expansion \eqref{eq_asym_ACND_in_taylor2} and the remarks there. We just replace $h_j, \rho_\varepsilon$ by $\check{h}_j, \check{\rho}_\varepsilon$.

For the higher orders in the expansion we use analogous definitions as in the scalar case:
\begin{Definition}[\textbf{Notation for Inner Expansion of (vAC)}]\upshape\phantomsection{\label{th_asym_vAC_in_def}}
	\begin{enumerate}
		\item We call $(\vec{\theta}_0,\vec{u}^I_1)$ the \textit{zero-th inner order} and $(\check{h}_j,\vec{u}^I_{j+1})$ the \textit{$j$-th inner order} for $j=1,...,M$. 
		\item Let $k\in\{-1,...,M+2\}$ and $\beta>0$. We denote with $\check{R}^I_{k,(\beta)}$ the set of smooth vector-valued functions $\vec{R}:\R\times\Sigma\times[0,T]\rightarrow\R^m$ that depend only on the $j$-th inner orders for $0\leq j\leq \min\{k,M\}$ and satisfy uniformly in $(\rho,s,t)$:
		\[
		|\partial_\rho^i(\nabla_\Sigma)_{n_1}...(\nabla_\Sigma)_{n_d}\partial_t^n \vec{R}(\rho,s,t)|=\Oc(e^{-\beta|\rho|})
		\]
		for all $n_1,...,n_d\in\{1,...,N\}$ and $d,i,n\in\N_0$.
		\item For $k\in\{-1,...,M+2\}$ and $\beta>0$ the set $\tilde{R}^I_{k,(\beta)}$ is defined analogously to $\check{R}^I_{k,(\beta)}$ with functions of type $\vec{R}:\R\times\partial\Sigma\times[0,T]\rightarrow\R^m$ instead.
	\end{enumerate}
\end{Definition}

Now we expand \eqref{eq_vAC1} for $\vec{u}_\varepsilon=\vec{u}^I_\varepsilon$ into $\varepsilon$-series. This is analogous to the scalar case in Section \ref{sec_asym_ACND_in}. In the following $(\rho,s,t)$ are always in $\R\times \Sigma\times[0,T]$ and sometimes omitted.

\paragraph{Inner Expansion: $\Oc(\varepsilon^{-2})$}\label{sec_asym_vAC_in_m2}
Using $|\nabla r|^2|_{\overline{X}_0(s,t)}=1$ due to Theorem \ref{th_coordND}, the $\Oc(\frac{1}{\varepsilon^2})$-order is zero if
\begin{align}\label{eq_asym_vAC_in_u0}
\quad -\partial_\rho^2\check{u}_0^I(\rho,s,t)+\nabla W(\check{u}_0^I(\rho,s,t))=0.
\end{align}
Since we want to connect the minima $\vec{u}_\pm$ of $W$, we require $\lim_{\rho\rightarrow\pm\infty}\check{u}_0^I(\rho,s,t)=\vec{u}_\pm$. Moreover, it is natural to ask for $R_{\vec{u}_-,\vec{u}_+}\check{u}_0^I|_{\rho=0}=\check{u}_0^I|_{\rho=0}$, since then $\check{u}_0^I|_{\rho=0}$ can be interpreted as being in the middle of the two phases $\vec{u}_\pm$. Here $R_{\vec{u}_-,\vec{u}_+}$ is as in Definition \ref{th_vAC_W}. By Theorem \ref{th_ODE_vect} there is a smooth $R_{\vec{u}_-,\vec{u}_+}$-odd $\vec{\theta}_0:\R\rightarrow\R^m$ such that $\check{u}^I_0(\rho,s,t):=\vec{\theta}_0(\rho)$ solves \eqref{eq_asym_vAC_in_u0} and
\[
\partial_z^l[\vec{\theta}_0-\vec{u}_\pm](\rho)=\Oc(e^{-\beta|\rho|})\quad\text{ for }\rho\rightarrow\pm\infty\text{ and all }l\in\N_0, \beta\in\left(0,\sqrt{\lambda/2}\right),
\] 
where $\lambda>0$ is such that $D^2W(\vec{u}_\pm)\geq\lambda I$.
Moreover, it holds $R_{\vec{u}_-,\vec{u}_+}\vec{\theta}_0'|_{\rho=0}\neq \vec{\theta}_0'|_{\rho=0}$.
\paragraph{Inner Expansion: $\Oc(\varepsilon^{-1})$}\label{sec_asym_vAC_in_m1}
Analogously to the scalar case, cf.~Section \ref{sec_asym_ACND_in_m1}, it follows that the $\Oc(\frac{1}{\varepsilon})$-order cancels if  
\[
\check{\Lc}_0\check{u}_1^I(\rho,s,t)+\vec{\theta}_0'(\rho)(\partial_tr-\Delta r )|_{\overline{X}_0(s,t)}=0,\quad\text{ where }\check{\Lc}_0:=-\partial_\rho^2+D^2W(\vec{\theta}_0).
\]
Moreover, it is natural to require that $R_{\vec{u}_-,\vec{u}_+}\check{u}_1^I|_{\rho=0}=\check{u}_1^I|_{\rho=0}$ since then heuristically $\check{u}_1^I|_{\rho=0}$ is in the middle of the two phases $\vec{u}_\pm$. We assume $\dim\ker\check{L}_0=1$ with respect to the spaces in \eqref{eq_ODE_vect_L0}, cf.~Remark \ref{th_ODE_vect_lin_op_rem}. Then due to Theorem \ref{th_ODE_vect_lin} and Remarks \ref{th_ODE_vect_lin_rem1}, \ref{th_ODE_vect_lin_rem2},~1.~this parameter-dependent ODE together with the additional condition and suitable decay in $|\rho|$ has a unique solution $\check{u}_1^I$ if and only if $(\partial_tr-\Delta r)|_{\overline{X}_0(s,t)}=0$. The latter holds since it is equivalent to \eqref{MCF} for $\Gamma$ by Theorem \ref{th_coordND}. Therefore we set $\check{u}_1^I:=0$.

\paragraph{Inner Expansion: $\Oc(\varepsilon^0)$}\label{sec_asym_vAC_in_0}
In the analogous way as in the scalar case, cf.~Section \ref{sec_asym_ACND_in_0}, the $\Oc(1)$-term in the expansion cancels if we require
\begin{align}\label{eq_asym_vAC_in_u2}
-\check{\Lc}_0\check{u}_2^I(\rho,s,t)&=\vec{R}_1(\rho, s,t),\\ \notag
\vec{R}_1(\rho,s,t):=\vec{\theta}_0'(\rho)&\left[-\partial_t\check{h}_1+\sum_{i,l=1}^N\nabla s_i\cdot\nabla s_l(\nabla_\Sigma)_i(\nabla_\Sigma)_l\check{h}_1 \right.\\\notag
&\left. +(\rho+\check{h}_1)\partial_r((\partial_tr-\Delta r)\circ\overline{X})|_{(0,s,t)}-(\partial_ts-\Delta s)|_{\overline{X}_0(s,t)}\cdot\nabla_\Sigma\check{h}_1\right]\\ \notag
+\vec{\theta}_0''(\rho)&\left[-\frac{1}{2}(\rho+\check{h}_1)^2\partial_r^2(|\nabla r|^2\circ\overline{X})|_{(0,s,t)}\right.\\
&\left.+2(\rho+\check{h}_1)\partial_r((D_xs\nabla r)^\top\circ\overline{X})|_{(0,s,t)}\nabla_\Sigma\check{h}_1-\left|(D_xs)^\top|_{\overline{X}_0(s,t)}\nabla_\Sigma\check{h}_1\right|^2\right].\notag
\end{align}
If $\check{h}_1$ is smooth, then $\vec{R}_1$ is smooth and together with all derivatives decays exponentially in $|\rho|$ uniformly in $(s,t)$ with rate $\beta$ for every $\beta\in(0,\sqrt{\lambda/2})$ because of Theorem \ref{th_ODE_vect}. Therefore Theorem \ref{th_ODE_vect_lin} (applied in local coordinates for $\Sigma$) yields that there is a unique solution $\check{u}_2^I$ to \eqref{eq_asym_vAC_in_u2} together with suitable regularity and decay as well as $R_{\vec{u}_-,\vec{u}_+}\check{u}_2^I|_{\rho=0}=\check{u}_2^I|_{\rho=0}$ if and only if $\int_\R \vec{R}_1(\rho,s,t)\cdot\vec{\theta}_0'(\rho)\,d\rho=0$. Because of integration by parts it holds $\int_\R\vec{\theta}_0'(\rho)\cdot\vec{\theta}_0''(\rho)\,d\rho=0$. Therefore the nonlinearities in $\check{h}_1$ cancel and we obtain a linear non-autonomous parabolic equation for $\check{h}_1$ on $\Sigma$:
\begin{align}\label{eq_asym_vAC_in_h1}
\partial_t\check{h}_1-\sum_{i,l=1}^N\nabla s_i\cdot\nabla s_l|_{\overline{X}_0(s,t)}(\nabla_\Sigma)_i(\nabla_\Sigma)_l \check{h}_1+\check{a}_1\cdot \nabla_\Sigma \check{h}_1+\check{a}_0\check{h}_1=\check{f}_0
\end{align} 
in $\Sigma\times[0,T]$. Here with
\begin{alignat*}{3}
\check{d}_1&:=\int_\R|\vec{\theta}_0'(\rho)|^2\,d\rho,&\quad \check{d}_2&:=\int_\R|\vec{\theta}_0'(\rho)|^2\rho\,d\rho,&\quad
\check{d}_3&:=\int_\R|\vec{\theta}_0'(\rho)|^2\rho^2\,d\rho,\\
\check{d}_4&:=\int_\R\vec{\theta}_0'(\rho)\cdot\vec{\theta}_0''(\rho)\rho\,d\rho,&\quad \check{d}_5&:=\int_\R\vec{\theta}_0'(\rho)\cdot\vec{\theta}_0''(\rho)\rho^2\,d\rho,&\quad
\check{d}_6&:=\int_\R\vec{\theta}_0'(\rho)\cdot\vec{\theta}_0''(\rho)\rho^3\,d\rho,
\end{alignat*}
we have defined for all $(s,t)\in\Sigma\times[0,T]$:
\begin{align}
\check{a}_1(s,t)&:=(\partial_ts-\Delta s)|_{\overline{X}_0(s,t)}-2\frac{\check{d}_4}{\check{d}_1}\partial_r((D_xs\nabla r)^\top \circ\overline{X})|_{(0,s,t)}\in \R^N,\label{eq_asym_vAC_in_a1}\\
\check{a}_0(s,t)&:=-\partial_r((\partial_tr-\Delta r)\circ\overline{X})|_{(0,s,t)}+\frac{\check{d}_4}{\check{d}_1}\partial_r^2(|\nabla r|^2\circ\overline{X})|_{(0,s,t)}\in\R,\label{eq_asym_vAC_in_a0}\\
\check{f}_0(s,t)&:=\frac{\check{d}_2}{\check{d}_1}\partial_r((\partial_tr-\Delta r)\circ\overline{X})|_{(0,s,t)}-\frac{\check{d}_5}{2\check{d}_1}\partial_r^2(|\nabla r|^2\circ\overline{X})|_{(0,s,t)}\in\R.
\end{align}
Note that since $\vec{\theta}_0$ is $R_{\vec{u}_-,\vec{u}_+}$-odd and due to the isometry properties of $R_{\vec{u}_-,\vec{u}_+}$, it follows that $\check{d}_2=\check{d}_5=0$ and hence also $\check{f}_0=0$. Therefore the equation \eqref{eq_asym_vAC_in_h1} for $\check{h}_1$ is homogeneous. Note that this corresponds to the case of symmetric $f$ in the scalar case, cf.~Remark \ref{th_asym_ACND_feven_rem1}. This is due to the fact that we restricted to symmetric $W$ in the vector-valued case.

If $\check{h}_1$ is smooth and solves \eqref{eq_asym_vAC_in_h1}, then Theorem \ref{th_ODE_vect_lin} (applied in local coordinates for $\Sigma$) yields a smooth solution $\check{u}_2^I$ to \eqref{eq_asym_vAC_in_u2} with $R_{\vec{u}_-,\vec{u}_+}\check{u}_2^I|_{\rho=0}=\check{u}_2^I|_{\rho=0}$ and we get decay estimates. With Remark \ref{th_nabla_sigma_equiv} and compactness we obtain $\check{u}_2^I\in \check{R}^I_{1,(\beta)}$ for any $\beta\in(0,\min\{\sqrt{\lambda/2}, \check{\beta}_0\})$, where $\check{\beta}_0>0$ is as in Theorem \ref{th_ODE_vect_lin}.

\paragraph{Inner Expansion: $\Oc(\varepsilon^k)$}\label{sec_asym_vAC_in_k}
Let $k\in\{1,...,M-1\}$ and suppose that the $j$-th inner order has already been constructed for $j=0,...,k$, that it is smooth and $\check{u}_{j+1}^I\in\check{R}_{j,(\beta)}^I$ for every $\beta\in(0,\min\{\sqrt{\lambda/2},\check{\beta}_0\})$ with $\check{\beta}_0>0$ as in Theorem \ref{th_ODE_vect_lin}. Analogously to the scalar case one can compute the $\Oc(\varepsilon^k)$-order in \eqref{eq_vAC1} for $\vec{u}_\varepsilon=\vec{u}^I_\varepsilon$. This yields that the order $\Oc(\varepsilon^k)$ is zero if 
\begin{align}\label{eq_asym_vAC_in_uk}
-\check{\Lc}_0\check{u}_{k+2}^{I}(\rho,s,t)&=\vec{R}_{k+1}(\rho,s,t),\\
\vec{R}_{k+1}(\rho,s,t):=\vec{\theta}_0'(\rho)\notag
&\left[
-\partial_t\check{h}_{k+1}
+\sum_{i,l=1}^N\nabla s_i\cdot\nabla s_l|_{\overline{X}_0(s,t)}(\nabla_\Sigma)_i(\nabla_\Sigma)_l \check{h}_{k+1}\right. \\ \notag
&\left.-(\partial_ts-\Delta s)|_{\overline{X}_0(s,t)}\cdot \nabla_\Sigma\check{h}_{k+1}
+\check{h}_{k+1}\partial_r((\partial_tr-\Delta r)\circ\overline{X})|_{(0,s,t)}
\right]\\ \notag
+\vec{\theta}_0''(\rho)
&\left[
-(\rho+\check{h}_1)\check{h}_{k+1} \partial_r^2(|\nabla r|^2\circ\overline{X})|_{(0,s,t)}\right.\\ \notag
&-2(\nabla_\Sigma \check{h}_1)^\top D_xs(D_xs)^\top|_{\overline{X}_0(s,t)}\nabla_\Sigma \check{h}_{k+1}\\ \notag
&\left.+2\partial_r((D_xs\nabla r)^\top\circ\overline{X})|_{(0,s,t)}[(\rho+\check{h}_1)\nabla_\Sigma \check{h}_{k+1}+\check{h}_{k+1}\nabla_\Sigma \check{h}_1]
\right]\\ 
+ &\check{R}_{k}(\rho,s,t), \notag
\end{align}
where $\check{R}_{k}\in\check{R}_{k,(\beta)}^I$. If $\check{h}_{k+1}$ is smooth, then due to Theorem \ref{th_ODE_vect_lin} equation \eqref{eq_asym_vAC_in_uk} admits a unique solution $\check{u}_{k+2}^I$ with suitable regularity and decay as well as $R_{\vec{u}_-,\vec{u}_+}\check{u}_2^I|_{\rho=0}=\check{u}_2^I|_{\rho=0}$ if and only if $\int_\R \vec{R}_{k+1}(\rho,s,t)\cdot\vec{\theta}_0'(\rho)\,d\rho=0$. Because of $\int_\R\vec{\theta}_0''\cdot\vec{\theta}_0'=0$, the latter is equivalent to 
\begin{align}\label{eq_asym_vAC_in_hk}
\partial_t\check{h}_{k+1}-\sum_{i,l=1}^N\nabla s_i\cdot\nabla s_l|_{\overline{X}_0(s,t)}(\nabla_\Sigma)_i(\nabla_\Sigma)_l \check{h}_{k+1}+\check{a}_1\cdot\nabla_\Sigma \check{h}_{k+1} + \check{a}_0 \check{h}_{k+1} = \check{f}_k,
\end{align}
where 
\[
\check{f}_k(s,t):=\int_\R \check{R}_k(\rho,s,t)\cdot\vec{\theta}_0'(\rho)\,d\rho/\|\vec{\theta}_0'\|_{L^2(\R)^m}^2
\] 
is a smooth function of $(s,t)$ and depends only on the $j$-th inner orders for $0\leq j\leq k$. Here $\check{a}_0, \check{a}_1$ are defined in \eqref{eq_asym_vAC_in_a1}-\eqref{eq_asym_vAC_in_a0}. If $\check{h}_{k+1}$ is smooth and solves \eqref{eq_asym_vAC_in_hk}, then Theorem \ref{th_ODE_vect_lin} yields as in the last Section \ref{sec_asym_vAC_in_0} a smooth solution $\check{u}_{k+2}^I$ to \eqref{eq_asym_vAC_in_uk} such that $\check{u}_{k+2}^I\in \check{R}_{k+1,(\beta)}^I$ for all $\beta\in(0,\min\{\sqrt{\lambda/2}, \check{\beta}_0\})$.

\subsubsection{Contact Point Expansion of (vAC) in ND}\label{sec_asym_vAC_cp}
This is analogous to the scalar case, cf.~Section \ref{sec_asym_ACND_cp}. We make the ansatz $\vec{u}_\varepsilon=\vec{u}^I_\varepsilon+\vec{u}^C_\varepsilon$ in $\Gamma(2\delta)$ close to the contact points. Let $\sigma,b:\overline{\Gamma^C(2\delta,2\mu_1)}\rightarrow\partial\Sigma\times[0,2\mu_1]$ be as in Theorem \ref{th_coordND}. Then with $H_\varepsilon:=\frac{b}{\varepsilon}$ we set
\[
\vec{u}^C_\varepsilon:=\sum_{j=1}^M\varepsilon^j \vec{u}_j^C,\quad \vec{u}_j^C(x,t):=\check{u}_j^C(\check{\rho}_\varepsilon(x,t),H_\varepsilon(x,t),\sigma(x,t),t)\quad\text{ for }(x,t)\in\overline{\Gamma^C(2\delta,2\mu_1)},
\]
where 
\[
\check{u}_j^C:\overline{\R^2_+}\times\partial\Sigma\times[0,T]\rightarrow\R^m:
(\rho,H,\sigma,t)\mapsto \check{u}_j^C(\rho,H,\sigma,t)
\] 
for $j=1,...,M$. Moreover, we define $\vec{u}^C_{M+1}:=\vec{u}^C_{M+2}:=0$ and $\check{u}^C_\varepsilon:=\sum_{j=1}^M\varepsilon^j \check{u}_j^C$. As in the scalar case, instead of \eqref{eq_vAC1} for $\vec{u}_\varepsilon=\vec{u}^I_\varepsilon+\vec{u}^C_\varepsilon$, we will expand the \enquote{bulk equation}
\begin{align}\label{eq_asym_vAC_cp} 
\partial_t\vec{u}^C_\varepsilon-\Delta \vec{u}^C_\varepsilon+\frac{1}{\varepsilon^2}\left[\nabla W(\vec{u}^I_\varepsilon+\vec{u}^C_\varepsilon)-\nabla W(\vec{u}^I_\varepsilon)\right]=0
\end{align}
into $\varepsilon$-series with coefficients in $(\check{\rho}_\varepsilon,H_\varepsilon,\sigma,t)$ up to $\Oc(\varepsilon^{M-2})$. Moreover, we will expand \eqref{eq_vAC2} for $\vec{u}_\varepsilon=\vec{u}^I_\varepsilon+\vec{u}^C_\varepsilon$ into $\varepsilon$-series with coefficients in $(\check{\rho}_\varepsilon,\sigma,t)$ up to $\Oc(\varepsilon^{M-1})$. Altogether we end up with analogous equations as in the scalar case. The solvability condition \eqref{eq_hp_vect_comp} will yield the boundary conditions on $\partial\Sigma\times[0,T]$ for the height functions $\check{h}_j$.

For the expansions we calculate the action of the differential operators on $\vec{u}^C_\varepsilon$ in the next lemma. Here we use the same conventions as in Lemma \ref{th_asym_vAC_in_trafo} and define $D_{\partial\Sigma}$ in an analogous way as $D_\Sigma$.

\begin{Lemma}\label{th_asym_vAC_cp_trafo}
	Let $\overline{\R^2_+}\times\partial\Sigma\times[0,T]\ni(\rho,H,\sigma,t)\mapsto\check{w}(\rho,H,\sigma,t)\in\R^m$ be sufficiently smooth and let $\vec{w}:\overline{\Gamma^C(2\delta,2\mu_1)}\rightarrow\R^m:(x,t)\mapsto\check{w}(\check{\rho}_\varepsilon(x,t),H_\varepsilon(x,t),\sigma(x,t),t)$. Then
	\begin{align*}
	\partial_t\vec{w}&=\partial_\rho\check{w}\left[\frac{\partial_tr}{\varepsilon}-\left(\partial_t\check{h}_\varepsilon+\partial_ts\cdot \nabla_\Sigma \check{h}_\varepsilon\right)\right]+\partial_H\check{w}\frac{\partial_tb}{\varepsilon}+D_{\partial\Sigma}\check{w}\,\partial_t\sigma+\partial_t\check{w},\\
	D_x\vec{w}&=\partial_\rho\check{w}\left[\frac{\nabla r}{\varepsilon}-(D_xs)^\top\nabla_\Sigma \check{h}_\varepsilon\right]^\top+\partial_H\check{w}\left[\frac{\nabla b}{\varepsilon}\right]^\top+D_{\partial\Sigma}\check{w}D_x\sigma,\\
	\Delta\vec{w}&=\partial_\rho\check{w}\left[\frac{\Delta r}{\varepsilon}-\left(\Delta s\cdot\nabla_\Sigma \check{h}_\varepsilon+\sum_{i,l=1}^N\nabla s_i\cdot\nabla s_l(\nabla_\Sigma)_i(\nabla_\Sigma)_l\check{h}_\varepsilon\right)\right]+\partial_H\check{w}\frac{\Delta b}{\varepsilon}\\
	&+\partial_H^2\check{w}\frac{|\nabla b|^2}{\varepsilon^2}
	+\partial_\rho^2\check{w}\left|\frac{\nabla r}{\varepsilon}-(D_xs)^\top\nabla_\Sigma \check{h}_\varepsilon\right|^2+ 2\partial_\rho\partial_H\check{w}\,\frac{\nabla b}{\varepsilon}\cdot\left[\frac{\nabla r}{\varepsilon}-(D_xs)^\top\nabla_\Sigma \check{h}_\varepsilon\right]\\
	&+2D_{\partial\Sigma}\partial_\rho \check{w}\,D_x\sigma\,
	\left[\frac{\nabla r}{\varepsilon}
	-(D_xs)^\top\nabla_\Sigma\check{h}_\varepsilon\right]
	+2D_{\partial\Sigma}\partial_H\check{w}\,D_x\sigma\,\frac{\nabla b}{\varepsilon}\\
	&+D_{\partial\Sigma}\check{w}\,\Delta\sigma+\sum_{i,l=1}^N \nabla\sigma_i\cdot\nabla\sigma_l(\nabla_{\partial\Sigma})_i(\nabla_{\partial\Sigma})_l\check{w},
	\end{align*}
	where the $\vec{w}$-terms on the left hand side and derivatives of $r$ or $s$ are evaluated at $(x,t)$, the $\check{h}_\varepsilon$-terms at $(s(x,t),t)$ and the $\check{w}$-terms at $(\check{\rho}_\varepsilon(x,t),H_\varepsilon(x,t),\sigma(x,t),t)$.
\end{Lemma}

\begin{proof}
	This can be shown by applying Lemma \ref{th_asym_ACND_cp_trafo} to every component.
\end{proof}

\paragraph{Contact Point Expansion: The Bulk Equation}\label{sec_asym_vAC_cp_bulk}
We expand the $\nabla W$-part in \eqref{eq_asym_vAC_cp}: If the $\vec{u}_j^I, \vec{u}_j^C$ are bounded, the Taylor expansion yields on $\overline{\Gamma(2\delta)}$
\begin{align*}
\nabla W(\vec{u}^I_\varepsilon+\vec{u}^C_\varepsilon)=\nabla W(\vec{\theta}_0)+\sum_{\nu\in\N_0^m,|\nu|=1}^{M+2}\frac{1}{\nu!}\partial_y^\nu\nabla W(\vec{\theta}_0)\left[\sum_{j=1}^{M+1}\varepsilon^j(\vec{u}_j^I+\vec{u}_j^C)\right]^\nu+\Oc(\varepsilon^{M+3}).
\end{align*}
As in the scalar case one can combine the latter with the expansion for $\nabla W(\vec{u}^I_\varepsilon)$ in \eqref{eq_asym_vAC_in_taylor_W} and use $\vec{u}^I_1=0$. This yields that the terms in the asymptotic expansion for $\nabla W(\vec{u}^I_\varepsilon+\vec{u}^C_\varepsilon)-\nabla W(\vec{u}^I_\varepsilon)$ are for $k=1,...,M+1$:
\begin{alignat*}{2}
\Oc(1)&:\quad 0,&\\
\Oc(\varepsilon)&:\quad D^2W(\vec{\theta}_0)\vec{u}^C_1,& \\
\Oc(\varepsilon^k)&:\quad D^2W(\vec{\theta}_0)\vec{u}_k^C +\quad & [\text{some polynomial in entries of }(\vec{u}_1^I,...,\vec{u}_{k-1}^I, \vec{u}^C_1,...,\vec{u}_{k-1}^C)\text{ of}\\
& &\text{order }\leq k, \text{ where the coefficients are multiples of }\partial_y^\nu\nabla W(\vec{\theta}_0),\\
& & \nu\in\N_0^m,|\nu|=2,...,k\text{ and every term contains a }(\vec{u}_j^C)_n\text{-factor}].
\end{alignat*}
The other explicit terms in $\nabla W(\vec{u}^I_\varepsilon+\vec{u}^C_\varepsilon)-\nabla W(\vec{u}^I_\varepsilon)$ are of order $\Oc(\varepsilon^{M+3})$.

Functions of $(s,t)$, $(\rho,s,t)$ and $(x,t)$ are expanded in the analogous way as in the scalar case, cf.~\eqref{eq_asym_ACND_cp_taylor2}-\eqref{eq_asym_ACND_cp_taylor3} and the remarks there. We just replace $h_j,\rho_\varepsilon$ by $\check{h}_j,\check{\rho}_\varepsilon$.

As in the scalar case we use some notation for the higher orders in the expansion:
\begin{Definition}[\textbf{Notation for Contact Point Expansion of (vAC)}]\upshape\phantomsection{\label{th_asym_vAC_cp_def}}
	\begin{enumerate}
		\item We call $(\vec{\theta}_0,\vec{u}^I_1)$ the \textit{zero-th order} and $(\check{h}_j,\vec{u}^I_{j+1},\vec{u}^C_j)$ the \textit{$j$-th order} for $j=1,...,M$. 
		\item Let $k\in\{-1,...,M+2\}$ and $\beta,\gamma>0$. Then $\check{R}^C_{k,(\beta,\gamma)}$ denotes the set of smooth functions $\vec{R}:\overline{\R^2_+}\times\partial\Sigma\times[0,T]\rightarrow\R^m$ depending only on the $j$-th orders for $0\leq j\leq \min\{k,M\}$ and such that uniformly in $(\rho,H,\sigma,t)$:
		\[
		|\partial_\rho^i\partial_H^l(\nabla_{\partial\Sigma})_{n_1}...(\nabla_{\partial\Sigma})_{n_d}\partial_t^n\vec{R}(\rho,H,\sigma,t)|=\Oc(e^{-(\beta|\rho|+\gamma H)})
		\]
		for all $n_1,...,n_d\in\{1,...,N\}$ and $d,i,l,n\in\N_0$.
		\item The set $\check{R}^C_{k,(\beta)}$ is defined in an analogous way without the $H$-dependence.
	\end{enumerate}
\end{Definition}
In the following we expand \eqref{eq_asym_vAC_cp} into $\varepsilon$-series with coefficients in $(\check{\rho}_\varepsilon,H_\varepsilon,\sigma,t)$.

\subparagraph{Bulk Equation: $\Oc(\varepsilon^{-1})$}\label{sec_asym_vAC_cp_bulk_m1}
The lowest order $\Oc(\frac{1}{\varepsilon})$ in  \eqref{eq_asym_vAC_cp} cancels if 
\begin{align}\label{eq_asym_vAC_cp_bulk1}
[-\Delta^{\sigma,t}+D^2W(\vec{\theta}_0(\rho))]\check{u}^C_1(\rho,H,\sigma,t)=0,
\end{align}
where $\Delta^{\sigma,t}:=\partial_\rho^2+|\nabla b|^2|_{\overline{X}_0(\sigma,t)}\partial_H^2$ and we used $\nabla r\cdot\nabla b|_{\overline{X}_0(\sigma,t)}=0$ for all $(\sigma,t)\in\partial\Sigma\times[0,T]$.

\subparagraph{Bulk Equation: $\Oc(\varepsilon^{k-1})$}\label{sec_asym_vAC_cp_bulk_km1}
For $k=1,...,M-1$ we assume that the $j$-th order is constructed for all $j=0,...,k$, that it is smooth and that $\check{u}^I_{j+1}\in \check{R}^I_{j,(\beta)}$ (bounded and all derivatives bounded is enough here) and $\check{u}^C_j\in\check{R}^C_{j,(\beta,\gamma)}$ for every $\beta\in(0,\min\{\check{\beta}(\gamma),\sqrt{\lambda/2},\check{\beta}_0\})$, $\gamma\in(0,\check{\gamma})$, where $\check{\beta}_0$ is from Theorem \ref{th_ODE_vect_lin} and $\check{\beta}, \check{\gamma}$ are as in Theorem \ref{th_hp_vect_exp_sol}. With analogous computations as in the scalar case, the $\Oc(\varepsilon^{k-1})$-order in the expansion for the bulk equation \eqref{eq_asym_vAC_cp} is zero if 
\begin{align}\label{eq_asym_vAC_cp_bulkk}
[-\Delta^{\sigma,t}+D^2W(\vec{\theta}_0)]\check{u}^C_{k+1}=\vec{G}_k(\rho,H,\sigma,t),
\end{align}
where $\vec{G}_k \in\check{R}^C_{k,(\beta,\gamma)}$.

\paragraph{Contact Point Expansion: The Neumann Boundary Condition}\label{sec_asym_vAC_cp_neum}
As in the scalar case, the boundary conditions complementing \eqref{eq_asym_vAC_cp_bulk1}-\eqref{eq_asym_vAC_cp_bulkk} will be obtained from the expansion of the Neumann boundary condition \eqref{eq_vAC2} for $\vec{u}_\varepsilon=\vec{u}^I_\varepsilon+\vec{u}^C_\varepsilon$, i.e.~$D_x(\vec{u}^I_\varepsilon+\vec{u}^C_\varepsilon)|_{\partial Q_T}N_{\partial\Omega}=0$. Lemma \ref{th_asym_vAC_in_trafo} and Lemma \ref{th_asym_vAC_cp_trafo} yield on $\overline{\Gamma^C(2\delta,2\mu_1)}$
\begin{align*}
D_x \vec{u}^I_\varepsilon|_{(x,t)}&=\partial_\rho\check{u}^I_\varepsilon|_{(\rho,s,t)}\left[\frac{\nabla r|_{(x,t)}}{\varepsilon}-(D_xs)^\top|_{(x,t)}\nabla_\Sigma h_\varepsilon|_{(s,t)}\right]^\top+D_\Sigma\check{u}^I_\varepsilon|_{(\rho,s,t)}\,D_xs|_{(x,t)},\\
D_x\vec{u}^C_\varepsilon|_{(x,t)}&=\partial_\rho\check{u}^C_\varepsilon|_{(\rho,H,\sigma,t)}\left[\frac{\nabla r|_{(x,t)}}{\varepsilon}-(D_xs)^\top|_{(x,t)}\nabla_\Sigma h_\varepsilon|_{(s,t)}\right]^\top\\
&\phantom{=}+\partial_H\check{u}^C_\varepsilon|_{(\rho,H,\sigma,t)}\left[\frac{\nabla b|_{(x,t)}}{\varepsilon}\right]^\top+D_{\partial\Sigma}\check{u}^C_\varepsilon|_{(\rho,H,\sigma,t)}\,D_x\sigma|_{(x,t)},
\end{align*}
where $\rho=\check{\rho}_\varepsilon(x,t), H=H_\varepsilon(x,t)$, $s=s(x,t)$ and $\sigma=\sigma(x,t)$. We consider the points $x=X(r,\sigma,t)$ for $(r,\sigma,t)\in[-2\delta,2\delta]\times\partial\Sigma\times[0,T]$, in particular $H=0$ and $s=\sigma$. 

For $g:\overline{\Gamma(2\delta)}\cap\partial Q_T\rightarrow\R$ smooth we use an expansion as in the scalar case, cf.~\eqref{eq_asym_ACND_cp_taylor4} and the remarks there. We just use $\check{h}_j,\check{\rho}_\varepsilon$ instead of $h_j,\rho_\varepsilon$.

In the following we expand the Neumann boundary condition into $\varepsilon$-series with coefficients in $(\check{\rho}_\varepsilon,\sigma,t)$ up to the order $\Oc(\varepsilon^{M-1})$. 

\subparagraph{Neumann Boundary Condition: $\Oc(\varepsilon^{-1})$}\label{sec_asym_vAC_cp_neum_m1}
At the lowest order $\Oc(\frac{1}{\varepsilon})$ we have $(N_{\partial\Omega}\cdot\nabla r)|_{\overline{X}_0(\sigma,t)}\vec{\theta}_0'(\rho)=0$. This is valid due to the $90$°-contact angle condition.

\subparagraph{Neumann Boundary Condition: $\Oc(\varepsilon^0)$}\label{sec_asym_vAC_cp_neum_0}
The order $\Oc(1)$ vanishes if
\begin{align}\label{eq_asym_vAC_cp_bc1}
&(N_{\partial\Omega}\cdot\nabla b)|_{\overline{X}_0(\sigma,t)}\partial_H\check{u}^C_1|_{H=0}(\rho,\sigma,t)=\vec{g}_1(\rho,\sigma,t),\\ \notag
\vec{g}_1(\rho,\sigma,t)&:=\vec{\theta}_0'(\rho)[(D_xsN_{\partial\Omega})^\top|_{\overline{X}_0(\sigma,t)}\nabla_\Sigma \check{h}_1|_{(\sigma,t)}-\partial_r((N_{\partial\Omega}\cdot\nabla r)\circ\overline{X})|_{(0,\sigma,t)}\check{h}_1|_{(\sigma,t)}]\\ \notag
&+\check{g}_0(\rho,\sigma,t),\notag
\end{align}
where $\check{g}_0(\rho,\sigma,t):=-\rho \vec{\theta}_0'(\rho)\partial_r((N_{\partial\Omega}\cdot\nabla r)\circ\overline{X})|_{(0,\sigma,t)}$. For $j=1,...,M$ let
\begin{align}\label{eq_asym_vAC_cp_ubar}
\underline{u}^C_j:\overline{\R^2_+}\times\partial\Sigma\times[0,T]\rightarrow\R^m:(\rho,H,\sigma,t)\mapsto\check{u}^C_j(\rho,|\nabla b|(\overline{X}_0(\sigma,t))H,\sigma,t).
\end{align}
Due to Theorem \ref{th_coordND} it holds $|\nabla b|_{\overline{X}_0(\sigma,t)}|\geq c>0$ and $|N_{\partial\Omega}\cdot\nabla b|_{\overline{X}_0(\sigma,t)}|\geq c>0$ for all $(\sigma,t)\in\partial\Sigma\times[0,T]$. Therefore \eqref{eq_asym_vAC_cp_bulk1} and \eqref{eq_asym_vAC_cp_bc1} for $\check{u}^C_1$ are equivalent to
\begin{align}\label{eq_asym_vAC_cp_uC1}
[-\Delta+D^2W(\vec{\theta}_0(\rho))]\underline{u}^C_1&=0,\\
-\partial_{H}\underline{u}^C_1|_{H=0}&=(|\nabla b|/N_{\partial\Omega}\cdot\nabla b)|_{\overline{X}_0(\sigma,t)} \vec{g}_1(\rho,\sigma,t).\label{eq_asym_vAC_cp_uC2}
\end{align}
The solvability condition \eqref{eq_hp_vect_comp} corresponding to \eqref{eq_asym_vAC_cp_uC1}-\eqref{eq_asym_vAC_cp_uC2} is
\[
(|\nabla b|/N_{\partial\Omega}\cdot\nabla b)|_{\overline{X}_0(\sigma,t)}
\int_\R \vec{g}_1(\rho,\sigma,t)\cdot\vec{\theta}_0'(\rho)\,d\rho=0.
\]
Due to the symmetry properties of $\vec{\theta}_0$, the term coming from $\check{g}_0$ vanishes. Therefore the latter condition yields the following boundary condition for $\check{h}_1$:
\begin{align}\label{eq_asym_vAC_cp_h1}
b_1(\sigma,t)\cdot\nabla_\Sigma\check{h}_1|_{(\sigma,t)}+ b_0(\sigma,t)\check{h}_1|_{(\sigma,t)}=0\quad\text{ for }(\sigma,t)\in\partial\Sigma\times[0,T],
\end{align}
where $b_1,b_0$ are as in the scalar case, cf.~the formulas below \eqref{eq_asym_ACND_cp_h1}. Together with the linear parabolic equation \eqref{eq_asym_vAC_in_h1} for $\check{h}_1$ from Subsection \ref{sec_asym_vAC_in_0}, we obtain a time-dependent linear parabolic boundary value problem for $\check{h}_1$, where the initial value $\check{h}_1|_{t=0}$ is not prescribed yet. However, since $\check{f}_0$ is zero, the equations for $\check{h}_1$ are homogeneous and we can take $\check{h}_1=0$.

Hence we get $\check{u}^I_2$ from Section \ref{sec_asym_vAC_in_0} with $\check{u}^I_2\in \check{R}^I_{1,(\beta_1)}$ for all $\beta_1\in(0,\min\{\sqrt{\lambda/2}, \check{\beta}_0\})$, where $\check{\beta}_0>0$ is as in Theorem \ref{th_ODE_vect_lin}. In particular the first inner order is determined. Furthermore, we have $\vec{g}_1\in \tilde{R}^I_{1,(\beta_1)}$ for all $\beta_1\in(0,\sqrt{\lambda/2})$ due to Theorem \ref{th_ODE_vect}. With Theorem \ref{th_hp_vect_exp_sol} (applied in local coordinates for $\partial\Sigma$) there is a unique smooth solution $\underline{u}^C_1$ to \eqref{eq_asym_vAC_cp_uC1}-\eqref{eq_asym_vAC_cp_uC2} and we get decay properties. By compactness and Remark \ref{th_nabla_sigma_equiv} with $\partial\Sigma$ instead of $\Sigma$ we obtain the decay $\underline{u}^C_1\in \check{R}^C_{1,(\beta,\gamma)}$ for all $\beta\in(0,\min\{\check{\beta}(\gamma),\sqrt{\lambda/2}\})$, $\gamma\in(0,\check{\gamma})$, where $\check{\beta}, \check{\gamma}$ are as in Theorem \ref{th_hp_vect_exp_sol}. Altogether we computed the first order.

\subparagraph{Neumann Boundary Condition: $\Oc(\varepsilon^k)$ and Induction}\label{sec_asym_vAC_cp_neum_k}
For $k=1,...,M-1$ we compute $\Oc(\varepsilon^k)$ in \eqref{eq_vAC2} for $\vec{u}_\varepsilon=\vec{u}^I_\varepsilon+\vec{u}^C_\varepsilon$ and obtain equations for the ($k+1$)-th order. We use the following induction hypothesis: assume that the $j$-th order is constructed for all $j=0,...,k$, that it is smooth and has the decay $\check{u}^I_{j+1}\in \check{R}^I_{j,(\beta_1)}$ for all $\beta_1\in(0,\min\{\sqrt{\lambda/2},\check{\beta}_0\})$ as well as $\check{u}^C_j\in \check{R}^C_{j,(\beta,\gamma)}$ for all $\beta\in(0,\min\{\check{\beta}(\gamma),\sqrt{\lambda/2},\check{\beta}_0\})$, $\gamma\in(0,\check{\gamma})$, where $\check{\beta}_0$ is from Theorem \ref{th_ODE_vect_lin} and $\check{\beta}, \check{\gamma}$ are as in Theorem \ref{th_hp_vect_exp_sol}. The assumption holds for $k=1$ due to Section \ref{sec_asym_vAC_cp_neum_0}. 

Analogously as in the scalar case, the $\Oc(\varepsilon^k)$-order in \eqref{eq_vAC2} for $\vec{u}_\varepsilon=\vec{u}^I_\varepsilon + \vec{u}^C_\varepsilon$ is zero if 
\begin{align}\label{eq_asym_vAC_cp_bck}
&(N_{\partial\Omega}\cdot\nabla b)|_{\overline{X}_0(\sigma,t)}\partial_H\check{u}^C_{k+1}|_{H=0}(\rho,\sigma,t)=\vec{g}_{k+1}(\rho,\sigma,t),\\ \notag
\vec{g}_{k+1}|_{(\rho,\sigma,t)}&:=\vec{\theta}_0'(\rho)[(D_xsN_{\partial\Omega})^\top|_{\overline{X}_0(\sigma,t)}\nabla_\Sigma \check{h}_{k+1}|_{(\sigma,t)}-\partial_r((N_{\partial\Omega}\cdot\nabla r)\circ\overline{X})|_{(0,\sigma,t)}\check{h}_{k+1}|_{(\sigma,t)}]\\ \notag
&+\check{g}_k(\rho,\sigma,t),\notag
\end{align}
where $\check{g}_k\in \check{R}^C_{k,(\beta)}$ and hence $\vec{g}_{k+1}\in \tilde{R}^I_{k+1,(\beta)}+ \check{R}^C_{k,(\beta)}$, if $\check{h}_{k+1}$ is smooth. 

As in the last Section \ref{sec_asym_vAC_cp_neum_0}, the equations \eqref{eq_asym_vAC_cp_bulkk}, \eqref{eq_asym_vAC_cp_bck} are equivalent to
\begin{align}\label{eq_asym_vAC_cp_uCk_1}
[-\Delta+D^2W(\vec{\theta}_0(\rho))]\underline{u}^C_{k+1}&=\underline{G}_k(\rho,H,\sigma,t),\\
-\partial_H\underline{u}^C_{k+1}|_{H=0}&=(|\nabla b|/N_{\partial\Omega}\cdot\nabla b)|_{\overline{X}_0(\sigma,t)} \vec{g}_{k+1}(\rho,H,\sigma,t),\label{eq_asym_vAC_cp_uCk_2}
\end{align}
where we defined $\underline{u}^C_{k+1}$ in \eqref{eq_asym_vAC_cp_ubar} and $\underline{G}_k$ is defined in the analogous way with the $\vec{G}_k\in \check{R}^C_{k,(\beta,\gamma)}$ from Section \ref{sec_asym_vAC_cp_bulk_km1}. The compatibility condition \eqref{eq_hp_vect_comp} for \eqref{eq_asym_vAC_cp_uCk_1}-\eqref{eq_asym_vAC_cp_uCk_2}, i.e.
\[
\int_{\R^2_+}\underline{G}_k(\rho,H,\sigma,t)\cdot\vec{\theta}_0'(\rho)\,d(\rho,H)
+(|\nabla b|/N_{\partial\Omega}\cdot\nabla b)|_{\overline{X}_0(\sigma,t)}\int_\R \vec{g}_{k+1}(\rho,\sigma,t)\cdot\vec{\theta}_0'(\rho)\,d\rho=0,
\]
implies a linear boundary condition for $\check{h}_{k+1}$:
\begin{align}\label{eq_asym_vAC_cp_hk}
b_1(\sigma,t)\cdot\nabla_\Sigma \check{h}_{k+1}|_{(\sigma,t)}+ b_0(\sigma,t)\check{h}_{k+1}|_{(\sigma,t)}=\check{f}^B_k(\sigma,t)\quad\text{ for }(\sigma,t)\in\partial\Sigma\times[0,T],
\end{align}
where $b_0,b_1$ are defined below \eqref{eq_asym_ACND_cp_h1} and 
\[
\check{f}^B_k|_{(\sigma,t)} := \frac{-1}{\|\vec{\theta}_0'\|_{L^2(\R)^m}^2}
\left[
\int_{\R^2_+}\underline{G}_k|_{(\rho,H,\sigma,t)}\cdot\vec{\theta}_0'|_{\rho}\,d\rho +\left.\frac{|\nabla b|}{N_{\partial\Omega}\cdot\nabla b}\right|_{\overline{X}_0(\sigma,t)}\int_\R \check{g}_k|_{(\rho,\sigma,t)}\cdot\vec{\theta}_0'|_{\rho}\,d\rho
\right]
\] 
is smooth in $(\sigma,t)\in\partial\Sigma\times[0,T]$.

Because of the remarks and computations in Section \ref{sec_asym_ACND_cp_neum_0} we can solve \eqref{eq_asym_vAC_in_hk} from Section \ref{sec_asym_vAC_in_k} together with \eqref{eq_asym_vAC_cp_hk} and obtain a smooth solution $\check{h}_{k+1}$. Therefore Section \ref{sec_asym_vAC_in_k} yields $\check{u}^I_{k+2}$ (solving \eqref{eq_asym_vAC_in_uk}) with $\check{u}^I_{k+2}\in \check{R}^I_{k+1,(\beta_1)}$ for all $\beta_1\in(0,\min\{\sqrt{\lambda/2},\check{\beta}_0\})$. In particular the $(k+1)$-th inner order is computed and it holds $\vec{G}_k\in \check{R}^C_{k,(\beta,\gamma)}$ as well as $\vec{g}_{k+1}\in \tilde{R}^I_{k+1,(\beta)}+ \check{R}^C_{k,(\beta)}$ for all $\beta\in(0,\min\{\check{\beta}(\gamma),\sqrt{\lambda/2},\check{\beta}_0\})$, $\gamma\in(0,\check{\gamma})$. As in the last Section \ref{sec_asym_vAC_cp_neum_0} we obtain a unique smooth solution $\underline{u}^C_{k+1}$ to \eqref{eq_asym_vAC_cp_uCk_1}-\eqref{eq_asym_vAC_cp_uCk_2} with the decay $\check{u}^C_{k+1}\in \check{R}^C_{k,(\beta,\gamma)}$ for all $(\beta,\gamma)$ as above. Altogether, the $(k+1)$-th order is determined.

Finally, by induction the $j$-th order is constructed for all $j=0,...,M$, the $\check{h}_j$ are smooth and $\check{u}^I_{j+1}\in \check{R}^I_{j,(\beta_1)}$ for all for all $\beta_1\in(0,\min\{\sqrt{\lambda/2},\check{\beta}_0\})$ as well as $\check{u}^C_j\in \check{R}^C_{j,(\beta,\gamma)}$ for every $\beta\in(0,\min\{\check{\beta}(\gamma),\sqrt{\lambda/2},\check{\beta}_0\})$, $\gamma\in(0,\check{\gamma})$.

\subsubsection{The Approximate Solution for (vAC) in ND}\label{sec_asym_vAC_uA}
Let $N\geq2$ and $\Gamma=(\Gamma_t)_{t\in[0,T]}$ be as in Section \ref{sec_coord_surface_requ} with contact angle $\alpha=\frac{\pi}{2}$ and a solution to \eqref{MCF} in $\Omega$. Moreover, let $\delta>0$ be such that the assertions of Theorem \ref{th_coordND} hold for $2\delta$ instead of $\delta$ and let $r,s,b,\sigma,\mu_0$ be as in the theorem. Furthermore, let $W:\R^m\rightarrow\R$ be as in Definition \ref{th_vAC_W} and $\vec{u}_\pm$ be any distinct pair of minimizers of $W$. Moreover, let $M\in\N$, $M\geq 2$ be as in the beginning of Section \ref{sec_asym_vAC}. Let $\eta:\R\rightarrow[0,1]$ be smooth with $\eta(r)=1$ for $|r|\leq 1$ and $\eta(r)=0$ for $|r|\geq 2$. Then for $\varepsilon>0$ we set
\begin{align*}
\vec{u}^A_\varepsilon:=
\begin{cases}
\eta(\frac{r}{\delta})\left[\vec{u}^I_\varepsilon+\vec{u}^C_\varepsilon \eta(\frac{b}{\mu_1})\right]+(1-\eta(\frac{r}{\delta}))\vec{u}_{\textup{sign}(r)}&\quad\text{ in }\overline{\Gamma(2\delta)},\\
\vec{u}_\pm&\quad\text{ in }Q_T^\pm\setminus\Gamma(2\delta),
\end{cases}
\end{align*}
where $\vec{u}^I_\varepsilon$ and $\vec{u}^C_\varepsilon$ were constructed in Sections \ref{sec_asym_vAC_in} and \ref{sec_asym_vAC_cp}. Analogously as in Section \ref{sec_asym_ACND_uA} one can prove that $\vec{u}^A_\varepsilon$ is an approximate solution for \eqref{eq_vAC1}-\eqref{eq_vAC3} in the following sense:
\begin{Lemma}\label{th_asym_vAC_uA}
	The function $\vec{u}^A_\varepsilon$ is smooth, uniformly bounded with respect to $x,t,\varepsilon$ and for the remainder 
	$\vec{r}^A_\varepsilon := 
	\partial_t\vec{u}^A_\varepsilon-\Delta\vec{u}^A_\varepsilon +\frac{1}{\varepsilon^2}\nabla W(\vec{u}^A_\varepsilon)$ in \eqref{eq_vAC1} and $\vec{s}^A_\varepsilon:=\partial_{N_{\partial\Omega}} \vec{u}^A_\varepsilon$ in \eqref{eq_vAC2} it holds 
	\begin{alignat*}{2}
	|\vec{r}^A_\varepsilon|&\leq C(\varepsilon^M e^{-c|\check{\rho}_\varepsilon|}+\varepsilon^{M+1})&\quad &\text{ in }\Gamma(2\delta,\mu_1),\\
	|\vec{r}^A_\varepsilon|&
	\leq 
	C(\varepsilon^{M-1} e^{-c(|\check{\rho}_\varepsilon|+H_\varepsilon)}
	+\varepsilon^M e^{-c|\check{\rho}_\varepsilon|}+\varepsilon^{M+1})
	&\quad &\text{ in }\Gamma^C(2\delta,2\mu_1),\\
	\vec{r}^A_\varepsilon &
	=0&\quad&\text{ in }Q_T\setminus\Gamma(2\delta),\\
	|\vec{s}^A_\varepsilon|&
	\leq C\varepsilon^M e^{-c|\check{\rho}_\varepsilon|}&\quad&\text{ on }\partial Q_T\cap\Gamma(2\delta),\\
	\vec{s}^A_\varepsilon &
	=0&\quad&\text{ on }\partial Q_T\setminus\Gamma(2\delta)
	\end{alignat*}
	for $\varepsilon>0$ small and some $c,C>0$. Here $\check{\rho}_\varepsilon$ is defined in \eqref{eq_asym_vAC_rho} and $H_\varepsilon=\frac{b}{\varepsilon}$.
\end{Lemma}

\begin{Remark}\upshape
	The analogous assertions as in Remark \ref{th_asym_ACND_uA_rem} hold.
\end{Remark}

\section{Spectral Estimates}\label{sec_SE}
The second step in the method of de Mottoni and Schatzman \cite{deMS} consists of estimating the difference of the exact and approximate solutions. To this end one employs a Gronwall-type argument together with the idea of linearization at the approximate solution, since the structure of the latter is known in detail. In order to estimate all terms in a suitable way, it is important to have a spectral estimate for a linear operator corresponding to the diffuse interface model and the approximate solution, i.e.~an estimate for the related bilinear form. 

In order to sketch the idea, let us consider the scalar case in the following. In this situation we have the (at the approximate solution) linearized Allen-Cahn operator
\[
\Lc_{\varepsilon,t}:=-\Delta + \frac{1}{\varepsilon^2} f''(u^A_\varepsilon(.,t))\quad\text{ on }\Omega
\]
together with homogeneous Neumann boundary condition, where $u^A_\varepsilon$ is from Section \ref{sec_asym_ACND_uA}. We will show a spectral estimate of the following form (see also Theorem \ref{th_SE_ACND} below): there are constants $c_0,C,\varepsilon_0>0$ such that for all $u\in H^1(\Omega)$ and $\varepsilon\in(0,\varepsilon_0]$
\begin{align}\label{eq_SE1}
\int_\Omega |\nabla u|^2 + \frac{1}{\varepsilon^2} f''(u^A_\varepsilon(.,t))u^2 \geq 
-C\|u\|_{L^2(\Omega)}^2
+\|\nabla u\|_{L^2(\Omega\setminus\Gamma_t(\delta))}^2
+c_0\|\nabla_\tau u\|_{L^2(\Gamma_t(\delta))}^2,
\end{align}
where $\nabla_\tau$ is the tangential derivative defined in Remark \ref{th_coordND_rem},~2. The estimate (also without the two additional last terms in \eqref{eq_SE1}) implies that the spectrum of $\Lc_{\varepsilon,t}$ is bounded from below by $-C$, where $\Lc_{\varepsilon,t}$ is viewed as an unbounded operator on $\{u\in H^2(\Omega):\partial_{N_{\partial\Omega}}u=0 \}$ with values in $L^2(\Omega)$. $\Lc_{\varepsilon,t}$ is selfadjoint and has spectrum in $\R$ in this setting. This explains the name \enquote{spectral estimate}. The last two additional terms are not fundamental for our case, but will help to optimize some estimates. This decreases the number of terms necessary in the asymptotic expansion which could be interesting for regularity questions or couplings.

For the proof of the spectral estimates the control of perturbed 1D-operators on normal modes and large intervals will be a crucial ingredient. We show such estimates similar as in Chen \cite{ChenSpectrums} (and \cite{ALiu}, \cite{Marquardt}). However, in the latter publications the formulation of the 1D-problems is always linked to the situation of a given approximate solution over an interface. Here we will treat the 1D-problems separately in Section \ref{sec_SE_1Dprelim} for better readability. Therefore we introduce an abstract setting in 1D in Section \ref{sec_SE_1Dsetting} that is applicable in both cases. We prove integral transformations and remainder estimates in Section \ref{sec_SE_1Dtrafo_remainder}. In Sections \ref{sec_SE_1Dscal} and \ref{sec_SE_1Dvect} we show the spectral estimates for (unperturbed and perturbed) operators in 1D in the scalar and vector-valued case on finite large intervals, respectively. In the appendix, Section \ref{sec_fredh}, for Section \ref{sec_SE_1Dprelim} we summarize an abstract Fredholm Alternative that can be applied for all the cases in order to obtain discrete eigenvalues and orthonormal bases of eigenfunctions. 

Equipped with this we prove the spectral estimate for the scalar case in Section \ref{sec_SE_ACND} for a slightly more general structure concerning $u^A_\varepsilon$. The approach is similar to \cite{AbelsMoser}, Section 4, and motivated from Alikakos, Chen, Fusco \cite{ACF}. Roughly, the idea is as follows. The approximate solution always has a specific structure. For parts away from the contact points the estimate directly follows with the 1D-estimates from Section \ref{sec_SE_1Dprelim} and an integral transformation as in \cite{ChenSpectrums}, Theorem 2.3. Therefore by an argument with a partition of unity we can reduce the spectral estimate to a corresponding one close to the contact points. Moreover, via Taylor expansions, we can replace the potential part by a term with simpler structure. Then we use a suitable ansatz to get an approximate first eigenfunction $\phi^A_\varepsilon(.,t)$ which leads to the model problem studied in Section \ref{sec_hp_90}. Finally, we split the space of $H^1$-functions over the domain with the help of a subspace consisting of suitable tangential alterations $a(s(.,t))\phi^A_\varepsilon(.,t)$ and analyze the bilinear form corresponding to $\Lc_{\varepsilon,t}$ on every part. 

For the case of the vector-valued Allen-Cahn equation \hyperlink{vAC}{(vAC)} the above procedure is adapted correspondingly. This is done in Section \ref{sec_SE_vAC}.

\begin{Remark}\label{th_SE_rem_alternative}\upshape
	Note that a general reduction strategy in analogy to \cite{ChenSpectrums} also might work in our cases, i.e.~the idea would be to reduce via perturbation arguments to the spectral properties of corresponding unperturbed operators on large domains approximating $\R^2_+$. However, that would require tedious estimates and the degeneracy is a difficulty, cf.~the tangential alterations in the eigenfunctions on the rectangle before. Therefore we work with the simpler strategy above. 
\end{Remark}

\subsection{Preliminaries in 1D}\label{sec_SE_1Dprelim}
\subsubsection{The Setting}\label{sec_SE_1Dsetting}
We introduce an abstract setting in 1D that is taylored for integrals over normal modes. Therefore let $\delta>0$ be fixed, $h_\varepsilon\in C^1([-\delta,\delta],\R)$ for $\varepsilon>0$ small such that
\begin{align}\label{eq_SE_h}
\|h_\varepsilon\|_{C_b^1([-\delta,\delta])}\leq C_0.
\end{align}
Then we set
\begin{align}\label{eq_SE_r_eps}
r_\varepsilon:[-\delta,\delta]\rightarrow\R:r\mapsto r-\varepsilon h_\varepsilon(r)\quad\text{ and }\quad \rho_\varepsilon:=\frac{r_\varepsilon}{\varepsilon}.
\end{align}
At this point let us already note Remark \ref{th_SE_1Dprelim_rem},~1.~below, where the correspondence to the application is explained.
\begin{Lemma}\label{th_SE_1Dprelim1}
	There is an $\varepsilon_0=\varepsilon_0(C_0)>0$ such that
	\begin{enumerate}
	\item $r_\varepsilon:[-\tilde{\delta},\tilde{\delta}]\rightarrow[-\tilde{\delta}-\varepsilon h_\varepsilon(-\tilde{\delta}),\tilde{\delta}-\varepsilon h_\varepsilon(\tilde{\delta})]$ is $C^1$ and invertible for all $\tilde{\delta}\in(0,\delta]$, $\varepsilon\in(0,\varepsilon_0]$. Moreover,
	\[
	\left|\frac{d}{dr}r_\varepsilon -1\right|\leq C_0\varepsilon\leq \frac{1}{2},\quad 
	\left|\frac{d}{d\tilde{r}}(r_\varepsilon^{-1})-1\right|\leq 2C_0\varepsilon
	\]
	and $|r_\varepsilon^{-1}(\tilde{r})|\leq (1+2C_0\varepsilon)(|\tilde{r}|+\varepsilon |h_\varepsilon(0)|)$ for all $\tilde{r}\in r_\varepsilon([-\delta,\delta])$ and $\varepsilon\in(0,\varepsilon_0]$.
	\item If additionally $h_\varepsilon\in C^2([-\delta,\delta])$ with $\|\frac{d^2}{dr^2}h_\varepsilon\|_{C_b^0([-\delta,\delta])}\leq\tilde{C}_0$ for $\varepsilon\in(0,\varepsilon_1]$ , $0<\varepsilon_1\leq\varepsilon_0$, then $r_\varepsilon$ is $C^2$ for $\varepsilon\in(0,\varepsilon_1]$ and it holds
	\[
	\left|\frac{d^2}{dr^2}r_\varepsilon\right|\leq \tilde{C}_0\varepsilon\quad\text{ and }\quad
	\left|\frac{d^2}{d\tilde{r}^2}(r_\varepsilon^{-1})\right|\leq 8\tilde{C}_0\varepsilon.
	\]
	\end{enumerate}
\end{Lemma}
\begin{proof}[Proof. Ad 1] 	
Since $h_\varepsilon$ is $C^1$, this is also true for $r_\varepsilon$. Moreover, it holds 
\[
\frac{d}{dr}r_\varepsilon(r)=1-\varepsilon\frac{d}{dr}h_\varepsilon(r)\quad\text{ and }\quad\left|\frac{d}{dr}r_\varepsilon(r)-1\right|\leq C_0\varepsilon\leq\frac{1}{2}
\] 
for all $r\in[-\delta,\delta]$ if $\varepsilon\in(0,\varepsilon_0]$ and $\varepsilon_0=\varepsilon_0(C_0)>0$ is small. In particular, $r_\varepsilon$ is strictly monotone and invertible on $[-\tilde{\delta},\tilde{\delta}]$ onto $[r_\varepsilon(-\tilde{\delta}),r_\varepsilon(\tilde{\delta})]$ for all $\tilde{\delta}\in(0,\delta]$. The inverse is also $C^1$ and
\[
\left|\frac{d}{d\tilde{r}}(r_\varepsilon^{-1})-1\right|=\left|\frac{1-\frac{d}{dr}r_\varepsilon(r_\varepsilon^{-1})}{\frac{d}{dr}r_\varepsilon(r_\varepsilon^{-1})}
\right|\leq 2C_0\varepsilon
\]
due to $|\frac{d}{dr}r_\varepsilon(r_\varepsilon^{-1})|\geq\frac{1}{2}$. Finally, note that $r_\varepsilon(0)=-\varepsilon h_\varepsilon(0)$. This yields for all $\tilde{r}\in r_\varepsilon([-\delta,\delta])$:
\[
r_\varepsilon^{-1}(\tilde{r}) = \int_{-\varepsilon h_\varepsilon(0)}^{\tilde{r}}\frac{d}{d\overline{r}}(r_\varepsilon^{-1})(\overline{r})\,d\overline{r}.
\]
Since the modulus of the integrand is bounded by $1+2C_0\varepsilon$, we obtain the estimate for $|r_\varepsilon^{-1}|$.\qedhere$_{1.}$\end{proof}

\begin{proof}[Ad 2] Let additionally $h_\varepsilon\in C^2([-\delta,\delta])$ with $\|\frac{d^2}{dr^2}h_\varepsilon\|_{C_b^0([-\delta,\delta])}\leq\tilde{C}_0$ for $\varepsilon\in(0,\varepsilon_1]$. Then $r_\varepsilon$ and $r_\varepsilon^{-1}$ are $C^2$ for all $\varepsilon\in(0,\varepsilon_1]$. The estimates follow from
\[
\frac{d^2}{dr^2}r_\varepsilon=-\varepsilon\frac{d^2}{dr^2}h_\varepsilon\quad\text{ and }\quad 
\frac{d^2}{d\tilde{r}^2}(r_\varepsilon^{-1})=-\frac{\frac{d^2}{dr^2}r_\varepsilon(r_\varepsilon^{-1})}{(\frac{d}{dr}r_\varepsilon(r_\varepsilon^{-1}))^3}
\]
together with $|\frac{d}{dr}r_\varepsilon(r_\varepsilon^{-1})|\geq\frac{1}{2}$.
\qedhere$_{2.}$\end{proof}

In particular $\rho_\varepsilon:[-\delta,\delta]\rightarrow \frac{1}{\varepsilon}r_\varepsilon([-\delta,\delta])$ is $C^1$ for $\varepsilon\in(0,\varepsilon_0]$ and invertible with inverse
\begin{align}\label{eq_SE_F_eps}
F_\varepsilon:\frac{1}{\varepsilon}r_\varepsilon([-\delta,\delta])\rightarrow[-\delta,\delta]:z\mapsto r_\varepsilon^{-1}(\varepsilon z).
\end{align}
Finally, let $J\in C^2([-\delta,\delta],\R)$ with
\begin{align}\label{eq_SE_J}
J\geq c_1>0\quad\text{ and }\quad \|J\|_{C_b^2([-\delta,\delta])}\leq C_2.
\end{align}
Then we define $J_\varepsilon:=J(F_\varepsilon):\frac{1}{\varepsilon}r_\varepsilon([-\delta,\delta])\rightarrow\R$ for $\varepsilon\in(0,\varepsilon_0]$.

\begin{Corollary}\label{th_SE_1Dprelim2}
Let $h_\varepsilon\in C^2([-\delta,\delta],\R)$ with $\|h_\varepsilon\|_{C^2([-\delta,\delta])}\leq \overline{C}_0$ for small $\varepsilon>0$ and $J$ be as above. Let $\varepsilon_0=\varepsilon_0(\overline{C}_0)>0$ be such that Lemma \ref{th_SE_1Dprelim1} holds. Then $F_\varepsilon$, $J_\varepsilon$ are well-defined for $\varepsilon\in(0,\varepsilon_0]$, $C^2$ and we obtain for all $z\in \frac{1}{\varepsilon}r_\varepsilon([-\delta,\delta])$ the estimates
\begin{alignat*}{3}
|F_\varepsilon(z)|&\leq 2\varepsilon(|z|+\overline{C}_0),& \quad \left|\frac{d}{dz}F_\varepsilon(z)\right| &\leq 2\varepsilon,&
\quad \left|\frac{d^2}{dz^2}F_\varepsilon(z)\right| &\leq 8\overline{C}_0\varepsilon^3,\\
c_1&\leq J_\varepsilon(z)\leq C_2,&\quad  \left|\frac{d}{dz}J_\varepsilon(z)\right| &\leq 2C_2\varepsilon,&
\quad \left|\frac{d^2}{dz^2}J_\varepsilon(z)\right| &\leq C(\overline{C}_0,C_2)\varepsilon^2.
\end{alignat*}
\end{Corollary}

\begin{proof}
	This directly follows from Lemma \ref{th_SE_1Dprelim1} and the chain rule.
\end{proof}

\begin{Remark}\upshape\phantomsection{\label{th_SE_1Dprelim_rem}}
	\begin{enumerate}
	\item In the applications later $\delta$ corresponds to the one from Theorem \ref{th_coordND}. Moreover, $J$ will correlate to the determinant in Remark \ref{th_coordND_rem},~3. The $h_\varepsilon$, $\rho_\varepsilon$ here stand for the height function and the rescaled normal variable in the asymptotic expansions. However, $h_\varepsilon$ depends on the normal variable in the abstract setting. The latter will be constant in the application later. But the generalization is interesting for other situations, e.g.~an Allen-Cahn equation with a nonlinear Robin boundary condition designed to approximate \eqref{MCF} with a constant $\alpha$-contact angle for $\alpha$ close to $\frac{\pi}{2}$. See \cite{MoserDiss} or \cite{AbelsMoserAlpha}. In the abstract setting in this section additional variables like tangential ones or time do not appear. Therefore we prove uniform estimates with respect to the constants above.
	\item Consider the situation of Corollary \ref{th_SE_1Dprelim2}. Later it will be convenient to consider suitable symmetric subintervals of $\frac{1}{\varepsilon}r_\varepsilon([-\delta,\delta])$. Therefore note that for $\varepsilon_1=\varepsilon_1(\delta,\overline{C}_0)>0$ small it holds $\varepsilon_1\leq\varepsilon_0$ and
	\[
	[-\tilde{\delta},\tilde{\delta}]\subseteq r_\varepsilon([-\delta,\delta])\quad\text{ for all }\tilde{\delta}\in(0,\frac{3\delta}{4}], \varepsilon\in(0,\varepsilon_1].
	\] 
	Hence $F_\varepsilon, J_\varepsilon$ are well-defined and the assertions of Corollary \ref{th_SE_1Dprelim2} hold on $\overline{I_{\varepsilon,\tilde{\delta}}}$ for all $\tilde{\delta}\in(0,\frac{3\delta}{4}]$ and $\varepsilon\in(0,\varepsilon_1]$, where $I_{\varepsilon,\tilde{\delta}}:=(-\tilde{\delta}/\varepsilon,\tilde{\delta}/\varepsilon)$. Typically remainder terms on $\frac{1}{\varepsilon}r_\varepsilon([-\tilde{\delta},\tilde{\delta}])\setminus I_{\varepsilon,\tilde{\delta}}$ will behave nicely and it is sufficient to prove finer estimates on $I_{\varepsilon,\tilde{\delta}}$. This simplifies some notation.
	\end{enumerate}
	
\end{Remark}

\subsubsection{Transformations and Remainder Terms}\label{sec_SE_1Dtrafo_remainder}
Let $\delta,C_0>0$ and $h_\varepsilon\in C^1([-\delta,\delta],\R)$ such that \eqref{eq_SE_h} holds for $\varepsilon>0$ small. Moreover, let $r_\varepsilon, \rho_\varepsilon$ be as in \eqref{eq_SE_r_eps} and $\varepsilon_0=\varepsilon_0(C_0)>0$ be such that Lemma \ref{th_SE_1Dprelim1} holds. Then $F_\varepsilon$ as in \eqref{eq_SE_F_eps} is well-defined. We obtain the following lemma for transformation arguments and estimates for remainder terms.

\begin{Lemma}\label{th_SE_1Dtrafo_remainder}
	Let $\tilde{\varepsilon}_0\in(0,\varepsilon_0]$ and $R_\varepsilon:\R\times[-\delta,\delta]\rightarrow\R$ be integrable for $\varepsilon\in(0,\tilde{\varepsilon}_0]$. Moreover, let $\mathcal{J}\subseteq[-\delta,\delta]$ be an interval.
	\begin{enumerate}
		\item For all $\varepsilon\in(0,\tilde{\varepsilon}_0]$ it holds
		\[
		\int_{\mathcal{J}} R_\varepsilon(\rho_\varepsilon(r),r)\,dr=
		\int_{r_\varepsilon(\mathcal{J})/\varepsilon} R_\varepsilon(z,F_\varepsilon(z))\frac{d}{dz}F_\varepsilon(z)\,dz,
		\]
		where $\frac{d}{dz}F_\varepsilon(z)=\varepsilon\frac{d}{d\tilde{r}}(r_\varepsilon^{-1})(\varepsilon z)>0$ and $\left|\frac{d}{d\tilde{r}}(r_\varepsilon^{-1})(\varepsilon z)-1\right|\leq 2C_0\varepsilon$ for all $z\in\frac{r_\varepsilon([-\delta,\delta])}{\varepsilon}$.
		\item If additionally $|R_\varepsilon(\rho,r)|\leq \overline{C}|r|^k e^{-\beta|\rho|}$ for all $(\rho,r)\in\R\times[-\delta,\delta]$ and some $k\geq 0$, $\overline{C},\beta>0$, then for all $\varepsilon\in(0,\tilde{\varepsilon}_0]$ it follows that with constants independent of $\mathcal{J}$ we have the estimate
		\[
		\int_{\mathcal{J}} |R_\varepsilon(\rho_\varepsilon(r),r)|\,dr\leq \overline{C} C(C_0,k,\beta)\varepsilon^{k+1}.
		\] 
	\end{enumerate}
\end{Lemma}
\begin{proof} The first assertion follows from Lemma \ref{th_SE_1Dprelim1} and the transformation rule. Using the latter for the second part, we obtain for all $\varepsilon\in(0,\tilde{\varepsilon}_0]$ that
	\[
	\int_{\mathcal{J}}|R_\varepsilon(\rho_\varepsilon(r),r)|\,dr\leq \overline{C}\varepsilon(1+2C_0\varepsilon)\int_{r_\varepsilon(\mathcal{J})/\varepsilon}|F_\varepsilon(z)|^k e^{-\beta|z|}\,dz.
	\]
	Here it holds $|F_\varepsilon(z)|=|r_\varepsilon^{-1}(\varepsilon z)|\leq (1+2C_0\varepsilon)\varepsilon(|z|+C_0)$ for all $z\in r_\varepsilon([-\delta,\delta])/\varepsilon$ and $C_0\varepsilon\leq\frac{1}{2}$ due to Lemma \ref{th_SE_1Dprelim1},~1. This yields 
	\[
	\int_{\mathcal{J}}|R_\varepsilon(\rho_\varepsilon(r),r)|\,dr\leq \overline{C}2^{k+1} \int_\R(|z|+C_0)^k e^{-\beta|z|}\,dz\,\varepsilon^{k+1}.
	\]
	This shows the estimate with constants independent of $\mathcal{J}$.
\end{proof}

\subsubsection{Spectral Estimates for Scalar-Valued Allen-Cahn-Type Operators in 1D}\label{sec_SE_1Dscal}
\paragraph{Unperturbed Scalar-Valued Allen-Cahn-Type Operators in 1D}\label{sec_SE_1Dscal_unpert}
We consider the operator $\Lc_0:=-\frac{d^2}{dz^2}+f''(\theta_0)$ on finite large intervals together with homogeneous Neumann boundary condition and we prove spectral properties in an $L^2$-setting. Note that in the following we only use Theorem \ref{th_theta_0}. For the spectral properties of $\Lc_0$ on the real line $\R$ see \cite{MoserDiss}, Section 4.1.2 or \cite{AbelsMoser}, Lemma 2.5. Note that the latter is important for the proof of Theorem \ref{th_hp_weak_sol} which is concerned with weak solutions for the model problem \eqref{eq_hp1}-\eqref{eq_hp2} on $\R^2_+$. 

Let $\tilde{\delta}>0$ be fixed, $\varepsilon>0$, $I_{\varepsilon,\tilde{\delta}}:=(-\frac{\tilde{\delta}}{\varepsilon},\frac{\tilde{\delta}}{\varepsilon})$ and $\K=\R$ or $\C$. Moreover, here one can reduce to $\tilde{\delta}=1$ by scaling in $\varepsilon$. However, introducing the $\tilde{\delta}$ in the notation already here will simplify the notation later. We consider the unbounded operator 
\[
L_{0,\varepsilon}:H^2_N(I_{\varepsilon,\tilde{\delta}},\K)
\rightarrow L^2(I_{\varepsilon,\tilde{\delta}},\K):u\mapsto \Lc_0u=\left[-\frac{d^2}{dz^2}+f''(\theta_0)\right]u,
\]
where $H^2_N(I_{\varepsilon,\tilde{\delta}},\K)$ is the space of $H^2$-functions $u$ on $I_{\varepsilon,\tilde{\delta}}$ satisfying the homogeneous Neumann boundary condition $\frac{d}{dz}u|_z=0$ for $z=\pm\tilde{\delta}/\varepsilon$.
The corresponding sesquilinearform is
\[
B_{0,\varepsilon}:H^1(I_{\varepsilon,\tilde{\delta}},\K)\times H^1(I_{\varepsilon,\tilde{\delta}},\K)\rightarrow\K: (\Phi,\Psi)\mapsto \int_{I_{\varepsilon,\tilde{\delta}}}\frac{d}{dz}\Phi\overline{\frac{d}{dz}\Psi} + f''(\theta_0)\Phi\overline{\Psi}\,dz.
\] 
As in \cite{ChenSpectrums}, Lemma 2.1 and \cite{Marquardt}, Section 3.1 we obtain the following lemma:
\begin{Lemma}\phantomsection{\label{th_SE_1Dscal_unpert}}
	\begin{enumerate}
		\item $L_{0,\varepsilon}$ is selfadjoint and the spectrum is given by discrete eigenvalues $(\lambda_{0,\varepsilon}^k)_{k\in\N}$ in $\R$ with $\lambda_{0,\varepsilon}^1\leq\lambda_{0,\varepsilon}^2\leq...  $ and $\lambda_{0,\varepsilon}^k\overset{k\rightarrow\infty}\longrightarrow\infty$. Moreover, there is an orthonormal basis $(\Psi_{0,\varepsilon}^k)_{k\in\N}$ of $L^2(I_{\varepsilon,\tilde{\delta}},\K)$ consisting of smooth $\R$-valued eigenfunctions $\Psi_{0,\varepsilon}^k$ to $\lambda_{0,\varepsilon}^k$. 
		\item $\lambda_{0,\varepsilon}^1$ is simple and the corresponding eigenfunction $\Psi_{0,\varepsilon}^1$ has a sign. We take $\Psi_{0,\varepsilon}^1$ positive.
		\item Let $c_0>0$ be such that $\inf_{|z|\geq c_0}f''(\theta_0(z))\geq \frac{3}{4}\min\{f''(\pm 1)\}$. Then for $\varepsilon>0$ small and any normalized eigenfunction $\Psi_{0,\varepsilon}$ of $L_{0,\varepsilon}$ to an eigenvalue $\lambda_{0,\varepsilon}\leq\frac{1}{4}\min\{f''(\pm 1)\}$ it holds
		\[
		|\Psi_{0,\varepsilon}(z)|\leq Ce^{-|z|\sqrt{\min\{f''(\pm 1)\}/3}}\quad\text{ for all }z\in I_{\varepsilon,\tilde{\delta}}, |z|\geq c_0+1,
		\]
		where $C>0$ only depends on $c_0$ and $\min\{f''(\pm 1)\}$.
		\item There is $\varepsilon_0=\varepsilon_0(\tilde{\delta})>0$ small such that for all $\varepsilon\in(0,\varepsilon_0]$
		\[
		\lambda_{0,\varepsilon}^1=\inf_{\Psi\in H^1(I_{\varepsilon,\tilde{\delta}}), \|\Psi\|_{L^2}=1} B_{0,\varepsilon}(\Psi,\Psi)= B_{0,\varepsilon}(\Psi_{0,\varepsilon}^1,\Psi_{0,\varepsilon}^1),\quad
		|\lambda_{0,\varepsilon}^1|\leq Ce^{-\frac{3\tilde{\delta}\sqrt{\min\{f''(\pm 1)\}}}{2\varepsilon}},
		\]
		where $C>0$ is independent of $\tilde{\delta}$, $\varepsilon$.
		\item There is $\nu_1>0$ independent of $\tilde{\delta}$, $\varepsilon$ and $\varepsilon_0=\varepsilon_0(\tilde{\delta})>0$ small such that
		\[
		\lambda_{0,\varepsilon}^2 = \inf_{\Psi\in H^1(I_{\varepsilon,\tilde{\delta}}), \|\Psi\|_{L^2}=1, \Psi\perp_{L^2}\Psi_{0,\varepsilon}^1} B_{0,\varepsilon}(\Psi,\Psi) \geq \nu_1\quad\text{ for all }\varepsilon\in(0,\varepsilon_0].
		\]
		\item Let $\beta_\varepsilon:=\|\theta_0'\|_{L^2(I_{\varepsilon,\tilde{\delta}})}^{-1}$ and $\Psi_{0,\varepsilon}^R:=\Psi_{0,\varepsilon}^1-\beta_\varepsilon\theta_0'$. For $\varepsilon_0=\varepsilon_0(\tilde{\delta})>0$ small and $\varepsilon\in(0,\varepsilon_0]$ we have
		\[\textstyle
		\left\|\Psi_{0,\varepsilon}^R\right\|_{L^2(I_{\varepsilon,\tilde{\delta}})}^2
		+\left\|\frac{d}{dz}\Psi_{0,\varepsilon}^R\right\|_{L^2(I_{\varepsilon,\tilde{\delta}})}^2
		\leq Ce^{-\frac{3\tilde{\delta}\sqrt{\min\{f''(\pm 1)\}}}{2\varepsilon}},
		\]
		where $C>0$ is independent of $\tilde{\delta}$, $\varepsilon$.
	\end{enumerate}
\end{Lemma}
\begin{proof} 
	By scaling in $\varepsilon$ it is enough to consider the case $\tilde{\delta}=1$. We set $I_\varepsilon:=I_{\varepsilon,1}$.\phantom{\qedhere}
	
	\begin{proof}[Ad 1] $L_{0,\varepsilon}$ is selfadjoint because $L_{0,\varepsilon}$ is symmetric and the resolvent set nonempty. The latter follows from the Lax-Milgram Theorem applied to a constant shift of $B_{0,\varepsilon}$ due to $f''(\theta_0)\geq -C$. The spectral properties in 1.~and the existence of the orthonormal basis follow from the abstract Fredholm alternative in Theorem \ref{th_fred} below (and standard regularity and integration by parts arguments) applied to
	\[
	A_{0,\varepsilon}:H^1(I_\varepsilon,\C)\rightarrow H^1(I_\varepsilon,\C)^*: u\mapsto [v\mapsto B_{0,\varepsilon}(u,v)],
	\]	 
	where $H^1(I_\varepsilon,\C)^*$ is the anti-dual space of $H^1(I_\varepsilon,\C)$, i.e.~the space of conjugate linear functionals on $H^1(I_\varepsilon,\C)$. Note that here also in the case $\K=\C$ one can obtain an $\R$-valued orthonormal basis since for an eigenfunction $u$ also $\overline{u}$ is an eigenfunction to the same eigenvalue and $u,\overline{u}$ are $\C$-linearly independent if and only if $\textup{Re}\,u, \textup{Im}\,u$ are $\R$-linearly independent. The latter can be seen with elementary arguments.\qedhere$_{1.}$\end{proof}
	
	\begin{proof}[Ad 2] The properties of $\lambda_{0,\varepsilon}^1$ and $\Psi_{0,\varepsilon}^1$ can be shown with the Krein-Rutman-Theorem and the maximum principle, cf.~\cite{Marquardt}, Proposition 3.6, 2.\qedhere$_{2.}$\end{proof}

    \begin{proof}[Ad 3.-6] For the remaining assertions it is enough to consider the case $\K=\R$. The eigenvalues are the same for $\K=\R, \C$ and the orthonormal basis of $\R$-valued eigenfunctions can be chosen to be the same. Moreover, the \textup{inf}-characterizations are known, cf.~e.g.~\cite{Marquardt}, proof of Proposition 3.6. The other assertions can be deduced with the comparison principle, Theorem \ref{th_theta_0}, the Harnack-inequality and the Hopf maximum principle, cf.~\cite{ChenSpectrums}, Lemma 2.1 and \cite{Marquardt}, Proposition 3.7 and Lemma 3.8. For the proof of 6.~see also the analogous computation in the proof of Lemma \ref{th_SE_1Dscal_unpert},~6.~below in the vector-valued case.\end{proof}
\end{proof}

\paragraph{Perturbed Scalar-Valued Allen-Cahn-Type Operators in 1D}\label{sec_SE_1Dscal_pert}
In this section we derive a result for perturbed and weighted operators in 1D. Let $\delta>0$ and $h_\varepsilon, J\in C^2([-\delta,\delta],\R)$ with $\|h_\varepsilon\|_{C^2([-\delta,\delta])}\leq \overline{C}_0$ for $\varepsilon>0$ small and $c_1,C_2>0$ be such that \eqref{eq_SE_J} holds. Then let $\rho_\varepsilon, F_\varepsilon, J_\varepsilon$ for $\varepsilon>0$ small be as in Section \ref{sec_SE_1Dsetting}. We consider
\begin{align}
\phi_\varepsilon:[-\delta,\delta]\rightarrow\R:r\mapsto \theta_0(\frac{r}{\varepsilon}) 
+ \varepsilon p_\varepsilon \theta_1(\frac{r}{\varepsilon}) + q_\varepsilon(r)\varepsilon^2, 
\end{align}
where $p_\varepsilon\in\R$ and $q_\varepsilon:[-\delta,\delta]\rightarrow\R$ is measurable with $|p_\varepsilon|+\frac{\varepsilon}{\varepsilon+|r|}|q_\varepsilon(r)|\leq C_3$ for all $r\in[-\delta,\delta]$, some $C_3>0$, and $\varepsilon>0$ small. Moreover, let $\theta_1\in L^\infty(\R)$ with $\|\theta_1\|_\infty\leq C_4$ for a $C_4>0$ and
\begin{align}\label{eq_SE_1Dscal_pert_theta1}
\int_\R f''(\theta_0)(\theta_0')^2\,\theta_1=0.
\end{align}

Let $\tilde{\delta}\in(0,\frac{3\delta}{4}]$ be fixed. Then $F_\varepsilon$, $J_\varepsilon$ are well-defined on $\overline{I_{\varepsilon,\tilde{\delta}}}$ for $\varepsilon\in(0,\varepsilon_1(\delta,\overline{C}_0)]$ and Corollary \ref{th_SE_1Dprelim2} is applicable due to Remark \ref{th_SE_1Dprelim_rem},~2. We consider the operators
	\[
	L_\varepsilon:H^2_N(I_{\varepsilon,\tilde{\delta}},\K)\rightarrow L^2_{J_\varepsilon}(I_{\varepsilon,\tilde{\delta}},\K):
	u\mapsto \Lc_\varepsilon u:= \left[-J_\varepsilon^{-1} \frac{d}{dz}\left(J_\varepsilon\frac{d}{dz}\right)+f''(\phi_\varepsilon(\varepsilon .))\right]u,
	\]
where $L^2_{J_\varepsilon}(I_{\varepsilon,\tilde{\delta}},\K)$ is the space of $L^2$-functions defined on $I_{\varepsilon,\tilde{\delta}}$ with the weight $J_\varepsilon$. We write $(.,.)_{J_\varepsilon}$, $\|.\|_{J_\varepsilon}$ and $\perp_{J_\varepsilon}$ for the corresponding scalar product, norm and orthogonal relation.
The sesquilinearform associated to $L_\varepsilon$ is given by
\[
B_\varepsilon:H^1(I_{\varepsilon,\tilde{\delta}},\K)\times H^1(I_{\varepsilon,\tilde{\delta}},\K)\rightarrow\K:(\Phi,\Psi)\mapsto \int_{I_{\varepsilon,\tilde{\delta}}}\left[\frac{d}{dz}\Phi\overline{\frac{d}{dz}\Psi} + f''(\phi_\varepsilon(\varepsilon .))\Phi\overline{\Psi}\right] J_\varepsilon\,dz.
\]
We obtain the analogue of Lemma \ref{th_SE_1Dscal_unpert},~1.-3.

\begin{Lemma}\phantomsection{\label{th_SE_1Dscal_pert1}}
	\begin{enumerate}
		\item $L_\varepsilon$ is selfadjoint and the spectrum is given by a sequence of discrete eigenvalues $(\lambda_\varepsilon^k)_{k\in\N}$ in $\R$ with $\lambda_\varepsilon^1\leq\lambda_\varepsilon^2\leq...  $ and $\lambda_\varepsilon^k\overset{k\rightarrow\infty}\longrightarrow\infty$. Moreover, there is an orthonormal basis $(\Psi_\varepsilon^k)_{k\in\N}$ of $L^2_{J_\varepsilon}(I_{\varepsilon,\tilde{\delta}},\K)$ consisting of smooth $\R$-valued eigenfunctions $\Psi_\varepsilon^k$ to $\lambda_\varepsilon^k$. 
		\item $\lambda_\varepsilon^1$ is simple and the corresponding eigenfunction $\Psi_\varepsilon^1$ has a sign. We take $\Psi_\varepsilon^1$ positive.
		\item Let $c_0>0$ be such that $\inf_{|z|\geq c_0}f''(\theta_0(z))\geq \frac{3}{4}\min\{f''(\pm 1)\}$. There is an $\varepsilon_0>0$ (only depending on $\delta$, $\tilde{\delta}$, $\overline{C}_0$, $c_1$, $C_2$, $C_3$, $C_4$) such that for all $\varepsilon\in(0,\varepsilon_0]$ and any normalized $\R$-valued eigenfunction $\Psi_\varepsilon$ of $L_\varepsilon$ to an eigenvalue $\lambda_\varepsilon\leq\frac{1}{4}\min\{f''(\pm 1)\}$ it holds 
		\[
		|\Psi_\varepsilon(z)|\leq  Ce^{-|z|\sqrt{\min\{f''(\pm 1)\}/3}}\quad\text{ for all }z\in I_{\varepsilon,\tilde{\delta}}, |z|\geq c_0+1,
		\]
		where $C>0$ only depends on $c_0$, $\min\{f''(\pm 1)\}$ and $c_1$.
	\end{enumerate}	
\end{Lemma}
\begin{proof}
	This follows in the analogous way as in the unperturbed case, cf.~the proof of Lemma \ref{th_SE_1Dscal_unpert},~1.-3.~above. For 3.~consider the proof of Proposition 3.7 in \cite{Marquardt}. Here the abstract Fredholm alternative in Theorem \ref{th_fred} below is applied to 
	\[
	A_\varepsilon:H^1_{J_\varepsilon}(I_{\varepsilon,\tilde{\delta}},\C)\rightarrow H^1_{J_\varepsilon}(I_{\varepsilon,\tilde{\delta}},\C)^*: u\mapsto [v\mapsto B_\varepsilon(u,v)],
	\]
	where $H^1_{J_\varepsilon}(I_{\varepsilon,\tilde{\delta}},\C)$ is $H^1(I_{\varepsilon,\tilde{\delta}},\C)$ with the weight $J_\varepsilon$ in the norm (both in the $L^2$-norm for the function and the derivative) and $H^1_{J_\varepsilon}(I_{\varepsilon,\tilde{\delta}},\C)^*$ is the anti-dual space.
\end{proof}

Now we obtain assertions that correspond to Lemma \ref{th_SE_1Dscal_unpert},~3.-6.~in the unperturbed case.

\begin{Theorem}\label{th_SE_1Dscal_pert2}
	There is an $\varepsilon_0>0$ only depending on $\delta$, $\tilde{\delta}$, $\overline{C}_0$ ,$c_1$, $C_2$, $C_3$, $C_4$ and $C>0$ only depending on $\tilde{\delta}$, $\overline{C}_0$, $c_1$, $C_2$, $C_3$, $C_4$ such that
	\begin{enumerate}
		\item For $\varepsilon\in(0,\varepsilon_0]$ it holds
		\[
		\lambda_\varepsilon^1=\inf_{\Psi\in H^1(I_{\varepsilon,\tilde{\delta}},\K), \|\Psi\|_{J_\varepsilon}=1} B_\varepsilon(\Psi,\Psi)= B_\varepsilon(\Psi_\varepsilon^1,\Psi_\varepsilon^1), \quad |\lambda_\varepsilon^1|\leq C\varepsilon^2.
		\]
		\item Let $\Psi^R_\varepsilon:=\Psi_\varepsilon^1-J(0)^{-\frac{1}{2}} \beta_\varepsilon\theta_0'$, where $\beta_\varepsilon=\|\theta_0'\|_{L^2(I_{\varepsilon,\tilde{\delta}})}$. Then for $\varepsilon\in(0,\varepsilon_0]$
		\[
		\textstyle\left\|\Psi^R_\varepsilon\right\|_{J_\varepsilon}
		+\left\|\frac{d}{dz}\Psi^R_\varepsilon\right\|_{J_\varepsilon}\leq C\varepsilon.
		\]
		\item With $\nu_1$ from Lemma \ref{th_SE_1Dscal_unpert},~5.~it holds for all $\varepsilon\in(0,\varepsilon_0]$
		\[
		\lambda_\varepsilon^2 = \inf_{\Psi\in H^1(I_{\varepsilon,\tilde{\delta}},\K), \|\Psi\|_{J_\varepsilon}=1, \Psi\perp_{J_\varepsilon} \Psi_\varepsilon^1} B_\varepsilon(\Psi,\Psi) \geq \nu_2:=\min\left\{\frac{\nu_1}{2},\frac{\min\{f''(\pm 1)\}}{4}\right\}>0.
		\]
	\end{enumerate}
\end{Theorem}

Although the proof is analogous to the ones of Lemma 2.2, 1.-2.~in \cite{ChenSpectrums} and Lemma 3.8 in \cite{Marquardt}, we give some details for the convenience of the reader. This will also help to understand the vector-valued case later.

\begin{proof} Note that it is enough to consider the case $\K=\R$. The eigenvalues are the same for $\K=\R, \C$ and the orthonormal basis of $\R$-valued eigenfunctions can be chosen to be the same. The \textup{inf}-characterizations can be shown as in the unperturbed case. For convenience, if we write \enquote{for $\varepsilon$ small} in the following it is always meant \enquote{for all $\varepsilon\in(0,\varepsilon_0]$ for some $\varepsilon_0>0$ small only depending on $\delta$, $\tilde{\delta}$, $\overline{C}_0$ ,$c_1$, $C_2$, $C_3$, $C_4$}. Similarly, all appearing constants (also in $\Oc$-notation) below only depend on $\tilde{\delta}$, $\overline{C}_0$ ,$c_1$, $C_2$, $C_3$, $C_4$, but we do not explicitly state this.\phantom{\qedhere}
	
First, we derive an identity for $B_\varepsilon(\Psi,\Psi)$ for all $\Psi\in H^1(I_{\varepsilon,\tilde{\delta}},\R)$. We define $\hat{\Psi}:=J_\varepsilon^{1/2}\Psi$. Then 
\[
\frac{d}{dz}\Psi = -\frac{1}{2}J_\varepsilon^{-\frac{3}{2}}(\frac{d}{dz}J_\varepsilon)\hat{\Psi}+ J_\varepsilon^{-\frac{1}{2}}\frac{d}{dz}\hat{\Psi}.
\]
Therefore
\[
B_\varepsilon(\Psi,\Psi)=\int_{I_{\varepsilon,\tilde{\delta}}} (\frac{d}{dz}\hat{\Psi})^2 + \left[f''(\phi_\varepsilon(\varepsilon .))+\frac{1}{4}J_\varepsilon^{-2}(\frac{d}{dz}J_\varepsilon)^2\right]\hat{\Psi}^2
-J_\varepsilon^{-1}(\frac{d}{dz}J_\varepsilon)\frac{1}{2}\frac{d}{dz}(\hat{\Psi}^2)\,dz.
\]
In order to use results from the unperturbed case, we would like to replace $f''(\phi_\varepsilon(\varepsilon.))$ by $f''(\theta_0)$. Therefore we use a Taylor expansion and obtain for all $|z|\leq\frac{\delta}{\varepsilon}$
\begin{align*}
&|f''(\phi_\varepsilon(\varepsilon z))-f''(\theta_0(z))-\varepsilon p_\varepsilon f'''(\theta_0(z))\theta_1(z)|\\
&\leq 
C|q_\varepsilon(\varepsilon z)|\varepsilon^2 + C|\varepsilon p_\varepsilon\theta_1(z)+q_\varepsilon(\varepsilon z)\varepsilon^2|^2\leq C(1+|z|)\varepsilon^2.
\end{align*}
Rewriting the last term in the above identity for $B_\varepsilon(\Psi,\Psi)$ with integration by parts yields
\begin{align}\label{eq_SE_1Dscal_pert_Beps_id}
B_\varepsilon(\Psi,\Psi)=
B_{0,\varepsilon}(\hat{\Psi},\hat{\Psi})
+ \int_{I_{\varepsilon,\tilde{\delta}}}  \left[\varepsilon p_\varepsilon f'''(\theta_0)\theta_1
+\tilde{q}_\varepsilon\right]
\hat{\Psi}^2
-\frac{1}{2}\left[J_\varepsilon^{-1}(\frac{d}{dz}J_\varepsilon)\hat{\Psi}^2 \right]_{z=-\frac{\tilde{\delta}}{\varepsilon}}^{\frac{\tilde{\delta}}{\varepsilon}},\\\notag
\tilde{q}_\varepsilon:=f''(\phi_\varepsilon(\varepsilon .))-f''(\theta_0)-f'''(\theta_0)\varepsilon p_\varepsilon\theta_1 + \frac{1}{4}\left(2J_\varepsilon^{-1}(\frac{d^2}{dz^2}J_\varepsilon)-J_\varepsilon^{-2}(\frac{d}{dz}J_\varepsilon)^2\right).
\end{align}
The first part of $\tilde{q}_\varepsilon$ is estimated above, the second part can be controlled with Corollary \ref{th_SE_1Dprelim2}. This implies $|\tilde{q}_\varepsilon(z)|\leq C\varepsilon^2(1+|z|)$ for all $z\in I_{\varepsilon,\tilde{\delta}}$.

\begin{proof}[Ad 1] First we show an upper bound on $\lambda_\varepsilon^1$ using \eqref{eq_SE_1Dscal_pert_Beps_id}. To this end we consider $\Psi=J_\varepsilon^{-1/2}\beta_\varepsilon\theta_0'$. It holds $\|J_\varepsilon^{-1/2}\beta_\varepsilon\theta_0'\|_{J_\varepsilon}=1$ due to the definitions. Hence with  \eqref{eq_SE_1Dscal_pert_Beps_id} and Corollary \ref{th_SE_1Dprelim2} we obtain
\[
\lambda_\varepsilon^1\leq \beta_\varepsilon^2\left[ B_{0,\varepsilon}(\theta_0',\theta_0') + \varepsilon p_\varepsilon \int_{I_{\varepsilon,\tilde{\delta}}} f'''(\theta_0)\theta_1(\theta_0')^2 + C\varepsilon^2\int_{I_{\varepsilon,\tilde{\delta}}}(1+|z|)\theta_0'(z)^2\,dz+C\varepsilon e^{-c\tilde{\delta}/\varepsilon}\right].
\]
The identity $\int_\R(\theta_0'')^2+f''(\theta_0)(\theta_0')^2=0$ due to integration by parts, \eqref{eq_SE_1Dscal_pert_theta1} and the decay for $\theta_0$ and its derivatives from Theorem \ref{th_theta_0} imply $\lambda_\varepsilon^1\leq C\varepsilon^2$ for $\varepsilon$ small.

In particular, Lemma \ref{th_SE_1Dscal_pert1},~3.~yields for $\varepsilon$ small the decay 
\[
|\Psi_\varepsilon^1(z)|\leq Ce^{-\frac{\sqrt{\min\{f''(\pm 1)\}}}{2}|z|}\quad\text{ for all }z\in I_{\varepsilon,\tilde{\delta}}, |z|\geq c_0+1,
\]
where $c_0>0$ is such that $\inf_{|z|\geq c_0}f''(\theta_0(z))\geq \frac{3}{4}\min\{f''(\pm 1)\}$.

Hence with \eqref{eq_SE_1Dscal_pert_Beps_id} and estimates as before we get for $\varepsilon$ small
\[
\lambda_\varepsilon^1=
B_{0,\varepsilon}(\hat{\Psi}_\varepsilon^1,\hat{\Psi}_\varepsilon^1)
+\varepsilon p_\varepsilon \int_{I_{\varepsilon,\tilde{\delta}}} f'''(\theta_0)\theta_1(\hat{\Psi}_\varepsilon^1)^2\,dz+\Oc(\varepsilon^2).
\]
In order to use \eqref{eq_SE_1Dscal_pert_theta1}, we note that $\Psi_{0,\varepsilon}^R:=\Psi_{0,\varepsilon}^1-\beta_\varepsilon\theta_0'$ satisfies good estimates due to Lemma \ref{th_SE_1Dscal_unpert},~6. Therefore we split
\begin{align}\label{eq_SE_1Dscal_pert_split}
\hat{\Psi}_\varepsilon^1=J_\varepsilon^\frac{1}{2}\Psi_\varepsilon^1=a_\varepsilon\Psi_{0,\varepsilon}^1 + \Psi_\varepsilon^\perp
\end{align}
orthogonally in $L^2(I_{\varepsilon,\tilde{\delta}})$, where $a_\varepsilon:=(\hat{\Psi}_\varepsilon^1,\Psi_{0,\varepsilon}^1)_{L^2(I_{\varepsilon,\tilde{\delta}})}$. It holds $|a_\varepsilon|\leq 1$ due to the Cauchy-Schwarz-Inequality and $a_\varepsilon^2=1-\|\Psi_\varepsilon^\perp\|_{L^2(I_{\varepsilon,\tilde{\delta}})}^2$. Moreover, due to positivity of $\hat{\Psi}_\varepsilon^1$ and $\Psi_{0,\varepsilon}^1$ we have $a_\varepsilon>0$. Note that the latter is only needed for the estimate of $\Psi_\varepsilon^R$ later. Hence
\[
\left|\int_{I_{\varepsilon,\tilde{\delta}}} f'''(\theta_0)\theta_1(\hat{\Psi}_\varepsilon^1)^2\,dz\right|\leq a_\varepsilon^2\left|\int_{I_{\varepsilon,\tilde{\delta}}} f'''(\theta_0)\theta_1(\Psi_{0,\varepsilon}^1)^2\,dz\right|+C\|\Psi_\varepsilon^\perp\|_{L^2(I_{\varepsilon,\tilde{\delta}})},
\]
where we used $\|\Psi_\varepsilon^\perp\|_{L^2(I_{\varepsilon,\tilde{\delta}})}\leq 1$. We substitute $\Psi_{0,\varepsilon}^1$ by $\Psi_{0,\varepsilon}^R+\beta_\varepsilon\theta_0'$. With \eqref{eq_SE_1Dscal_pert_theta1}, the decay for $\theta_0'$ from Theorem \ref{th_theta_0} and Lemma \ref{th_SE_1Dscal_unpert},~6.~we obtain for $\varepsilon$ small
\[
\left|\int_{I_{\varepsilon,\tilde{\delta}}} f'''(\theta_0)\theta_1(\hat{\Psi}_\varepsilon^1)^2\,dz\right|\leq C(e^{-c/\varepsilon}+\|\Psi_\varepsilon^\perp\|_{L^2(I_{\varepsilon,\tilde{\delta}})})\leq \tilde{C}(\varepsilon+\|\Psi_\varepsilon^\perp\|_{L^2(I_{\varepsilon,\tilde{\delta}})}).
\]
Moreover, due to integration by parts we have $B_{0,\varepsilon}(\Psi_{0,\varepsilon}^1,\Psi_\varepsilon^\perp)=0$ and therefore
\[
B_{0,\varepsilon}(\hat{\Psi}_\varepsilon^1,\hat{\Psi}_\varepsilon^1)=a_\varepsilon^2B_{0,\varepsilon}(\Psi_{0,\varepsilon}^1,\Psi_{0,\varepsilon}^1)+B_{0,\varepsilon}(\Psi_\varepsilon^\perp,\Psi_\varepsilon^\perp)=a_\varepsilon^2\lambda_{0,\varepsilon}^1+B_{0,\varepsilon}(\Psi_\varepsilon^\perp,\Psi_\varepsilon^\perp).
\]
Together with Lemma \ref{th_SE_1Dscal_unpert},~4.-5.~this yields for $\varepsilon$ small
\[
C\varepsilon^2\geq\lambda_\varepsilon^1\geq \nu_1\|\Psi_\varepsilon^\perp\|_{L^2(I_{\varepsilon,\tilde{\delta}})}^2 + \Oc(\varepsilon)\|\Psi_\varepsilon^\perp\|_{L^2(I_{\varepsilon,\tilde{\delta}})} +\Oc(\varepsilon^2)\geq\frac{\nu_1}{2}\|\Psi_\varepsilon^\perp\|_{L^2(I_{\varepsilon,\tilde{\delta}})}^2-\tilde{C}\varepsilon^2,
\]
where we used Young's inequality for the last estimate. This shows $\|\Psi_\varepsilon^\perp\|_{L^2(I_{\varepsilon,\tilde{\delta}})}=\Oc(\varepsilon)$ and hence $\lambda_\varepsilon^1=\Oc(\varepsilon^2)$ for $\varepsilon$ small. Moreover, we get $a_\varepsilon^2=1+\Oc(\varepsilon^2)$ for $\varepsilon$ small.\qedhere$_{1.}$\end{proof}

\begin{proof}[Ad 2] The estimates above yield $|B_{0,\varepsilon}(\Psi_\varepsilon^\perp,\Psi_\varepsilon^\perp)|=\Oc(\varepsilon^2)$. Therefore $\|\Psi_\varepsilon^\perp\|_{L^2(I_{\varepsilon,\tilde{\delta}})}=\Oc(\varepsilon)$ and the definition of $B_{0,\varepsilon}$ imply $\|\frac{d}{dz}\Psi_\varepsilon^\perp\|_{L^2(I_{\varepsilon,\tilde{\delta}})}=\Oc(\varepsilon)$. Using this together with properties of $\Psi_{0,\varepsilon}^R=\Psi_{0,\varepsilon}^1-\beta_\varepsilon\theta_0'$ from Lemma \ref{th_SE_1Dscal_unpert},~6.~we will deduce estimates for $\Psi_\varepsilon^R:=\Psi_\varepsilon^1-J(0)^{-1/2}\beta_\varepsilon\theta_0'$. First, we use the splitting \eqref{eq_SE_1Dscal_pert_split} and the definition of $\Psi_{0,\varepsilon}^R$ to rewrite
\[
\Psi_\varepsilon^R=J_\varepsilon^{-\frac{1}{2}}
\left[(a_\varepsilon-J_\varepsilon^{\frac{1}{2}}J(0)^{-\frac{1}{2}})\beta_\varepsilon\theta_0'+a_\varepsilon\Psi_{0,\varepsilon}^R+\Psi_\varepsilon^\perp\right].
\]
With $a_\varepsilon^2=1+\Oc(\varepsilon^2)$ for $\varepsilon$ small, $a_\varepsilon>0$ and $(a_\varepsilon-1)(a_\varepsilon+1)=a_\varepsilon^2-1$ we get $a_\varepsilon=1+\Oc(\varepsilon^2)$. Moreover, a Taylor expansion and Corollary \ref{th_SE_1Dprelim2} yields for all $z\in I_{\varepsilon,\tilde{\delta}}$
\[
J_\varepsilon(z)^{\frac{1}{2}}=J(0)^{\frac{1}{2}}+\Oc(|F_\varepsilon(z)|)\quad\text{ and }\quad|F_\varepsilon(z)|\leq C\varepsilon(|z|+1).
\]
Therefore $|a_\varepsilon-J_\varepsilon(z)^{\frac{1}{2}}J(0)^{-\frac{1}{2}}|\leq \tilde{C}\varepsilon(|z|+1)$ for all $z\in I_{\varepsilon,\tilde{\delta}}$ and $\varepsilon$ small. Together with the decay for $\theta_0'$ and the estimates for $\|\Psi_\varepsilon^\perp\|_{L^2(I_{\varepsilon,\tilde{\delta}})}$ and $\|\Psi_{0,\varepsilon}^R\|_{L^2(I_{\varepsilon,\tilde{\delta}})}$ we obtain $\|\Psi_\varepsilon^R\|_{J_\varepsilon}=\Oc(\varepsilon)$. Moreover, the derivative is given by
\begin{align*}
\frac{d}{dz}\Psi_\varepsilon^R&=
J_\varepsilon^{-\frac{1}{2}}\left[(a_\varepsilon-J_\varepsilon^{\frac{1}{2}}J(0)^{-\frac{1}{2}})\beta_\varepsilon\theta_0''+(-\frac{1}{2}J_\varepsilon^{-\frac{1}{2}}\frac{d}{dz}J_\varepsilon J(0)^{-\frac{1}{2}})\beta_\varepsilon\theta_0'+a_\varepsilon\frac{d}{dz}\Psi_{0,\varepsilon}^R+\frac{d}{dz}\Psi_\varepsilon^\perp\right]\\
&-\frac{1}{2}J_\varepsilon^{-1}(\frac{d}{dz}J_\varepsilon)\Psi_\varepsilon^R.
\end{align*}
The estimates for $|a_\varepsilon-J_\varepsilon(z)^{\frac{1}{2}}J(0)^{-\frac{1}{2}}|$, $\|\Psi_\varepsilon^R\|_{J_\varepsilon}$, $\|\frac{d}{dz}\Psi_{0,\varepsilon}^R\|_{L^2}$ and $\|\frac{d}{dz}\Psi_\varepsilon^\perp\|_{L^2}$ as well as Corollary \ref{th_SE_1Dprelim2} yield $\|\frac{d}{dz}\Psi_\varepsilon^R\|_{J_\varepsilon}=\Oc(\varepsilon)$.\qedhere$_{2.}$\end{proof}

\begin{proof}[Ad 3] Consider any normalized eigenfunction $\Psi_\varepsilon^2$ to $\lambda_\varepsilon^2$. If $\lambda_\varepsilon^2\geq\frac{1}{4}\min\{f''(\pm1)\}$, then there is nothing to show. Therefore we assume that $\lambda_\varepsilon^2\leq\frac{1}{4}\min\{f''(\pm1)\}$. Then $\Psi_\varepsilon^2$ satisfies the decay in Lemma \ref{th_SE_1Dscal_pert1},~3.~and computations as before yield 
\[
\lambda_\varepsilon^2=
B_{0,\varepsilon}(\hat{\Psi}_\varepsilon^2,\hat{\Psi}_\varepsilon^2)
+\varepsilon p_\varepsilon \int_{I_{\varepsilon,\tilde{\delta}}} f'''(\theta_0)\theta_1(\hat{\Psi}_\varepsilon^2)^2\,dz+\Oc(\varepsilon^2)=B_{0,\varepsilon}(\hat{\Psi}_\varepsilon^2,\hat{\Psi}_\varepsilon^2)+\Oc(\varepsilon).
\]
Analogously as above we split 
\[
\hat{\Psi}_\varepsilon^2=\tilde{a}_\varepsilon\Psi_{0,\varepsilon}^1+\Psi_\varepsilon^{2,\perp}
\]
orthogonally in $L^2(I_{\varepsilon,\tilde{\delta}})$ and obtain with Lemma \ref{th_SE_1Dscal_unpert},~5.~that
\[
\lambda_\varepsilon^2=
\tilde{a}_\varepsilon^2\lambda_{0,\varepsilon}^1
+B_{0,\varepsilon}(\Psi_\varepsilon^{2,\perp},\Psi_\varepsilon^{2,\perp})
+\Oc(\varepsilon)
\geq\tilde{a}_\varepsilon^2\lambda_{0,\varepsilon}^1
+\nu_1(1-\tilde{a}_\varepsilon^2)-C\varepsilon.
\]
In order to get an estimate for $\tilde{a}_\varepsilon=(\hat{\Psi}_\varepsilon^2,\Psi_{0,\varepsilon}^1)_{L^2(I_{\varepsilon,\tilde{\delta}})}$, note that $\hat{\Psi}_\varepsilon^1\perp_{L^2}\hat{\Psi}_\varepsilon^2$. Therefore with the splitting \eqref{eq_SE_1Dscal_pert_split} we obtain for $\varepsilon$ small
\[
|\tilde{a}_\varepsilon|=
\left|-\frac{1}{a_\varepsilon}(\Psi_\varepsilon^\perp,\hat{\Psi}_\varepsilon^2)_{L^2(I_{\varepsilon,\tilde{\delta}})}\right|
\leq C\|\Psi_\varepsilon^\perp\|_{L^2(I_{\varepsilon,\tilde{\delta}})}
\leq C\varepsilon.
\]
This yields $\lambda_\varepsilon^2\geq\frac{\nu_1}{2}$ for $\varepsilon$ small if $\lambda_\varepsilon^2\leq\frac{1}{4}\min\{f''(\pm1)\}$.\qedhere$_{3.}$\end{proof}
\end{proof}

\subsubsection{Spectral Estimates for Vector-Valued Allen-Cahn-Type Operators in 1D}\label{sec_SE_1Dvect}
In the scalar case we frequently used theorems and estimates that are not available in the vector-valued case, e.g.~the comparison principle, the Harnack-inequality and the Hopf maximum principle. Looking closely into the last Section \ref{sec_SE_1Dscal}, we observe that these arguments were used explicitly only for the proofs of Lemma \ref{th_SE_1Dscal_unpert},~2.-5.~and Lemma \ref{th_SE_1Dscal_pert1},~2.-3. For the vector-valued case we have to adjust suitably. The goal is to obtain analogous assertions based on the operator $\check{\Lc}_0:=-\frac{d^2}{dz^2} + D^2W(\vec{\theta}_0)$, where $W:\R^m\rightarrow\R$ is as in Definition \ref{th_vAC_W} and $\vec{\theta}_0$ is as in Remark \ref{th_ODE_vect_rem},~1. In Lemma \ref{th_ODE_vect_lin_op} we already showed properties of $\check{\Lc}_0$ viewed as an unbounded operator 
$\check{L}_0:H^2(\R,\K)^m\rightarrow L^2(\R,\K)^m$ and we obtained a spectral gap provided $\dim\ker\check{L}_0=1$. Under this assumption we show analogous properties as in the scalar case in the last Section \ref{sec_SE_1Dscal}. To this end we use contradiction arguments and further assertions in Kusche \cite{Kusche}, in particular \cite{Kusche}, Chapter 1, where abstract vector-valued Sturm-Liouville operators are considered.

\paragraph{Unperturbed Vector-Valued Allen-Cahn-Type Operators in 1D}\label{sec_SE_1Dvect_unpert}
We consider $\check{\Lc}_0$ on finite large intervals together with homogeneous Neumann boundary condition. 

Let $\tilde{\delta}>0$ fixed, $\varepsilon>0$, $I_{\varepsilon,\tilde{\delta}}:=(-\frac{\tilde{\delta}}{\varepsilon},\frac{\tilde{\delta}}{\varepsilon})$ and $\K=\R$ or $\C$. We consider the unbounded operator 
\[
\check{L}_{0,\varepsilon}:H^2_N(I_{\varepsilon,\tilde{\delta}},\K)^m\rightarrow L^2(I_{\varepsilon,\tilde{\delta}},\K)^m:
\vec{u}\mapsto \check{\Lc}_0\vec{u}=\left[-\frac{d^2}{dz^2}+D^2W(\vec{\theta}_0)\right]\vec{u}.
\]
The associated sesquilinearform is $\check{B}_{0,\varepsilon}:H^1(I_{\varepsilon,\tilde{\delta}},\K)^m\times H^1(I_{\varepsilon,\tilde{\delta}},\K)^m\rightarrow\K$,
\[
\check{B}_{0,\varepsilon}(\vec{\Phi},\vec{\Psi})
:=\int_{I_{\varepsilon,\tilde{\delta}}}(\frac{d}{dz}\vec{\Phi},\frac{d}{dz}\vec{\Psi})_{\K^m} + (D^2W(\vec{\theta}_0)\vec{\Phi},\vec{\Psi})_{\K^m}\,dz.
\] 
We obtain the analogue of Lemma \ref{th_SE_1Dscal_unpert}.
\begin{Lemma}\phantomsection{\label{th_SE_1Dvect_unpert}}Assume $\dim\ker\check{L}_0=1$, cf.~Remark \ref{th_ODE_vect_lin_op_rem}. Then
	\begin{enumerate}
		\item $\check{L}_{0,\varepsilon}$ is selfadjoint and the spectrum is given by discrete eigenvalues $(\check{\lambda}_{0,\varepsilon}^k)_{k\in\N}$ in $\R$ with $\check{\lambda}_{0,\varepsilon}^1\leq\check{\lambda}_{0,\varepsilon}^2\leq...  $ and $\check{\lambda}_{0,\varepsilon}^k\overset{k\rightarrow\infty}\longrightarrow\infty$. Moreover, there is an orthonormal basis $(\vec{\Psi}_{0,\varepsilon}^k)_{k\in\N}$ of $L^2(I_{\varepsilon,\tilde{\delta}},\K)^m$ consisting of smooth $\R^m$-valued eigenfunctions $\vec{\Psi}_{0,\varepsilon}^k$ to $\check{\lambda}_{0,\varepsilon}^k$. 
		\item $\check{\lambda}_{0,\varepsilon}^1$ is simple for $\varepsilon>0$ small.
		\item For any normalized eigenfunction $\vec{\Psi}_{0,\varepsilon}$ to an eigenvalue $\check{\lambda}_{0,\varepsilon}\leq\frac{1}{4}\min\{\sigma(D^2W(\vec{u}_\pm))\}$ of $\check{L}_{0,\varepsilon}$ and $\varepsilon>0$ small it holds 
		\[
		|\vec{\Psi}_{0,\varepsilon}(z)|\leq Ce^{-|z|\sqrt{\min\{\sigma(D^2W(\vec{u}_\pm))\}/6}}\quad\text{ for all }z\in I_{\varepsilon,\tilde{\delta}},
		\]
		where $C>0$ is independent of $\varepsilon$, $\tilde{\delta}$.
		\item There is $\check{\varepsilon}_0=\check{\varepsilon}_0(\tilde{\delta})>0$ small such that for all $\varepsilon\in(0,\check{\varepsilon}_0]$
		\[
		\check{\lambda}_{0,\varepsilon}^1=\inf_{\vec{\Psi}\in H^1(I_{\varepsilon,\tilde{\delta}})^m, \|\vec{\Psi}\|_{L^2}=1} \check{B}_{0,\varepsilon}(\vec{\Psi},\vec{\Psi})= \check{B}_{0,\varepsilon}(\vec{\Psi}_{0,\varepsilon}^1,\vec{\Psi}_{0,\varepsilon}^1)=
		\Oc(e^{-\frac{3\tilde{\delta}\sqrt{\min\{\sigma(D^2W(\vec{u}_\pm))\}}}{2\sqrt{2}\varepsilon}}),
		\]
		where the constant in the $\Oc$-estimate is independent of $\tilde{\delta}$, $\varepsilon$.
		\item There is $\check{\nu}_1>0$ independent of $\tilde{\delta}$, $\varepsilon$ and $\check{\varepsilon}_0=\check{\varepsilon}_0(\tilde{\delta})>0$ small such that
		\[
		\check{\lambda}_{0,\varepsilon}^2 = \inf_{\vec{\Psi}\in H^1(I_{\varepsilon,\tilde{\delta}})^m, \|\vec{\Psi}\|_{L^2}=1, \vec{\Psi}\perp \vec{\Psi}_{0,\varepsilon}^1} \check{B}_{0,\varepsilon}(\vec{\Psi},\vec{\Psi}) \geq \check{\nu}_1\quad\text{ for all }\varepsilon\in(0,\check{\varepsilon}_0].
		\]
		\item Let $\check{\beta}_\varepsilon:=\|\vec{\theta}_0'\|_{L^2(I_{\varepsilon,\tilde{\delta}})^m}^{-1}$. For $\check{\varepsilon}_0=\check{\varepsilon}_0(\tilde{\delta})>0$ small and $\varepsilon\in(0,\check{\varepsilon}_0]$ there are $\check{c}_{0,\varepsilon}\in\{\pm 1\}$ such that for $\vec{\Psi}_{0,\varepsilon}^R:=\check{c}_{0,\varepsilon}\vec{\Psi}_{0,\varepsilon}^1-\check{\beta}_\varepsilon\vec{\theta}_0'$ we have
		\[\textstyle
		\left\|\vec{\Psi}_{0,\varepsilon}^R\right\|_{L^2(I_{\varepsilon,\tilde{\delta}})^m}^2
		+\left\|\frac{d}{dz}\vec{\Psi}_{0,\varepsilon}^R\right\|_{L^2(I_{\varepsilon,\tilde{\delta}})^m}^2
		\leq Ce^{-\frac{3\tilde{\delta}\sqrt{\min\{\sigma(D^2W(\vec{u}_\pm))\}}}{2\sqrt{2}\varepsilon}},
		\]
		where $C>0$ is independent of $\tilde{\delta}$, $\varepsilon$.
	\end{enumerate}
\end{Lemma}
\begin{Remark}\upshape
	Lemma \ref{th_SE_1Dvect_unpert},~1.~and 3.-4.~also work without the assumption $\dim\ker\check{L}_0=1$, cf.~\cite{Kusche}, Lemma 1.1 and Lemma 2.1. However, one has to modify the decay parameters by some scalar factor independent of $W$. This is because the decay properties for $\vec{\theta}_0$ from Theorem \ref{th_ODE_vect} are better than the ones obtained from \cite{Kusche}. More precisely, the maximal rate is $\sqrt{\min\{\sigma(D^2W(\vec{u}_\pm))\}/2}$ instead of $\sqrt{\min\{\sigma(D^2W(\vec{u}_\pm))\}}/2$ in \cite{Kusche}. Nevertheless, the precise rates in Lemma \ref{th_SE_1Dvect_unpert},~3.-4.~are not so important anyway.
\end{Remark}

\begin{proof} By scaling in $\varepsilon$ it is enough to consider the case $\tilde{\delta}=1$. We set $I_\varepsilon:=I_{\varepsilon,1}$.\phantom{\qedhere}
	
\begin{proof}[Ad 1]This can be seen as in the scalar case, cf.~the proof of Lemma \ref{th_SE_1Dscal_unpert},~1. Here the abstract Fredholm alternative in Theorem \ref{th_fred} below is used for 
	\[
	\check{A}_{0,\varepsilon}:H^1(I_\varepsilon,\C^m)\rightarrow H^1(I_\varepsilon,\C^m)^*:\vec{u}\mapsto [\vec{v}\mapsto \check{B}_{0,\varepsilon}(\vec{u},\vec{v})],
	\]
	where $H^1(I_\varepsilon,\C^m)^*$ is the anti-dual space.\qedhere$_{1.}$\end{proof}
	
\begin{proof}[Ad 2] Assume the contrary. Then there is a zero sequence $(\varepsilon_n)_{n\in\N}$ and normalized, pairwise orthogonal eigenfunctions $\vec{\Psi}_{0,\varepsilon_n}^1$, $\vec{\Psi}_{0,\varepsilon_n}^2$ of $\check{L}_{0,\varepsilon_n}$ to the eigenvalue $\check{\lambda}_{0,\varepsilon_n}^1$ for all $n\in\N$.
	Now note that the upper bound on $\check{\lambda}_{0,\varepsilon}^1$ in 4.~can be shown solely with the decay properties of $\vec{\theta}_0$ from Theorem \ref{th_ODE_vect}, cf.~\cite{Kusche}, proof of Lemma 2.1, 1. Therefore due to \cite{Kusche}, Lemma 1.2 (and its proof) there is a subsequence $(\varepsilon_{n_k})_{k\in\N}$ such that $\vec{\Psi}_{0,\varepsilon_{n_k}}^j$ converges uniformly in $C^2$ on compact subsets of $\R$ to a normalized eigenfunction $\vec{\Psi}_0^j\in H^2(\R,\K)^m\cap C^2(\R,\K)^m$ to the eigenvalue $0$ of $\check{L}_0$ for $j=1,2$. Because of \cite{Kusche}, Lemma 1.1 all $\vec{\Psi}_{0,\varepsilon_n}^j$, $\vec{\Psi}_0^j$ for $n\in\N$ and $j=1,2$ satisfy uniform pointwise exponential bounds. Hence the Dominated Convergence Theorem yields that $\vec{\Psi}_0^1$ is orthogonal to $\vec{\Psi}_0^2$. This is a contradiction to $\dim\ker\check{L}_0=1$.\qedhere$_{2.}$\end{proof}
	
	\begin{proof}[Ad 3] This follows from \cite{Kusche}, Lemma 1.1.\qedhere$_{3.}$\end{proof}
	
	\begin{proof}[Ad 4] The \textup{inf}-characterization can be shown as in the scalar case. As mentioned in the proof of 2. above, the upper bound on $\check{\lambda}_{0,\varepsilon}^1$ follows from Theorem \ref{th_ODE_vect}. \cite{Kusche}, Lemma 1.2 and a contradiction argument yield $|(\vec{\Psi}_{0,\varepsilon}^1,\vec{\theta}_0')_{L^2(I_\varepsilon)^m}|\geq C>0$ for $\varepsilon$ small. Together with the uniform decay for eigenfunctions from 2.~this implies the estimate, cf.~also the proof of Lemma 2.1, 1.~in \cite{Kusche}.\qedhere$_{4.}$\end{proof}
	
	\begin{proof}[Ad 5] The \textup{inf}-characterization follows as in the scalar case. For $\check{\nu}_0$ as in Lemma \ref{th_ODE_vect_lin_op} let $\check{\nu}_1:=\min\{\frac{1}{2}\check{\nu}_0,\frac{1}{4}\min\{\sigma(D^2(\vec{u}_\pm))\}\}>0$. Assume the estimate on $\check{\lambda}_{0,\varepsilon}^2$ does not hold with this $\check{\nu}_1$. Then there is a zero sequence $(\varepsilon_n)_{n\in\N}$ such that $\check{\lambda}_{0,\varepsilon_n}^2<\check{\nu}_1$. Due to \cite{Kusche}, Lemma 1.2, there is a subsequence $(\varepsilon_{n_k})_{k\in\N}$ such that some normalized eigenvectors $\vec{\Psi}_{0,\varepsilon_{n_k}}^2$ to $\check{\lambda}_{0,\varepsilon_{n_k}}^2$ converge uniformly in $C^2$ on compact subsets of $\R$ to an eigenfunction $\vec{\Psi}_0^2$ of $\check{L}_0$. Due to the assumption on $\check{\nu}_1$, the eigenvalue corresponding to $\vec{\Psi}_0^2$ is necessarily zero. In particular $\dim\ker\check{L}_0=1$ yields $(\vec{\Psi}_0^2,\vec{\theta}_0')_{L^2(\R)^m}\neq 0$. On the other hand, since $\vec{\Psi}_{0,\varepsilon}^2$ and $\vec{\Psi}_{0,\varepsilon}^1$ are orthogonal in $L^2(I_\varepsilon)^m$, we obtain with Lemma 2.1 in \cite{Kusche}, the Dominated Convergence Theorem applied to another subsequence using the uniform decay in 3.~that $(\vec{\Psi}_0^2,\vec{\theta}_0')_{L^2(\R)^m}=0$. This is a contradiction.\qedhere$_{5.}$\end{proof}
	
	\begin{proof}[Ad 6] The proof is analogous to the one of \cite{Marquardt}, Lemma 3.8, 3. We decompose the function $\check{\beta}_\varepsilon\vec{\theta}_0'=\check{a}_{0,\varepsilon}\vec{\Psi}_{0,\varepsilon}^1+\vec{\Psi}_{0,\varepsilon}^\perp$ orthogonally in $L^2(I_\varepsilon)^m$, where $\check{\beta}_\varepsilon=\|\vec{\theta}_0'\|_{L^2(I_\varepsilon)^m}^{-1}$. Due to 4.~and Theorem \ref{th_ODE_vect} we obtain with integration by parts for $\varepsilon$ small that
	\[
	\Oc(e^{-\frac{3\sqrt{\min\{\sigma(D^2W(\vec{u}_\pm))\}}}{2\sqrt{2}\varepsilon}})=\check{B}_{0,\varepsilon}(\check{\beta}_\varepsilon\vec{\theta}_0',\check{\beta}_\varepsilon\vec{\theta}_0')\geq \check{a}_{0,\varepsilon}^2 \check{\lambda}_{0,\varepsilon}^1+\check{\nu}_1\|\vec{\Psi}_{0,\varepsilon}^\perp\|_{L^2(I_\varepsilon)^m}^2.
	\]
	Now note that $1=\check{a}_{0,\varepsilon}^2+\|\vec{\Psi}_{0,\varepsilon}^\perp\|_{L^2(I_\varepsilon)^m}^2$. Therefore 4.-5.~and the above estimate yield for $\varepsilon$ small
	\[
	\|\vec{\Psi}_{0,\varepsilon}^\perp\|_{L^2(I_\varepsilon)^m}^2=\Oc(e^{-\frac{3\sqrt{\min\{\sigma(D^2W(\vec{u}_\pm))\}}}{2\sqrt{2}\varepsilon}}).
	\]
	Hence with $1-\check{a}_{0,\varepsilon}^2=(1-\check{a}_{0,\varepsilon})(1+\check{a}_{0,\varepsilon})$ we obtain the estimate in the lemma for $\|\vec{\Psi}_{0,\varepsilon}^R\|_{L^2(I_\varepsilon)^m}$ if we set $\check{c}_{0,\varepsilon}:=\textup{sign}\,\check{a}_{0,\varepsilon}\in\{\pm 1\}$. For notational simplicity assume w.l.o.g. $\check{c}_{0,\varepsilon}=1$, otherwise one can replace $\vec{\Psi}_{0,\varepsilon}^1$ by $\check{c}_{0,\varepsilon}\vec{\Psi}_{0,\varepsilon}^1$. Then it holds $\frac{d}{dz}\vec{\Psi}_{0,\varepsilon}^R=\frac{d}{dz}\vec{\Psi}_{0,\varepsilon}^1-\check{\beta}_\varepsilon\vec{\theta}_0''$ and
	\[
	\left\|\frac{d}{dz}\vec{\Psi}_{0,\varepsilon}^R\right\|_{L^2(I_\varepsilon)^m}^2=\int_{I_\varepsilon}|\check{\beta}_\varepsilon\vec{\theta}_0''|^2-2\check{\beta}_\varepsilon\vec{\theta}_0''\cdot\frac{d}{dz}\vec{\Psi}_{0,\varepsilon}^1+\left|\frac{d}{dz}\vec{\Psi}_{0,\varepsilon}^1\right|^2\,dz.
	\]
	The first term is a problem. Therefore we rewrite 
	\begin{align*}
    -\int_{I_\varepsilon}\vec{\theta}_0''\cdot\frac{d}{dz}\vec{\Psi}_{0,\varepsilon}^1
    =\int_{I_\varepsilon}\vec{\theta}_0'\cdot\frac{d^2}{dz^2}\vec{\Psi}_{0,\varepsilon}^1, \quad
    \int_{I_\varepsilon}\left|\frac{d}{dz}\vec{\Psi}_{0,\varepsilon}^1\right|^2
    =\check{\lambda}_{0,\varepsilon}^1-\int_{I_\varepsilon}(D^2W(\vec{\theta}_0)\vec{\Psi}_{0,\varepsilon}^1,\vec{\Psi}_{0,\varepsilon}^1)_{\R^m}.
    \end{align*}
    We use $\frac{d^2}{dz^2}\vec{\Psi}_{0,\varepsilon}^1
    =-\check{\lambda}_{0,\varepsilon}^1\vec{\Psi}_{0,\varepsilon}^1+D^2W(\vec{\theta}_0)\vec{\Psi}_{0,\varepsilon}^1$ and insert 
    $\vec{\Psi}_{0,\varepsilon}^1=\vec{\Psi}_{0,\varepsilon}^R
    +\check{\beta}_\varepsilon\vec{\theta}_0'$ everywhere. Then integration by parts yields that the quadratic terms in $\vec{\theta}_0$ cancel up to an appropriately decaying term. Moreover, $\check{\lambda}_{0,\varepsilon}^1$ has the decay due to 4.~and the other terms (without the one with $\check{\lambda}_{0,\varepsilon}^1$) where $\vec{\theta}_0'$ is combined with $\vec{\Psi}_{0,\varepsilon}^R$ cancel. Together with the estimate on the $L^2$-norm of $\vec{\Psi}$ we obtain the estimate for the derivative.\qedhere$_{6.}$\end{proof}
\end{proof}

\paragraph{Perturbed Vector-Valued Allen-Cahn-Type Operators in 1D}\label{sec_SE_1Dvect_pert}
In this section we consider perturbed and weighted vector-valued operators in 1D. Let $\delta>0$ and $h_\varepsilon, J\in C^2([-\delta,\delta],\R)$ with $\|h_\varepsilon\|_{C^2([-\delta,\delta])}\leq \overline{C}_0$ for $\varepsilon>0$ small and $c_1,C_2>0$ be such that \eqref{eq_SE_J} holds. Then let $\rho_\varepsilon, F_\varepsilon, J_\varepsilon$ for $\varepsilon>0$ small be as in Section \ref{sec_SE_1Dsetting}. We define
\begin{align}
\vec{\phi}_\varepsilon:[-\delta,\delta]\rightarrow\R^m:r\mapsto \vec{\theta}_0(\frac{r}{\varepsilon}) 
+ \varepsilon p_\varepsilon \vec{\theta}_1(\frac{r}{\varepsilon}) + \vec{q}_\varepsilon(r)\varepsilon^2, 
\end{align}
where $p_\varepsilon\in\R$ and $\vec{q}_\varepsilon:[-\delta,\delta]\rightarrow\R$ is measurable with $|p_\varepsilon|+\frac{\varepsilon}{\varepsilon+|r|}|\vec{q}_\varepsilon(r)|\leq \check{C}_3$ for $r\in[-\delta,\delta]$, a $\check{C}_3>0$, and $\varepsilon>0$ small. Moreover, let $\vec{\theta}_1\in L^\infty(\R)^m$ with $\|\vec{\theta}_1\|_\infty\leq \check{C}_4$ for a $\check{C}_4>0$ and
\begin{align}\label{eq_SE_1Dvect_pert_theta1}
\int_\R (\vec{\theta}_0',\sum_{\xi\in\N_0^m,|\xi|=1}\partial^{\xi}D^2W(\vec{\theta}_0)(\vec{\theta}_1)^\xi\vec{\theta}_0')_{\R^m}=0.
\end{align}

Let $\tilde{\delta}\in(0,\frac{3\delta}{4}]$ be fixed. Then $F_\varepsilon$, $J_\varepsilon$ are well-defined on $\overline{I_{\varepsilon,\tilde{\delta}}}$ for $\varepsilon\in(0,\varepsilon_1(\delta,\overline{C}_0)]$ and Corollary \ref{th_SE_1Dprelim2} is applicable due to Remark \ref{th_SE_1Dprelim_rem},~2. We consider the operators
\[
\check{L}_\varepsilon:H^2_N(I_{\varepsilon,\tilde{\delta}},\K)^m\rightarrow L^2_{J_\varepsilon}(I_{\varepsilon,\tilde{\delta}},\K)^m:
\vec{u}\mapsto \check{\Lc}_\varepsilon \vec{u}:= \left[-J_\varepsilon^{-1} \frac{d}{dz}\left(J_\varepsilon\frac{d}{dz}\right)+D^2W(\vec{\phi}_\varepsilon(\varepsilon .))\right]\vec{u},
\]
where $L^2_{J_\varepsilon}(I_{\varepsilon,\tilde{\delta}},\K)^m$ is the space of $\K^m$-valued $L^2$-functions defined on $I_{\varepsilon,\tilde{\delta}}$ with the weight $J_\varepsilon$. We write $(.,.)_{J_\varepsilon}$, $\|.\|_{J_\varepsilon}$ and $\perp_{J_\varepsilon}$ for the corresponding scalar product, norm and orthogonal relation. Note that for convenience we use the same notation for the latter as in the scalar case.
The sesquilinearform associated to $\check{L}_\varepsilon$ is given by $\check{B}_\varepsilon:H^1(I_{\varepsilon,\tilde{\delta}},\K)^m\times H^1(I_{\varepsilon,\tilde{\delta}},\K)^m\rightarrow\K$,
\[
\check{B}_\varepsilon(\vec{\Phi},\vec{\Psi}):= \int_{I_{\varepsilon,\tilde{\delta}}}\left[(\frac{d}{dz}\vec{\Phi},\frac{d}{dz}\vec{\Psi})_{\K^m} + (D^2W(\vec{\phi}_\varepsilon(\varepsilon .))\vec{\Phi},\vec{\Psi})_{\K^m}\right] J_\varepsilon\,dz.
\]
We obtain the analogue of Lemma \ref{th_SE_1Dvect_unpert},~1.-3.

\begin{Lemma}\phantomsection{\label{th_SE_1Dvect_pert1}}
	\begin{enumerate}
		\item $\check{L}_\varepsilon$ is selfadjoint and the spectrum is given by a sequence of discrete eigenvalues $(\check{\lambda}_\varepsilon^k)_{k\in\N}$ in $\R$ with $\check{\lambda}_\varepsilon^1\leq\check{\lambda}_\varepsilon^2\leq...  $ and $\check{\lambda}_\varepsilon^k\overset{k\rightarrow\infty}\longrightarrow\infty$. Moreover, there is an orthonormal basis $(\vec{\Psi}_\varepsilon^k)_{k\in\N}$ of $L^2_{J_\varepsilon}(I_{\varepsilon,\tilde{\delta}},\K)^m$ consisting of smooth $\R$-valued eigenfunctions $\vec{\Psi}_\varepsilon^k$ to $\check{\lambda}_\varepsilon^k$. 
		\item $\check{\lambda}_\varepsilon^1$ is simple for $\varepsilon>0$ small.
		\item There is an $\check{\varepsilon}_0>0$ (only depending on $\delta$, $\tilde{\delta}$, $\overline{C}_0$, $c_1$, $C_2$, $\check{C}_3$, $\check{C}_4$) such that for all $\varepsilon\in(0,\varepsilon_0]$ and any normalized eigenfunction $\vec{\Psi}_\varepsilon$ of $L_\varepsilon$ to an eigenvalue $\check{\lambda}_\varepsilon
		\leq\frac{1}{4}\min\{\sigma(D^2W(\vec{u}_\pm))\}$ it holds 
		\[
		|\vec{\Psi}_\varepsilon(z)|\leq Ce^{-|z|\sqrt{\min\{\sigma(D^2W(\vec{u}_\pm))\}/6}}\quad\text{ for all }z\in I_{\varepsilon,\tilde{\delta}},
		\]
		where $C>0$ only depends on $c_1$.
	\end{enumerate}	
\end{Lemma}
\begin{proof}
	We need some properties of the weight and the perturbation. Corollary \ref{th_SE_1Dprelim2} and the assumptions yield $|\frac{d}{dz}J_\varepsilon|\leq 2C_2\varepsilon$, $|J_\varepsilon-J(0)|\leq C(\overline{C}_0,C_2)\varepsilon$ in $I_{\varepsilon,\tilde{\delta}}$ and
	\[
	|\vec{\phi}_\varepsilon(\varepsilon z)-\vec{\theta}_0(z)|\leq C(\delta,\check{C}_3,\check{C}_4)\varepsilon\quad\text{ for all }z\in I_{\varepsilon,\tilde{\delta}}.
	\]
	With these uniform estimates one can show that the abstract results in \cite{Kusche} are applicable. Hence the assertions follow in the analogous way as in the unperturbed case, cf.~the proof of Lemma \ref{th_SE_1Dvect_unpert},~1.-3.~above. Here the abstract Fredholm alternative in Theorem \ref{th_fred} below is applied to 
	\[
	\check{A}_\varepsilon:H^1_{J_\varepsilon}(I_{\varepsilon,\tilde{\delta}},\C^m)\rightarrow H^1_{J_\varepsilon}(I_{\varepsilon,\tilde{\delta}},\C^m)^*: \vec{u}\mapsto [\vec{v}\mapsto \check{B}_\varepsilon(\vec{u},\vec{v})],
	\]
	where $H^1_{J_\varepsilon}(I_{\varepsilon,\tilde{\delta}},\C^m)$ is $H^1(I_{\varepsilon,\tilde{\delta}},\C^m)$ with the weight $J_\varepsilon$ in the norm and $H^1_{J_\varepsilon}(I_{\varepsilon,\tilde{\delta}},\C^m)^*$ is the anti-dual space.
\end{proof}

Now we obtain the analogue to Theorem \ref{th_SE_1Dscal_pert2}.

\begin{Theorem}\label{th_SE_1Dvect_pert2}
	There is an $\check{\varepsilon}_0>0$ only depending on $\delta$, $\tilde{\delta}$, $\overline{C}_0$ ,$c_1$, $C_2$, $\check{C}_3$, $\check{C}_4$ and $C>0$ only depending on $\tilde{\delta}$, $\overline{C}_0$, $c_1$, $C_2$, $\check{C}_3$, $\check{C}_4$ such that
	\begin{enumerate}
		\item For $\varepsilon\in(0,\check{\varepsilon}_0]$ it holds
		\[
		\check{\lambda}_\varepsilon^1=\inf_{\vec{\Psi}\in H^1(I_{\varepsilon,\tilde{\delta}},\K)^m, \|\vec{\Psi}\|_{J_\varepsilon}=1} \check{B}_\varepsilon(\vec{\Psi},\vec{\Psi})= \check{B}_\varepsilon(\vec{\Psi}_\varepsilon^1,\vec{\Psi}_\varepsilon^1), \quad |\check{\lambda}_\varepsilon^1|\leq C\varepsilon^2.
		\]
		\item There are $\check{c}_\varepsilon\in\{\pm 1\}$ such that for $\vec{\Psi}^R_\varepsilon:=\check{c}_\varepsilon\vec{\Psi}_\varepsilon^1-J(0)^{-\frac{1}{2}} \check{\beta}_\varepsilon\vec{\theta}_0'$, where $\beta_\varepsilon=\|\vec{\theta}_0'\|_{L^2(I_{\varepsilon,\tilde{\delta}})^m}$, and $\varepsilon\in(0,\check{\varepsilon}_0]$ it holds 
		\[\textstyle
		\|\vec{\Psi}^R_\varepsilon\|_{J_\varepsilon}+\|\frac{d}{dz}\vec{\Psi}^R_\varepsilon\|_{J_\varepsilon}\leq C\varepsilon.
		\]
		\item With $\check{\nu}_1$ from Lemma \ref{th_SE_1Dvect_unpert},~5.~it holds for all $\varepsilon\in(0,\check{\varepsilon}_0]$
		\[
		\check{\lambda}_\varepsilon^2 = \inf_{\vec{\Psi}\in H^1(I_{\varepsilon,\tilde{\delta}},\K)^m, \|\vec{\Psi}\|_{J_\varepsilon}=1, \vec{\Psi}\perp_{J_\varepsilon} \vec{\Psi}_\varepsilon^1} B_\varepsilon(\vec{\Psi},\vec{\Psi}) \geq \check{\nu}_2:=\min\left\{\frac{\check{\nu}_1}{2},\frac{\sigma(D^2W(\vec{u}_\pm))}{4}\right\}>0.
		\]
	\end{enumerate}
\end{Theorem}

\begin{proof} The proof is analogous to the scalar case, cf.~Theorem \ref{th_SE_1Dscal_pert2}. It is enough to consider $\K=\R$. Moreover, the \textup{inf}-characterizations can be shown as in the scalar case. If we write \enquote{for $\varepsilon$ small} in the following it is always meant \enquote{for all $\varepsilon\in(0,\check{\varepsilon}_0]$ for some $\check{\varepsilon}_0>0$ small only depending on $\delta$, $\tilde{\delta}$, $\overline{C}_0$ ,$c_1$, $C_2$, $\check{C}_3$, $\check{C}_4$}. Similarly, all appearing constants (also in $\Oc$-notation) below only depend on $\tilde{\delta}$, $\overline{C}_0$ ,$c_1$, $C_2$, $\check{C}_3$, $\check{C}_4$, but we do not explicitly state this.\phantom{\qedhere}
	
	As in the scalar case, we prove an identity for $\check{B}_\varepsilon(\Psi,\Psi)$ for all $\Psi\in H^1(I_{\varepsilon,\tilde{\delta}},\R)^m$ first. Let $\check{\Psi}:=J_\varepsilon^{1/2}\vec{\Psi}$. Then 
	\[
	\frac{d}{dz}\vec{\Psi} = -\frac{1}{2}J_\varepsilon^{-\frac{3}{2}}(\frac{d}{dz}J_\varepsilon)\check{\Psi}+ J_\varepsilon^{-\frac{1}{2}}\frac{d}{dz}\check{\Psi}.
	\]
	Therefore
	\[
	\check{B}_\varepsilon(\vec{\Psi},\vec{\Psi})=\int_{I_{\varepsilon,\tilde{\delta}}}|\frac{d}{dz}\check{\Psi}|^2 + \check{\Psi}\cdot\left[D^2W(\vec{\phi}_\varepsilon(\varepsilon .))+\frac{1}{4}J_\varepsilon^{-2}(\frac{d}{dz}J_\varepsilon)^2\right]\check{\Psi}
	-J_\varepsilon^{-1}(\frac{d}{dz}J_\varepsilon)\frac{1}{2}\frac{d}{dz}|\check{\Psi}^2|.
	\]
	To use the result from the unperturbed case, we replace $D^2(\vec{\phi}_\varepsilon(\varepsilon.))$ by $D^2W(\vec{\theta}_0)$. To this end we use a Taylor expansion and get for all $|z|\leq\frac{\delta}{\varepsilon}$
	\begin{align*}
	\left|D^2W(\vec{\phi}_\varepsilon(\varepsilon z))-D^2W(\vec{\theta}_0(z))-\varepsilon p_\varepsilon\sum_{\xi\in\N_0^m,|\xi|=1}\partial^\xi D^2W(\vec{\theta}_0(z))(\vec{\theta}_1)^\xi(z)\right|\\
	\leq 
	C|\vec{q}_\varepsilon(\varepsilon z)|\varepsilon^2 + C\varepsilon^2(| p_\varepsilon\vec{\theta}_1(z)|+\varepsilon|\vec{q}_\varepsilon(\varepsilon z)|)^2
	\leq \tilde{C}(1+|z|)\varepsilon^2.
	\end{align*}
	We use integration by parts for the last term in the above identity for $\check{B}_\varepsilon(\vec{\Psi},\vec{\Psi})$. This yields
	\begin{align}\label{eq_SE_1Dvect_pert_Beps_id}
	\check{B}_\varepsilon(\vec{\Psi},\vec{\Psi})&=
	\check{B}_{0,\varepsilon}(\check{\Psi},\check{\Psi})
	-\frac{1}{2}\left[J_\varepsilon^{-1}(\frac{d}{dz}J_\varepsilon)\check{\Psi}^2 \right]_{z=-\frac{\tilde{\delta}}{\varepsilon}}^{\frac{\tilde{\delta}}{\varepsilon}}\\\notag
	&+ \int_{I_{\varepsilon,\tilde{\delta}}} \check{\Psi}\cdot \left[\varepsilon p_\varepsilon\sum_{\xi\in\N_0^m,|\xi|=1}\partial^\xi D^2W(\vec{\theta}_0)(\vec{\theta}_1)^\xi
	+\check{q}_\varepsilon\right]
	\check{\Psi}\,dz,
	\end{align}
	where 
	\begin{align*}
	\check{q}_\varepsilon:=
	D^2W(\vec{\phi}_\varepsilon(\varepsilon .))-D^2W(\vec{\theta}_0)-\varepsilon p_\varepsilon\sum_{\xi\in\N_0^m,|\xi|=1}\partial^\xi D^2W(\vec{\theta}_0)(\vec{\theta}_1)^\xi\\
	+\frac{1}{4}\left(2J_\varepsilon^{-1}(\frac{d^2}{dz^2}J_\varepsilon)-J_\varepsilon^{-2}(\frac{d}{dz}J_\varepsilon)^2\right)\textup{Id}_{\R^{m\times m}}.
	\end{align*}
	The first part of $\check{q}_\varepsilon$ is estimated above, for the second part we use Corollary \ref{th_SE_1Dprelim2}. This yields $|\check{q}_\varepsilon(z)|\leq C\varepsilon^2(1+|z|)$ for all $z\in I_{\varepsilon,\tilde{\delta}}$.
	
	\begin{proof}[Ad 1] First we prove an upper bound on $\check{\lambda}_\varepsilon^1$ with \eqref{eq_SE_1Dvect_pert_Beps_id}. Let $\vec{\Psi}=J_\varepsilon^{-1/2}\check{\beta}_\varepsilon\vec{\theta}_0'$. Then it holds $\|J_\varepsilon^{-1/2}\check{\beta}_\varepsilon\vec{\theta}_0'\|_{J_\varepsilon}=1$. Therefore \eqref{eq_SE_1Dvect_pert_Beps_id} and Corollary \ref{th_SE_1Dprelim2} yield
	\begin{align*}
	\check{\lambda}_\varepsilon^1\leq \check{\beta}_\varepsilon^2\left[ \check{B}_{0,\varepsilon}(\vec{\theta}_0',\vec{\theta}_0')
	+\varepsilon p_\varepsilon \int_{I_{\varepsilon,\tilde{\delta}}} (\vec{\theta}_0',\sum_{\xi\in\N_0^m,|\xi|=1}\partial^\xi D^2W(\vec{\theta}_0)(\vec{\theta}_1)^\xi\vec{\theta}_0')_{\R^m}\,dz\right. \\
	\left.+C\varepsilon^2\int_{I_{\varepsilon,\tilde{\delta}}}(1+|z|)|\vec{\theta}_0'(z)|^2\,dz+C\varepsilon e^{-c\tilde{\delta}/\varepsilon}\right].
	\end{align*}
	It holds $\int_\R|\vec{\theta}_0''|^2+\vec{\theta}_0'\cdot D^2W(\vec{\theta}_0)\vec{\theta}_0'=0$ because of integration by parts. Together with \eqref{eq_SE_1Dvect_pert_theta1} and the decay properties of $\vec{\theta}_0$ from Theorem \ref{th_ODE_vect} this implies $\check{\lambda}_\varepsilon^1\leq C\varepsilon^2$ for $\varepsilon$ small.
	
	In particular, Lemma \ref{th_SE_1Dvect_pert1},~3.~yields for $\varepsilon$ small the decay 
	\[
	|\vec{\Psi}_\varepsilon^1(z)|\leq Ce^{-|z|\sqrt{\min\{\sigma(D^2W(\vec{u}_\pm))\}/6}}\quad\text{ for all }z\in I_{\varepsilon,\tilde{\delta}}.
	\]
	Hence \eqref{eq_SE_1Dvect_pert_Beps_id} and estimates as before imply for $\varepsilon$ small
	\[
	\check{\lambda}_\varepsilon^1=
	\check{B}_{0,\varepsilon}(\check{\Psi}_\varepsilon^1,\check{\Psi}_\varepsilon^1)
	+\varepsilon p_\varepsilon \int_{I_{\varepsilon,\tilde{\delta}}} (\check{\Psi}_\varepsilon^1,\sum_{\xi\in\N_0^m,|\xi|=1}\partial^\xi D^2W(\vec{\theta}_0)(\vec{\theta}_1)^\xi\check{\Psi}_\varepsilon^1)_{\R^m}\,dz+\Oc(\varepsilon^2).
	\]
	To estimate the second term we will use \eqref{eq_SE_1Dvect_pert_theta1}. For notational convenience we assume w.l.o.g. that $\check{c}_{0,\varepsilon}=1$ in Lemma \ref{th_SE_1Dvect_unpert},~6. If this is not the case then one can simply exchange $\vec{\Psi}_{0,\varepsilon}^1$. We split
	\begin{align}\label{eq_SE_1Dvect_pert_split}
	\check{\Psi}_\varepsilon^1=J_\varepsilon^\frac{1}{2}\vec{\Psi}_\varepsilon^1=\check{a}_\varepsilon\vec{\Psi}_{0,\varepsilon}^1 + \vec{\Psi}_\varepsilon^\perp
	\end{align}
	orthogonally in $L^2(I_{\varepsilon,\tilde{\delta}})^m$, where $\check{a}_\varepsilon:=(\check{\Psi}_\varepsilon^1,\vec{\Psi}_{0,\varepsilon}^1)_{L^2(I_{\varepsilon,\tilde{\delta}})^m}$. Due to the Cauchy-Schwarz-Inequality we have $|\check{a}_\varepsilon|\leq 1$. Moreover, it holds $\check{a}_\varepsilon^2=1-\|\vec{\Psi}_\varepsilon^\perp\|_{L^2(I_{\varepsilon,\tilde{\delta}})^m}^2$. Hence
	\begin{align*}
	&\left|\int_{I_{\varepsilon,\tilde{\delta}}} (\check{\Psi}_\varepsilon^1,\sum_{\xi\in\N_0^m,|\xi|=1}\partial^\xi D^2W(\vec{\theta}_0)(\vec{\theta}_1)^\xi\check{\Psi}_\varepsilon^1)_{\R^m}\,dz\right|\\
	&\leq \check{a}_\varepsilon^2\left|\int_{I_{\varepsilon,\tilde{\delta}}} (\vec{\Psi}_{0,\varepsilon}^1,\sum_{\xi\in\N_0^m,|\xi|=1}\partial^\xi D^2W(\vec{\theta}_0)(\vec{\theta}_1)^\xi\vec{\Psi}_{0,\varepsilon}^1)_{\R^m}\,dz\right|+C\|\vec{\Psi}_\varepsilon^\perp\|_{L^2(I_{\varepsilon,\tilde{\delta}})^m},
	\end{align*}
	where we used $\|\vec{\Psi}_\varepsilon^\perp\|_{L^2(I_{\varepsilon,\tilde{\delta}})^m}\leq 1$. We insert $\vec{\Psi}_{0,\varepsilon}^1=\vec{\Psi}_{0,\varepsilon}^R+\check{\beta}_\varepsilon\vec{\theta}_0'$. The assumption \eqref{eq_SE_1Dvect_pert_theta1} on $\vec{\theta}_1$, the decay for $\vec{\theta}_0'$ from Theorem \ref{th_ODE_vect} and Lemma \ref{th_SE_1Dvect_unpert},~6.~yield for $\varepsilon$ small
	\[
	\left|\int_{I_{\varepsilon,\tilde{\delta}}} (\check{\Psi}_\varepsilon^1,\sum_{\xi\in\N_0^m,|\xi|=1}\partial^\xi D^2W(\vec{\theta}_0)(\vec{\theta}_1)^\xi\check{\Psi}_\varepsilon^1)_{\R^m}\,dz\right|
	\leq C(e^{-c/\varepsilon}+\|\vec{\Psi}_\varepsilon^\perp\|_{L^2(I_{\varepsilon,\tilde{\delta}})^m}).
	\]
	Moreover, integration by parts yields $\check{B}_{0,\varepsilon}(\vec{\Psi}_{0,\varepsilon}^1,\vec{\Psi}_\varepsilon^\perp)=0$. Hence we obtain
	\[
	\check{B}_{0,\varepsilon}(\check{\Psi}_\varepsilon^1,\check{\Psi}_\varepsilon^1)=\check{a}_\varepsilon^2\check{B}_{0,\varepsilon}(\vec{\Psi}_{0,\varepsilon}^1,\vec{\Psi}_{0,\varepsilon}^1)+\check{B}_{0,\varepsilon}(\vec{\Psi}_\varepsilon^\perp,\vec{\Psi}_\varepsilon^\perp)=\check{a}_\varepsilon^2\check{\lambda}_{0,\varepsilon}^1+\check{B}_{0,\varepsilon}(\vec{\Psi}_\varepsilon^\perp,\vec{\Psi}_\varepsilon^\perp).
	\]
	Therefore Lemma \ref{th_SE_1Dvect_unpert},~4.-5.~implies for $\varepsilon$ small
	\[
	C\varepsilon^2\geq\check{\lambda}_\varepsilon^1\geq \check{\nu}_1\|\vec{\Psi}_\varepsilon^\perp\|_{L^2(I_{\varepsilon,\tilde{\delta}})^m}^2 + \Oc(\varepsilon)\|\vec{\Psi}_\varepsilon^\perp\|_{L^2(I_{\varepsilon,\tilde{\delta}})^m} +\Oc(\varepsilon^2)\geq\frac{\check{\nu}_1}{2}\|\vec{\Psi}_\varepsilon^\perp\|_{L^2(I_{\varepsilon,\tilde{\delta}})^m}^2-\tilde{C}\varepsilon^2,
	\]
	where the last estimate follows from Young's inequality. Hence we obtain $\|\vec{\Psi}_\varepsilon^\perp\|_{L^2(I_{\varepsilon,\tilde{\delta}})^m}=\Oc(\varepsilon)$ and $\check{\lambda}_\varepsilon^1=\Oc(\varepsilon^2)$ for $\varepsilon$ small. Moreover, it holds $\check{a}_\varepsilon^2=1+\Oc(\varepsilon^2)$ for $\varepsilon$ small.\qedhere$_{1.}$\end{proof}
	
	\begin{proof}[Ad 2] The estimates above also imply $|\check{B}_{0,\varepsilon}(\vec{\Psi}_\varepsilon^\perp,\vec{\Psi}_\varepsilon^\perp)|=\Oc(\varepsilon^2)$. Therefore the definition of $\check{B}_{0,\varepsilon}$ and $\|\vec{\Psi}_\varepsilon^\perp\|_{L^2(I_{\varepsilon,\tilde{\delta}})^m}=\Oc(\varepsilon)$ yield $\|\frac{d}{dz}\vec{\Psi}_\varepsilon^\perp\|_{L^2(I_{\varepsilon,\tilde{\delta}})^m}=\Oc(\varepsilon)$. Let $\check{c}_\varepsilon:=\textup{sign}\,\check{a}_\varepsilon\in\{\pm 1\}$. For notational simplicity let $\check{c}_\varepsilon=1$. If this is not the case, one can modify $\vec{\Psi}_\varepsilon^1$. We consider $\vec{\Psi}_\varepsilon^R:=\vec{\Psi}_\varepsilon^1-J(0)^{-1/2}\check{\beta}_\varepsilon\vec{\theta}_0'$. Then with the splitting \eqref{eq_SE_1Dvect_pert_split} and $\vec{\Psi}_{0,\varepsilon}^1=\vec{\Psi}_{0,\varepsilon}^R+\check{\beta}_\varepsilon\vec{\theta}_0'$ it follows that
	\[
	\vec{\Psi}_\varepsilon^R=J_\varepsilon^{-\frac{1}{2}}
	\left[(\check{a}_\varepsilon-J_\varepsilon^{\frac{1}{2}}J(0)^{-\frac{1}{2}})\check{\beta}_\varepsilon\vec{\theta}_0'+\check{a}_\varepsilon\vec{\Psi}_{0,\varepsilon}^R+\vec{\Psi}_\varepsilon^\perp\right].
	\]
	As in the scalar case one can show with Corollary \ref{th_SE_1Dprelim2} that $|\check{a}_\varepsilon-J_\varepsilon(z)^{\frac{1}{2}}J(0)^{-\frac{1}{2}}|\leq C\varepsilon(|z|+1)$ for all $z\in I_{\varepsilon,\tilde{\delta}}$ and $\varepsilon$ small. Then the decay properties of $\vec{\theta}_0'$ and the estimates for the $L^2$-norms of $\vec{\Psi}_\varepsilon^\perp$ and $\vec{\Psi}_{0,\varepsilon}^R$ yield $\|\vec{\Psi}_\varepsilon^R\|_{J_\varepsilon}=\Oc(\varepsilon)$. Finally, as in the scalar case one can directly compute and estimate the derivative, cf.~the proof of Theorem \ref{th_SE_1Dscal_pert2},~2.\qedhere$_{2.}$\end{proof}
	
	\begin{proof}[Ad 3] Let $\vec{\Psi}_\varepsilon^2$ be a normalized eigenfunction to $\check{\lambda}_\varepsilon^2$. If $\check{\lambda}_\varepsilon^2\geq\frac{1}{4}\min\{\sigma(D^2W(\vec{u}_\pm))\}$, then we are done. Therefore let $\check{\lambda}_\varepsilon^2\leq\frac{1}{4}\min\{\sigma(D^2W(\vec{u}_\pm))\}$. Then $\vec{\Psi}_\varepsilon^2$ has the decay property in Lemma \ref{th_SE_1Dvect_pert1},~3. Hence computations as before yield 
	\[
	\check{\lambda}_\varepsilon^2=
	\check{B}_{0,\varepsilon}(\check{\Psi}_\varepsilon^2,\check{\Psi}_\varepsilon^2)
	+\varepsilon p_\varepsilon \int_{I_{\varepsilon,\tilde{\delta}}} (\check{\Psi}_\varepsilon^2,\sum_{\xi\in\N_0^m,|\xi|=1}\partial^\xi D^2W(\vec{\theta}_0)(\vec{\theta}_1)^\xi\check{\Psi}_\varepsilon^2)_{\R^m}\,dz+\Oc(\varepsilon^2)
	\]
	and the second term is $\Oc(\varepsilon)$. Therefore analogous computations as in the scalar case yield $\check{\lambda}_\varepsilon^2\geq\frac{\check{\nu}_1}{2}$ provided that $\check{\lambda}_\varepsilon^2\leq\frac{1}{4}\min\{\sigma(D^2W(\vec{u}_\pm))\}$, cf.~the proof of Theorem \ref{th_SE_1Dscal_pert2},~3.\qedhere$_{3.}$\end{proof}
\end{proof}

\subsubsection{Appendix: An Abstract Fredholm Alternative}\label{sec_fredh}
We use an abstract Fredholm Alternative in the setting of a Gelfand-Triple. The result is basically well-known, but hard to find in the literature in the form presented below. Therefore we state the result for the convenience of the reader. The presentation is taken directly from the lecture notes Abels \cite{AbelsPDEI_17}, Section 6.3. First, let us recall the definition of a Gelfand-Triple:
\begin{Remark}[\textbf{Gelfand-Triple}]\upshape\label{th_fred_gelfand}
	Let $V,H$ be $\K$-Hilbert spaces such that there exists $i\in\Lc(V,H)$ injective with $i(V)$ dense in $H$. We write $(.,.)_V, (.,.)_H$ for the scalar product in $V,H$, respectively. Moreover, we identify $V$ with the subspace $i(V)$ of $H$ and write $V\subseteq H$. Let $V^\ast,H^\ast$ be the \textit{anti-dual space} of $V,H$, respectively, i.e.~the space of all conjugate-linear functionals. We write $\langle.,.\rangle_{V^\ast,V},\langle.,.\rangle_{H^\ast,H}$ for the dual product on $V^\ast\times V$, $H^\ast\times H$, respectively. Then due to the Riesz-Representation Theorem we can identify $H\cong H^\ast$ via $y\mapsto(y,.)_H$. Moreover, $i^\ast:H^\ast\rightarrow V^\ast:y^\ast\mapsto y^\ast\circ i$ defines $i^\ast\in\Lc(H^\ast,V^\ast)$ injective and we identify $H^\ast\subseteq V^\ast$. Hence
	\[
	V\overset{i}{\hookrightarrow}_d H\cong H^\ast \overset{i^\ast}{\hookrightarrow} V^\ast\quad\text{ and shortened }\quad V\subseteq H\cong H^\ast\subseteq V^\ast.
	\]
	The triple $(V,H,V^\ast)$ is then called \textit{Gelfand-Triple}.
\end{Remark}

\begin{Theorem}[\textbf{An Abstract Fredholm Alternative}]\label{th_fred}
	Let $(V,H,V^\ast)$ be a Gelfand-Triple as in Remark \ref{th_fred_gelfand} with $\K=\C$, let $H$ be infinite dimensional and $i$ compact. Moreover, let $A\in\Lc(V,V^\ast)$ be such that for some $c_0>0, c_1\in\R$ it holds
	\begin{alignat*}{2}
	\textup{Re}\langle Au,u\rangle_{V^\ast,V}&\geq c_0\|u\|_V^2-c_1\|u\|_H^2&\quad&\text{ for all }u\in V,\\
    \langle Au,v\rangle_{V^\ast,V}&=\overline{\langle Av,u\rangle_{V^\ast,V}}&\quad&\text{ for all }u,v\in V.
	\end{alignat*}
	Then there is a sequence of real numbers $\lambda_1\leq\lambda_2\leq...\leq\lambda_k\overset{k\rightarrow\infty}{\longrightarrow}\infty$ and $e_k\in V, k\in\N$ such that $(e_k)_{k\in\N}$ is an orthonormal base of $H$ and $e_k$ is an eigenvector of $A$ to the eigenvalue $\lambda_k$, i.e.~$(\lambda_k-A)e_k=0$ in $V^\ast$. Moreover:
	\begin{enumerate}
		\item For all $\lambda\in\C\setminus\{\lambda_k:k\in\N\}$ and $f\in H$ there is a unique solution $u\in V$ of
		\begin{align}\label{eq_fred}
		(\lambda-A)u=f\quad\text{ in }V^\ast.
		\end{align}
		The solution $u$ can be represented as 
		\[
		u=\sum_{k=1}^\infty \frac{1}{\lambda-\lambda_k}(f,e_k)_H e_k\quad\text{ in }H.
		\]
		\item Let $\lambda\in\{\lambda_k:k\in\N\}$ and $J_\lambda:=\{j\in\N:\lambda_j=\lambda\}$. Then for $f\in H$ there exists a solution $u\in V$ of \eqref{eq_fred} if and only if $(f,e_j)_H=0$ for all $j\in J_\lambda$. If the latter holds, then all solutions of \eqref{eq_fred} can be represented as
		\[
		\sum_{k\in\N\setminus J_\lambda}\frac{1}{\lambda-\lambda_k}(f,e_k)_H e_k + \textup{span}\{e_j:j\in J_\lambda\}.
		\]
	\end{enumerate}
\end{Theorem}

\begin{proof}
	One applies the Spectral Theorem for Compact Self-Adjoint Operators to $(\mu-A)^{-1}$ for some $\mu\leq -c_1$ viewed as a bounded linear operator in $H$. The existence of the resolvent for these $\mu$ follows from the Lax-Milgram Theorem. For spectral theorems see Alt \cite{AltFA}, Theorem 11.9 and Theorem 12.12. A similar application can be found in Renardy, Rogers \cite{RenardyRogers}, Section 9.3.
\end{proof}

\subsection{Spectral Estimate for (AC) in ND}
\label{sec_SE_ACND}
In this section we show the spectral estimate \eqref{eq_SE1}. This works in a similar way as in the 2D-case in \cite{AbelsMoser}, Section 4, but some computations are more technical. For convenience we often use the same notation. The construction of the approximate solution in Section \ref{sec_asym_ACND} yields the precise structure of $u^A_\varepsilon$, but as in the 2D-case a sightly more general structure is enough for the spectral estimate. In the following we state the assumptions for this section.

Let $\Omega\subset\R^N$ and $\Gamma=(\Gamma_t)_{t\in[0,T]}$ for $T>0$ be as in Section \ref{sec_coord_surface_requ} for $N\geq 2$ with $90$°-contact angle (\eqref{MCF} not needed). Moreover, we consider $\delta>0$ such that Theorem \ref{th_coordND} holds for $2\delta$ instead of $\delta$. In the following we use the same notation for $\vec{n}_{\partial\Sigma}, \vec{n}_{\partial\Gamma}, Y, X_0, X, \mu_0, \mu_1, r, s, \sigma, b$ as in Theorem~\ref{th_coordND}. Furthermore, we use the definitions of some sets and of $\partial_n$, $\nabla_\tau$, $J$ from Remark \ref{th_coordND_rem}. In this section we assume for the height functions $h_1$ and $h_2=h_2(\varepsilon)$ (with a slight abuse of notation) that
\[
h_j\in B([0,T],C^0(\Sigma)\cap C^2(\hat{\Sigma})), j=1,2,\quad \hat{\Sigma}:= Y(\partial\Sigma\times[0,2\mu_0]),\quad C^2(\hat{\Sigma}):=C^2(\overline{\hat{\Sigma}^\circ}).
\] 
Moreover, let $\overline{C}_0>0$ be such that $\|h_j\|_{B([0,T],C^0(\Sigma)\cap C^2(\hat{\Sigma}))}
\leq\overline{C}_0$ for $j=1,2$. Then we define $h_\varepsilon:=h_1+\varepsilon h_2$ for $\varepsilon>0$ small and introduce the scaled variables
\[
\rho_\varepsilon:=\frac{r-\varepsilon h_\varepsilon(s,t)}{\varepsilon}\quad\text{ in }\overline{\Gamma(2\delta)}, \quad H_\varepsilon:=\frac{b}{\varepsilon}\quad\text{ in }\overline{\Gamma^C(2\delta,2\mu_1)}.
\]
Furthermore, let $\hat{u}^C_1:\overline{\R^2_+}\times\partial\Sigma\times[0,T]\rightarrow\R:(\rho,H,\sigma,t)\mapsto\hat{u}^C_1(\rho,H,\sigma,t)$ be in the space
$B([0,T];C^2(\partial\Sigma,H^2_{(0,\gamma)}(\R^2_+)))$ for some $\gamma>0$. Then we define
\[
u^C_1(x,t):=\hat{u}^C_1(\rho_\varepsilon(x,t),H_\varepsilon(x,t),\sigma(x,t),t)\quad\text{ for }(x,t)\in\overline{\Gamma^C(2\delta,2\mu_1)}.
\] 
For $\varepsilon>0$ small let
\begin{align*}
u^A_\varepsilon=
\begin{cases}
\theta_0(\rho_\varepsilon)+\Oc(\varepsilon^2)&\quad\text{ in }\Gamma(\delta,\mu_0),\\
\theta_0(\rho_\varepsilon)+\varepsilon u^C_1+\Oc(\varepsilon^2)&\quad\text{ in }\Gamma^C(\delta,2\mu_0),\\
\pm 1+\Oc(\varepsilon)&\quad\text{ in }Q_T^\pm\setminus\Gamma(\delta),
\end{cases}
\end{align*}
where $\theta_0$ is from Theorem \ref{th_theta_0} and $\Oc(\varepsilon^k)$ are measurable functions bounded by $C\varepsilon^k$. 
\begin{Remark}\upshape\label{th_SE_ACND_rem}
	It is also possible to include an additional term of the form $\varepsilon\theta_1(\rho_\varepsilon)p_\varepsilon(s,t)$ in $u^A_\varepsilon$ on $\Gamma(\delta)$, where $p_\varepsilon\in B([0,T],C^0(\Sigma)\cap C^2(\hat{\Sigma}))$ fulfils a uniform estimate for $\varepsilon$ small and
	\begin{align}\label{eq_SE_ACND_theta1}
	\theta_1\in C_b^0(\R)\quad\text{ with }\quad\int_\R f'''(\theta_0)\theta_1(\theta_0')^2\,d\rho=0.
	\end{align} 
	See Remark \ref{th_SE_ACND_asym_rem},~2.~below.
\end{Remark}

We obtain the following spectral estimate:
\begin{Theorem}[\textbf{Spectral Estimate for (AC) in ND}]\label{th_SE_ACND}
	There are $\varepsilon_0,C,c_0>0$ independent of the $h_j$ for fixed $\overline{C}_0$ such that for all $\varepsilon\in(0,\varepsilon_0], t\in[0,T]$ and $\psi\in H^1(\Omega)$ it holds
	\[
	\int_\Omega|\nabla\psi|^2+\frac{1}{\varepsilon^2}f''(u^A_\varepsilon(.,t))\psi^2\,dx\geq -C\|\psi\|_{L^2(\Omega)}^2+\|\nabla\psi\|_{L^2(\Omega\setminus\Gamma_t(\delta))}^2+c_0\|\nabla_\tau\psi\|_{L^2(\Gamma_t(\delta))}^2.
	\]
\end{Theorem}

As in the 2D-case in \cite{AbelsMoser} we separately prove a spectral estimate on 
\[
\Omega^C_t:=\Gamma^C_t(\delta,2\mu_0)=X((-\delta,\delta)\times \hat{\Sigma}^\circ\times\{t\})\quad\text{ for }t\in[0,T].
\] 
In the following we need several properties of Sobolev spaces on the appearing sets.
\begin{Remark}\phantomsection{\label{th_SE_ACND_sob_rem}}\upshape 
	The results of Sections \ref{sec_Leb}-\ref{sec_SobMfd} can be applied for sets such as $U=\Sigma^\circ$, $\hat{\Sigma}^\circ$, $(-\delta,\delta)\times\Sigma^\circ$, $(-\delta,\delta)\times\hat{\Sigma}^\circ$, $\Gamma(\delta)$, $\Gamma_t(\delta)$ and $\Omega_t^C$ for all $t\in[0,T]$. These sets can all be viewed as an open subset of a smooth compact Riemannian submanifold of some $\R^n$ with the Euclidean metric and they have Lipschitz boundary. For $\Sigma^\circ$, $(-\delta,\delta)\times\Sigma^\circ$ etc.~one can simply consider local charts to prove this. For $\Gamma(\delta)$, $\Gamma_t(\delta)$, $\Omega_t^C$ and similar sets one can show this with the extension of $\overline{X}$ to a diffeomorphism due to Theorem \ref{th_coordND} and Remark \ref{th_SobMfd_LipRem}. In particular
	\begin{enumerate}
		\item We can  transform integrals via $\overline{X}$ and $X(.,t)$ for $t\in[0,T]$ with the usual transformation formula due to Theorem \ref{th_Leb_trafo_mfd} with the factor $J$ and $J_t$, $t\in[0,T]$ from Remark \ref{th_coordND_rem}, 3.
		\item Density and trace theorems for Sobolev spaces on the above sets hold due to Theorem \ref{th_SobDom_LipThm}, Theorem \ref{th_SobMfd_LipThm} and we can use integration by parts on $\Omega_t^C$ because of Theorem \ref{th_SobDom_LipThm},~4.
		\item $H^k(\Gamma_t(\delta))\cong H^k((-\delta,\delta)\times\Sigma^\circ)$ and $H^k(\Omega_t^C)\cong H^k((-\delta,\delta)\times\hat{\Sigma}^\circ)$ etc.~for $k\in\N_0$ via $X(.,t)$ for all $t\in[0,T]$ due to Theorem \ref{th_SobMfd_trafo}. Therefore Corollary~\ref{th_coordND_nabla_tau_n},~1.-2.~carries over to $H^1$-functions. In particular the gradients are pointwise a.e.~uniformly equivalent. Moreover, note that for $k=0$, i.e.~$L^2$-spaces, the operator norms of the transformations can be estimated uniformly in $t\in[0,T]$ because of Theorem \ref{th_Leb_trafo_mfd} and Remark \ref{th_coordND_rem},~3. Hence this also holds for the $L^2$-norms of the gradients and due to Lemma \ref{th_SobMfd_def_lemma},~3.~also for $k=1$. 
		\item Lemma \ref{th_SobMfd_prod_set} yields $L^2((-\delta,\delta)\times\Sigma)\cong L^2(-\delta,\delta,L^2(\Sigma))\cong L^2(\Sigma,L^2(-\delta,\delta))$ as well as
		\begin{align*}
		H^1((-\delta,\delta)\times\Sigma^\circ)&\cong H^1(-\delta,\delta,L^2(\Sigma^\circ))\cap L^2(-\delta,\delta,H^1(\Sigma^\circ))\\
		&\cong H^1(\Sigma^\circ,L^2(-\delta,\delta))\cap L^2(\Sigma^\circ,H^1(-\delta,\delta))
		\end{align*}
		and the derivatives $\nabla_\Sigma:=\nabla_{\Sigma^\circ}$ and $\partial_r$ are compatible. The analogous assertion holds for $\hat{\Sigma}$ instead of $\Sigma$ and similar sets. 
	\end{enumerate}
\end{Remark}

The spectral estimate on $\Omega_t^C$ is as follows:
\begin{Theorem}\label{th_SE_ACND_cp} 
	There are $\tilde{\varepsilon}_0, C, \tilde{c}_0>0$ independent of the $h_j$ for fixed $\overline{C}_0$ such that for all $\varepsilon\in(0,\tilde{\varepsilon}_0]$, $t\in[0,T]$ and $\psi\in H^1(\Omega^C_t)$ with $\psi|_{X(.,s,t)}=0$ for a.e.~$s\in Y(\partial\Sigma\times[\frac{3}{2}\mu_0,2\mu_0])$:
	\[
	\int_{\Omega^C_t}|\nabla\psi|^2+\frac{1}{\varepsilon^2}f''(u^A_\varepsilon(.,t))\psi^2\,dx\geq -C\|\psi\|_{L^2(\Omega_t^C)}^2+\tilde{c}_0\|\nabla_\tau\psi\|_{L^2(\Omega_t^C)}^2.
	\]
\end{Theorem}
The additional assumption on $\psi$ is not needed but simplifies the proof, cf.~Remark \ref{th_SE_ACND_asym_rem},~3.~below. The latter is enough to show Theorem \ref{th_SE_ACND}:

\begin{proof}[Proof of Theorem \ref{th_SE_ACND}] 
	For $\varepsilon_0>0$ small and all $\varepsilon\in(0,\varepsilon_0]$ it holds $f''(u^A_\varepsilon)\geq 0$ on $Q_T^\pm\setminus\Gamma(\delta)$. Therefore it is enough to show the estimate in Theorem \ref{th_SE_ACND} for $\Gamma_t(\delta)$ instead of $\Omega$. On $\Gamma_t(\delta)$ we reduce to further subsets. 
	
	Due to Theorem \ref{th_SE_ACND_cp} we have an estimate for $\Omega_t^C=\Gamma^C_t(\delta,2\mu_0)$ instead of $\Omega$. Moreover, the estimate holds for $\Gamma_t(\delta,\mu_0)$ instead of $\Omega$ with $c_0=1$: 
	there our curvilinear coordinate system is the usual tubular neighbourhood coordinate system, cf.~Theorem \ref{th_coordND}. Let $\psi\in H^1(\Gamma_t(\delta,\mu_0))$.
	Then $|\nabla\psi|^2=|\partial_n\psi|^2+|\nabla_\tau\psi|^2$ on $\Gamma_t(\delta,\mu_0)$ due to Corollary \ref{th_coordND_nabla_tau_n},~2.~and Remark \ref{th_SE_ACND_sob_rem},~3. Due to Taylor's Theorem we can replace $u^A_\varepsilon(.,t)$ by $\theta_0(\rho_\varepsilon(.,t))$ in the integral. Therefore an integral transformation with $X(.,t)$, cf.~Remark \ref{th_SE_ACND_sob_rem},~1., and the Fubini Theorem yield for $\tilde{\psi}_t:=\psi|_{X(.,t)}$
	\begin{align*}
	\int_{\Gamma_t(\delta,\mu_0)}|\nabla\psi|^2+\frac{1}{\varepsilon^2}f''(u^A_\varepsilon(.,t))\psi^2\,dx
	\geq -\tilde{C}\|\psi\|_{L^2(\Gamma_t(\delta,\mu_0))}^2
	+\|\nabla_\tau\psi\|_{L^2(\Gamma_t(\delta,\mu_0))}^2\\
	+\int_{\Sigma\setminus [Y(\partial\Sigma\times[0,\mu_0])]}\int_{-\delta}^\delta
	\left[|\partial_r\tilde{\psi}_t|^2+\frac{1}{\varepsilon^2}f''(\theta_0(\rho_\varepsilon|_{\overline{X}(.,t)}))\tilde{\psi}_t^2\right]J_t\,dr\,d\Hc^{N-1}(s).
	\end{align*}
	Due to Remark \ref{th_SE_ACND_sob_rem},~4.~we can estimate the inner integral in the second line in the analogous way as in \cite{ALiu}, proof of Theorem 2.13. Note that the results of Section \ref{sec_SE_1Dprelim} are  applicable with a constant $h_\varepsilon$ there, in particular we use the transformations in Lemma \ref{th_SE_1Dtrafo_remainder},~1.~and Theorem \ref{th_SE_1Dscal_pert2},~1. See also \cite{MoserDiss}, proof of Theorem 6.16 for the calculation in the 2D-case. This yields the estimate for $\Gamma_t(\delta,\mu_0)$ instead of $\Omega$ with $c_0=1$.
	
	Finally, we combine the above estimates with a suitable partition of unity for 
	\begin{align}\label{eq_SE_ACND_proof}
	\Gamma_t(\delta)\subseteq\overline{\Gamma_t(\delta,\mu_0)}\cup\overline{\Gamma_t^C(\delta,2\mu_0)}.
	\end{align}
	The arguments are similar to the 2D-case, cf.~the proof of Theorem 4.1 in \cite{AbelsMoser}. Here one uses $b$ from Theorem \ref{th_coordND},~4.~to construct the cut-off functions.
\end{proof}

\subsubsection{Outline for the Proof of the Spectral Estimate close to the Contact Points}\label{sec_SE_ACND_outline}
For the proof of Theorem \ref{th_SE_ACND_cp} we can replace $\frac{1}{\varepsilon^2}f''(u^A_\varepsilon(.,t))$ by
\[
\frac{1}{\varepsilon^2}f''(\theta_0|_{\rho_\varepsilon(.,t)})+\frac{1}{\varepsilon}f'''(\theta_0|_{\rho_\varepsilon(.,t)})u^C_1|_{(.,t)}
\] 
due to a Taylor expansion. We construct an approximation $\phi^A_\varepsilon(.,t)$ to the first eigenfunction of 
\[
\Lc_{\varepsilon,t}^C:=-\Delta
+\frac{1}{\varepsilon^2}f''(\theta_0|_{\rho_\varepsilon(.,t)})
+\frac{1}{\varepsilon}f'''(\theta_0|_{\rho_\varepsilon(.,t)})u^C_1|_{(.,t)}\quad\text{ on }\Omega_t^C
\] 
together with a homogeneous Neumann boundary condition. Then we split
\begin{align}\label{eq_SE_ACND_H1tilde_Omega}
\tilde{H}^1(\Omega^C_t):=
\left\{\psi\in H^1(\Omega^C_t): \psi|_{X(.,s,t)}=0\text{ for a.e.~}s\in Y(\partial\Sigma\times[\frac{3}{2}\mu_0,2\mu_0])\right\}
\end{align}
with respect to the subspace of tangential alterations of $\phi^A_\varepsilon(.,t)$. Therefore we set up the ansatz
\begin{alignat*}{2}
\phi^A_\varepsilon(.,t)&:=\frac{1}{\sqrt{\varepsilon}}[v^I_\varepsilon(.,t)+\varepsilon v^C_\varepsilon(.,t)]&\quad&\text{ on }\Omega^C_t,\\
v^I_\varepsilon(.,t)&:=\hat{v}^I(\rho_\varepsilon(.,t),s(.,t),t):=\theta_0'|_{\rho_\varepsilon(.,t)}\,q(s(.,t),t)&\quad&\text{ on }\Omega^C_t,\\ v^C_\varepsilon(.,t)&:=\hat{v}^C(\rho_\varepsilon(.,t),H_\varepsilon(.,t),\sigma(.,t),t)&\quad&\text{ on }\Omega^C_t,
\end{alignat*}
where $q:\hat{\Sigma}\times[0,T]\rightarrow\R$ and $\hat{v}^C:\overline{\R^2_+}\times\partial\Sigma\times[0,T]\rightarrow\R$. The $\frac{1}{\sqrt{\varepsilon}}$-factor normalizes in a suitable way, see Lemma \ref{th_SE_ACND_split_L2} below. 

In Subsection \ref{sec_SE_ACND_asym} we expand $\Lc_{\varepsilon,t}^C\phi^A_\varepsilon(.,t)$ and $\partial_{N_{\partial\Omega}}\phi^A_\varepsilon(.,t)$ similarly as in Section \ref{sec_asym_ACND} and choose $q$ and $\hat{v}^C$ appropriately. The $q$-term will be used to enforce the compatibility condition for the equations for $\hat{v}^C$. In Subsection \ref{sec_SE_ACND_splitting} we characterize the $L^2$-orthogonal splitting of $\tilde{H}^1(\Omega^C_t)$ with respect to the subspace
\begin{align}\label{eq_SE_ACND_V_def}
V_{\varepsilon,t}
&:=\left\{\phi=a(s(.,t))\phi^A_\varepsilon(.,t):
a\in\tilde{H}^1(\hat{\Sigma}^\circ)\right\},\\
\tilde{H}^1(\hat{\Sigma}^\circ)&:=
\left\{a\in H^1(\hat{\Sigma}^\circ) : a|_{Y(.,b)}=0\text{ for a.e.~}b\in[\frac{3}{2}\mu_0,2\mu_0]\right\}.\label{eq_SE_ACND_H1tilde_interval}
\end{align}
Finally, in Subsection \ref{sec_SE_ACND_BLF} we analyze the bilinear form $B_{\varepsilon,t}^C$ corresponding to $\Lc_{\varepsilon,t}^C$ on $V_{\varepsilon,t}\times V_{\varepsilon,t}$, $V_{\varepsilon,t}^\perp\times V_{\varepsilon,t}^\perp$ and $V_{\varepsilon,t}\times V_{\varepsilon,t}^\perp$. Here for $\phi,\psi\in H^1(\Omega_t^C)$ we set
\begin{align}\label{eq_SE_ACND_Bepst}
B_{\varepsilon,t}^C(\phi,\psi):=\int_{\Omega_t^C}\nabla\phi\cdot\nabla\psi+\left[\frac{1}{\varepsilon^2}f''(\theta_0|_{\rho_\varepsilon(.,t)})+\frac{1}{\varepsilon}f'''(\theta_0|_{\rho_\varepsilon(.,t)})u^C_1(.,t)\right]\phi\psi\,dx.
\end{align}
\subsubsection{Asymptotic Expansion for the Approximate Eigenfunction}\label{sec_SE_ACND_asym}
\begin{proof}[Asymptotic Expansion of $\sqrt{\varepsilon}\Lc_{\varepsilon,t}^C\phi^A_\varepsilon(.,t)$.] 
	First, we expand $\Delta v^I_\varepsilon$ as in the inner expansion in Section \ref{sec_asym_ACND_in}. The lowest order $\Oc(\frac{1}{\varepsilon^2})$ is given by $\frac{1}{\varepsilon^2}|\nabla r|^2|_{\overline{X}_0(s,t)}\theta_0'''(\rho)q(s,t)=\frac{1}{\varepsilon^2}\theta_0'''(\rho)q(s,t)$. In $\sqrt{\varepsilon}\Lc_{\varepsilon,t}^C\phi^A_\varepsilon(.,t)$ this cancels with $\frac{1}{\varepsilon^2}f''(\theta_0(\rho))\theta_0'(\rho)q(s,t)$. For the $\frac{1}{\varepsilon}$-order of $\Delta v^I_\varepsilon$ we get 
	\begin{align*}
	&\frac{1}{\varepsilon}\theta_0'''(\rho)q(s,t)\left[(\rho+h_1)\partial_r(|\nabla r|^2\circ\overline{X})|_{(0,s,t)}-2(D_xs\nabla r)^\top|_{\overline{X}_0(s,t)}\nabla_\Sigma h_1\right]\\
	&+\frac{1}{\varepsilon}\theta_0''(\rho)\left[\Delta r|_{\overline{X}_0(s,t)}q(s,t)+2(D_xs\nabla r)^\top|_{\overline{X}_0(s,t)}\nabla_\Sigma q(s,t)\right]=\frac{1}{\varepsilon}\theta_0''(\rho)\Delta r|_{\overline{X}_0(s,t)}q(s,t).
	\end{align*} 
	We leave $\frac{1}{\varepsilon}\Delta r|_{\overline{X}_0(s,t)}q(s,t)\theta_0''(\rho)$ as a remainder.\phantom{\qedhere}
	
	For $\varepsilon\Delta v^C_\varepsilon$ we apply the expansion in Section \ref{sec_asym_ACND_cp_bulk}, but without using a Taylor expansion for the $h_j$ since we just need the lowest order and we intended to minimize the regularity assumptions. More precisely, the $(x,t)$-terms in the formula for $\Delta v^C_\varepsilon$ in Lemma \ref{th_asym_ACND_cp_trafo} are expanded only with \eqref{eq_asym_ACND_cp_taylor3}. At the lowest order  $\Oc(\frac{1}{\varepsilon})$ we obtain $\frac{1}{\varepsilon}\Delta^{\sigma,t}\hat{v}^C$, where $\Delta^{\sigma,t}:=\partial_\rho^2+|\nabla b|^2|_{\overline{X}_0(\sigma,t)}\partial_H^2$ for $(\sigma,t)\in\partial\Sigma\times[0,T]$. Moreover, the $f$-parts yield $\frac{1}{\varepsilon}f''(\theta_0(\rho))\hat{v}^C+\frac{1}{\varepsilon}f'''(\theta_0(\rho))\hat{u}^C_1\hat{v}^I$. To get an equation for $\hat{v}^C$ in $(\rho,H,\sigma,t)$ we apply a Taylor expansion for $q(Y(\sigma,.),t)|_{[0,2\mu_0]}$: 
	\[
	q(Y(\sigma,\varepsilon H),t)=q(\sigma,t)+\Oc(\varepsilon H)\quad\text{ for }(\sigma,\varepsilon H)\in\partial\Sigma\times[0,2\mu_0].
	\]
	Therefore we require
	\begin{align}\label{eq_SE_ACND_asym_bulk}
	\left[-\Delta^{\sigma,t}+f''(\theta_0(\rho))\right]\hat{v}^C=
	-f'''(\theta_0)\theta_0'|_\rho\hat{u}^C_1|_{(\rho,H,\sigma,t)}q|_{(\sigma,t)}\quad\text{ in }\overline{\R^2_+}\times\partial\Sigma\times[0,T].
	\end{align}\end{proof}

\begin{proof}[Asymptotic Expansion of $\sqrt{\varepsilon}\partial_{N_{\partial\Omega}}\phi^A_\varepsilon(.,t)$.] We proceed as in Section \ref{sec_asym_ACND_cp_neum}. 
	Note that in $\overline{\Omega^C_t}$
	\begin{align}\begin{split}\label{eq_SE_ACND_nabla_vI_vC}
	\nabla v^I_\varepsilon
	&=q|_{(s,t)}\theta_0''(\rho_\varepsilon)
	\left[\frac{\nabla r}{\varepsilon}-(D_xs)^\top \nabla_\Sigma h_\varepsilon|_{(s,t)}\right]
	+\theta_0'(\rho_\varepsilon)(D_xs)^\top\nabla_\Sigma q|_{(s,t)},\\
	\nabla v^C_\varepsilon  &=\partial_\rho\hat{v}^C
	\left[\frac{\nabla r}{\varepsilon}-(D_xs)^\top\nabla_\Sigma h_\varepsilon|_{(s,t)}\right]
	+\frac{\nabla b}{\varepsilon}\partial_H\hat{v}^C
	+(D_x\sigma)^\top\nabla_{\partial\Sigma}\hat{v}^C,\end{split}
	\end{align}
	where the $\hat{v}^C$-terms are evaluated at $(\rho_\varepsilon,H_\varepsilon,\sigma,t)$.
	The lowest order $\Oc(\frac{1}{\varepsilon})$ in $\sqrt{\varepsilon}\partial_{N_{\partial\Omega}}\phi^A_\varepsilon(.,t)$ is $\frac{1}{\varepsilon}(N_{\partial\Omega}\cdot\nabla r)|_{\overline{X}_0(\sigma,t)}\theta_0''(\rho)q|_{(\sigma,t)}=0$ due to the $90$°-contact angle condition. At $\Oc(1)$ we get
	\begin{align*}
	q|_{(\sigma,t)}\theta_0''(\rho)\left[(\rho+h_1|_{(\sigma,t)})\partial_r((N_{\partial\Omega}\cdot\nabla r)\circ\overline{X})|_{(0,\sigma,t)}-(D_xsN_{\partial\Omega})^\top|_{\overline{X}_0(\sigma,t)}\nabla_\Sigma h_1|_{(\sigma,t)}\right]\\
	+(D_xsN_{\partial\Omega})^\top|_{\overline{X}_0(\sigma,t)} \nabla_\Sigma q|_{(\sigma,t)}\theta_0'(\rho)+0\cdot\partial_\rho\hat{v}^C|_{H=0}+(N_{\partial\Omega}\cdot\nabla b)|_{\overline{X}_0(\sigma,t)}\partial_H\hat{v}^C|_{H=0}.
	\end{align*}
	This vanishes if and only if\phantom{\qedhere}
	\begin{align*}
	&(N_{\partial\Omega}\cdot\nabla b)|_{\overline{X}_0(\sigma,t)}\partial_H\hat{v}^C|_{H=0}=-(D_xsN_{\partial\Omega})^\top|_{\overline{X}_0(\sigma,t)}\nabla_\Sigma q|_{(\sigma,t)}\theta_0'(\rho)\\
	&+q|_{(\sigma,t)}\theta_0''(\rho)
	\left[(D_xsN_{\partial\Omega})^\top|_{\overline{X}_0(\sigma,t)}\nabla_\Sigma h_1|_{(\sigma,t)}-(\rho+h_1|_{(\sigma,t)})\partial_r((N_{\partial\Omega}\cdot\nabla r)\circ\overline{X})|_{(0,\sigma,t)}\right].
	\end{align*}
	Here note that for the desired regularity of $\hat{v}^C$ the term with $\nabla_\Sigma h_1$ is not good enough. One option is to require additionally $\nabla_\Sigma h_1|_{\partial\Sigma\times[0,T]}\in B([0,T],C^2(\partial\Sigma))$. However, we can also leave the term as a remainder similar to the one in $\sqrt{\varepsilon}\Lc_{\varepsilon,t}^C\phi^A_\varepsilon(.,t)$. Hence due to \eqref{eq_SE_ACND_asym_bulk} we require
	\begin{alignat}{2}\label{eq_SE_ACND_asym_vbar1}
	[-\Delta+f''(\theta_0)]\overline{v}^C&=-f'''(\theta_0)\theta_0'\overline{u}^C_1q|_{(\sigma,t)}&\quad&\text{ in }\overline{\R^2_+}\times\partial\Sigma\times[0,T],\\
	-\partial_H\overline{v}^C|_{H=0}&=(|\nabla b|/N_{\partial\Omega}\cdot\nabla b)|_{\overline{X}_0(\sigma,t)}g^{C}&\quad&\text{ in }\R\times\partial\Sigma\times[0,T],\label{eq_SE_ACND_asym_vbar2}
	\end{alignat}
	where $\overline{v}^C,\overline{u}^C_1:\overline{\R^2_+}\times\partial\Sigma\times[0,T]\rightarrow\R$ correspond to $\hat{v}^C$ and $\hat{u}^C_1$ in the same way as in \eqref{eq_asym_ACND_cp_ubar} in Section \ref{sec_asym_ACND_cp_neum_0}, respectively, and we define $g^C(\rho,\sigma,t)$ for all $(\rho,\sigma,t)\in\R\times\partial\Sigma\times[0,T]$ as
	\[
	-q|_{(\sigma,t)}\theta_0''(\rho)(\rho+h_1|_{(\sigma,t)})\partial_r((N_{\partial\Omega}\cdot\nabla r)\circ\overline{X})|_{(0,\sigma,t)}
	-(D_xsN_{\partial\Omega})^\top|_{\overline{X}_0(\sigma,t)}\nabla_\Sigma q|_{(\sigma,t)}\theta_0'(\rho).
	\]
	The right hand sides in \eqref{eq_SE_ACND_asym_vbar1}-\eqref{eq_SE_ACND_asym_vbar2} are contained in $B([0,T];C^2(\partial\Sigma,H^2_{(\beta,\gamma)}(\R^2_+)\times H^{5/2}_{(\beta)}(\R)))$ for some $\beta,\gamma>0$ provided that $(q,\nabla_\Sigma q)|_{\partial\Sigma\times[0,T]}\in B([0,T],C^2(\partial\Sigma))^{1+N}$. We require $q=1$ on $\partial\Sigma\times[0,T]$. Then the compatibility condition \eqref{eq_hp_comp} associated to \eqref{eq_SE_ACND_asym_vbar1}-\eqref{eq_SE_ACND_asym_vbar2} is equivalent to 
	\begin{align}\label{eq_SE_ACND_nabla_q}
	-(D_xsN_{\partial\Omega})^\top|_{\overline{X}_0(\sigma,t)}\nabla_\Sigma q|_{(\sigma,t)}&=\hat{g}^C|_{(\sigma,t)},\\
	\hat{g}^C|_{(\sigma,t)}:=\frac{1}{\|\theta_0'\|_{L^2(\R)}^2}\left[\partial_r((N_{\partial\Omega}\cdot\nabla r)\circ\overline{X})|_{(0,\sigma,t)}\right.&\int_\R\rho\theta_0'(\rho)\theta_0''(\rho)\,d\rho \notag\\
	+\left.\frac{N_{\partial\Omega}\cdot\nabla b}{|\nabla b|}\right|_{\overline{X}_0(\sigma,t)}&\int_{\R^2_+}\left.f'''(\theta_0(\rho))\theta_0'(\rho)^2\overline{u}^C_1|_{(\rho,H,\sigma,t)}\,d(\rho,H)\right].\notag
	\end{align}
	Note that $\hat{g}^C\in B([0,T],C^2(\partial\Sigma))$.
	
	To construct $q$ we consider $\tilde{q}:\partial\Sigma\times[0,2\mu_0]\times[0,T]\rightarrow\R:(\sigma,b,t)\mapsto q(Y(\sigma,b),t)$. Then due to $Y(.,0)=\textup{id}_{\partial\Sigma}$ and \eqref{eq_coordND_DbY} it holds
	\[
	\nabla_\Sigma q|_{(\sigma,t)}\cdot(\vec{v}-\vec{n}_{\partial\Sigma}|_\sigma w)=\nabla_\Sigma q|_{(\sigma,t)}\cdot d_{(\sigma,0)}Y(\vec{v},w)
	=\nabla_{\partial\Sigma}[\tilde{q}(.,0,t)]|_\sigma\cdot\vec{v}
	+w\,\partial_b\tilde{q}(\sigma,0,t)
	\]
	for all $(\vec{v},w)\in T_\sigma\partial\Sigma\times\R$ and $(\sigma,t)\in\partial\Sigma\times[0,T]$. If $q=1$ on $\partial\Sigma\times[0,T]$, then we obtain $\nabla_{\partial\Sigma}[\tilde{q}(.,0,t)]|_\sigma=0$ and therefore 
	\begin{align}\label{eq_SE_ACND_nabla_q2}
	\nabla_\Sigma q|_{(\sigma,t)}=-\vec{n}_{\partial\Sigma}|_\sigma\partial_b\tilde{q}(\sigma,0,t)\quad\text{ for all }(\sigma,t)\in\partial\Sigma\times[0,T].
	\end{align}
	Note that $|(D_xs N_{\partial\Omega})|_{\overline{X}_0(\sigma,t)}\cdot\vec{n}_{\partial\Sigma}|_\sigma|\geq c>0$ for all $(\sigma,t)\in\partial\Sigma\times[0,T]$ due to Theorem~ \ref{th_coordND},~4. Hence with a simple ansatz and cutoff we can construct $q\in B([0,T],C^2(\hat{\Sigma}))$ such that $q=1$ on $(\partial\Sigma\cup Y(\partial\Sigma,[\mu_0,2\mu_0]))\times[0,T]$ and $c\leq q\leq C$ for some $c,C>0$ and such that \eqref{eq_SE_ACND_nabla_q} holds. Together with \eqref{eq_SE_ACND_nabla_q2} the latter yields $\nabla_\Sigma q|_{\partial\Sigma\times[0,T]}\in B([0,T],C^2(\partial\Sigma))$.  
	Therefore Remark~\ref{th_hp_exp_sol_rem} yields a unique solution of \eqref{eq_SE_ACND_asym_vbar1}-\eqref{eq_SE_ACND_asym_vbar2} such that for some $\beta,\gamma>0$
	\[\overline{v}^C\in B([0,T];C^2(\partial\Sigma,H^4_{(\beta,\gamma)}(\R^2_+)))\hookrightarrow B([0,T];C^2(\partial\Sigma,C^2_{(\beta,\gamma)}(\overline{\R^2_+}))).
	\]\end{proof}

\begin{Remark}\upshape\phantomsection{\label{th_SE_ACND_asym_rem}}
	\begin{enumerate}
		\item Consider the situation of Section \ref{sec_asym_ACND}. Then $h_1$ is smooth and the $\nabla_\Sigma h_1$-term can be included above. Moreover, $\overline{u}^C_1$ is smooth and solves \eqref{eq_asym_ACND_cp_uC1}-\eqref{eq_asym_ACND_cp_uC2}. Hence we can use $\hat{v}^C:=\partial_\rho\overline{u}^C_1$ in this situation, cf.~Remark 4.3, 2.~in \cite{AbelsMoser} for the 2D-case.
		\item In the case of additional terms in $u^A_\varepsilon$ as in Remark \ref{th_SE_ACND_rem} one can proceed analogously as in the 2D-case, cf.~\cite{MoserDiss}, Remark 6.18, 2.~or Remark 4.1, 2.~in \cite{AbelsMoser}.
		\item The behaviour of $\phi^A_\varepsilon(x,t)$ for $x\in\Omega^C_t$ with $b(x,t)\in[\frac{7}{4}\mu_0,2\mu_0]$ is not important because we only consider $\psi\in\tilde{H}^1(\Omega^C_t)$ in Theorem \ref{th_SE_ACND_cp}, where $\tilde{H}^1(\Omega^C_t)$ was defined in \eqref{eq_SE_ACND_H1tilde_Omega}.
	\end{enumerate}
\end{Remark}
\begin{Lemma}\label{th_SE_ACND_phi_A}
	The function $\phi^A_\varepsilon(.,t)$ is $C^2(\overline{\Omega_t^C})$ and satisfies uniformly in $t\in[0,T]$:
	\begin{alignat*}{2}
	\left|\sqrt{\varepsilon}\Lc_{\varepsilon,t}^C\phi^A_\varepsilon(.,t)+\frac{1}{\varepsilon}\Delta r|_{\overline{X}_0(s(.,t),t)}q|_{(s(.,t),t)}\theta_0''|_{\rho_\varepsilon(.,t)}\right|&\leq Ce^{-c|\rho_\varepsilon(.,t)|}&\quad&\text{ in }\Omega_t^C,\\
	\left|\sqrt{\varepsilon}\partial_{N_{\partial\Omega}}\phi^A_\varepsilon|_{(.,t)}+D_xsN_{\partial\Omega}|_{\overline{X}_0(\sigma(.,t),t)}\cdot\nabla_\Sigma h_1|_{(\sigma(.,t),t)}\theta_0''\right|&\leq C\varepsilon e^{-c|\rho_\varepsilon(.,t)|}&\quad&\text{ on }\partial\Omega_t^C\cap\partial\Omega,\\
	\left|\sqrt{\varepsilon}N_{\partial\Omega^C_t}\cdot\nabla\phi^A_\varepsilon|_{(.,t)}\right|&\leq Ce^{-c/\varepsilon}&\quad&\text{ on }\partial\Omega_t^C\setminus\Gamma_t(\delta).
	\end{alignat*}
\end{Lemma}
\begin{proof}
	The regularity for $\phi^A_\varepsilon$ is obtained from the construction. The estimates follow from rigorous estimates for the remainder terms in the expansions above and the decay properties of the involved terms, cf.~the proof of Lemma 4.4 in \cite{AbelsMoser} in the 2D-case.
\end{proof}

\subsubsection{The Splitting}\label{sec_SE_ACND_splitting} 
Similar as in the 2D-case we show a characterization for the splitting of $\tilde{H}^1(\Omega^C_t)$.
\begin{Lemma}\phantomsection{\label{th_SE_ACND_split_L2}}
	Let $\tilde{H}^1(\Omega^C_t)$, $V_{\varepsilon,t}$ and $\tilde{H}^1(\hat{\Sigma}^\circ)$ be as in \eqref{eq_SE_ACND_H1tilde_Omega}-\eqref{eq_SE_ACND_H1tilde_interval}. Then
	\begin{enumerate}
		\item $V_{\varepsilon,t}$ is a subspace of $\tilde{H}^1(\Omega_t^C)$ and for $\varepsilon_0>0$ small there are $c_1,C_1>0$ such that 
		\[
		c_1\|a\|_{L^2(\hat{\Sigma})}\leq\|\psi\|_{L^2(\Omega_t^C)}\leq C_1\|a\|_{L^2(\hat{\Sigma})}
		\]
		for all $\psi=a(s(.,t))\phi^A_\varepsilon(.,t)\in V_{\varepsilon,t}$ and $\varepsilon\in(0,\varepsilon_0]$, $t\in[0,T]$.
		\item Let $V_{\varepsilon,t}^\perp$ be the $L^2$-orthogonal complement of $V_{\varepsilon,t}$ in $\tilde{H}^1(\Omega_t^C)$. Then for $\psi\in\tilde{H}^1(\Omega_t^C)$:
		\[
		\psi\in V_{\varepsilon,t}^\perp\quad\Leftrightarrow\quad\int_{-\delta}^\delta(\phi^A_\varepsilon(.,t)\psi)|_{X(r,s,t)}J_t(r,s)\,dr=0\quad\text{ for a.e.~}s\in\hat{\Sigma}.
		\]
		Moreover, $\tilde{H}^1(\Omega_t^C)=V_{\varepsilon,t}\oplus  V_{\varepsilon,t}^\perp$ for all $\varepsilon\in(0,\varepsilon_0]$ and $\varepsilon_0>0$ small.
	\end{enumerate}
\end{Lemma}
\begin{proof}[Proof. Ad 1] 
	It holds $\phi^A_\varepsilon(.,t)\in C^2(\overline{\Omega_t^C})$ for fixed $t\in[0,T]$ due to Lemma \ref{th_SE_ACND_phi_A}. 
	Moreover, $a(s(.,t))\in H^1(\Omega_t^C)$ for all $a\in H^1(\hat{\Sigma}^\circ)$ because of Remark \ref{th_SE_ACND_sob_rem},~3.-4. Therefore $V_{\varepsilon,t}$ is a subspace of $\tilde{H}^1(\Omega_t^C)$. Now let $\psi=a(s(.,t))\phi^A_\varepsilon(.,t)\in V_{\varepsilon,t}$ be arbitrary. Then the transformation rule, cf.~Remark \ref{th_SE_ACND_sob_rem},~1., and the Fubini Theorem yield
	\begin{align}\label{eq_SE_ACND_split_L2}
	\|\psi\|_{L^2(\Omega_t^C)}^2=\int_{\hat{\Sigma}}a(s)^2\int_{-\delta}^\delta(\phi^A_\varepsilon|_{\overline{X}(r,s,t)})^2 J_t(r,s)\,dr\,d\Hc^{N-1}(s).
	\end{align}
	Since there are $c,C>0$ with $c\leq J, q\leq C$, we can transform and estimate the inner integral with Lemma \ref{th_SE_1Dtrafo_remainder} as in the 2D-case, cf.~the proof of Lemma 4.6 ,~1.~in \cite{AbelsMoser}.\qedhere$_{1.}$\end{proof}

\begin{proof}[Ad 2] Let $t\in[0,T]$ be fixed. By definition
	\[
	V_{\varepsilon,t}^\perp=\left\{\psi\in \tilde{H}^1(\Omega_t^C):\int_{\Omega_t^C}\psi\, a(s(.,t))\phi^A_\varepsilon(.,t)\,dx=0\text{ for all }a\in\tilde{H}^1(\hat{\Sigma}^\circ)\right\}.
	\]
	The integral equals
	$\int_{\hat{\Sigma}}a(s)\int_{-\delta}^\delta(\phi^A_\varepsilon(.,t)\psi)|_{X(r,s,t)}J_t(r,s)\,dr\,d\Hc^{N-1}(s)$ due to the transformation rule and the Fubini Theorem. Therefore with the Fundamental Theorem of Calculus of Variations we obtain the characterization. Since by definition $V_{\varepsilon,t}\cap V_{\varepsilon,t}^\perp=\{0\}$, it remains to show $V_{\varepsilon,t}+V_{\varepsilon,t}^\perp=\tilde{H}^1(\Omega_t^C)$. We define
	\[
	w_\varepsilon:\hat{\Sigma}\rightarrow\R:s\mapsto\int_{-\delta}^\delta(\phi^A_\varepsilon|_{\overline{X}(r,s,t)})^2 J_t(r,s)\,dr.
	\]
	It holds $w_\varepsilon\in C^1(\hat{\Sigma})$ and with Lemma \ref{th_SE_1Dtrafo_remainder} one can prove $w_\varepsilon\geq c>0$ for small $\varepsilon$. Now let $\psi\in \tilde{H}^1(\Omega_t^C)$ be arbitrary. Then we set
	\[
	a_\varepsilon:\hat{\Sigma}\rightarrow\R:s\mapsto\frac{1}{w_\varepsilon(s)}\int_{-\delta}^\delta(\phi^A_\varepsilon(.,t)\psi)|_{X(r,s,t)}J_t(r,s)\,dr.
	\] 
	Due to Remark \ref{th_SE_ACND_sob_rem},~3.-4.~and since integration yields a bounded linear functional on $L^2(-\delta,\delta)$, we obtain $a_\varepsilon\in\tilde{H}^1(\hat{\Sigma}^\circ)$. For $\psi^\perp_\varepsilon:=\psi-a_\varepsilon(s(.,t))\phi^A_\varepsilon(.,t)\in\tilde{H}^1(\Omega_t^C)$ it holds
	\[
	\int_{-\delta}^\delta(\phi^A_\varepsilon(.,t)\psi_\varepsilon^\perp)|_{X(r,s,t)}J_t(r,s)\,dr=a_\varepsilon(s)w_\varepsilon(s)-a_\varepsilon(s)w_\varepsilon(s)=0
	\]
	for a.e.~$s\in\hat{\Sigma}$. Therefore by the integral characterization above we obtain $\psi^\perp_\varepsilon\in V_{\varepsilon,t}^\perp$.\qedhere$_{2.}$\end{proof}

\subsubsection{Analysis of the Bilinear Form}\label{sec_SE_ACND_BLF}
First we analyze $B_{\varepsilon,t}^C$ on $V_{\varepsilon,t}\times V_{\varepsilon,t}$.
\begin{Lemma}\label{th_SE_ACND_VxV}
	There are $\varepsilon_0,C,c>0$ such that 
	\[
	B_{\varepsilon,t}^C(\phi,\phi)\geq-C\|\phi\|_{L^2(\Omega_t^C)}^2+c\|a\|_{H^1(\hat{\Sigma}^\circ)}^2
	\]
	for all $\phi=a(s(.,t))\phi^A_\varepsilon(.,t)\in V_{\varepsilon,t}$ and $\varepsilon\in(0,\varepsilon_0],t\in[0,T]$.
\end{Lemma}
\begin{proof} 
	Let $\phi$ be as in the lemma. We rewrite $B_{\varepsilon,t}^C(\phi,\phi)$ in order to use Lemma \ref{th_SE_ACND_phi_A}. Therefore we compute
	$\nabla\phi=
	\nabla(a|_{s(.,t)})\phi^A_\varepsilon(.,t)+a|_{s(.,t)}\nabla\phi^A_\varepsilon(.,t)$ and \phantom{\qedhere}
	\[
	|\nabla\phi|^2=|\nabla(a(s))\phi^A_\varepsilon|^2|_{(.,t)}+a^2(s)|\nabla\phi^A_\varepsilon|^2|_{(.,t)}+\nabla(a^2(s))\cdot\nabla\phi^A_\varepsilon\phi^A_\varepsilon|_{(.,t)}.
	\]
	Due to Remark \ref{th_SE_ACND_sob_rem},~2.~we can use integration by parts on $\Omega_t^C$. Therefore
	\begin{align*}
	\int_{\Omega_t^C}\left[\nabla(a^2(s))\cdot\nabla\phi^A_\varepsilon\phi^A_\varepsilon\right]|_{(.,t)}\,dx&=-\int_{\Omega_t^C}\left[a^2(s)(\Delta\phi^A_\varepsilon\,\phi^A_\varepsilon+|\nabla\phi^A_\varepsilon|^2)\right]|_{(.,t)}\,dx\\
	&+\int_{\partial\Omega_t^C}\left[N_{\partial\Omega^C_t}\cdot\nabla\phi^A_\varepsilon\,\tr(a^2(s)\phi^A_\varepsilon|_{(.,t)})\right]\,d\Hc^{N-1}.
	\end{align*}
	Therefore we obtain
	\begin{align*}
	B_{\varepsilon,t}^C(\phi,\phi)&=\int_{\Omega_t^C}|\nabla(a(s))\phi^A_\varepsilon|^2|_{(.,t)}\,dx+\int_{\Omega_t^C}(a^2(s)\phi^A_\varepsilon)|_{(.,t)}\Lc_{\varepsilon,t}^C\phi^A_\varepsilon|_{(.,t)}\,dx\\
	&+\int_{\partial\Omega_t^C}\left[N_{\partial\Omega^C_t}\cdot\nabla\phi^A_\varepsilon\,\tr(a^2(s)\phi^A_\varepsilon|_{(.,t)})\right]\,d\Hc^{N-1}=:(I)+(II)+(III).
	\end{align*}
	\begin{proof}[Ad $(I)$] 
		Because of $\nabla(a|_{s(.,t)})=(D_xs)^\top|_{\overline{X}(.,t)}\nabla_{\hat{\Sigma}}a|_{s(.,t)}$ and Theorem \ref{th_coordND},~3.~it follows that $|\nabla(a|_{s(.,t)})|^2\geq c|\nabla_{\hat{\Sigma}} a|^2|_{s(.,t)}$. Therefore\phantom{\qedhere}
		\[
		(I)\geq \int_{\hat{\Sigma}}|\nabla_{\hat{\Sigma}} a|^2|_s\int_{-\delta}^\delta(\phi^A_\varepsilon)^2|_{\overline{X}(r,s,t)}J_t(r,s)\,dr\,d\Hc^{N-1}(s).
		\]
		The proof of Lemma \ref{th_SE_ACND_split_L2},~1.~yields that the inner integral is estimated from below by a uniform positive constant. Hence $(I)\geq c_0\|\nabla_{\hat{\Sigma}} a\|_{L^2(\hat{\Sigma})}^2$ for a $c_0>0$ independent of $\phi\in V_{\varepsilon,t}$ and  all $\varepsilon\in(0,\varepsilon_0]$, $t\in[0,T]$, if $\varepsilon_0>0$ is small. Lemma \ref{th_SobMfd_def_lemma},~3.~and Lemma \ref{th_SE_ACND_split_L2} yield the estimate.\end{proof}
	
	\begin{proof}[Ad $(II)$] We write\phantom{\qedhere}
		\[
		(II)=\int_{\hat{\Sigma}}a^2(s)\int_{-\delta}^\delta \phi^A_\varepsilon|_{\overline{X}(r,s,t)}(\Lc_{\varepsilon,t}^C\phi^A_\varepsilon(.,t))|_{X(r,s,t)}J_t(r,s)\,dr\,d\Hc^{N-1}(s)
		\]
		and estimate the inner integral. Lemma \ref{th_SE_ACND_phi_A} yields
		\[
		\left|\sqrt{\varepsilon}\Lc_{\varepsilon,t}^C\phi^A_\varepsilon(.,t)+\frac{1}{\varepsilon}\Delta r|_{\overline{X}_0(s(.,t),t)}q(s(.,t),t)\theta_0''(\rho_\varepsilon(.,t))\right|
		\leq Ce^{-c|\rho_\varepsilon(.,t)|}\quad\text{ in }\Omega_t^C.
		\]
		The lowest $\varepsilon$-order term in the inner integral in $(II)$ is
		\begin{align}\label{eq_SE_ACND_VxV_bad}
		\frac{1}{\varepsilon^2}\Delta r|_{\overline{X}_0(s,t)}q(s,t)^2\int_{-\delta}^\delta \theta_0''\theta_0'(\rho_\varepsilon|_{\overline{X}(r,s,t)})J_t(r,s)\,dr.
		\end{align}
		It holds $|J_t(r,s)-J_t(0,s)|\leq \tilde{C}|r|$ with $\tilde{C}>0$ independent of $(r,s,t)$. Due to Lemma \ref{th_SE_1Dtrafo_remainder},~1.~and $\int_\R\theta_0''\theta_0'\,dz=0$ the term \eqref{eq_SE_ACND_VxV_bad} with $J_t(0,s)$ instead of $J_t(r,s)$ can be estimated by a constant $C>0$ independent of $s, t$ and $\varepsilon\in(0,\varepsilon_0]$. The remaining terms in \eqref{eq_SE_ACND_VxV_bad} and $(II)$ can be controlled with Lemma \ref{th_SE_1Dtrafo_remainder}. Altogether we obtain $|(II)|\leq C\|a\|_{L^2(\hat{\Sigma})}^2$ with $C>0$ independent of $\phi\in V_{\varepsilon,t}$ and all $\varepsilon\in(0,\varepsilon_0]$, $t\in[0,T]$ if $\varepsilon_0>0$ is small.\end{proof}
	
	\begin{proof}[Ad $(III)$] 
		We transform the integral over $\partial\Omega^C_t$ to the boundary of $(-\delta,\delta)\times\hat{\Sigma}^\circ$ with the aid of Theorem \ref{th_Leb_trafo_mfd} and Theorem \ref{th_coordND}. This makes sense because of Remark \ref{th_SobDom_LipRem},~3. Note that the traces transform naturally by a density argument (possible due to Remark \ref{th_SE_ACND_sob_rem},~2.). Therefore we obtain
		\begin{align*}
		(III)&=\sum_\pm\int_{\hat{\Sigma}}a^2(s)[\phi^A_\varepsilon N_{\partial\Omega^C_t}\cdot\nabla\phi^A_\varepsilon]|_{\overline{X}(\pm\delta,s,t)}|\det d_s [X(\pm\delta,.,t)]|\,d\Hc^{N-1}(s)\\
		&+\int_{\partial\Sigma}\tr\,a^2|_{\sigma}\int_{-\delta}^\delta[\phi^A_\varepsilon \partial_{N_{\partial\Omega}}\phi^A_\varepsilon]|_{\overline{X}(r,Y(\sigma,0),t)}|\det d_{(r,\sigma)}[X(.,Y(.,0),t)]|\,dr\,d\Hc^{N-2}(\sigma).
		\end{align*}
		We apply Lemma \ref{th_SE_ACND_phi_A} and for the last integral we use Lemma \ref{th_SE_1Dtrafo_remainder} and $\int_\R\theta_0'\theta_0''=0$. This yields
		\[
		|(III)|\leq Ce^{-c/\varepsilon}\|a\|_{L^2(\hat{\Sigma})}^2+C\varepsilon \|\tr\, a|_{\partial\Sigma}\|_{L^2(\partial\Sigma)}^2.
		\]
		Theorem \ref{th_SobMfd_LipThm} implies
		$\|\tr\,a|_{\partial\Sigma}\|_{L^2(\partial\Sigma)}\leq C\|a\|_{H^1(\hat{\Sigma}^0)}$ and Lemma \ref{th_SE_ACND_split_L2}, 1.~yields the claim.\end{proof}
\end{proof}

Next we consider $B^C_{\varepsilon,t}$ on $V_{\varepsilon,t}^\perp\times V_{\varepsilon,t}^\perp$.
\begin{Lemma}\label{th_SE_ACND_VpxVp}
	There are $\nu,\varepsilon_0>0$ such that
	\[
	B^C_{\varepsilon,t}(\psi,\psi)\geq
	\nu\left[\frac{1}{\varepsilon^2}\|\psi\|_{L^2(\Omega_t^C)}^2+\|\nabla\psi\|_{L^2(\Omega_t^C)}^2\right]
	\]  
	for all $\psi\in V_{\varepsilon,t}^\perp$ and $\varepsilon\in(0,\varepsilon_0]$, $t\in[0,T]$.
\end{Lemma}
\begin{proof}
	It is enough to show that there are $\tilde{\nu}, \tilde{\varepsilon}_0>0$ such that
	\begin{align}\label{eq_SE_ACND_VpxVp_1}
	\tilde{B}_{\varepsilon,t}(\psi,\psi):=\int_{\Omega^C_t}|\nabla\psi|^2+\frac{1}{\varepsilon^2}f''(\theta_0|_{\rho_\varepsilon(.,t)})\psi^2\,dx\geq\frac{\tilde{\nu}}{\varepsilon^2}\|\psi\|_{L^2(\Omega^C_t)}^2
	\end{align}
	for all $\psi\in V_{\varepsilon,t}^\perp$ and $\varepsilon\in(0,\tilde{\varepsilon}_0]$, $t\in[0,T]$.
	Then the claim follows as in the 2D-case, cf.~the proof of Lemma 4.8 in \cite{AbelsMoser}.
	
	Analogously to \cite{AbelsMoser} we show \eqref{eq_SE_ACND_VpxVp_1} by reducing to Neumann boundary problems in normal direction. To this end let $\tilde{\psi}_t:=\psi|_{X(.,t)}$ for $\psi\in V_{\varepsilon,t}^\perp$. Then $\tilde{\psi}_t\in H^1((-\delta,\delta)\times\hat{\Sigma}^\circ)$ and
	\[
	\nabla \psi|_{X(.,t)}=
	\nabla r|_{\overline{X}(.,t)}\partial_r\tilde{\psi}_t
	+(D_xs)^\top|_{\overline{X}(.,t)}\nabla_{\hat{\Sigma}}\tilde{\psi}_t
	\]
	in $\Omega_t^C$ due to Corollary~\ref{th_coordND_nabla_tau_n} and Remark \ref{th_SE_ACND_sob_rem}, 3.-4. Therefore
	\[
	|\nabla\psi|^2|_{X(.,t)}=(\partial_r,\nabla_{\hat{\Sigma}})^\top\tilde{\psi}_t
	\begin{pmatrix}
	|\nabla r|^2 & (D_xs\nabla r)^\top\\
	D_xs\nabla r & D_xs(D_xs)^\top
	\end{pmatrix}|_{\overline{X}(.,t)}
	\begin{pmatrix}
	\partial_r\\
	\nabla_{\hat{\Sigma}}
	\end{pmatrix}\tilde{\psi}_t
	\]
	in $\Omega_t^C$. Theorem \ref{th_coordND}, a Taylor expansion and Young's inequality imply 
	\begin{align}\label{eq_SE_ACND_VpxVp_2}
	|\nabla\psi|^2|_{X(.,t)}\geq
	(1-Cr^2)(\partial_r\tilde{\psi}_t)^2 
	+c|\nabla_{\hat{\Sigma}}\tilde{\psi}_t|^2
	\end{align}
	in $\Omega_t^C$ for some $c,C>0$. The second term is not needed here. In order to get $Cr^2$ small enough (which will be precise later), we fix $\tilde{\delta}>0$ small and estimate separately for $r$ in
	\[
	I_{s,t}^\varepsilon:=(-\tilde{\delta},\tilde{\delta})+\varepsilon h_\varepsilon(s,t)\quad\text{ and }\quad \hat{I}_{s,t}^\varepsilon:=(-\delta,\delta)\setminus I_{s,t}^\varepsilon.
	\]
	If $\varepsilon_0=\varepsilon_0(\tilde{\delta},\overline{C}_0)>0$ is small, then for all $\varepsilon\in(0,\varepsilon_0]$ and $s\in\hat{\Sigma}$, $t\in[0,T]$ we have
	\[
	f''(\theta_0(\rho_\varepsilon|_{\overline{X}(r,s,t)}))\geq c_0>0\quad \text{ for }r\in\hat{I}_{s,t}^{\varepsilon},\quad
	|r|\leq \tilde{\delta}+\varepsilon|h_\varepsilon(s,t)|\leq 2\tilde{\delta}\quad \text{ for }r\in I_{s,t}^{\varepsilon}.
	\]
	Let $\tilde{c}=\tilde{c}(\tilde{\delta}):=4C\tilde{\delta}^2$ with $C$ from \eqref{eq_SE_ACND_VpxVp_2}. Then for all $\varepsilon\in(0,\varepsilon_0]$ it holds
	\begin{align*}
	\tilde{B}_{\varepsilon,t}(\psi,\psi)&\geq\int_{\hat{\Sigma}}\int_{\hat{I}_{s,t}^\varepsilon}\frac{c_0}{\varepsilon^2}\tilde{\psi}_t^2 J_t|_{(r,s)}\,dr\,d\Hc^{N-1}(s)\\
	&+\int_{\hat{\Sigma}}\int_{I_{s,t}^\varepsilon}\left[(1-\tilde{c})(\partial_r\tilde{\psi}_t)^2+\frac{1}{\varepsilon^2}f''(\theta_0(\rho_\varepsilon|_{\overline{X}(.,t)}))\tilde{\psi}_t^2\right] J_t|_{(r,s)}\,dr\,d\Hc^{N-1}(s).
	\end{align*}
	
	We set $F_{\varepsilon,s,t}(z):=\varepsilon(z+h_\varepsilon(s,t))$ and $\tilde{J}_{\varepsilon,s,t}(z):=J_t(F_{\varepsilon,s,t}(z),s)$ for all $z\in[-\frac{\delta}{\varepsilon},\frac{\delta}{\varepsilon}]-h_\varepsilon(s,t)$ and $(s,t)\in\Sigma\times[0,T]$. Moreover, let $I_{\varepsilon,\tilde{\delta}}:=(-\frac{\tilde{\delta}}{\varepsilon},\frac{\tilde{\delta}}{\varepsilon})$ 
	and $\Psi_{\varepsilon,s,t}:=\sqrt{\varepsilon}
	\tilde{\psi}_t(F_{\varepsilon,s,t}(.),s)$. Due to Remark \ref{th_SE_ACND_sob_rem} it holds $\Psi_{\varepsilon,s,t}\in H^1(I_{\varepsilon,\tilde{\delta}})$ for a.e.~$s\in\hat{\Sigma}$ and all $t\in[0,T]$ and together with Lemma \ref{th_SE_1Dtrafo_remainder},~1.~we obtain that the second inner integral in the estimate above equals $1/\varepsilon^2$ times
	\begin{align*}
	B_{\varepsilon,s,t}^{\tilde{c}}(\Psi_{\varepsilon,s,t},\Psi_{\varepsilon,s,t})
	:=\int_{I_{\varepsilon,\tilde{\delta}}}\left[(1-\tilde{c})(\frac{d}{dz}\Psi_{\varepsilon,s,t})^2+f''(\theta_0(z))(\Psi_{\varepsilon,s,t})^2\right]\tilde{J}_{\varepsilon,s,t}\,dz
	\end{align*} 
	for a.e.~$s\in\hat{\Sigma}$ and all $t\in[0,T]$.
	Therefore \eqref{eq_SE_ACND_VpxVp_1} follows if we show
	with the same $c_0$ as above 
	\begin{align}\label{eq_SE_ACND_VpxVp_3}
	B_{\varepsilon,s,t}^{\tilde{c}}(\Psi_{\varepsilon,s,t},\Psi_{\varepsilon,s,t})
	\geq \overline{c}\|\Psi_{\varepsilon,s,t}\|^2_{L^2(I_{\varepsilon,\tilde{\delta}},\tilde{J}_{\varepsilon,s,t})}-\frac{c_0}{2}\|\tilde{\psi}_t(.,s)\|^2_{L^2(\hat{I}_{s,t}^\varepsilon,J_t(.,s))}
	\end{align}
	for $\varepsilon\in(0,\varepsilon_0]$,  a.e.~$s\in\hat{\Sigma}$ and all $t\in[0,T]$ with some $\varepsilon_0,\overline{c}>0$ independent of $\varepsilon,s,t$.
	
	The estimate \eqref{eq_SE_ACND_VpxVp_3} can be proven for appropriately small $\tilde{\delta}$ in the analogous way as in the 2D-case, cf.~the proof of Lemma 4.8 in \cite{AbelsMoser}. One uses the integral characterization for $\psi\in V_{\varepsilon,t}^\perp$ from Lemma \ref{th_SE_ACND_split_L2},~2.~and results from Section \ref{sec_SE_1Dscal_pert} for the operator
	\[
	\Lc_{\varepsilon,s,t}^0:=-(\tilde{J}_{\varepsilon,s,t})^{-1}\frac{d}{dz}\left(\tilde{J}_{\varepsilon,s,t}\frac{d}{dz}\right)+f''(\theta_0)
	\]
	on $H^2(I_{\varepsilon,\tilde{\delta}})$ with homogeneous Neumann boundary condition, in particular Theorem \ref{th_SE_1Dscal_pert2}. 
\end{proof}

For $B_{\varepsilon,t}^C$ on $V_{\varepsilon,t}\times V_{\varepsilon,t}^\perp$ it holds
\begin{Lemma}\label{th_SE_ACND_VxVp}
	There are $\varepsilon_0,C>0$ such that
	\[
	|B_{\varepsilon,t}^C(\phi,\psi)|\leq\frac{C}{\varepsilon}\|\phi\|_{L^2(\Omega_t^C)}\|\psi\|_{L^2(\Omega_t^C)}+\frac{1}{4}B_{\varepsilon,t}^C(\psi,\psi)+C\varepsilon\|a\|_{H^1(\hat{\Sigma}^\circ)}^2
	\]
	for all $\phi=a(s(.,t))\phi^A_\varepsilon(.,t)\in V_{\varepsilon,t}$, $\psi\in V_{\varepsilon,t}^\perp$ and $\varepsilon\in(0,\varepsilon_0]$, $t\in[0,T]$.
\end{Lemma}

First we prove the following auxiliary estimate. 
\begin{Lemma}\label{th_SE_ACND_intpol_tr}
	Let $\overline{\varepsilon}>0$ be fixed. Then there is a $\overline{C}>0$ (independent of $\psi$, $\varepsilon$, $t$) such that 
	\[
	\|\tr\,\psi\|_{L^2(\partial\Omega_t^C)}^2
	\leq \overline{C}\left[\varepsilon\|\nabla \psi\|_{L^2(\Omega_t^C)}^2+\frac{1}{\varepsilon}\|\psi\|_{L^2(\Omega_t^C)}^2\right]
	\]
	for all $\psi\in H^1(\Omega_t^C)$ and $\varepsilon\in(0,\overline{\varepsilon}]$, $t\in[0,T]$.
\end{Lemma}
\begin{proof}
	Because of Remark \ref{th_SE_ACND_sob_rem}, Theorem \ref{th_Leb_trafo_mfd} and Theorem \ref{th_coordND} it is equivalent to prove the estimate for $S:=(-\delta,\delta)\times\hat{\Sigma}^\circ$ instead of $\Omega_t^C$ and $\nabla_S=(\partial_r,\nabla_{\hat{\Sigma}})$ instead of $\nabla$. For the $S$-case we use the idea from Evans \cite{Evans},~5.10, problem 7. Note that $S$ is a smooth manifold with thin singular set in the sense of Amann, Escher \cite{AmannEscherIII}, Chapter 3.1. Therefore the outer unit normal $N_{\partial S}$ is defined $\Hc^{N-1}$-a.e.~on $\partial S$ and the Gauß-Theorem holds for $C^2$-vector fields on $\overline{S}$ due to \cite{AmannEscherIII}, Theorem XII.3.15 and Remark XII.3.16(c). Let $w_1\in C^2([-\delta,\delta])$ with $w_1|_{\pm\delta}=\pm 1$ and $w_2$ be a $C^2$-vector field on $\hat{\Sigma}$ such that $\vec{w}_2|_{\partial\hat{\Sigma}}=N_{\partial\hat{\Sigma}}$. Then 
	\[
	\vec{w}:\overline{S}=[-\delta,\delta]\times\hat{\Sigma}\rightarrow\R^{N+1}:(r,s)\mapsto (w_1(r),0)+(0,\vec{w}_2(s))
	\] 
	is a $C^2$-vector field on $\overline{S}$ such that $\vec{w}\cdot N_{\partial S}\geq 1$ holds $\Hc^{N-1}$-a.e.~on $\partial S$. Hence for all $\psi\in C^2(\overline{S})$:
	\[
	\|\tr\,\psi\|_{L^2(\partial S)}^2\leq\int_{\partial S}\psi^2 \vec{w}\cdot N_{\partial S}\,d\Hc^{N-1}=\int_S\diverg_S(\psi^2\vec{w})\,d\Hc^N=\int_S\psi^2\diverg_S\vec{w}+2\psi \vec{w}\cdot\nabla_S\psi\,d\Hc^N.
	\]
	Therefore Young's inequality and $1\leq \overline{\varepsilon}/\varepsilon$ yields
	\[
	\|\tr\,\psi\|_{L^2(\partial S)}^2\leq C\left[\varepsilon\|\nabla_S\psi\|_{L^2(S)}^2+(1+\frac{1}{\varepsilon})\|\psi\|_{L^2(S)}^2\right]
	\leq \overline{C}\left[\varepsilon\|\nabla_S \psi\|_{L^2(S)}^2+\frac{1}{\varepsilon}\|\psi\|_{L^2(S)}^2\right]
	\]
	for all $\psi\in C^2(\overline{S})$ and $\varepsilon\in(0,\overline{\varepsilon}]$, where $\overline{C}>0$ is independent of $\psi$, $\varepsilon$.
	Hence the estimate also follows for all $\psi\in H^1(S)$ via density due to Remark \ref{th_SE_ACND_sob_rem} and Theorem \ref{th_SobMfd_LipThm}.
\end{proof}

\begin{proof}[Proof of Lemma \ref{th_SE_ACND_VxVp}]
	We rewrite $B_{\varepsilon,t}^C(\phi,\psi)$ in order to use Lemma \ref{th_SE_ACND_phi_A} and Lemma \ref{th_SE_ACND_split_L2}. It holds $\nabla\phi=\nabla(a|_{s(.,t)})\phi^A_\varepsilon+a|_{s(.,t)}\nabla\phi^A_\varepsilon|_{(.,t)}$ and integration by parts, cf.~Remark \ref{th_SE_ACND_sob_rem},~2., yields \phantom{\qedhere}
	\begin{align*}
	\int_{\Omega_t^C}a(s)\nabla\phi^A_\varepsilon|_{(.,t)}\cdot\nabla\psi\,dx
	=-\int_{\Omega_t^C}\left[\nabla(a(s))\cdot\nabla\phi^A_\varepsilon+a(s)\Delta\phi^A_\varepsilon|_{(.,t)}\right]\psi\,dx\\
	+\int_{\partial\Omega_t^C}N_{\partial\Omega^C_t}\cdot\nabla\phi^A_\varepsilon|_{(.,t)}\tr\left[a(s(.,t))\psi\right]\,d\Hc^{N-1}.
	\end{align*}
	Therefore we obtain
	\begin{align*}
	B_{\varepsilon,t}^C(\phi,\psi)
	=\int_{\Omega_t^C} a(s)|_{(.,t)}\psi\Lc_{\varepsilon,t}^C\phi^A_\varepsilon|_{(.,t)}\,dx
	+\int_{\partial\Omega_t^C}N_{\partial\Omega^C_t}\cdot\nabla\phi^A_\varepsilon|_{(.,t)}\tr\left[a(s(.,t))\psi\right]\,d\Hc^{N-1}\\
	+\int_{\Omega_t^C}\nabla(a(s))|_{(.,t)}\cdot\left[\phi^A_\varepsilon|_{(.,t)}\nabla\psi-\nabla\phi^A_\varepsilon|_{(.,t)}\psi\right]\,dx=:(I)+(II)+(III).
	\end{align*}
	
	\begin{proof}[Ad $(I)$] 
		The Hölder Inequality yields $|(I)|\leq\|a(s|_{(.,t)})\Lc_{\varepsilon,t}^C\phi^A_\varepsilon(.,t)\|_{L^2(\Omega_t^C)}\|\psi\|_{L^2(\Omega_t^C)}$, where
		\[
		\|a(s|_{(.,t)})\Lc_{\varepsilon,t}^C\phi^A_\varepsilon(.,t)\|_{L^2(\Omega_t^C)}^2=\int_{\hat{\Sigma}}a^2(s)\int_{-\delta}^\delta (\Lc_{\varepsilon,t}^C\phi^A_\varepsilon(.,t))^2|_{X(r,s,t)}\,J_t(r,s)\,dr\,d\Hc^{N-1}(s)
		\]
		due to Theorem \ref{th_Leb_trafo_mfd}. By Lemma \ref{th_SE_ACND_phi_A} and Lemma \ref{th_SE_1Dtrafo_remainder} the inner integral is estimated by $\frac{C}{\varepsilon^2}$, cf.~also the estimate of $|(I)|$ for the 2D-case in the proof of Lemma 4.10 in \cite{AbelsMoser} for more details. Hence Lemma~\ref{th_SE_ACND_split_L2},~1.~yields
		\[
		|(I)|\leq \frac{C}{\varepsilon}\|a\|_{L^2(\hat{\Sigma})}\|\psi\|_{L^2(\Omega_t^C)}
		\leq\frac{\tilde{C}}{\varepsilon}\|\phi\|_{L^2(\Omega_t^C)}\|\psi\|_{L^2(\Omega_t^C)}
		\] 
		for all $t\in[0,T]$ and $\varepsilon\in(0,\varepsilon_0]$ if $\varepsilon_0>0$ is small. \phantom{\qedhere}\end{proof}
	
	\begin{proof}[Ad $(II)$] 
		The Hölder Inequality yields\phantom{\qedhere} 
		\[	
		|(II)|\leq\|\tr\,\psi\|_{L^2(\partial\Omega_t^C)}\|\tr(a(s|_{(.,t)}))N_{\partial\Omega^C_t}\cdot\nabla\phi^A_\varepsilon|_{(.,t)}\|_{L^2(\partial\Omega_t^C)}.
		\] 
		We transform the second integral as in the estimate of $(III)$ in the proof of Lemma \ref{th_SE_ACND_VxV} with Theorem \ref{th_Leb_trafo_mfd} and Remark \ref{th_SE_ACND_sob_rem}. Then Lemma \ref{th_SE_ACND_phi_A}, Lemma \ref{th_SE_1Dtrafo_remainder} and Theorem \ref{th_SobMfd_LipThm} yield
		\[
		\|a(s)N_{\partial\Omega^C_t}\cdot\nabla\phi^A_\varepsilon|_{(.,t)}\|_{L^2(\partial\Omega_t^C)}\leq C\|\tr\,a|_{\partial\Sigma}\|_{L^2(\partial\Sigma)}+ Ce^{-c/\varepsilon}\|a\|_{L^2(\hat{\Sigma})}\leq C\|a\|_{H^1(\hat{\Sigma}^\circ)}.
		\]
		We estimate $\|\tr\,\psi\|_{L^2(\partial\Omega_t^C)}$ with Lemma \ref{th_SE_ACND_intpol_tr}. Hence
		Young's inequality and Lemma \ref{th_SE_ACND_VpxVp} imply
		\[
		|(II)|\leq
		\frac{\nu}{8\varepsilon \overline{C}}\|\tr\,\psi\|_{L^2(\partial\Omega_t^C)}^2
		+\tilde{C}\varepsilon\|a\|_{H^1(\hat{\Sigma}^\circ)}^2\leq\frac{1}{8}B_{\varepsilon,t}^C(\psi,\psi)
		+\tilde{C}\varepsilon\|a\|_{H^1(\hat{\Sigma}^\circ)}^2,
		\]
		where $\overline{C}$ is as in Lemma \ref{th_SE_ACND_intpol_tr}.
	\end{proof}
	
	\begin{proof}[Ad $(III)$] 
		With Remark \ref{th_SE_ACND_sob_rem},~1.~we can transform $(III)=\int_{\hat{\Sigma}}\nabla_{\hat{\Sigma}}a(s)\cdot g_t(s)\,d\Hc^{N-1}(s)$, where
		\[
		g_t(s):=\int_{-\delta}^\delta D_xs|_{\overline{X}(r,s,t)}\left[\phi^A_\varepsilon(.,t)\nabla\psi-\nabla\phi^A_\varepsilon(.,t)\psi\right]|_{X(r,s,t)}\,J_t(r,s)\,dr.
		\]
		It holds
		$\nabla\psi|_{X(.,t)}=
		\nabla r|_{\overline{X}(.,t)}\partial_r\tilde{\psi}_t
		+(D_xs)^\top|_{\overline{X}(.,t)} \nabla_{\hat{\Sigma}}\tilde{\psi}_t$ with
		$\tilde{\psi}_t:=\psi|_{X(.,t)}$ in $\Omega_t^C$ because of Remark \ref{th_SE_ACND_sob_rem},~3.~and Corollary~\ref{th_coordND_nabla_tau_n}.
		For the $\nabla_{\hat{\Sigma}}\tilde{\psi}_t$-term in $g_t$ we use 
		\[
		\left|D_xs(D_xs)^\top|_{\overline{X}(r,s,t)}-D_xs(D_xs)^\top|_{\overline{X}(0,s,t)}\right|\leq C|r|.
		\] 
		Therefore $|g_t(s)|$ is for a.e.~$s\in\hat{\Sigma}$ estimated by
		\begin{align*}
		&|D_xs(D_xs)^\top|_{\overline{X}_0(s,t)}\left|\int_{-\delta}^\delta\left[\phi^A_\varepsilon|_{\overline{X}(.,t)}\nabla_{\hat{\Sigma}}\tilde{\psi}_t J_t\right]|_{(r,s)}\,dr\right|+\int_{-\delta}^\delta\left|(\tilde{\psi}_t J_t)|_{(r,s)}D_xs\nabla\phi^A_\varepsilon|_{\overline{X}(r,s,t)}\right|\,dr\\
		&+\int_{-\delta}^\delta\left[
		C\left|r\, \nabla_{\hat{\Sigma}}\tilde{\psi}_t|_{(r,s)}\right|
		+\left|D_xs\nabla r|_{\overline{X}(r,s,t)} \partial_r\tilde{\psi}_t|_{(r,s)}\right|\right] \cdot\left|\phi^A_\varepsilon|_{\overline{X}(r,s,t)}
		J_t|_{(r,s)}\right|\,dr.
		\end{align*}
		We use $\psi\in V_{\varepsilon,t}^\perp$ to estimate the first term. Due to Lemma \ref{th_SobMfd_prod_set} and since integration gives a bounded linear operator on $L^2(-\delta,\delta)$, we can apply $\nabla_{\hat{\Sigma}}$ to the identity in Lemma \ref{th_SE_ACND_split_L2},~2.~and commute integration with $\nabla_{\hat{\Sigma}}$. Therefore the first term is bounded by 
		\[
		C\left|\int_{-\delta}^\delta\left[\left(\nabla_{\hat{\Sigma}}(\phi^A_\varepsilon|_{\overline{X}(.,t)})J_t+\phi^A_\varepsilon|_{\overline{X}(.,t)}\nabla_{\hat{\Sigma}} J_t\right)\tilde{\psi}_t\right]|_{(r,s)}\,dr\right|
		\]
		for a.e.~$s\in\hat{\Sigma}$.
		Now we make use of the structure of $\phi^A_\varepsilon$. Due to \eqref{eq_SE_ACND_nabla_vI_vC} it holds in $\Omega^C_t$
		\begin{align*}
		&\nabla\phi^A_\varepsilon|_{(.,t)}=
		\frac{1}{\sqrt{\varepsilon}}\left(
		\theta_0''|_{\rho_\varepsilon}q|_{(s,t)}
		+\varepsilon\partial_\rho\hat{v}^C|_{(\rho_\varepsilon,H_\varepsilon,\sigma,t)}\right)
		\left[\frac{\nabla r}{\varepsilon}-(D_xs)^\top\nabla_{\hat{\Sigma}}h_\varepsilon|_{(s,t)}\right]\\
		&+\frac{1}{\sqrt{\varepsilon}}\left[(D_xs)^\top
		\nabla_{\hat{\Sigma}} q|_{(s,t)}\theta_0'|_{\rho_\varepsilon}
		+\nabla b\,
		\partial_H\hat{v}^C|_{(\rho_\varepsilon,H_\varepsilon,\sigma,t)}\right]+\sqrt{\varepsilon}(D_x\sigma)^\top\nabla_{\partial\Sigma}\hat{v}^C|_{(\rho_\varepsilon,H_\varepsilon,\sigma,t)},
		\end{align*}
		where all terms are evaluated at $(.,t)$. Moreover, instead of $\nabla_{\hat{\Sigma}}(\phi^A_\varepsilon|_{\overline{X}})$ it is equivalent to estimate $\nabla_\tau\phi^A_\varepsilon|_{\overline{X}}=\nabla\phi^A_\varepsilon|_{\overline{X}}-\nabla r\partial_r(\phi^A_\varepsilon|_{\overline{X}})$ due to Corollary~\ref{th_coordND_nabla_tau_n}. The latter identity yields
		\begin{align}\begin{split}\label{eq_SE_ACND_dsigma_phiA}\nabla_\tau\phi^A_\varepsilon|_{\overline{X}}
		=-\frac{1}{\sqrt{\varepsilon}}\left(
		\theta_0''|_{\rho_\varepsilon}q|_{(s,t)}
		+\varepsilon\partial_\rho\hat{v}^C|_{(\rho_\varepsilon,H_\varepsilon,\sigma,t)}\right)
		(D_xs)^\top\nabla_{\hat{\Sigma}}h_\varepsilon|_{(s,t)}\qquad\qquad\qquad\\
		+\frac{1}{\sqrt{\varepsilon}}\left[(D_xs)^\top
		\nabla_{\hat{\Sigma}} q|_{(s,t)}\theta_0'|_{\rho_\varepsilon}
		+\nabla b\,
		\partial_H\hat{v}^C|_{(\rho_\varepsilon,H_\varepsilon,\sigma,t)}\right]+\sqrt{\varepsilon}(D_x\sigma)^\top\nabla_{\partial\Sigma}\hat{v}^C|_{(\rho_\varepsilon,H_\varepsilon,\sigma,t)},
		\end{split}
		\end{align} 
		where the terms on the right hand side are evaluated at $\overline{X}$. Note that compared to $\nabla\phi^A_\varepsilon$ the $\varepsilon$-order is better by one. Therefore we can control the terms in the above estimate for $|g_t(s)|$. The Hölder Inequality, Lemma \ref{th_SE_1Dtrafo_remainder} and $D_xs\nabla r|_{\overline{X}_0(s,t)}=0$ yield for a.e.~$s\in\hat{\Sigma}$
		\[
		|g_t(s)|\leq C\|\tilde{\psi}_t(.,s)\|_{L^2(-\delta,\delta;J_t(.,s))}+C\varepsilon\|(\partial_r,\nabla_{\hat{\Sigma}})\tilde{\psi}_t(.,s)\|_{L^2(-\delta,\delta;J_t(.,s))}.
		\]
		Due to Remark \ref{th_SE_ACND_sob_rem} and Corollary~\ref{th_coordND_nabla_tau_n} it holds  $|(\partial_r,\nabla_{\hat{\Sigma}})\tilde{\psi}_t|\leq C|\nabla\psi|_{X(.,t)}|$. Therefore
		\begin{align*}
		|(III)|
		&\leq C\|\nabla_{\hat{\Sigma}}a\|_{L^2(\hat{\Sigma})}(\|\psi\|_{L^2(\Omega_t^C)}+\varepsilon\|\nabla\psi\|_{L^2(\Omega_t^C)})\\
		&\leq C\varepsilon^2\|\nabla_{\hat{\Sigma}}a\|_{L^2(\hat{\Sigma})}^2+\frac{\nu}{8}\left[\frac{1}{\varepsilon^2}\|\psi\|_{L^2(\Omega_t^C)}^2+\|\nabla\psi\|_{L^2(\Omega_t^C)}^2\right],
		\end{align*}
		where we used Young's inequality in the second step and $\nu$ is as in Lemma \ref{th_SE_ACND_VpxVp}. The last term is estimated by $\frac{1}{8}B_{\varepsilon,t}^C(\psi,\psi)$ due to Lemma \ref{th_SE_ACND_VpxVp}. With Lemma \ref{th_SobMfd_def_lemma} the claim follows.\end{proof}
\end{proof}

Finally, we put together Lemma \ref{th_SE_ACND_VxV}-\ref{th_SE_ACND_VxVp}.

\begin{Theorem}\label{th_SE_ACND_cp2}
	There are $\varepsilon_0, C, c_0>0$ such that for all $\varepsilon\in(0,\varepsilon_0]$, $t\in[0,T]$ and $\psi\in H^1(\Omega_t^C)$ with $\psi|_{X(.,s,t)}=0$ for a.e.~$s\in Y(\partial\Sigma\times[\frac{3}{2}\mu_0,2\mu_0])$ it holds
	\[
	B_{\varepsilon,t}^C(\psi,\psi)\geq -C\|\psi\|_{L^2(\Omega_t^C)}^2+c_0\|\nabla_\tau\psi\|_{L^2(\Omega_t^C)}^2.
	\] 
\end{Theorem}

\begin{Remark}\upshape\begin{enumerate}
		\item The estimate can be refined, cf.~the proof below.
		\item Theorem \ref{th_SE_ACND_cp2} directly yields Theorem \ref{th_SE_ACND_cp}, cf.~the beginning of Section \ref{sec_SE_ACND_outline}.
	\end{enumerate}
\end{Remark}

\begin{proof}[Proof of Theorem \ref{th_SE_ACND_cp2}]
	Due to Lemma \ref{th_SE_ACND_split_L2} we can uniquely represent any $\psi\in\tilde{H}^1(\Omega_t^C)$ as
	\[
	\psi=\phi+\phi^\perp\quad\text{ with } \phi=[a(s)\phi^A_\varepsilon]|_{(.,t)}\in V_{\varepsilon,t}\text{ and }
	\phi^\perp\in V_{\varepsilon,t}^\perp.
	\]
	Lemma \ref{th_SE_ACND_VxV} and Lemma \ref{th_SE_ACND_VxVp} yield for $\varepsilon_0>0$ small and all $\varepsilon\in(0,\varepsilon_0]$, $t\in[0,T]$ that
	\begin{align*}
	&B_{\varepsilon,t}^C(\psi,\psi)=B_{\varepsilon,t}^C(\phi,\phi)+2B_{\varepsilon,t}^C(\phi,\phi^\perp)+B_{\varepsilon,t}^C(\phi^\perp,\phi^\perp)\\
	&\geq -C\|\phi\|_{L^2(\Omega_t^C)}^2+(c_0-C\varepsilon)\|a\|_{H^1(\hat{\Sigma}^\circ)}^2-\frac{C}{\varepsilon}\|\phi\|_{L^2(\Omega_t^C)}\|\phi^\perp\|_{L^2(\Omega_t^C)}+\frac{B_{\varepsilon,t}^C(\phi^\perp,\phi^\perp)}{2}.
	\end{align*}
	The third term is estimated by $\frac{\nu}{4\varepsilon^2}\|\phi^\perp\|_{L^2(\Omega_t^C)}^2+\tilde{C}\|\phi\|_{L^2(\Omega_t^C)}^2$ due to Young's inequality, where $\nu$ is as in Lemma \ref{th_SE_ACND_VpxVp}. Hence we obtain 
	\[
	B_{\varepsilon,t}^C(\psi,\psi)\geq-C\|\phi\|_{L^2(\Omega_t^C)}^2+\frac{\nu}{4\varepsilon^2}\|\phi^\perp\|_{L^2(\Omega_t^C)}^2+\frac{c_0}{2}\|a\|_{H^1(\hat{\Sigma}^\circ)}^2+\frac{\nu}{2}\|\nabla(\phi^\perp)\|_{L^2(\Omega_t^C)}^2
	\]
	for all $\varepsilon\in(0,\varepsilon_0]$ and $t\in[0,T]$, where $\varepsilon_0>0$ is small.
	
	It remains to include the $\nabla_\tau\psi$-term in the estimate. Because of the triangle inequality it holds $\|\nabla_\tau\psi\|_{L^2(\Omega_t^C)}\leq\|\nabla_\tau\phi\|_{L^2(\Omega_t^C)}+\|\nabla_\tau(\phi^\perp)\|_{L^2(\Omega_t^C)}$. Here 
	\[
	\nabla_\tau\phi|_{X(r,s,t)}
	=(D_xs)^\top|_{\overline{X}(r,s,t)}\nabla_{\hat{\Sigma}}a|_s \phi^A_\varepsilon|_{\overline{X}(r,s,t)}
	+a|_s\nabla_\tau\phi^A_\varepsilon|_{\overline{X}(r,s,t)}.
	\]
	We already computed $\nabla_\tau\phi^A_\varepsilon$ in \eqref{eq_SE_ACND_dsigma_phiA}. An integral transformation with Remark \ref{th_SE_ACND_sob_rem},~1., the Fubini Theorem, Lemma \ref{th_SE_1Dtrafo_remainder} and Lemma \ref{th_SobMfd_def_lemma},~3.~yield $\|\nabla_\tau\phi\|_{L^2(\Omega_t^C)}\leq C\|a\|_{H^1(\hat{\Sigma}^\circ)}$. Moreover, Remark \ref{th_SE_ACND_sob_rem},~3.~and Corollary~\ref{th_coordND_nabla_tau_n} imply $\|\nabla_\tau(\phi^\perp)\|_{L^2(\Omega_t^C)}\leq C\|\nabla(\phi^\perp)\|_{L^2(\Omega_t^C)}$.
	Therefore
	\[
	\|\nabla_\tau\psi\|_{L^2(\Omega_t^C)}^2\leq C(\|a\|_{H^1(\hat{\Sigma}^\circ)}^2+\|\nabla(\phi^\perp)\|_{L^2(\Omega_t^C)}^2).
	\]
	Finally, together with the above estimate for $B_{\varepsilon,t}^C$ this shows the claim.\end{proof}

\subsection{Spectral Estimate for (vAC) in ND}
\label{sec_SE_vAC}
In this section we prove the analogue of the spectral estimate \eqref{eq_SE1} from the scalar case for the vector-valued Allen-Cahn equation \eqref{eq_vAC1}-\eqref{eq_vAC3} when the diffuse interface meets the boundary in the case of $N$ dimensions, $N\geq 2$. The procedure is analogous to the scalar case in the last Section \ref{sec_SE_ACND}. The coordinates are the same, in particular we can use Remark \ref{th_SE_ACND_sob_rem}. Hence the only new difficulty is that we have to consider the potential $W:\R^m\rightarrow\R$ from Definition \ref{th_vAC_W} and vector-valued functions, i.e.~the image space is $\R^m$ instead of $\R$. However, we already laid all the foundations to adapt the arguments from the scalar case. Under the assumption in Remark \ref{th_ODE_vect_lin_op_rem} we solved the model problems for the vector-valued case in Sections \ref{sec_ODE_vect}-\ref{sec_hp_vect} and proved spectral properties for vector-valued Allen-Cahn-type operators in 1D in Section \ref{sec_SE_1Dvect}. We have a specific $\vec{u}^A_\varepsilon$ in mind, cf.~Section \ref{sec_asym_vAC_uA}. Nevertheless, as in the scalar case a slightly more general form is enough to prove the spectral estimate. Now we fix the assumptions for this section.

Let $\Omega\subset\R^N$ and $\Gamma=(\Gamma_t)_{t\in[0,T]}$ for $T>0$ be as in Section \ref{sec_coord_surface_requ} for $N\geq 2$ with contact angle $\alpha=\frac{\pi}{2}$ (\eqref{MCF} not needed). Moreover, let $\delta>0$ be such that Theorem \ref{th_coordND} holds for $2\delta$ instead of $\delta$. We use the notation for $\vec{n}_{\partial\Sigma}, \vec{n}_{\partial\Gamma}, Y, X_0, X, \mu_0, \mu_1, r, s, \sigma, b$ as in Theorem~\ref{th_coordND} and the definitions of several sets and of $\partial_n$, $\nabla_\tau$, $J$ from Remark \ref{th_coordND_rem}. Here for suitable $\R^m$-valued functions $\vec{\psi}$ we define $\partial_n\vec{\psi}$ and $\nabla_\tau\vec{\psi}$ component-wise. More precisely $\partial_n\vec{\psi}:=(\partial_n\psi_1,...,\partial_n\psi_m)$ and $\nabla_\tau\vec{\psi}:=((\nabla_\tau\psi_1)^\top,...,(\nabla_\tau\psi_m)^\top)$. Note that Corollary \ref{th_coordND_nabla_tau_n} carries over to $\R^m$-valued functions, in particular 
\[
\nabla\vec{\psi}:=(D_x\vec{\psi})^\top=\nabla r \partial_n\vec{\psi}+\nabla_\tau\vec{\psi}.
\] 
We consider height functions $\check{h}_1$ and $\check{h}_2=\check{h}_2(\varepsilon)$ and assume (with a slight abuse of notation)
\[
\check{h}_j\in B([0,T],C^0(\Sigma)\cap C^2(\hat{\Sigma})), j=1,2,\quad \hat{\Sigma}:= Y(\partial\Sigma\times[0,2\mu_0]),\quad C^2(\hat{\Sigma}):=C^2(\overline{\hat{\Sigma}^\circ}).
\] 
Moreover, consider $\check{C}_0>0$ such that $\|\check{h}_j\|_{B([0,T],C^0(\Sigma)\cap C^2(\hat{\Sigma}))}
\leq\check{C}_0$ for $j=1,2$. We set $\check{h}_\varepsilon:=\check{h}_1+\varepsilon\check{h}_2$ for $\varepsilon>0$ small and introduce the scaled variables
\[
\check{\rho}_\varepsilon:=\frac{r-\varepsilon \check{h}_\varepsilon(s,t)}{\varepsilon}\quad\text{ in }\overline{\Gamma(2\delta)}, \quad H_\varepsilon:=\frac{b}{\varepsilon}\quad\text{ in }\overline{\Gamma^C(2\delta,2\mu_1)}.
\]
For $W:\R^m\rightarrow\R$ as in Definition \ref{th_vAC_W} and any fixed distinct pair $\vec{u}_\pm$ of minimizers of $W$ let $\vec{\theta}_0$ be as in Remark \ref{th_ODE_vect_rem},~1. We make the assumption $\dim\ker\check{L}_0=1$, cf.~Remark \ref{th_ODE_vect_lin_op_rem}, where $\check{L}_0$ is as in \eqref{eq_ODE_vect_L0}. Moreover, let $\check{u}^C_1:\overline{\R^2_+}\times\partial\Sigma\times[0,T]\rightarrow\R^m:(\rho,H,\sigma,t)\mapsto\check{u}^C_1(\rho,H,\sigma,t)$ be in the space
$B([0,T];C^2(\partial\Sigma,H^2_{(0,\gamma)}(\R^2_+)^m))$ for some $\gamma>0$. Then we define
\[
\vec{u}^C_1(x,t):=\check{u}^C_1(\check{\rho}_\varepsilon(x,t),H_\varepsilon(x,t),\sigma(x,t),t)\quad\text{ for }(x,t)\in\overline{\Gamma^C(2\delta,2\mu_1)}.
\] 
For $\varepsilon>0$ small we consider
\begin{align*}
\vec{u}^A_\varepsilon=
\begin{cases}
\vec{\theta}_0(\check{\rho}_\varepsilon)+\Oc(\varepsilon^2)&\quad\text{ in }\Gamma(\delta,\mu_0),\\
\vec{\theta}_0(\check{\rho}_\varepsilon)+\varepsilon \vec{u}^C_1+\Oc(\varepsilon^2)&\quad\text{ in }\Gamma^C(\delta,2\mu_0),\\
\vec{u}_\pm+\Oc(\varepsilon)&\quad\text{ in }Q_T^\pm\setminus\Gamma(\delta),
\end{cases}
\end{align*}
where $\Oc(\varepsilon^k)$ are $\R^m$-valued measurable functions bounded by $C\varepsilon^k$. 
\begin{Remark}\upshape\label{th_SE_vAC_rem}
	We can also include an additional term of the form $\varepsilon\vec{\theta}_1(\check{\rho}_\varepsilon)\check{p}_\varepsilon(s,t)$ in $\vec{u}^A_\varepsilon$ on $\Gamma(\delta)$, where $\check{p}_\varepsilon\in B([0,T],C^0(\Sigma)\cap C^2(\hat{\Sigma}))$ satisfies a uniform estimate for $\varepsilon$ small and
	\begin{align}\label{eq_SE_vAC_theta1}
	\vec{\theta}_1\in C_b^0(\R)^m\quad\text{ with }\quad\int_\R (\vec{\theta}_0',\sum_{\xi\in\N_0^m,|\xi|=1}\partial^{\xi}D^2W(\vec{\theta}_0)(\vec{\theta}_1)^\xi\vec{\theta}_0')_{\R^m}=0.
	\end{align} 
	See Remark \ref{th_SE_vAC_asym_rem},~2.~below for details.
\end{Remark}

We obtain the following spectral estimate:
\begin{Theorem}[\textbf{Spectral Estimate for (vAC) in ND}]\label{th_SE_vAC}
	There are $\varepsilon_0,C,c_0>0$ independent of the $\check{h}_j$ for fixed $\check{C}_0$ such that for all $\varepsilon\in(0,\varepsilon_0], t\in[0,T]$ and $\vec{\psi}\in H^1(\Omega)^m$ it holds
	\begin{align*}
	\int_\Omega|\nabla\vec{\psi}|^2+\frac{1}{\varepsilon^2}&(\vec{\psi},D^2W(\vec{u}^A_\varepsilon(.,t))\vec{\psi})_{\R^m}\,dx\\
	&\geq -C\|\vec{\psi}\|_{L^2(\Omega)^m}^2+\|\nabla\vec{\psi}\|_{L^2(\Omega\setminus\Gamma_t(\delta))^{N\times m}}^2+c_0\|\nabla_\tau\vec{\psi}\|_{L^2(\Gamma_t(\delta))^{N\times m}}^2.
	\end{align*}
\end{Theorem}

We prove a spectral estimate on $\Omega^C_t:=\Gamma^C_t(\delta,2\mu_0)=X((-\delta,\delta)\times \hat{\Sigma}^\circ)$ for $t\in[0,T]$: 
\begin{Theorem}\label{th_SE_vAC_cp} 
	There are $\check{\varepsilon}_0, \check{C}, \check{c}_0>0$ independent of the $\check{h}_j$ for fixed $\check{C}_0$ such that for all $\varepsilon\in(0,\check{\varepsilon}_0]$, $t\in[0,T]$ and $\vec{\psi}\in H^1(\Omega^C_t)^m$ with $\vec{\psi}|_{X(.,s,t)}=0$ for a.e.~$s\in Y(\partial\Sigma\times[\frac{3}{2}\mu_0,2\mu_0])$:
	\[
	\int_{\Omega^C_t}|\nabla\vec{\psi}|^2+\frac{1}{\varepsilon^2}(\vec{\psi},D^2W(\vec{u}^A_\varepsilon(.,t))\vec{\psi})_{\R^m}\,dx\geq -\check{C}\|\vec{\psi}\|_{L^2(\Omega_t^C)^m}^2+\check{c}_0\|\nabla_\tau\vec{\psi}\|_{L^2(\Omega_t^C)^{N\times m}}^2.
	\]
\end{Theorem}

The extra assumption on $\vec{\psi}$ is not needed but simplifies the proof, cf.~Remark \ref{th_SE_vAC_asym_rem},~3.~below.  

\begin{proof}[Proof of Theorem \ref{th_SE_vAC}] 
	Note that $D^2W(\vec{u}^A_\varepsilon)$ is positive (semi-)definite on $Q_T^\pm\setminus\Gamma(\delta)$ for all $\varepsilon\in(0,\varepsilon_0]$ if $\varepsilon_0>0$ is small. Therefore it is enough to show the estimate in Theorem \ref{th_SE_vAC} for $\Gamma_t(\delta)$ instead of $\Omega$. This can be done with Theorem \ref{th_SE_vAC_cp} and Theorem \ref{th_SE_1Dvect_pert2},~1.~in the analogous way as in the scalar case, cf.~the proof of Theorem \ref{th_SE_ACND}.
\end{proof}

\subsubsection{Outline for the Proof of the Spectral Estimate close to the Contact Points}\label{sec_SE_vAC_outline}
For the proof of Theorem \ref{th_SE_vAC_cp} we can replace $\frac{1}{\varepsilon^2}D^2W(\vec{u}^A_\varepsilon(.,t))$ by
\[
\frac{1}{\varepsilon^2}D^2W(\vec{\theta}_0|_{\check{\rho}_\varepsilon(.,t)})+\frac{1}{\varepsilon}\sum_{\xi\in\N_0^m,|\xi|=1}\partial^{\xi}D^2W(\vec{\theta}_0|_{\check{\rho}_\varepsilon(.,t)})(\vec{u}^C_1|_{(.,t)})^\xi
\] 
with a Taylor expansion. We construct an approximation $\vec{\phi}^A_\varepsilon(.,t)$ to the first eigenfunction of 
\[
\check{\Lc}_{\varepsilon,t}^C:=-\Delta+
\frac{1}{\varepsilon^2}D^2W(\vec{\theta}_0|_{\check{\rho}_\varepsilon(.,t)})+\frac{1}{\varepsilon}\sum_{\xi\in\N_0^m,|\xi|=1}\partial^{\xi}D^2W(\vec{\theta}_0|_{\check{\rho}_\varepsilon(.,t)})(\vec{u}^C_1|_{(.,t)})^\xi\quad\text{ on }\Omega_t^C
\] 
together with a homogeneous Neumann boundary condition. Then we split
\begin{align}\label{eq_SE_vAC_H1tilde_Omega}
\check{H}^1(\Omega^C_t):=
\left\{\vec{\psi}\in H^1(\Omega^C_t)^m:\vec{\psi}|_{X(.,s,t)}=0\text{ for a.e.~}s\in Y(\partial\Sigma\times[\frac{3}{2}\mu_0,2\mu_0])\right\}
\end{align}
along the subspace of tangential alterations of $\vec{\phi}^A_\varepsilon(.,t)$. Therefore we consider the ansatz
\begin{alignat*}{2}
\vec{\phi}^A_\varepsilon(.,t)&:=\frac{1}{\sqrt{\varepsilon}}[\vec{v}^I_\varepsilon(.,t)+\varepsilon\vec{v}^C_\varepsilon(.,t)]&\quad&\text{ on }\Omega^C_t,\\
\vec{v}^I_\varepsilon(.,t)&:=\check{v}^I(\check{\rho}_\varepsilon(.,t),s(.,t),t):=\vec{\theta}_0'(\check{\rho}_\varepsilon(.,t))\,\check{q}(s(.,t),t)&\quad&\text{ on }\Omega^C_t,\\
\vec{v}^C_\varepsilon(.,t)&:=\check{v}^C(\check{\rho}_\varepsilon(.,t),H_\varepsilon(.,t),\sigma(.,t),t)&\quad&\text{ on }\Omega^C_t,
\end{alignat*}
where $\check{q}:\hat{\Sigma}\times[0,T]\rightarrow\R$ and $\check{v}^C:\overline{\R^2_+}\times\partial\Sigma\times[0,T]\rightarrow\R^m$. The $\frac{1}{\sqrt{\varepsilon}}$-factor is used for a normalization, see Lemma \ref{th_SE_vAC_split_L2} below. 

In Subsection \ref{sec_SE_vAC_asym} we expand $\check{\Lc}_{\varepsilon,t}^C\vec{\phi}^A_\varepsilon(.,t)$ and $\partial_{N_{\partial\Omega}}\vec{\phi}^A_\varepsilon(.,t)$ in a similar way as in Section \ref{sec_asym_vAC} and choose $\check{q}$ and $\check{v}^C$ suitably. The $\check{q}$-term is introduced in order to enforce the compatibility condition for the equations for $\check{v}^C$. In Subsection \ref{sec_SE_vAC_splitting} we characterize the $L^2(\Omega^C_t)^m$-orthogonal splitting of $\check{H}^1(\Omega^C_t)^m$ with respect to the subspace
\begin{align}\label{eq_SE_vAC_V_def}
\check{V}_{\varepsilon,t}
&:=\left\{\vec{\phi}=\check{a}(s(.,t))\vec{\phi}^A_\varepsilon(.,t):
\check{a}\in\check{H}^1(\hat{\Sigma}^\circ)\right\},\\
\check{H}^1(\hat{\Sigma}^\circ)&:=
\left\{\check{a}\in H^1(\hat{\Sigma}^\circ) : \check{a}|_{Y(.,b)}=0\text{ for a.e.~}b\in[\frac{3}{2}\mu_0,2\mu_0]\right\}.\label{eq_SE_vAC_H1check_interval}
\end{align}
Finally, in Subsection \ref{sec_SE_vAC_BLF} we prove estimates for the bilinear form $\check{B}_{\varepsilon,t}^C$ associated to $\check{\Lc}_{\varepsilon,t}^C$ on $\check{V}_{\varepsilon,t}\times\check{V}_{\varepsilon,t}$, $\check{V}_{\varepsilon,t}^\perp\times \check{V}_{\varepsilon,t}^\perp$ and $\check{V}_{\varepsilon,t}\times \check{V}_{\varepsilon,t}^\perp$. Here for $\vec{\phi},\vec{\psi}\in H^1(\Omega_t^C)^m$ we set
\begin{align}\label{eq_SE_vAC_Bepst}
&\check{B}_{\varepsilon,t}^C(\vec{\phi},\vec{\psi}):=\int_{\Omega_t^C}\nabla\vec{\phi}:\nabla\vec{\psi}\,dx\\
&+\int_{\Omega_t^C}(\vec{\phi},\left[\frac{1}{\varepsilon^2}D^2W(\vec{\theta}_0|_{\check{\rho}_\varepsilon(.,t)})+\frac{1}{\varepsilon}\sum_{\xi\in\N_0^m,|\xi|=1}\partial^{\xi}D^2W(\vec{\theta}_0|_{\check{\rho}_\varepsilon(.,t)})(\vec{u}^C_1|_{(.,t)})^\xi\right]\vec{\psi})_{\R^m}\,dx.\notag
\end{align}

\subsubsection{Asymptotic Expansion for the Approximate Eigenfunction}\label{sec_SE_vAC_asym}
\begin{proof}[Asymptotic Expansion of $\sqrt{\varepsilon}\check{\Lc}_{\varepsilon,t}^C\vec{\phi}^A_\varepsilon(.,t)$.] 
	First, we expand $\Delta\vec{v}^I_\varepsilon$ as in the inner expansion in Section \ref{sec_asym_vAC_in}. The lowest order $\Oc(\frac{1}{\varepsilon^2})$ equals $\frac{1}{\varepsilon^2}|\nabla r|^2|_{\overline{X}_0(s,t)}\vec{\theta}_0'''(\rho)\check{q}(s,t)=\frac{1}{\varepsilon^2}\vec{\theta}_0'''(\rho)\check{q}(s,t)$. In $\sqrt{\varepsilon}\check{\Lc}_{\varepsilon,t}^C\vec{\phi}^A_\varepsilon(.,t)$ the latter cancels with $\frac{1}{\varepsilon^2}D^2W(\vec{\theta}_0(\rho))\vec{\theta}_0'(\rho)\check{q}(s,t)$. At $\Oc(\frac{1}{\varepsilon})$ in $\Delta\vec{v}^I_\varepsilon$ we obtain 
	\begin{align*}
	&\frac{1}{\varepsilon}\vec{\theta}_0'''(\rho)\check{q}(s,t)\left[(\rho+\check{h}_1)\partial_r(|\nabla r|^2\circ\overline{X})|_{(0,s,t)}-2(D_xs\nabla r)^\top|_{\overline{X}_0(s,t)}\nabla_\Sigma\check{h}_1\right]\\
	&+\frac{1}{\varepsilon}\vec{\theta}_0''(\rho)\left[\Delta r|_{\overline{X}_0(s,t)}\check{q}(s,t)+2(D_xs\nabla r)^\top|_{\overline{X}_0(s,t)}\nabla_\Sigma \check{q}(s,t)\right]=\frac{1}{\varepsilon}\vec{\theta}_0''(\rho)\Delta r|_{\overline{X}_0(s,t)}\check{q}(s,t).
	\end{align*} 
	The term $\frac{1}{\varepsilon}\Delta r|_{\overline{X}_0(s,t)}\check{q}(s,t)\vec{\theta}_0''(\rho)$ is left as a remainder.\phantom{\qedhere}
	
	For $\varepsilon\Delta\vec{v}^C_\varepsilon$ we use the expansion in Section \ref{sec_asym_vAC_cp_bulk}, but without using a Taylor expansion for the $\check{h}_j$ because we only need the lowest order and we intended to reduce the regularity assumption on the $\check{h}_j$. More precisely, the $(x,t)$-terms in the formula for $\Delta\vec{v}^C_\varepsilon$ in Lemma \ref{th_asym_vAC_cp_trafo} are expanded solely with \eqref{eq_asym_ACND_cp_taylor3}. At the lowest order  $\Oc(\frac{1}{\varepsilon})$ we get $\frac{1}{\varepsilon}\Delta^{\sigma,t}\check{v}^C$, where $\Delta^{\sigma,t}:=\partial_\rho^2+|\nabla b|^2|_{\overline{X}_0(\sigma,t)}\partial_H^2$ for $(\sigma,t)\in\partial\Sigma\times[0,T]$. Moreover, the $W$-parts yield 
	\[
	\frac{1}{\varepsilon}D^2W(\vec{\theta}_0(\rho))\check{v}^C+\frac{1}{\varepsilon}\sum_{\xi\in\N_0^m,|\xi|=1}\partial^{\xi}D^2W(\vec{\theta}_0(\rho))(\check{u}^C_1)^\xi\check{v}^I.
	\] 
	To obtain an equation for $\check{v}^C$ in $(\rho,H,\sigma,t)$ we use a Taylor expansion for $\check{q}(Y(\sigma,.),t)|_{[0,2\mu_0]}$: 
	\[
	\check{q}(Y(\sigma,\varepsilon H),t)=\check{q}(\sigma,t)+\Oc(\varepsilon H)\quad\text{ for }(\sigma,\varepsilon H)\in\partial\Sigma\times[0,2\mu_0].
	\]
	Therefore we require in $\overline{\R^2_+}\times\partial\Sigma\times[0,T]$:
	\begin{align}\label{eq_SE_vAC_asym_bulk}
	\left[-\Delta^{\sigma,t}+D^2W(\vec{\theta}_0(\rho))\right]
	\check{v}^C=
	-\sum_{\xi\in\N_0^m,|\xi|=1}\partial^{\xi}D^2W(\vec{\theta}_0(\rho))(\check{u}^C_1|_{(\rho,H,\sigma,t)})^\xi\vec{\theta}_0'|_\rho \check{q}|_{(\sigma,t)}.
	\end{align}\end{proof}

\begin{proof}[Asymptotic Expansion of $\sqrt{\varepsilon}\partial_{N_{\partial\Omega}}\vec{\phi}^A_\varepsilon(.,t)$.] We expand as in Section \ref{sec_asym_vAC_cp_neum}. 
	Note that in $\overline{\Omega^C_t}$
	\begin{align}\begin{split}\label{eq_SE_vAC_nabla_vI_vC}
	D_x\vec{v}^I_\varepsilon
	&=\check{q}|_{(s,t)}\vec{\theta}_0''(\check{\rho}_\varepsilon)
	\left[\frac{\nabla r}{\varepsilon}-(D_xs)^\top \nabla_\Sigma \check{h}_\varepsilon|_{(s,t)}\right]^\top
	+\vec{\theta}_0'(\check{\rho}_\varepsilon)(\nabla_\Sigma \check{q}|_{(s,t)})^\top D_xs,\\
	D_x\vec{v}^C_\varepsilon  &=\partial_\rho\check{v}^C
	\left[\frac{\nabla r}{\varepsilon}-(D_xs)^\top\nabla_\Sigma \check{h}_\varepsilon|_{(s,t)}\right]^\top
	+\partial_H\check{v}^C\left[\frac{\nabla b}{\varepsilon}\right]^\top
	+D_{\partial\Sigma}\check{v}^C D_x\sigma,\end{split}
	\end{align}
	where the $\check{v}^C$-terms are evaluated at $(\check{\rho}_\varepsilon,H_\varepsilon,\sigma,t)$.
	In $\sqrt{\varepsilon}\partial_{N_{\partial\Omega}}\vec{\phi}^A_\varepsilon(.,t)=\sqrt{\varepsilon}D_x\vec{\phi}^A_\varepsilon N_{\partial\Omega}$ the lowest order is $\Oc(\frac{1}{\varepsilon})$ and given by $\frac{1}{\varepsilon}(N_{\partial\Omega}\cdot\nabla r)|_{\overline{X}_0(\sigma,t)}\vec{\theta}_0''(\rho)\check{q}|_{(\sigma,t)}=0$ due to the $90$°-contact angle condition. At $\Oc(1)$ we obtain
	\begin{align*}
	\check{q}|_{(\sigma,t)}\vec{\theta}_0''(\rho)\left[(\rho+\check{h}_1|_{(\sigma,t)})\partial_r((N_{\partial\Omega}\cdot\nabla r)\circ\overline{X})|_{(0,\sigma,t)}-(D_xsN_{\partial\Omega})^\top|_{\overline{X}_0(\sigma,t)}\nabla_\Sigma \check{h}_1|_{(\sigma,t)}\right]\\
	+(D_xsN_{\partial\Omega})^\top|_{\overline{X}_0(\sigma,t)} \nabla_\Sigma \check{q}|_{(\sigma,t)}\vec{\theta}_0'(\rho)+0\cdot\partial_\rho\check{v}^C|_{H=0}+(N_{\partial\Omega}\cdot\nabla b)|_{\overline{X}_0(\sigma,t)}\partial_H\check{v}^C|_{H=0}.
	\end{align*}
	The latter is zero if and only if\phantom{\qedhere}
	\begin{align*}
	&(N_{\partial\Omega}\cdot\nabla b)|_{\overline{X}_0(\sigma,t)}\partial_H\check{v}^C|_{H=0}=-(D_xsN_{\partial\Omega})^\top|_{\overline{X}_0(\sigma,t)}\nabla_\Sigma \check{q}|_{(\sigma,t)}\vec{\theta}_0'(\rho)\\
	&+\check{q}|_{(\sigma,t)}\vec{\theta}_0''(\rho)
	\left[(D_xsN_{\partial\Omega})^\top|_{\overline{X}_0(\sigma,t)}\nabla_\Sigma \check{h}_1|_{(\sigma,t)}-(\rho+\check{h}_1|_{(\sigma,t)})\partial_r((N_{\partial\Omega}\cdot\nabla r)\circ\overline{X})|_{(0,\sigma,t)}\right].
	\end{align*}
	Analogously to the scalar case, cf.~Section \ref{sec_SE_ACND_asym}, we leave the term with $\nabla_\Sigma\check{h}_1$ as a remainder in order to lower the needed regularity for $\check{h}_1$. Therefore due to \eqref{eq_SE_vAC_asym_bulk} we require
	\begin{alignat}{2}\label{eq_SE_vAC_asym_vbar1}
	\!\!\![-\Delta+D^2W(\vec{\theta}_0)]\underline{v}^C&=-\!\sum_{\xi\in\N_0^m,|\xi|=1}\partial^\xi D^2W(\vec{\theta}_0)(\underline{u}^C_1)^\xi\vec{\theta}_0'\check{q}|_{(\sigma,t)}&\quad\!\!&\text{ in }\overline{\R^2_+}\times\partial\Sigma\times[0,T],\\
	-\partial_H\underline{v}^C|_{H=0}&=(|\nabla b|/N_{\partial\Omega}\cdot\nabla b)|_{\overline{X}_0(\sigma,t)}\vec{g}^{C}&\quad\!\!&\text{ in }\R\times\partial\Sigma\times[0,T],\label{eq_SE_vAC_asym_vbar2}
	\end{alignat}
	where $\underline{v}^C,\underline{u}^C_1:\overline{\R^2_+}\times\partial\Sigma\times[0,T]\rightarrow\R^m$ are associated to $\check{v}^C$ and $\check{u}^C_1$ analogous to \eqref{eq_asym_vAC_cp_ubar} in Section \ref{sec_asym_vAC_cp_neum_0}, respectively, and we define $\vec{g}^C(\rho,\sigma,t)$ for all $(\rho,\sigma,t)\in\R\times\partial\Sigma\times[0,T]$ as
	\[
	-\check{q}|_{(\sigma,t)}\vec{\theta}_0''(\rho)(\rho+\check{h}_1|_{(\sigma,t)})\partial_r((N_{\partial\Omega}\cdot\nabla r)\circ\overline{X})|_{(0,\sigma,t)}
	-(D_xsN_{\partial\Omega})^\top|_{\overline{X}_0(\sigma,t)}\nabla_\Sigma\check{q}|_{(\sigma,t)}\vec{\theta}_0'(\rho).
	\]
	The right hand sides in \eqref{eq_SE_vAC_asym_vbar1}-\eqref{eq_SE_vAC_asym_vbar2} are of class $B([0,T];C^2(\partial\Sigma,H^2_{(\beta,\gamma)}(\R^2_+)^m\times H^{5/2}_{(\beta)}(\R)^m))$ for some $\beta,\gamma>0$ if $(\check{q},\nabla_\Sigma \check{q})|_{\partial\Sigma\times[0,T]}\in B([0,T],C^2(\partial\Sigma))^{1+N}$. For simplicity we require $\check{q}=1$ on $\partial\Sigma\times[0,T]$. Then the compatibility condition \eqref{eq_hp_vect_comp} for \eqref{eq_SE_vAC_asym_vbar1}-\eqref{eq_SE_vAC_asym_vbar2} is equivalent to 
	\begin{align}\label{eq_SE_vAC_nabla_q}
	&-(D_xsN_{\partial\Omega})^\top|_{\overline{X}_0(\sigma,t)}\nabla_\Sigma \check{q}|_{(\sigma,t)}=\check{g}^C|_{(\sigma,t)},\\
	\check{g}^C|_{(\sigma,t)}&:=\frac{1}{\|\vec{\theta}_0'\|_{L^2(\R)^m}^2}\left[\partial_r((N_{\partial\Omega}\cdot\nabla r)\circ\overline{X})|_{(0,\sigma,t)}\right.\int_\R\rho\vec{\theta}_0'(\rho)\cdot\vec{\theta}_0''(\rho)\,d\rho \notag\\
	&+\left.\frac{N_{\partial\Omega}\cdot\nabla b}{|\nabla b|}\right|_{\overline{X}_0(\sigma,t)}\int_{\R^2_+}(\vec{\theta}_0',\sum_{\xi\in\N_0^m,|\xi|=1}\left.\partial^\xi D^2W(\vec{\theta}_0)(\underline{u}^C_1|_{(\rho,H,\sigma,t)})^\xi\vec{\theta}_0')_{\R^m}\,d(\rho,H)\right].\notag
	\end{align}
    Here because of the assumptions it holds $\check{g}^C\in B([0,T],C^2(\partial\Sigma))$.
	
	Analogously to the scalar case, cf.~Section \ref{sec_SE_ACND_asym}, it is possible to construct $\check{q}\in B([0,T],C^2(\hat{\Sigma}))$ with $\nabla_\Sigma \check{q}|_{\partial\Sigma\times[0,T]}\in B([0,T],C^2(\partial\Sigma))$ and $\check{q}=1$ on $(\partial\Sigma\cup Y(\partial\Sigma,[\mu_0,2\mu_0]))\times[0,T]$ and such that $c\leq \check{q}\leq C$ for some $c,C>0$ and \eqref{eq_SE_vAC_nabla_q} holds. Hence Remark~\ref{th_hp_vect_exp_rem} yields a unique solution of \eqref{eq_SE_vAC_asym_vbar1}-\eqref{eq_SE_vAC_asym_vbar2} such that for some $\beta,\gamma>0$
	\[\underline{v}^C\in B([0,T];C^2(\partial\Sigma,H^4_{(\beta,\gamma)}(\R^2_+)^m))\hookrightarrow B([0,T];C^2(\partial\Sigma,C^2_{(\beta,\gamma)}(\overline{\R^2_+})^m)).
	\]\end{proof}

\begin{Remark}\upshape\phantomsection{\label{th_SE_vAC_asym_rem}}
	\begin{enumerate}
		\item Consider the situation of Section \ref{sec_asym_vAC}. Then $\check{h}_1$ is smooth and the $\nabla_\Sigma \check{h}_1$-term can be included above. Moreover, $\underline{u}^C_1$ is smooth and solves \eqref{eq_asym_vAC_cp_uC1}-\eqref{eq_asym_vAC_cp_uC2}. However note that in general it is not valid to use $\check{v}^C:=\partial_\rho\underline{u}^C_1$ in analogy to Remark \ref{th_SE_ACND_asym_rem},~1.~in the scalar case. This is because in general
		\[
		\partial_\rho(D^2W(\vec{\theta}_0))\underline{u}^C_1=\sum_{\xi\in\N_0^m,|\xi|=1}\partial_\xi D^2W(\vec{\theta}_0)(\vec{\theta}_0')^\xi\underline{u}^C_1\neq\sum_{\xi\in\N_0^m,|\xi|=1}\partial_\xi D^2W(\vec{\theta}_0)(\underline{u}^C_1)^\xi\vec{\theta}_0'.
		\]
		\item In the case of additional terms in $\vec{u}^A_\varepsilon$ as in Remark \ref{th_SE_vAC_rem} one can proceed analogously as in the scalar 2D-case, cf.~\cite{MoserDiss}, Remark 6.18, 2.~or Remark 4.1, 2.~in \cite{AbelsMoser}. More precisely, there is an additional term 
		\[
		\frac{1}{\varepsilon}\sum_{\xi\in\N_0^m,|\xi|=1}\partial^\xi D^2W(\vec{\theta}_0|_{\check{\rho}_\varepsilon(.,t)})(\vec{\theta}_1|_{\check{\rho}_\varepsilon(.,t)})^\xi \check{p}_\varepsilon(s(.,t),t)
		\] 
		in the operator $\check{\Lc}^C_{\varepsilon,t}$. Therefore in the ansatz for $\vec{\phi}^A_\varepsilon$ we add $\varepsilon \check{v}_1(\check{\rho}_\varepsilon(.,t))\check{q}_{1,\varepsilon}(s(.,t),t)$, where $\check{v}_1:\R\rightarrow\R^m$ and $\check{q}_{1,\varepsilon}$ is $\R$-valued. Hence in the asymptotic expansion of $\sqrt{\varepsilon}\check{\Lc}_{\varepsilon,t}^C\vec{\phi}^A_\varepsilon(.,t)$ new terms appear at order $\Oc(\frac{1}{\varepsilon})$, namely
		\[
		\frac{1}{\varepsilon}[-\partial_\rho^2+D^2W(\vec{\theta}_0)]\check{v}_1|_\rho\check{q}_{1,\varepsilon}|_{(s,t)}+\frac{1}{\varepsilon}\sum_{\xi\in\N_0^m,|\xi|=1}\partial^\xi D^2W(\vec{\theta}_0)(\vec{\theta}_1)^\xi\vec{\theta}_0'|_\rho(\check{p}_\varepsilon\check{q})|_{(s,t)}.
		\]
		Therefore we define $\check{q}_{1,\varepsilon}:=\check{p}_\varepsilon\check{q}$ and look for a solution of
		\[
		[-\partial_\rho^2+D^2W(\vec{\theta}_0)]\check{v}_1=-\sum_{\xi\in\N_0^m,|\xi|=1}\partial^\xi D^2W(\vec{\theta}_0)(\vec{\theta}_1)^\xi\vec{\theta}_0'.
		\] 
		The right hand side is an element of $C^0_{(\beta)}(\R)^m$ for some $\beta>0$ and \eqref{eq_SE_vAC_theta1} holds. Therefore Theorem \ref{th_ODE_vect_lin} yields a unique solution $\check{v}_1\in H^2_{(\beta)}(\R)^m$ for possibly smaller $\beta$. With embeddings and the equation it follows that $\check{v}_1\in C^2_{(\beta)}(\R)^m$ for some $\beta>0$. Then below analogous arguments can be used.
		\item The behaviour of $\vec{\phi}^A_\varepsilon(x,t)$ for $x\in\Omega^C_t$ with $b(x,t)\in[\frac{7}{4}\mu_0,2\mu_0]$ does not matter because we only have to consider $\vec{\psi}\in\check{H}^1(\Omega^C_t)$ in Theorem \ref{th_SE_vAC_cp}, where $\check{H}^1(\Omega^C_t)$ is from \eqref{eq_SE_vAC_H1tilde_Omega}.
	\end{enumerate}
\end{Remark}
\begin{Lemma}\label{th_SE_vAC_phi_A}
	The function $\vec{\phi}^A_\varepsilon(.,t)$ is $C^2(\overline{\Omega_t^C})^m$ and satisfies uniformly in $t\in[0,T]$:
	\begin{alignat*}{2}
	\left|\sqrt{\varepsilon}\check{\Lc}_{\varepsilon,t}^C\vec{\phi}^A_\varepsilon(.,t)+\frac{1}{\varepsilon}\Delta r|_{\overline{X}_0(s(.,t),t)}\check{q}|_{(s(.,t),t)}\vec{\theta}_0''|_{\check{\rho}_\varepsilon(.,t)}\right|&\leq Ce^{-c|\check{\rho}_\varepsilon(.,t)|}&\quad&\text{ in }\Omega_t^C,\\
	\left|\sqrt{\varepsilon}D_x\vec{\phi}^A_\varepsilon|_{(.,t)}N_{\partial\Omega^C_t}+[(D_xsN_{\partial\Omega})^\top|_{\overline{X}_0}\nabla_\Sigma \check{h}_1]|_{(\sigma(.,t),t)}\vec{\theta}_0''\right|&\leq C\varepsilon e^{-c|\check{\rho}_\varepsilon(.,t)|}&\quad&\text{ on }\partial\Omega_t^C\cap\partial\Omega,\\
	\left|\sqrt{\varepsilon}D_x\vec{\phi}^A_\varepsilon|_{(.,t)}N_{\partial\Omega^C_t}\right|&\leq Ce^{-c/\varepsilon}&\quad&\text{ on }\partial\Omega_t^C\setminus\Gamma_t(\delta).
	\end{alignat*}
\end{Lemma}
\begin{proof}
	The construction yields the regularity for $\vec{\phi}^A_\varepsilon$ and rigorous remainder estimates for the Taylor expansions above imply the estimates, cf.~the proof of Lemma 4.4 in \cite{AbelsMoser} in the scalar 2D-case.
\end{proof}

\subsubsection{The Splitting}\label{sec_SE_vAC_splitting} 
Analogously to the scalar case we use a characterization for the splitting of $\check{H}^1(\Omega^C_t)$.
\begin{Lemma}\phantomsection{\label{th_SE_vAC_split_L2}}
	Let $\check{H}^1(\Omega^C_t)$, $\check{V}_{\varepsilon,t}$ and $\check{H}^1(\hat{\Sigma}^\circ)$ be as in \eqref{eq_SE_vAC_H1tilde_Omega}-\eqref{eq_SE_vAC_H1check_interval}. Then
	\begin{enumerate}
		\item $\check{V}_{\varepsilon,t}$ is a subspace of $\check{H}^1(\Omega_t^C)$ and for $\varepsilon_0>0$ small there are $\check{c}_1,\check{C}_1>0$ such that 
		\[
		\check{c}_1\|\check{a}\|_{L^2(\hat{\Sigma})}\leq\|\vec{\psi}\|_{L^2(\Omega_t^C)^m}\leq \check{C}_1\|\check{a}\|_{L^2(\hat{\Sigma})}
		\]
		for all $\vec{\psi}=\check{a}(s(.,t))\vec{\phi}^A_\varepsilon(.,t)\in \check{V}_{\varepsilon,t}$ and $\varepsilon\in(0,\varepsilon_0]$, $t\in[0,T]$.
		\item Let $\check{V}_{\varepsilon,t}^\perp$ be the $L^2$-orthogonal complement of $\check{V}_{\varepsilon,t}$ in $\check{H}^1(\Omega_t^C)$. Then for $\vec{\psi}\in\check{H}^1(\Omega_t^C)$:
		\[
		\vec{\psi}\in \check{V}_{\varepsilon,t}^\perp\quad\Leftrightarrow\quad\int_{-\delta}^\delta(\vec{\phi}^A_\varepsilon(.,t)\cdot\vec{\psi})|_{X(r,s,t)}J_t(r,s)\,dr=0\quad\text{ for a.e.~}s\in\hat{\Sigma}.
		\]
		Moreover, $\check{H}^1(\Omega_t^C)=\check{V}_{\varepsilon,t}\oplus \check{V}_{\varepsilon,t}^\perp$ for all $\varepsilon\in(0,\varepsilon_0]$ and $\varepsilon_0>0$ small.
	\end{enumerate}
\end{Lemma}
\begin{proof}[Proof.] 
	This follows in the analogous way as in the scalar case, cf.~the proof of Lemma \ref{th_SE_ACND_split_L2}. Here note that $\vec{\theta}_0'(0)\neq 0$ and therefore $\int_\R|\vec{\theta}_0'|^2>0$ due to Theorem \ref{th_ODE_vect}.
\end{proof}

\subsubsection{Analysis of the Bilinear Form}\label{sec_SE_vAC_BLF}
First we consider $\check{B}_{\varepsilon,t}^C$ on $\check{V}_{\varepsilon,t}\times\check{V}_{\varepsilon,t}$.
\begin{Lemma}\label{th_SE_vAC_VxV}
	There are $\varepsilon_0,C,c>0$ such that 
	\[
	\check{B}_{\varepsilon,t}^C(\vec{\phi},\vec{\phi})\geq-C\|\vec{\phi}\|_{L^2(\Omega_t^C)^m}^2+c\|\check{a}\|_{H^1(\hat{\Sigma}^\circ)}^2
	\]
	for all $\vec{\phi}=\check{a}(s(.,t))\vec{\phi}^A_\varepsilon(.,t)\in \check{V}_{\varepsilon,t}$ and $\varepsilon\in(0,\varepsilon_0],t\in[0,T]$.
\end{Lemma}
\begin{proof} 
	Let $\vec{\phi}$ be as in the lemma. For $j=1,...,m$ we denote with $\phi_j$, $\phi^A_{\varepsilon,j}$ the $j$-th component of $\vec{\phi}$, $\vec{\phi}^A_\varepsilon$, respectively. Then
	$\nabla\phi_j=
	\nabla(\check{a}|_{s(.,t)})\phi^A_{\varepsilon,j}(.,t)+\check{a}|_{s(.,t)}\nabla(\phi^A_{\varepsilon,j})(.,t)$ for $j=1,...,m$ and 
	\[
	|\nabla\vec{\phi}|^2=|\nabla(\check{a}(s))|^2|\vec{\phi}^A_\varepsilon|^2|_{(.,t)}+\check{a}^2(s)|\nabla\vec{\phi}^A_\varepsilon|^2|_{(.,t)}+\sum_{j=1}^m\nabla(\check{a}^2(s))\cdot\nabla(\phi^A_{\varepsilon,j})\phi^A_{\varepsilon,j}|_{(.,t)}.
	\]
	Integration by parts, cf.~Remark \ref{th_SE_ACND_sob_rem},~2., yields
	\begin{align*}
	\sum_{j=1}^m\int_{\Omega_t^C}\left[\nabla(\check{a}^2(s))\cdot\nabla(\phi^A_{\varepsilon,j})\phi^A_{\varepsilon,j}\right]|_{(.,t)}\,dx&=-\int_{\Omega_t^C}\left[\check{a}^2(s)(\Delta\vec{\phi}^A_\varepsilon\cdot\vec{\phi}^A_\varepsilon+|\nabla\vec{\phi}^A_\varepsilon|^2)\right]|_{(.,t)}\,dx\\
	&+\int_{\partial\Omega_t^C}\left[D_x\vec{\phi}^A_\varepsilon N_{\partial\Omega^C_t}\cdot\,\tr(\check{a}^2(s)\vec{\phi}^A_\varepsilon|_{(.,t)})\right]\,d\Hc^{N-1}.
	\end{align*}
	Therefore we obtain
	\begin{align*}
	\check{B}_{\varepsilon,t}^C(\vec{\phi},\vec{\phi})&=\int_{\Omega_t^C}|\nabla(\check{a}(s))|^2|\vec{\phi}^A_\varepsilon|^2|_{(.,t)}\,dx+\int_{\Omega_t^C}(\check{a}^2(s)\vec{\phi}^A_\varepsilon)|_{(.,t)}\cdot\check{\Lc}_{\varepsilon,t}^C\vec{\phi}^A_\varepsilon|_{(.,t)}\,dx\\
	&+\int_{\partial\Omega_t^C}\left[D_x\vec{\phi}^A_\varepsilon N_{\partial\Omega^C_t}\cdot\,\tr(\check{a}^2(s)\vec{\phi}^A_\varepsilon|_{(.,t)})\right]\,d\Hc^{N-1}=:(I)+(II)+(III).
	\end{align*}
	
	Due to integration by parts it holds $\int_\R\vec{\theta}_0'\cdot\vec{\theta}_0''=0$.
	Therefore one can estimate $(I)$-$(III)$ in the analogous way as in the scalar case with Lemma \ref{th_SE_vAC_phi_A}, cf.~the proof of Lemma \ref{th_SE_ACND_VxV}. 
\end{proof}

Next we analyze $\check{B}^C_{\varepsilon,t}$ on $\check{V}_{\varepsilon,t}^\perp\times \check{V}_{\varepsilon,t}^\perp$.
\begin{Lemma}\label{th_SE_vAC_VpxVp}
	There are $\check{\nu},\check{\varepsilon}_0>0$ such that
	\[
	\check{B}^C_{\varepsilon,t}(\vec{\psi},\vec{\psi})\geq
	\check{\nu}\left[\frac{1}{\varepsilon^2}\|\vec{\psi}\|_{L^2(\Omega_t^C)^m}^2+\|\nabla\vec{\psi}\|_{L^2(\Omega_t^C)^{N\times m}}^2\right]
	\]  
	for all $\vec{\psi}\in\check{V}_{\varepsilon,t}^\perp$ and $\varepsilon\in(0,\check{\varepsilon}_0]$, $t\in[0,T]$.
\end{Lemma}
\begin{proof}
	As in the scalar 2D-case, cf.~the proof of Lemma 4.8 in \cite{AbelsMoser}, it is enough to show for some $\nu, \varepsilon_0>0$ 
	\begin{align}\label{eq_SE_vAC_VpxVp_1}
	\check{B}_{\varepsilon,t}(\vec{\psi},\vec{\psi}):=\int_{\Omega^C_t}|\nabla\vec{\psi}|^2+\frac{1}{\varepsilon^2}(\vec{\psi},D^2W(\vec{\theta}_0|_{\check{\rho}_\varepsilon(.,t)})\vec{\psi})_{\R^m}\,dx
	\geq\frac{\nu}{\varepsilon^2}\|\vec{\psi}\|_{L^2(\Omega^C_t)^m}^2
	\end{align}
	for all $\vec{\psi}\in \check{V}_{\varepsilon,t}^\perp$ and $\varepsilon\in(0,\varepsilon_0]$, $t\in[0,T]$.
	
	Analogously to the scalar case we prove \eqref{eq_SE_vAC_VpxVp_1} by reducing to Neumann boundary problems in normal direction. Therefore let $\check{\psi}_t:=\vec{\psi}|_{X(.,t)}$ for $\vec{\psi}\in \check{V}_{\varepsilon,t}^\perp$. Then $\check{\psi}_t\in H^1((-\delta,\delta)\times\hat{\Sigma}^\circ)^m$ and
	\begin{align}\label{eq_SE_vAC_VpxVp_2}
	|\nabla\vec{\psi}|^2|_{X(.,t)}\geq
	(1-Cr^2)|\partial_r\check{\psi}_t|^2 
	+c|\nabla_{\hat{\Sigma}}\check{\psi}_t|^2
	\end{align}
	in $\Omega_t^C$ for some $c,C>0$ due to \eqref{eq_SE_ACND_VpxVp_2} for every component. We do not use the second term here. To get $Cr^2$ small enough (which will be precise later), we fix $\check{\delta}>0$ small and estimate separately for $r$ in
	\[
	I_{s,t}^\varepsilon:=(-\check{\delta},\check{\delta})+\varepsilon \check{h}_\varepsilon(s,t)\quad\text{ and }\quad \check{I}_{s,t}^\varepsilon:=(-\delta,\delta)\setminus I_{s,t}^\varepsilon.
	\]
	If $\varepsilon_0=\varepsilon_0(\check{\delta},\check{C}_0)>0$ is small, then for all $\varepsilon\in(0,\varepsilon_0]$ and $s\in\hat{\Sigma}$, $t\in[0,T]$ it holds
	\[
	D^2W(\vec{\theta}_0(\check{\rho}_\varepsilon|_{\overline{X}(r,s,t)}))\geq \check{c}_0I\quad \text{ for }r\in\check{I}_{s,t}^{\varepsilon},\quad
	|r|\leq \check{\delta}+\varepsilon|\check{h}_\varepsilon(s,t)|\leq 2\check{\delta}\quad \text{ for }r\in I_{s,t}^{\varepsilon},
	\]
	where $\check{c}_0>0$. Let $\check{c}=\check{c}(\check{\delta}):=4C\check{\delta}^2$ with $C$ from \eqref{eq_SE_vAC_VpxVp_2}. Then for all $\varepsilon\in(0,\varepsilon_0]$ we obtain
	\begin{align*}
	&\check{B}_{\varepsilon,t}(\vec{\psi},\vec{\psi})\geq\int_{\hat{\Sigma}}\int_{\check{I}_{s,t}^\varepsilon}\frac{\check{c}_0}{\varepsilon^2}|\check{\psi}_t|^2 J_t|_{(r,s)}\,dr\,d\Hc^{N-1}(s)\\
	&+\int_{\hat{\Sigma}}\int_{I_{s,t}^\varepsilon}\left[(1-\check{c})|\partial_r\check{\psi}_t|^2+\frac{1}{\varepsilon^2}(\check{\psi}_t,D^2W(\vec{\theta}_0(\check{\rho}_\varepsilon|_{\overline{X}(.,t)}))\check{\psi}_t)_{\R^m}\right] J_t|_{(r,s)}\,dr\,d\Hc^{N-1}(s).
	\end{align*}
	
	We set $\check{F}_{\varepsilon,s,t}(z):=\varepsilon(z+\check{h}_\varepsilon(s,t))$ and $\check{J}_{\varepsilon,s,t}(z):=J_t(\check{F}_{\varepsilon,s,t}(z),s)$ for all $z\in[-\frac{\delta}{\varepsilon},\frac{\delta}{\varepsilon}]-\check{h}_\varepsilon(s,t)$ and $(s,t)\in\Sigma\times[0,T]$. Moreover, we define $I_{\varepsilon,\check{\delta}}:=(-\frac{\check{\delta}}{\varepsilon},\frac{\check{\delta}}{\varepsilon})$ 
	and $\vec{\Psi}_{\varepsilon,s,t}:=\sqrt{\varepsilon}
	\check{\psi}_t(\check{F}_{\varepsilon,s,t}(.),s)$. Due to Remark \ref{th_SE_ACND_sob_rem} it holds $\vec{\Psi}_{\varepsilon,s,t}\in H^1(I_{\varepsilon,\check{\delta}})^m$ for a.e.~$s\in\hat{\Sigma}$ and all $t\in[0,T]$. Together with Lemma \ref{th_SE_1Dtrafo_remainder},~1.~it follows that the second inner integral in the estimate above equals $1/\varepsilon^2$ times
	\begin{align*}
	\check{B}_{\varepsilon,s,t}^{\check{c}}(\vec{\Psi}_{\varepsilon,s,t},\vec{\Psi}_{\varepsilon,s,t})
	:=\int_{I_{\varepsilon,\check{\delta}}}\left[(1-\check{c})|\frac{d}{dz}\vec{\Psi}_{\varepsilon,s,t}|^2+(\vec{\Psi}_{\varepsilon,s,t},D^2W(\vec{\theta}_0(z))\vec{\Psi}_{\varepsilon,s,t}))_{\R^m}\right]\check{J}_{\varepsilon,s,t}\,dz
	\end{align*} 
	for a.e.~$s\in\hat{\Sigma}$ and all $t\in[0,T]$.
	Therefore \eqref{eq_SE_vAC_VpxVp_1} follows if we show with the same $\check{c}_0$ as above
	\begin{align}\label{eq_SE_vAC_VpxVp_3}
	\check{B}_{\varepsilon,s,t}^{\check{c}}(\vec{\Psi}_{\varepsilon,s,t},\vec{\Psi}_{\varepsilon,s,t})
	\geq \overline{c}\|\vec{\Psi}_{\varepsilon,s,t}\|^2_{L^2(I_{\varepsilon,\check{\delta}},\check{J}_{\varepsilon,s,t})^m}-\frac{\check{c}_0}{2}\|\check{\psi}_t(.,s)\|^2_{L^2(\check{I}_{s,t}^\varepsilon,J_t(.,s))^m}
	\end{align}
	for $\varepsilon\in(0,\varepsilon_0]$,  a.e.~$s\in\hat{\Sigma}$ and all $t\in[0,T]$ with some $\varepsilon_0,\overline{c}>0$ independent of $\varepsilon,s,t$.
	
	The estimate \eqref{eq_SE_vAC_VpxVp_3} follows for suitable small $\check{\delta}$ analogously as in the scalar 2D-case, cf.~the proof of Lemma 4.8 in \cite{AbelsMoser}. One uses the integral characterization for $\vec{\psi}\in \check{V}_{\varepsilon,t}^\perp$ due to Lemma \ref{th_SE_vAC_split_L2},~2.~and spectral properties for the operator
	\[
	\check{\Lc}_{\varepsilon,s,t}^0:=-(\check{J}_{\varepsilon,s,t})^{-1}\frac{d}{dz}\left(\check{J}_{\varepsilon,s,t}\frac{d}{dz}\right)+D^2W(\vec{\theta}_0)
	\]
	on $H^2(I_{\varepsilon,\check{\delta}})^m$ with homogeneous Neumann boundary condition, see Section \ref{sec_SE_1Dvect_pert} and in particular Theorem \ref{th_SE_1Dvect_pert2}. 
\end{proof}

For $\check{B}_{\varepsilon,t}^C$ on $\check{V}_{\varepsilon,t}\times \check{V}_{\varepsilon,t}^\perp$ we obtain
\begin{Lemma}\label{th_SE_vAC_VxVp}
	There are $\varepsilon_0,C>0$ such that
	\[
	|\check{B}_{\varepsilon,t}^C(\vec{\phi},\vec{\psi})|\leq\frac{C}{\varepsilon}\|\vec{\phi}\|_{L^2(\Omega_t^C)^m}\|\vec{\psi}\|_{L^2(\Omega_t^C)^m}+\frac{1}{4}\check{B}_{\varepsilon,t}^C(\vec{\psi},\vec{\psi})+C\varepsilon\|\check{a}\|_{H^1(\hat{\Sigma}^\circ)}^2
	\]
	for all $\vec{\phi}=\check{a}(s(.,t))\vec{\phi}^A_\varepsilon(.,t)\in \check{V}_{\varepsilon,t}$, $\vec{\psi}\in \check{V}_{\varepsilon,t}^\perp$ and $\varepsilon\in(0,\varepsilon_0]$, $t\in[0,T]$.
\end{Lemma}

\begin{proof}
	We rewrite $\check{B}_{\varepsilon,t}^C(\vec{\phi},\vec{\psi})$ in order to use Lemma \ref{th_SE_vAC_phi_A} and Lemma \ref{th_SE_vAC_split_L2}. Using notation as in the beginning of the proof of Lemma \ref{th_SE_vAC_VxV} we obtain with integration by parts
	\begin{align*}
	\int_{\Omega_t^C}\check{a}(s)\nabla\vec{\phi}^A_\varepsilon|_{(.,t)}:\nabla\vec{\psi}\,dx
	=-\sum_{j=1}^m\int_{\Omega_t^C}\left[\nabla(\check{a}(s))\cdot\nabla\phi^A_{\varepsilon,j}+\check{a}(s)\Delta\phi^A_{\varepsilon,j}|_{(.,t)}\right]\psi_j\,dx\\
	+\int_{\partial\Omega_t^C}D_x\phi^A_\varepsilon|_{(.,t)}N_{\partial\Omega^C_t}\cdot\tr\left[\check{a}(s(.,t))\vec{\psi}\right]\,d\Hc^{N-1}.
	\end{align*}
	Therefore it follows that
	\begin{align*}
	\check{B}_{\varepsilon,t}^C(\vec{\phi},\vec{\psi})
	=\int_{\Omega_t^C} \check{a}(s)|_{(.,t)}\vec{\psi}\cdot\check{\Lc}_{\varepsilon,t}^C\vec{\phi}^A_\varepsilon|_{(.,t)}\,dx
	+\int_{\partial\Omega_t^C}D_x\phi^A_\varepsilon|_{(.,t)}N_{\partial\Omega^C_t}\cdot\tr\left[\check{a}(s(.,t))\vec{\psi}\right]\,d\Hc^{N-1}\\
	+\int_{\Omega_t^C}\nabla(\check{a}(s))|_{(.,t)}\cdot\left[\sum_{j=1}^m\nabla\psi_j\,\phi^A_{\varepsilon,j}|_{(.,t)}-\psi_j\nabla\phi^A_{\varepsilon,j}|_{(.,t)}\right]\,dx=:(I)+(II)+(III).
	\end{align*}
	
	The terms $(I)$-$(III)$ can be estimated in the analogous way as in the scalar case, cf.~the proof of Lemma \ref{th_SE_ACND_VxVp}. The most important ingredients for the estimate of $(I)$ and $(II)$ are Lemma \ref{th_SE_vAC_phi_A} and Lemma \ref{th_SE_ACND_intpol_tr}. For $(III)$ one essentially uses the integral characterization for $\vec{\psi}\in\check{V}_{\varepsilon,t}^\perp$ from Lemma \ref{th_SE_vAC_split_L2},~2.~(by differentiating it) and the structure of $\vec{\phi}^A_\varepsilon$.
\end{proof}

Finally, we combine Lemma \ref{th_SE_vAC_VxV}-\ref{th_SE_vAC_VxVp}.

\begin{Theorem}\label{th_SE_vAC_cp2}
	There are $\check{\varepsilon}_0, \check{C}, \check{c}_0>0$ such that for all $\varepsilon\in(0,\check{\varepsilon}_0]$, $t\in[0,T]$ and every $\vec{\psi}\in H^1(\Omega_t^C)^m$ with $\vec{\psi}|_{X(.,s,t)}=0$ for a.e.~$s\in Y(\partial\Sigma\times[\frac{3}{2}\mu_0,2\mu_0])$ it holds
	\[
	\check{B}_{\varepsilon,t}^C(\vec{\psi},\vec{\psi})\geq -C\|\vec{\psi}\|_{L^2(\Omega_t^C)^m}^2+c_0\|\nabla_\tau\vec{\psi}\|_{L^2(\Omega_t^C)^{N\times m}}^2.
	\] 
\end{Theorem}

\begin{Remark}\upshape\begin{enumerate}
		\item The estimate can be refined, cf.~the proof below.
		\item Theorem \ref{th_SE_vAC_cp2} directly yields Theorem \ref{th_SE_vAC_cp}, cf.~the beginning of Section \ref{sec_SE_vAC_outline}.
	\end{enumerate}
\end{Remark}

\begin{proof}[Proof of Theorem \ref{th_SE_vAC_cp2}]
	Let $\vec{\psi}\in\check{H}^1(\Omega_t^C)^m$. Because of Lemma \ref{th_SE_vAC_split_L2} we can uniquely write
	\[
	\vec{\psi}=\vec{\phi}+\vec{\phi}^\perp\quad\text{ with } \vec{\phi}=[\check{a}(s)\vec{\phi}^A_\varepsilon]|_{(.,t)}\in \check{V}_{\varepsilon,t}\text{ and }
	\vec{\phi}^\perp\in \check{V}_{\varepsilon,t}^\perp.
	\]
	Analogously to the scalar case, cf.~the proof of Theorem \ref{th_SE_ACND_cp2}, we obtain from Lemma \ref{th_SE_vAC_VxV}-\ref{th_SE_vAC_VxVp}:
	\[
	\check{B}_{\varepsilon,t}^C(\vec{\psi},\vec{\psi})\geq-C\|\vec{\phi}\|_{L^2(\Omega_t^C)^m}^2+\frac{\check{\nu}}{4\varepsilon^2}\|\vec{\phi}^\perp\|_{L^2(\Omega_t^C)^m}^2+\frac{c_0}{2}\|\check{a}\|_{H^1(\hat{\Sigma}^\circ)}^2+\frac{\check{\nu}}{2}\|\nabla(\vec{\phi}^\perp)\|_{L^2(\Omega_t^C)^{N\times m}}^2
	\]
	for all $\varepsilon\in(0,\varepsilon_0]$ and $t\in[0,T]$, where $\varepsilon_0>0$ is small and $\check{\nu}$ is as in Lemma \ref{th_SE_vAC_VpxVp}. Moreover, as in the scalar case it follows that
	\[
	\|\nabla_\tau\vec{\psi}\|_{L^2(\Omega_t^C)^m}^2\leq C(\|\check{a}\|_{H^1(\hat{\Sigma}^\circ)}^2+\|\nabla(\vec{\phi}^\perp)\|_{L^2(\Omega_t^C)^{N\times m}}^2).
	\]
	Together with the estimate for $\check{B}_{\varepsilon,t}^C$ this shows the claim.\end{proof}

\section{Difference Estimates and Proofs of the Convergence Theorems}
\label{sec_DC}
In this section we estimate the difference of the exact and approximate solutions with a Gronwall-type argument in both cases. This is the second step in the method by de Mottoni and Schatzman \cite{deMS}. Here the major ingredients are the spectral estimates from Sections \ref{sec_SE_ACND}-\ref{sec_SE_vAC}. Moreover, we have to control certain nonlinear terms stemming from differences of potential terms. We will estimate these with interpolation inequalities. Therefore as preparation we prove uniform a priori bounds for exact classical solutions of \hyperlink{AC}{(AC)} and \hyperlink{vAC}{(vAC)} in Sections \ref{sec_DC_prelim_bdd_scal} and \ref{sec_DC_prelim_bdd_vect}, respectively. Moreover we recall some Gagliardo-Nirenberg estimates in Section \ref{sec_DC_prelim_GN}. For Allen-Cahn type models such uniform bounds in $\varepsilon$ for the exact solutions is typical, see de Mottoni, Schatzman \cite{deMS}, Section 6 for the standard Allen-Cahn equation as well as \cite{ALiu}, Remark 1.2 and \cite{ALiu}, Section 5.2 for the Allen-Cahn equation coupled with the Stokes system. In Section \ref{sec_DC_ACND} we prove the difference estimate in the scalar case and Theorem \ref{th_AC_conv}. This is similar to Section 5 and Section 6 in \cite{AbelsMoser}, but more technical. Finally, in Section \ref{sec_DC_vAC} we show the difference estimate in the vector-valued case and Theorem \ref{th_vAC_conv}. The latter works analogously to the scalar case.

\subsection{Preliminaries}\label{sec_DC_prelim}
\subsubsection{Uniform A Priori Bound for Classical Solutions of (AC)}\label{sec_DC_prelim_bdd_scal}
Let $N$, $\Omega$, $Q_T$, $\partial Q_T$ be as in Remark \ref{th_intro_coord},~1.~and $\varepsilon>0$. We prove uniform boundedness estimates for classical solutions of the Allen-Cahn equation \hyperlink{AC}{(AC)}. Let $f$ be as in \eqref{eq_AC_fvor1} and $R_0\geq 1$ be such that the condition \eqref{eq_AC_fvor2} for $f'$ holds.
\begin{Lemma}\label{th_DC_bdd_scal}
	Let $u_{0,\varepsilon}\in C^0(\overline{\Omega})$ and $u_{\varepsilon}\in 
	C^0(\overline{Q_T})\cap C^1(\overline{\Omega}\times(0,T])\cap C^2(\Omega\times(0,T])$ be a solution of \eqref{eq_AC1}-\eqref{eq_AC3}. Then 
	\[
	\|u_{\varepsilon}\|_{L^\infty(Q_T)}\leq\max\{R_0,\|u_{0,\varepsilon}\|_{L^\infty(\Omega)}\}.
	\]
\end{Lemma}
\begin{proof}
	We use a contradiction argument and ideas from the proof of the weak maximum principle for parabolic equations, cf.~Renardy, Rogers \cite{RenardyRogers}, Theorem 4.25. Variants of the proof may also work. We have chosen a proof that can be directly generalized to the vector-valued case, see below. Assume $\|u_{\varepsilon}\|_{L^\infty(Q_T)}>\max\{R_0,\|u_{0,\varepsilon}\|_{L^\infty(\Omega)}\}$. We consider $u_{\varepsilon,\beta}:=|u_{\varepsilon}|^2+\beta e^{-t}$ and $u_{0,\varepsilon,\beta}:=|u_{0,\varepsilon}|^2+\beta$ for $\beta>0$. Then for $\beta>0$ small
	\begin{align}\label{eq_DC_bdd_scal1}
	\|u_{\varepsilon,\beta}\|_{L^\infty(Q_T)}>
	\beta +\max\{R_0^2,\|u_{0,\varepsilon,\beta}\|_{L^\infty(\Omega)}\}.
	\end{align}
	Because of $u_{\varepsilon,\beta}\in C^0(\overline{Q_T})$, it follows that the maximum of $u_{\varepsilon,\beta}=|u_{\varepsilon,\beta}|$ is attained in some $(x_0,t_0)\in\overline{Q_T}$. Due to \eqref{eq_DC_bdd_scal1} it holds
	\begin{align}\label{eq_DC_bdd_scal2}
	|u_{\varepsilon}|^2|_{(x_0,t_0)}
	=u_{\varepsilon,\beta}|_{(x_0,t_0)}-\beta e^{-t_0}
	> \beta + R_0^2 -\beta = R_0^2.
	\end{align}
	Hence $f'(u_{\varepsilon}|_{(x_0,t_0)})u_{\varepsilon}|_{(x_0,t_0)}\geq 0$ due to \eqref{eq_AC_fvor2}. If $(x_0,t_0)\in\Omega\times(0,T]$, we get from \eqref{eq_AC1}:
	\begin{align*}
	(\partial_t-\Delta)u_{\varepsilon,\beta}|_{(x_0,t_0)}&=-\beta e^{-t_0}+
	2u_{\varepsilon}(\partial_tu_{\varepsilon}-\Delta u_{\varepsilon})|_{(x_0,t_0)}
	-2|\nabla u_{\varepsilon}|^2|_{(x_0,t_0)}\\
	&=-\beta e^{-t_0}-
	2\left[u_{\varepsilon}|_{(x_0,t_0)}\frac{1}{\varepsilon^2}f'(u_{\varepsilon}|_{(x_0,t_0)})
	+|\nabla u_{\varepsilon}|^2|_{(x_0,t_0)}\right]\\
	&\leq -\beta e^{-t_0}<0.
	\end{align*}
	
	If $(x_0,t_0)\in\overline
	\Omega\times\{0\}$, we get a contradiction from \eqref{eq_DC_bdd_scal1} and since $u_{\varepsilon,\beta}|_{t=0}=u_{0,\varepsilon,\beta}$ due to \eqref{eq_AC3}. In the case $(x_0,t_0)\in\Omega\times(0,T]$ we obtain a contradiction to $(\partial_t-\Delta)u_{\varepsilon,\beta}|_{(x_0,t_0)}\geq0$ as in the proof of the weak maximum principle, cf.~\cite{RenardyRogers}, Theorem 4.25. 
	
	Finally, let $(x_0,t_0)\in\partial\Omega\times(0,T]$. With the Hopf Lemma (cf.~Gilbarg, Trudinger \cite{GilbargTrudinger}, Lemma 3.4) we deduce a contradiction to the boundary condition \eqref{eq_AC2}. The above consideration yields  $u_{\varepsilon,\beta}|_{(x_0,t_0)}>u_{\varepsilon,\beta}|_{(x,t)}$ for all $(x,t)\in\Omega\times(0,T]$. Additionally, because of continuity and \eqref{eq_DC_bdd_scal2} it holds $|u_{\varepsilon}(x,t)|>R_0$ for all $(x,t)\in B_\eta(x_0,t_0)\cap(\Omega\times(0,T])$ and $\eta>0$ small. Hence as above $(\partial_t-\Delta)u_{\varepsilon,\beta}|_{(x,t)}\leq -\beta e^{-t}<0$ for these $(x,t)$. Moreover, since $(x_0,t_0)$ is a maximum of $u_{\varepsilon,\beta}$, it follows that $\partial_t u_{\varepsilon,\beta}|_{(x_0,t_0)}\geq 0$. Therefore continuity of $\partial_tu_{\varepsilon,\beta}$ yields $\Delta u_{\varepsilon,\beta}|_{(x,t)}<0$ for all $(x,t)\in B_\eta(x_0,t_0)\cap(\Omega\times(0,T])$ and $\eta>0$ small. Hence the Hopf Lemma is applicable on $B_\eta(x_0)\cap\Omega$ and yields $N_{\partial\Omega}\cdot\nabla u_{\varepsilon,\beta}|_{(x_0,t_0)}>0$. This gives a contradiction to \eqref{eq_AC2} because of
	$\nabla u_{\varepsilon,\beta}=2u_{\varepsilon}\nabla u_{\varepsilon}$ and \eqref{eq_DC_bdd_scal2}. Finally, we have considered all possible cases and obtained a contradiction. Hence the lemma is proven.
\end{proof}

\subsubsection{Uniform A Priori Bound for Classical Solutions of (vAC)}\label{sec_DC_prelim_bdd_vect}
Let $N$, $\Omega$, $Q_T$, $\partial Q_T$ be as in Remark \ref{th_intro_coord},~1.~and $\varepsilon>0$. Moreover, let $m\in\N$ and $W:\R^m\rightarrow\R$ be as in Definition \ref{th_vAC_W}. We prove uniform boundedness estimates for classical solutions of \hyperlink{vAC}{(vAC)}. This works analogously to the proof of Lemma \ref{th_DC_bdd_scal} for the scalar case in the last section.
\begin{Lemma}\label{th_DC_bdd_vect}
	Let $\vec{u}_{0,\varepsilon}\in C^0(\overline{\Omega})^m$ and $\vec{u}_\varepsilon\in 
	C^0(\overline{Q_T})^m\cap C^1(\overline{\Omega}\times(0,T])^m\cap C^2(\Omega\times(0,T])^m$ be a solution of \eqref{eq_vAC1}-\eqref{eq_vAC3}. Then with $\check{R}_0>0$ as in Definition \ref{th_vAC_W} it holds
	\[
	\|\vec{u}_\varepsilon\|_{L^\infty(Q_T,\R^m)}\leq\max\{\check{R}_0,\|\vec{u}_{0,\varepsilon}\|_{L^\infty(\Omega,\R^m)}\}.
	\]
\end{Lemma}
\begin{proof}
	Assume $\|\vec{u}_\varepsilon\|_{L^\infty(Q_T,\R^m)}>\max\{\check{R}_0,\|\vec{u}_{0,\varepsilon}\|_{L^\infty(\Omega,\R^m)}\}$. Let $\check{u}_{\varepsilon,\beta}:=|\vec{u}_\varepsilon|^2+\beta e^{-t}$ and $\check{u}_{0,\varepsilon,\beta}:=|\vec{u}_{0,\varepsilon}|^2+\beta$ for $\beta>0$. Then for $\beta>0$ small
	\begin{align}\label{eq_DC_bdd_vect1}
	\|\check{u}_{\varepsilon,\beta}\|_{L^\infty(Q_T)}>
	\beta +\max\{\check{R}_0^2,\|\check{u}_{0,\varepsilon,\beta}\|_{L^\infty(\Omega)}\}.
	\end{align}
	The maximum of $\check{u}_{\varepsilon,\beta}=|\check{u}_{\varepsilon,\beta}|$ is attained in some $(x_0,t_0)\in\overline{Q_T}$ due to $\vec{u}_\varepsilon\in C^0(\overline{Q_T})^m$. Then inequality \eqref{eq_DC_bdd_vect1} yields
	\begin{align}\label{eq_DC_bdd_vect2}
	|\vec{u}_\varepsilon|^2|_{(x_0,t_0)}
	=\check{u}_{\varepsilon,\beta}|_{(x_0,t_0)}-\beta e^{-t_0}
	> \beta + \check{R}_0^2 -\beta = \check{R}_0^2.
	\end{align}
	By the assumption on $W$ in Definition \ref{th_vAC_W} it follows that $\vec{u}_\varepsilon|_{(x_0,t_0)}\cdot\nabla W(\vec{u}_\varepsilon|_{(x_0,t_0)})\geq 0$. Hence equation \eqref{eq_vAC1} yields in the case $(x_0,t_0)\in\Omega\times(0,T]$ that
	\begin{align*}
	(\partial_t-\Delta)\check{u}_{\varepsilon,\beta}|_{(x_0,t_0)}&=-\beta e^{-t_0}+
	2\vec{u}_\varepsilon\cdot(\partial_t\vec{u}_\varepsilon-\Delta\vec{u}_\varepsilon)|_{(x_0,t_0)}
	-2|\nabla\vec{u}_\varepsilon|^2|_{(x_0,t_0)}\\
	&=-\beta e^{-t_0}-
	2\left[\vec{u}_\varepsilon|_{(x_0,t_0)}\cdot\frac{1}{\varepsilon^2}\nabla W(\vec{u}_\varepsilon|_{(x_0,t_0)})
	+|\nabla\vec{u}_\varepsilon|^2|_{(x_0,t_0)}\right]\\
	&\leq -\beta e^{-t_0}<0.
	\end{align*}
	In the case $(x_0,t_0)\in\overline
	\Omega\times\{0\}$ the contradiction follows from \eqref{eq_DC_bdd_vect1} and $\check{u}_{\varepsilon,\beta}|_{t=0}=\check{u}_{0,\varepsilon,\beta}$ because of \eqref{eq_vAC3}. In the case $(x_0,t_0)\in\Omega\times(0,T]$ we obtain a contradiction to $(\partial_t-\Delta)\check{u}_{\varepsilon,\beta}|_{(x_0,t_0)}\geq0$ as in the proof of the weak maximum principle, cf.~\cite{RenardyRogers}, Theorem 4.25.
	
	Finally, let $(x_0,t_0)\in\partial\Omega\times(0,T]$. The above consideration yields  $\check{u}_{\varepsilon,\beta}|_{(x_0,t_0)}>\check{u}_{\varepsilon,\beta}|_{(x,t)}$ for all $(x,t)\in\Omega\times(0,T]$. Moreover, due to continuity and \eqref{eq_DC_bdd_vect2} we obtain $|\vec{u}_\varepsilon(x,t)|>\check{R}_0$ for all $(x,t)\in B_\eta(x_0,t_0)\cap(\Omega\times(0,T])$ and $\eta>0$ small. Therefore as above it follows that $(\partial_t-\Delta)\check{u}_{\varepsilon,\beta}|_{(x,t)}\leq -\beta e^{-t}<0$ for these $(x,t)$. Furthermore, it holds $\partial_t \check{u}_{\varepsilon,\beta}|_{(x_0,t_0)}\geq 0$ because $(x_0,t_0)$ is a maximum of $\check{u}_{\varepsilon,\beta}$. By continuity of $\partial_t\check{u}_{\varepsilon,\beta}$ we obtain $\Delta \check{u}_{\varepsilon,\beta}|_{(x,t)}<0$ for all $(x,t)\in B_\eta(x_0,t_0)\cap(\Omega\times(0,T])$ and $\eta>0$ small. Therefore the Hopf Lemma (cf.~\cite{GilbargTrudinger}, Lemma 3.4) can be applied on $B_\eta(x_0)\cap\Omega$ and yields $N_{\partial\Omega}\cdot\nabla \check{u}_{\varepsilon,\beta}|_{(x_0,t_0)}>0$. Here 
	$\nabla \check{u}_{\varepsilon,\beta}=\sum_{j=1}^m\nabla(|u_{\varepsilon,j}|^2)=2\sum_{j=1}^m u_{\varepsilon,j}\nabla u_{\varepsilon,j}$. Therefore we obtain a contradiction to the boundary condition \eqref{eq_vAC2}. Finally, this yields the lemma.
\end{proof}

\subsubsection{Gagliardo-Nirenberg Inequalities}\label{sec_DC_prelim_GN}
Let us recall some Gagliardo-Nirenberg inequalities. 
\begin{Lemma}[\textbf{Gagliardo-Nirenberg Inequality}]\label{th_DC_GN}
	Let $n\in\N$, $1\leq p,q,r\leq\infty$ and $\theta\in[0,1]$ such that
	\[
	\theta\left(\frac{1}{p}-\frac{1}{n}\right)+\frac{1-\theta}{q}=\frac{1}{r},
	\]
	where $\frac{1}{\infty}:=0$. Moreover, if $p=n>1$, then we assume $\theta<1$. Then 
	\[
	\|u\|_{L^r(\R^n)}\leq c\|u\|_{L^q(\R^n)}^{1-\theta}\|\nabla u\|_{L^p(\R^n)}^{\theta}
	\]
	for all $u\in L^q(\R^n)\cap W^{1,p}(\R^n)$ and a constant $c=c(n,p,q,r)>0$.
\end{Lemma}
\begin{proof}
	See Leoni \cite{Leoni}, Theorem 12.83.
\end{proof}

\begin{Remark}\label{th_DC_GN_rem}\upshape
	With suitable extension operators the estimate in Lemma \ref{th_DC_GN} carries over to domains $\Omega$ with uniform Lipschitz boundary if the $\nabla u$-factor in the estimate is replaced by the $W^{1,p}$-norm and the constant in the estimate depends on $n,p,q,r$ and the operator norm of the extension operator. For the existence of such extension operators (going back to Stein) see Leoni \cite{Leoni}, Theorem 13.8 and Theorem 13.17. Note that the operator norms in \cite{Leoni} are estimated solely in terms of the usual parameters and the geometrical quantities of $\Omega$ and $\partial\Omega$. In particular if the geometrical quantities can be controlled in a uniform way, the operator norms and the constants in the above Gagliardo-Nirenberg inequalities can be taken uniformly with respect to $\Omega$.
\end{Remark}

\subsection[Difference Estimate and Proof of the Convergence Thm. for (AC) in ND]{Difference Estimate and Proof of the Convergence Theorem for (AC) in ND}\label{sec_DC_ACND}
We prove in Section \ref{sec_DC_ACND_DE} a rather abstract estimate for the difference of exact solutions and suitable approximate solutions for the Allen-Cahn equation \eqref{eq_AC1}-\eqref{eq_AC3} in ND. Then in Section \ref{sec_DC_ACND_conv} we show the Theorem \ref{th_AC_conv} about convergence by verifying the requirements for the difference estimate applied to the approximate solution from Section \ref{sec_asym_ACND_uA}.

\subsubsection{Difference Estimate}\label{sec_DC_ACND_DE}
\begin{Theorem}[\textbf{Difference Estimate for (AC)}]\label{th_DC_ACND}
	Let $N\geq2$, $\Omega$, $Q_T$ and $\partial Q_T$ be as in Remark \ref{th_intro_coord},~1. Moreover, let $\Gamma=(\Gamma_t)_{t\in[0,T_0]}$ for some $T_0>0$ be as in Section \ref{sec_coordND} and $\delta>0$ be such that Theorem \ref{th_coordND} holds for $2\delta$ instead of $\delta$. We use the notation for $\Gamma_t(\delta)$, $\Gamma(\delta)$, $\nabla_\tau$ and $\partial_n$ from Remark \ref{th_coordND_rem}. Additionally, let $f$ satisfy \eqref{eq_AC_fvor1}-\eqref{eq_AC_fvor2}.
	
	Moreover, let $\varepsilon_0>0$, $u^A_\varepsilon\in C^2(\overline{Q_{T_0}})$, $u_{0,\varepsilon}\in C^2(\overline{\Omega})$ with $\partial_{N_{\partial\Omega}} u_{0,\varepsilon}=0$ on $\partial\Omega$ and let $u_\varepsilon\in C^2(\overline{Q_{T_0}})$ be exact solutions to \eqref{eq_AC1}-\eqref{eq_AC3} with $u_{0,\varepsilon}$ in \eqref{eq_AC3} for $\varepsilon\in(0,\varepsilon_0]$. 
	
	For some $R>0$ and $M\in\N, M\geq k(N):=\max\{2,\frac{N}{2}\}$ we impose the following conditions:
	\begin{enumerate}
		\item \textup{Uniform Boundedness:} $\sup_{\varepsilon\in(0,\varepsilon_0]}\|u^A_\varepsilon\|_{L^\infty(Q_{T_0})}+\|u_{0,\varepsilon}\|_{L^\infty(\Omega)}<\infty$.
		\item \textup{Spectral Estimate:} There are $c_0,C>0$ such that
		\[
		\int_\Omega|\nabla\psi|^2+\frac{1}{\varepsilon^2}f''(u^A_\varepsilon(.,t))\psi^2\,dx\geq -C\|\psi\|_{L^2(\Omega)}^2+\|\nabla\psi\|_{L^2(\Omega\setminus\Gamma_t(\delta))}^2+c_0\|\nabla_\tau\psi\|_{L^2(\Gamma_t(\delta))}^2
		\]
		for all $\psi\in H^1(\Omega)$ and $\varepsilon\in(0,\varepsilon_0],t\in[0,T_0]$.
		\item \textup{Approximate Solution:} For the remainders 
		\[
		r^A_\varepsilon:=\partial_t u^A_\varepsilon-\Delta u^A_\varepsilon+\frac{1}{\varepsilon^2}f'(u^A_\varepsilon)\quad\text{ and }\quad s^A_\varepsilon:=\partial_{N_{\partial\Omega}}u^A_\varepsilon
		\]
		in \eqref{eq_AC1}-\eqref{eq_AC2} for $u^A_\varepsilon$ and the difference $\overline{u}_\varepsilon:=u_\varepsilon-u^A_\varepsilon$ it holds
		\begin{align}\begin{split}\label{eq_DC_ACND_uA}
		&\left|\int_{\partial\Omega}s^A_\varepsilon\tr\,\overline{u}_\varepsilon(t)\,d\Hc^{N-1}+\int_\Omega r^A_\varepsilon\overline{u}_\varepsilon(t)\,dx\right|\\
		&\leq C\varepsilon^{M+\frac{1}{2}}(\|\overline{u}_\varepsilon(t)\|_{L^2(\Omega)}+\|\nabla_\tau\overline{u}_\varepsilon(t)\|_{L^2(\Gamma_t(\delta))}+\|\nabla\overline{u}_\varepsilon(t)\|_{L^2(\Omega\setminus\Gamma_t(\delta))})
		\end{split}
		\end{align}
		for all $\varepsilon\in(0,\varepsilon_0]$ and $T\in(0,T_0]$.
		\item \textup{Well-Prepared Initial Data:} For all $\varepsilon\in(0,\varepsilon_0]$ it holds
		\begin{align}\label{eq_DC_ACND_u0} \|u_{0,\varepsilon}-u^A_\varepsilon|_{t=0}\|_{L^2(\Omega)}\leq R\varepsilon^{M+\frac{1}{2}}.
		\end{align}
	\end{enumerate} 
	Then we obtain
	\begin{enumerate}
		\item Let $M>k(N)$. Then there are $\beta,\varepsilon_1>0$ such that for $g_\beta(t):=e^{-\beta t}$ it holds
		\begin{align}
		\begin{split}
		\sup_{t\in[0,T]}\|(g_\beta\overline{u}_\varepsilon)(t)\|_{L^2(\Omega)}^2+\|g_\beta\nabla\overline{u}_\varepsilon\|_{L^2(Q_T\setminus\Gamma(\delta))}^2&\leq 2R^2\varepsilon^{2M+1},\\
		c_0\|g_\beta\nabla_\tau\overline{u}_\varepsilon\|^2_{L^2(Q_T\cap\Gamma(\delta))}+\varepsilon^2\|g_\beta\partial_n\overline{u}_\varepsilon\|^2_{L^2(Q_T\cap\Gamma(\delta))}&\leq 2R^2\varepsilon^{2M+1}\label{eq_DC_ACND_DE}\end{split}
		\end{align}
		for all $\varepsilon\in(0,\varepsilon_1]$ and $T\in(0,T_0]$.
		\item Let $k(N)\in\N$ and $M=k(N)$. Let \eqref{eq_DC_ACND_uA} hold for some $\tilde{M}>M$ instead of $M$. Then there are $\beta,\tilde{R},\varepsilon_1>0$ such that, if \eqref{eq_DC_ACND_u0} holds for $\tilde{R}$ instead of $R$, then $\eqref{eq_DC_ACND_DE}$ for $\tilde{R}$ instead of $R$ is valid for all $\varepsilon\in(0,\varepsilon_1], T\in(0,T_0]$. 
		\item Let $N\in\{2,3\}$ and $M=2(=k(N))$. Then there are $\varepsilon_1,T_1>0$ such that $\eqref{eq_DC_ACND_DE}$ holds for $\beta=0$ and for all $\varepsilon\in(0,\varepsilon_1], T\in(0,T_1]$.
	\end{enumerate}
\end{Theorem}
\begin{Remark}\upshape\phantomsection{\label{th_DC_ACND_rem}}
	\begin{enumerate}
	\item The parameter $M$ corresponds to the order of the approximate solution in Section \ref{sec_asym_ACND}.
	\item The parameter $\beta$ was introduced in order to obtain a result valid for all times $T\in(0,T_0]$.
	\item Note that weaker requirements in the theorem also work, e.g.~when one does not have the two additional terms on the right hand side of the spectral estimate or only an estimate with the full $H^1$-norm on the right hand side in \eqref{eq_DC_ACND_uA}. Moreover, a slightly less involved proof is also possible, see e.g.~Remark \ref{th_DC_ACND_GN_rem} below. However, then the result is also weaker, in particular the somewhat critical order $k(N)$ for $M$ could be increased and the $\varepsilon$-orders in \eqref{eq_DC_ACND_DE} could be weakened. This is because the $H^1$-norm can be controlled with the spectral term but one has to pay $\varepsilon^{-2}$ times the $L^2$-norm. Nevertheless, we intended to give an optimal result, also having in mind e.g.~couplings with other equations like the Stokes system as in \cite{ALiu}, where a low number of terms in the ansatz for the approximate solution is convenient.
	\item That the parameter $k(N)$ is critical for $M$ in our proof can be seen at \eqref{eq_DC_ACND_DE2} in the proof below. This is due to an estimate of a cubic term, see \eqref{eq_DC_ACND_proof4} and Lemma \ref{th_DC_ACND_GN} below. The results 2.-3.~are the best we could prove for the critical case $M=k(N)\in\N$. This situation is difficult because in the estimates there will be a term of order larger than $2$ in $R$ and a linear term in $R$, but the desired order is $2$ in $R$. The linear term in $R$ will enter due to \eqref{eq_DC_ACND_uA}, see \eqref{eq_DC_ACND_DE2} below. For the parameter $\beta$ there is a similar problem in the critical case.
	\end{enumerate}
\end{Remark}
\begin{proof}[Proof of Theorem \ref{th_DC_ACND}]
	The continuity of the objects on the left hand side in \eqref{eq_DC_ACND_DE} yields that
	\begin{align}\label{eq_DC_ACND_T_epsR}
	T_{\varepsilon,\beta,R}:=\sup\,\{\tilde{T}\in(0,T_0]: \eqref{eq_DC_ACND_DE}\text{ holds for }\varepsilon,R\text{ and all }T\in(0,\tilde{T}]\}
	\end{align}
	is well-defined for all $\varepsilon\in(0,\varepsilon_0],\beta\geq 0$ and  $T_{\varepsilon,\beta,R}>0$. In the different cases we have to show:
	\begin{enumerate}
	\item If $M>k(N)$, then there are $\beta,\varepsilon_1>0$ such that $T_{\varepsilon,\beta,R}=T_0$ for all $\varepsilon\in(0,\varepsilon_1]$.
	\item If $k(N)\in\N$ and $M=k(N)$, then there are $\beta,\tilde{R},\varepsilon_1>0$ such that $T_{\varepsilon,\beta,\tilde{R}}=T_0$ provided that $\varepsilon\in(0,\varepsilon_1]$ and \eqref{eq_DC_ACND_uA} is true for some $\tilde{M}>M$ instead of $M$ and \eqref{eq_DC_ACND_u0} is valid with $R$ replaced by $\tilde{R}$.
	\item If $N\in\{2,3\}$, $M=2$, then there are $T_1,\varepsilon_1>0$ such that $T_{\varepsilon,0,R}\geq T_1$ for all $\varepsilon\in(0,\varepsilon_1]$.
	\end{enumerate}
	
	We carry out a general computation first and return back to the different cases later. The difference of the left hand sides in \eqref{eq_AC1} for $u_\varepsilon$ and $u^A_\varepsilon$ yields
	\begin{align}\label{eq_DC_ACND_proof1}
	\left[\partial_t-\Delta+\frac{1}{\varepsilon^2}f''(u^A_\varepsilon)\right]\overline{u}_\varepsilon=-r^A_\varepsilon-r_\varepsilon(u_\varepsilon,u^A_\varepsilon),
	\end{align} 
	where $r_\varepsilon(u_\varepsilon,u^A_\varepsilon):=\frac{1}{\varepsilon^2}\left[f'(u_\varepsilon)-f'(u^A_\varepsilon)-f''(u^A_\varepsilon)\overline{u}_\varepsilon\right]$. We multiply \eqref{eq_DC_ACND_proof1} by $g_\beta^2\overline{u}_\varepsilon$ and integrate over $Q_T$ for $T\in(0,T_{\varepsilon,\beta,R}]$, where $\varepsilon\in(0,\varepsilon_0]$ and $\beta\geq 0$ are fixed. This yields 
	\begin{align}\label{eq_DC_ACND_proof2}
	\int_0^T g _\beta^2\int_\Omega\overline{u}_\varepsilon\left[\partial_t-\Delta+\frac{1}{\varepsilon^2}f''(u^A_\varepsilon)\right]\overline{u}_\varepsilon\,dx\,dt
	=-\int_0^T g _\beta^2\int_\Omega [r^A_\varepsilon+r_\varepsilon(u_\varepsilon,u^A_\varepsilon)]\overline{u}_\varepsilon\,dx\,dt
	\end{align} 
	for all $T\in(0,T_{\varepsilon,\beta,R}]$, $\varepsilon\in(0,\varepsilon_0]$ and $\beta\geq 0$.
	We have to estimate all terms in a suitable way. First, $\frac{1}{2}\partial_t|\overline{u}_\varepsilon|^2=\overline{u}_\varepsilon\partial_t\overline{u}_\varepsilon$, integration by parts in time and $\partial_tg_\beta=-\beta g_\beta$ imply
	\[
	\int_0^T\int_\Omega g_\beta^2\partial_t\overline{u}_\varepsilon\overline{u}_\varepsilon\,dx\,dt=\frac{1}{2}g_\beta(T)^2\|\overline{u}_\varepsilon(T)\|_{L^2(\Omega)}^2-\frac{1}{2}\|\overline{u}_\varepsilon(0)\|_{L^2(\Omega)}^2+\beta\int_0^T g_\beta^2 \|\overline{u}_\varepsilon\|_{L^2(\Omega)}^2\,dt,
	\]
	where $\|\overline{u}_\varepsilon(0)\|_{L^2(\Omega)}^2\leq R^2\varepsilon^{2M+1}$ due to \eqref{eq_DC_ACND_u0} (\enquote{well-prepared initial data}). For the other term on the left hand side in \eqref{eq_DC_ACND_proof2} we use integration by parts in space. This yields
	\begin{align*}
	&\int_0^Tg_\beta^2\int_\Omega \overline{u}_\varepsilon\left[-\Delta+\frac{1}{\varepsilon^2}f''(u^A_\varepsilon)\right]\overline{u}_\varepsilon\,dx\,dt\\
	&=\int_0^Tg_\beta^2\int_\Omega|\nabla\overline{u}_\varepsilon|^2+\frac{1}{\varepsilon^2}f''(u^A_\varepsilon)\overline{u}_\varepsilon^2\,dx\,dt
	+\int_0^Tg_\beta^2\int_{\partial\Omega}s^A_\varepsilon\,\tr\,\overline{u}_\varepsilon\,d\Hc^{N-1}\,dt.
	\end{align*}
	With requirement 2.~(\enquote{spectral estimate}) in the theorem it follows that the first integral on the right hand side in the latter equation is bounded from below by
	\[
	-C\int_0^T g_\beta^2\|\overline{u}_\varepsilon\|_{L^2(\Omega)}^2\,dt+\|g_\beta\nabla\overline{u}_\varepsilon\|_{L^2(Q_T\setminus\Gamma(\delta))}^2+c_0\|g_\beta\nabla_\tau\overline{u}_\varepsilon\|_{L^2(Q_T\cap\Gamma(\delta))}^2.
	\]
    For the remainder terms involving $r^A_\varepsilon$ and $s^A_\varepsilon$ we use \eqref{eq_DC_ACND_uA} (\enquote{approximate solution}). This yields 
    \[
    \left|\int_0^Tg_\beta^2\left[\int_{\partial\Omega}s^A_\varepsilon\tr\,\overline{u}_\varepsilon(t)\,d\Hc^{N-1}+\int_\Omega r^A_\varepsilon\overline{u}_\varepsilon(t)\,dx\right]dt\right|
    \leq \overline{C}_1R\|g_\beta\|_{L^2(0,T)}\varepsilon^{2M+1}
    \]
    due to \eqref{eq_DC_ACND_DE} for all $T\in(0,T_{\varepsilon,\beta,R}]$ and $\varepsilon\in(0,\varepsilon_0]$, where we used $\|g_\beta\|_{L^1(0,T)}\leq \sqrt{T_0}\|g_\beta\|_{L^2(0,T)}$.
	
	In the following we estimate the $r_\varepsilon$-term in \eqref{eq_DC_ACND_proof2}. The requirement 1.~(\enquote{uniform boundedness}) in the theorem and Lemma \ref{th_DC_bdd_scal} yield
	\begin{align}\label{eq_DC_ACND_proof3}
	\sup_{\varepsilon\in(0,\varepsilon_0]}\left[\|u_\varepsilon\|_{L^\infty(Q_{T_0})}+\|u^A_\varepsilon\|_{L^\infty(Q_{T_0})}\right]<\infty.
	\end{align}
	Therefore we can apply the Taylor Theorem and obtain
	\begin{align}\label{eq_DC_ACND_proof4}
	\left|\int_0^Tg_\beta^2\int_\Omega r_\varepsilon(u_\varepsilon,u^A_\varepsilon)\overline{u}_\varepsilon\,dx\,dt\right|
	\leq \frac{C}{\varepsilon^2}\int_0^Tg_\beta^2\|\overline{u}_\varepsilon\|_{L^3(\Omega)}^3\,dt.
	\end{align}
	In order to estimate the latter we use a standard Gagliardo-Nirenberg Inequality on $\Omega\setminus\Gamma_t(\delta)$ but on $\Gamma_t(\delta)$ we apply such inequalities in tangential and normal direction. This is similar to \cite{ALiu}, Lemma 5.3. The idea is to get a finer estimate and account for the fact that the estimate for $\nabla_\tau\overline{u}_\varepsilon$ in \eqref{eq_DC_ACND_DE} is better than that for $\partial_n\overline{u}_\varepsilon$. However, note that if $N$ is too large, estimating the full $L^3$-norm for $\overline{u}_\varepsilon$ will not work because of the requirements for the Gagliardo-Nirenberg Inequality or because we only have $L^2$-estimates in \eqref{eq_DC_ACND_DE} for $\nabla_\tau\overline{u}_\varepsilon$ and $\partial_n\overline{u}_\varepsilon$. Therefore we use the uniform boundedness \eqref{eq_DC_ACND_proof3} to lower the exponent. However, this will also decrease the resulting $\varepsilon$-order. Therefore we try to find the largest possible parameter. The estimates are lengthy and we decided to postpone them, see below. The result is
	\begin{Lemma}\label{th_DC_ACND_GN}
	Under the assumptions in Theorem \ref{th_DC_ACND} it holds
	\begin{align*}
	\left|\int_0^T\!g_\beta^2\int_\Omega r_\varepsilon(u_\varepsilon,u^A_\varepsilon)\overline{u}_\varepsilon\,dx\,dt\right|
	\leq
	CR^{2+K(N)}\varepsilon^{2M+1} \varepsilon^{K(N)(M-k(N))}\|g_\beta^{-K(N)}\|_{L^{\frac{4}{4-\min\{4,N\}}}(0,T)}
	\end{align*}
	for all $T\in(0,T_{\varepsilon,\beta,R}]$ and $\varepsilon\in(0,\varepsilon_0]$, where $T_{\varepsilon,\beta,R}$ is as in \eqref{eq_DC_ACND_T_epsR}, $k(N)=\max\{2,\frac{N}{2}\}$ and $K(N):=\min\{1,\frac{4}{N}\}\in(0,1]$.
	\end{Lemma}
	\begin{Remark}\upshape\label{th_DC_ACND_GN_rem}
	Note that one could apply the same standard Gagliardo-Nirenberg Inequality used for $\Omega\setminus\Gamma_t(\delta)$ also for whole $\Omega$, but then the estimate is weaker and the minimal order $k(N)$ for $M$ that is required for the difference estimate to work increases. 
	\end{Remark}
	
	It remains to estimate $\partial_n\overline{u}_\varepsilon$. Therefore we use that due to Corollary \ref{th_coordND_nabla_tau_n}
	\begin{align*}
	\varepsilon^2\|g_\beta\partial_n\overline{u}_\varepsilon\|_{L^2(Q_T\cap\Gamma(\delta))}^2
	&\leq C\varepsilon^2\int_0^T g_\beta^2\int_\Omega|\nabla\overline{u}_\varepsilon|^2+\frac{1}{\varepsilon^2}f''(u^A_\varepsilon)(\overline{u}_\varepsilon)^2\,dx\,dt\\
	&+C\sup_{\varepsilon\in(0,\varepsilon_0]}\|f''(u^A_\varepsilon)\|_{L^\infty(Q_{T_0})}\int_0^T g_\beta^2\|\overline{u}_\varepsilon(t)\|_{L^2(\Omega)}^2\,dt
	\end{align*}
	with a constant $C>0$ independent of $\varepsilon$, $T$ and $R$. The first term is absorbed with $\frac{1}{2}$ of the spectral term above if $\varepsilon\in(0,\varepsilon_1]$ and $\varepsilon_1>0$ is small (independent of $T$, $R$). Finally, all terms are estimated and we obtain
	\begin{align}\notag
	&\|\frac{g_\beta\overline{u}_\varepsilon}{2}|_T\|_{L^2(\Omega)}^2
	+\|\frac{g_\beta}{2}\nabla\overline{u}_\varepsilon\|_{L^2(Q_T\setminus\Gamma(\delta))}^2
	+\frac{c_0}{2}\|g_\beta\nabla_\tau\overline{u}_\varepsilon\|_{L^2(Q_T\cap\Gamma(\delta))}^2
	+\frac{\varepsilon^2}{2}\|g_\beta\partial_n\overline{u}_\varepsilon\|_{L^2(Q_T\cap\Gamma(\delta))}^2\\
	\begin{split}&\leq
	\frac{R^2}{2}\varepsilon^{2M+1}+\int_0^T(-\beta+\overline{C}_0)g_\beta^2\|\overline{u}_\varepsilon(t)\|_{L^2(\Omega)}^2\,dt+\overline{C}_1R\varepsilon^{2M+1}\|g_\beta\|_{L^2(0,T)}\\
	&+CR^{2+K(N)}\varepsilon^{2M+1} \varepsilon^{K(N)(M-k(N))}\|g_\beta^{-K(N)}\|_{L^{\frac{4}{4-\min\{4,N\}}}(0,T)}\label{eq_DC_ACND_DE2}\end{split}
	\end{align}
	for all $T\in(0,T_{\varepsilon,\beta,R}]$, $\varepsilon\in(0,\varepsilon_1]$ and constants $\overline{C}_0,\overline{C}_1,C>0$ independent of $\varepsilon,T,R$, where $k(N)=\max\{2,\frac{N}{2}\}$ and $K(N)=\min\{1,\frac{4}{N}\}$. 
	
	Now we consider the cases in the theorem.\phantom{\qedhere}
	
	\begin{proof}[Ad 1] 
	If $M>k(N)$, then we choose $\beta\geq \overline{C}_0$ large such that $\overline{C}_1R\|g_\beta\|_{L^2(0,T_0)}\leq\frac{R^2}{8}$.
	Then \eqref{eq_DC_ACND_DE2} is estimated by $\frac{3}{4}R^2\varepsilon^{2M+1}$ for all $T\in (0,T_{\varepsilon,\beta,R}]$ and $\varepsilon\in(0,\varepsilon_1]$, if $\varepsilon_1>0$ is small. By contradiction and continuity this shows $T_{\varepsilon,\beta,R}=T_0$ for all $\varepsilon\in(0,\varepsilon_1]$.\qedhere$_{1.}$
	\end{proof}

	\begin{proof}[Ad 2] 
	Let $k(N)\in\N$, $M=k(N)$ and let \eqref{eq_DC_ACND_uA} hold for some $\tilde{M}>M$ instead of $M$. Then the term in \eqref{eq_DC_ACND_DE2} where $R$ enters linearly is improved by a factor $\varepsilon^{\tilde{M}-M}$. Let $\beta\geq\overline{C}_0$ be fixed. Now we can first choose $R>0$ small such that the $R^{2+K(N)}$-term in \eqref{eq_DC_ACND_DE2} is bounded by $\frac{1}{8}R^2\varepsilon^{2M+1}$. Then $\varepsilon_1>0$ can be taken small such that \eqref{eq_DC_ACND_DE2} is estimated by $\frac{3}{4}R^2\varepsilon^{2M+1}$ for all $T\in (0,T_{\varepsilon,\beta,R}]$ and $\varepsilon\in(0,\varepsilon_1]$. By contradiction we get $T_{\varepsilon,\beta,R}=T_0$ for all $\varepsilon\in(0,\varepsilon_1]$.\qedhere$_{2.}$
	\end{proof}

	\begin{proof}[Ad 3] 
	Finally, let $N\in\{2,3\}$, $M=2$ and $\beta=0$. Then \eqref{eq_DC_ACND_DE2} is dominated by
	\[
	\left[\frac{R^2}{2}+CR^2T+CRT^{\frac{1}{2}}+CR^3 T^{\frac{4-N}{4}}\right]\varepsilon^{2M+1}.
	\]
	Due to $\frac{4-N}{4}>0$ there are $\varepsilon_1,T_1>0$ such that this is bounded by $\frac{3}{4}R^2\varepsilon^{2M+1}$ for every $T\in(0,\min(T_{\varepsilon,\beta,R},T_1)]$ and $\varepsilon\in(0,\varepsilon_1]$. Therefore $T_{\varepsilon,0,R}\geq T_1$ for all $\varepsilon\in(0,\varepsilon_1]$.
	\qedhere$_{3.}$\end{proof}
\end{proof}

\begin{proof}[Proof of Lemma \ref{th_DC_ACND_GN}]
First, let us estimate \eqref{eq_DC_ACND_proof4} for $\Omega\setminus\Gamma_t(\delta)$ instead of $\Omega$. The Gagliardo-Nirenberg Inequality in Lemma \ref{th_DC_GN} and Remark \ref{th_DC_GN_rem} imply for $2\leq N\leq 6$
\[
\|\overline{u}_\varepsilon(t)\|_{L^3(\Omega\setminus\Gamma_t(\delta))}^3\leq C\|\overline{u}_\varepsilon(t)\|_{L^2(\Omega\setminus\Gamma_t(\delta))}^{3(1-\frac{N}{6})}\|\overline{u}_\varepsilon(t)\|_{H^1(\Omega\setminus\Gamma_t(\delta))}^{3\frac{N}{6}},
\]
where the constant $C$ is independent of $t\in[0,T_0]$ because of Remark \ref{th_DC_GN_rem} and since $\Omega\setminus\Gamma_t(\delta)$ has a Lipschitz-boundary uniformly in $t\in[0,T_0]$, cf.~Remark \ref{th_SE_ACND_sob_rem}. For $3\frac{N}{6}>2$, i.e.~$N=5,6$, the right hand side can not be controlled with \eqref{eq_DC_ACND_DE}. But for $N\in\{2,3,4\}$ we can use this estimate and obtain that \eqref{eq_DC_ACND_proof4} for $\Omega\setminus\Gamma_t(\delta)$ instead of $\Omega$ is estimated by
\begin{align}\begin{split}\label{eq_DC_ACND_proof_GN1}
&\frac{C}{\varepsilon^2}\int_0^T g_\beta^{-1}
\!\left[g_\beta^3\|\overline{u}_\varepsilon(t)\|_{L^2(\Omega\setminus\Gamma_t(\delta))}^3
+g_\beta^{3-\frac{N}{2}}\|\overline{u}_\varepsilon\|_{L^2(\Omega\setminus\Gamma_t(\delta))}^{3-\frac{N}{2}} g_\beta^{\frac{N}{2}}\|\nabla\overline{u}_\varepsilon(t)\|_{L^2(\Omega\setminus\Gamma_t(\delta))}^{\frac{N}{2}}\right]\!dt\\
&\leq CR^3\varepsilon^{-2+3(M+\frac{1}{2})}\left[\|g_\beta^{-1}\|_{L^1(0,T)}+\|g_\beta^{-1}\|_{L^{\frac{4}{4-N}}(0,T)}\right]\end{split}
\end{align}
for all $T\in(0,T_{\varepsilon,\beta,R}]$ and $\varepsilon\in(0,\varepsilon_0]$, where we used \eqref{eq_DC_ACND_DE} and the Hölder Inequality with exponents $\infty$, $\frac{4}{N}$ and $\frac{4}{4-N}$ for the second term.
Now let $N\geq 5$. Then we consider $r(N)\in(2,3]$ (to be determined later) and estimate with the uniform boundedness from \eqref{eq_DC_ACND_proof3}
\[
\|\overline{u}_\varepsilon(t)\|_{L^3(\Omega\setminus\Gamma_t(\delta))}^3\leq
C\|\overline{u}_\varepsilon(t)\|_{L^{r(N)}(\Omega\setminus\Gamma_t(\delta))}^{r(N)}.
\]
By Lemma \ref{th_DC_GN} the Gagliardo-Nirenberg Inequality is applicable for all $\theta=\theta(N)\in[0,1]$ with
\[
\theta\left(\frac{1}{2}-\frac{1}{N}\right)+\frac{1-\theta}{2}=\frac{1}{r(N)}\quad\Leftrightarrow\quad \theta=N\left(\frac{1}{2}-\frac{1}{r(N)}\right).
\]
The condition $\theta\in[0,1]$ restricts $r(N)$ to be in $(2,2+\frac{4}{N-2}]$. Then Lemma \ref{th_DC_GN} and Remark \ref{th_DC_GN_rem} yield
\[
\|\overline{u}_\varepsilon(t)\|_{L^{r(N)}(\Omega\setminus\Gamma_t(\delta))}^{r(N)}
\leq C\|\overline{u}_\varepsilon(t)\|_{L^2(\Omega\setminus\Gamma_t(\delta))}^{r(N)(1-\theta)} \|\overline{u}_\varepsilon(t)\|_{H^1(\Omega\setminus\Gamma_t(\delta))}^{r(N)\theta}.
\]
Because we have to use \eqref{eq_DC_ACND_DE} to control the right hand side, we require $r(N)\theta\leq 2$. This is equivalent to
\[
N\left(\frac{r(N)}{2}-1\right)\leq 2\quad\Leftrightarrow\quad r(N)\leq 2+\frac{4}{N}.
\]
Since $2+\frac{4}{N}\leq\min\{3,2+\frac{4}{N-2}\}$ for $N\geq 5$, we can take $r(N):=2+\frac{4}{N}$ for $N\geq 5$. Therefore $r(N)\theta=2$ and $r(N)(1-\theta)=r(N)-r(N)\theta=\frac{4}{N}$. Hence for $N\geq 5$ we obtain the estimate
\[
\|\overline{u}_\varepsilon(t)\|_{L^3(\Omega\setminus\Gamma_t(\delta))}^3\leq C\|\overline{u}_\varepsilon(t)\|_{L^2(\Omega\setminus\Gamma_t(\delta))}^{\frac{4}{N}} \|\overline{u}_\varepsilon(t)\|_{H^1(\Omega\setminus\Gamma_t(\delta))}^2.
\]
Note that for $N=4$ the calculation also works but yields the same as before. Hence using \eqref{eq_DC_ACND_DE} we obtain for $N\geq 5$ that \eqref{eq_DC_ACND_proof4} with $\Omega$ replaced by $\Omega\setminus\Gamma_t(\delta)$ is estimated via
\begin{align}\notag
&\frac{C}{\varepsilon^2}\int_0^T g_\beta^{-\frac{4}{N}}\left[
g_\beta^{2+\frac{4}{N}}\|\overline{u}_\varepsilon(t)\|_{L^2(\Omega\setminus\Gamma_t(\delta))}^{2+\frac{4}{N}}
+g_\beta^{\frac{4}{N}}\|\overline{u}_\varepsilon(t)\|_{L^2(\Omega\setminus\Gamma_t(\delta))}^{\frac{4}{N}}g_\beta^2 \|\nabla\overline{u}_\varepsilon(t)\|_{L^2(\Omega\setminus\Gamma_t(\delta))}^2\right]dt\\
&\leq CR^{2+\frac{4}{N}}\varepsilon^{-2+(2+\frac{4}{N})(M+\frac{1}{2})}\left[\|g_\beta^{-\frac{4}{N}}\|_{L^1(0,T)}+\|g_\beta^{-\frac{4}{N}}\|_{L^\infty(0,T)}\right]\label{eq_DC_ACND_proof_GN2}
\end{align}
for all $T\in(0,T_{\varepsilon,\beta,R}]$ and $\varepsilon\in(0,\varepsilon_0]$.

Next we estimate \eqref{eq_DC_ACND_proof4} for $\Gamma_t(\delta)$ instead of $\Omega$. First we transform the integral to $(-\delta,\delta)\times\Sigma$ via $\overline{X}$. This yields
\[
\int_0^Tg_\beta^2\|\overline{u}_\varepsilon\|_{L^3(\Omega)}^3\,dt=\int_0^Tg_\beta^2\int_{-\delta}^\delta\int_\Sigma|\overline{u}_\varepsilon|_{\overline{X}(r,s,t)}|^3 J_t(r,s)\,d\Hc^{N-1}(s)\,dr\,dt.
\]
Here $0<c\leq J_t\leq C$ with $c,C>0$ independent of $t$ by Remark \ref{th_coordND_rem},~3. Because of Lemma \ref{th_SobMfd_def_lemma} and Remark \ref{th_DC_GN_rem} we can use the Gagliardo-Nirenberg Inequality for $\Sigma$ with $N-1$ instead of $N$ (and full $W^{1,p}$-norm on the right hand side). First we consider $N\in\{2,3,4\}$ since this was a special case for the estimate on $\Omega\setminus\Gamma_t(\delta)$, too. Then
\[
\int_0^T g_\beta^2\int_{-\delta}^\delta\|\overline{u}_\varepsilon|_{\overline{X}(r,.,t)}\|_{L^3(\Sigma)}^3\,dr\,dt
\leq \int_0^TCg_\beta^2\int_{-\delta}^\delta \|\overline{u}_\varepsilon|_{\overline{X}(r,.,t)}\|_{L^2(\Sigma)}^{\frac{7-N}{2}}\|\overline{u}_\varepsilon|_{\overline{X}(r,.,t)}\|_{H^1(\Sigma)}^{\frac{N-1}{2}}\,dr\,dt.
\]
By Lemma \ref{th_SobMfd_prod_set} we can use the Hölder Inequality with exponents $\frac{4}{5-N}$ and $\frac{4}{N-1}$. Therefore
\[
\int_{-\delta}^\delta\|\overline{u}_\varepsilon|_{\overline{X}(r,.,t)}\|_{L^3(\Sigma)}^3\,dr
\leq C \|\overline{u}_\varepsilon|_{\overline{X}(.,t)}\|_{L^{2\frac{7-N}{5-N}}(-\delta,\delta,L^2(\Sigma))}^{\frac{7-N}{2}}\|\overline{u}_\varepsilon|_{\overline{X}(.,t)}\|_{L^2(-\delta,\delta,H^1(\Sigma))}^{\frac{N-1}{2}}.
\]
Here $2\frac{7-N}{5-N}\in(2,\infty)$. Hence for the first term we can use the Gagliardo-Nirenberg Inequality for $(-\delta,\delta)$. The condition on the intermediate parameter $\theta=\theta(N)\in[0,1]$ is
\[
\theta\left(\frac{1}{2}-1\right)+\frac{1-\theta}{2}
=\frac{5-N}{2(7-N)}\quad\Leftrightarrow\quad 
\theta=\frac{1}{7-N}.
\]
Hence we obtain 
\[
\|\overline{u}_\varepsilon|_{\overline{X}(.,t)}\|_{L^{2\frac{7-N}{5-N}}(-\delta,\delta,L^2(\Sigma))}^{\frac{7-N}{2}}
\leq C\|\overline{u}_\varepsilon|_{\overline{X}(.,t)}\|_{L^2(-\delta,\delta,L^2(\Sigma))}^{\frac{6-N}{2}} \|\overline{u}_\varepsilon|_{\overline{X}(.,t)}\|_{H^1(-\delta,\delta,L^2(\Sigma))}^\frac{1}{2}.
\]
Hence \eqref{eq_DC_ACND_proof4} for $\Gamma_t(\delta)$ instead of $\Omega$ and $N\in\{2,3,4\}$ is estimated by
\begin{align}\label{eq_DC_ACND_proof_GN3}
\frac{C}{\varepsilon^2}\int_0^T g_\beta^2 \|\overline{u}_\varepsilon|_{\overline{X}(.,t)}\|_{L^2(-\delta,\delta,L^2(\Sigma))}^{\frac{6-N}{2}} \|\overline{u}_\varepsilon|_{\overline{X}(.,t)}\|_{H^1(-\delta,\delta,L^2(\Sigma))}^\frac{1}{2} 
\|\overline{u}_\varepsilon|_{\overline{X}(.,t)}\|_{L^2(-\delta,\delta,H^1(\Sigma))}^{\frac{N-1}{2}}\,dt.
\end{align}
Moreover, because of Lemma \ref{th_SobMfd_def_lemma}, Lemma \ref{th_SobMfd_prod_set} and Corollary \ref{th_coordND_nabla_tau_n},~1.~it holds 
\begin{align*}
\|\overline{u}_\varepsilon|_{\overline{X}(.,t)}\|_{H^1(-\delta,\delta,L^2(\Sigma))}
&\leq C(\|\overline{u}_\varepsilon|_{\overline{X}(.,t)}\|_{L^2((-\delta,\delta)\times\Sigma)}+\|\partial_n\overline{u}_\varepsilon|_{\overline{X}(.,t)}\|_{L^2((-\delta,\delta)\times\Sigma)}),\\
\|\overline{u}_\varepsilon|_{\overline{X}(.,t)}\|_{L^2(-\delta,\delta,H^1(\Sigma))}
&\leq C(\|\overline{u}_\varepsilon|_{\overline{X}(.,t)}\|_{L^2((-\delta,\delta)\times\Sigma)}+\|\nabla_\tau\overline{u}_\varepsilon|_{\overline{X}(.,t)}\|_{L^2((-\delta,\delta)\times\Sigma)}).
\end{align*}
For the product term in \eqref{eq_DC_ACND_proof_GN3} involving both $\partial_n\overline{u}_\varepsilon$ and $\nabla_\tau\overline{u}_\varepsilon$ as factors we apply the Hölder Inequality with exponents $\frac{4}{4-N}$, $4$ and $\frac{4}{N-1}$. For the term with $\partial_n\overline{u}_\varepsilon$ we use the exponents $\frac{4}{3}$, $4$ and for the one with $\nabla_\tau$ we use $\frac{4}{5-N}$, $\frac{4}{N-1}$.
Altogether \eqref{eq_DC_ACND_proof4} for $\Gamma_t(\delta)$ instead of $\Omega$ and $N\in\{2,3,4\}$ is controlled by
\begin{align}\label{eq_DC_ACND_proof_GN4}
CR^3\varepsilon^{3M-1}\!
\left[\|g_\beta^{-1}\|_{L^1(0,T)}\!+\!\|g_\beta^{-1}\|_{L^{\frac{4}{5-N}}\!(0,T)}\!+\!\|g_\beta^{-1}\|_{L^{\frac{4}{3}}\!(0,T)}\!+\!\|g_\beta^{-1}\|_{L^{\frac{4}{4-N}}\!(0,T)}\right]
\end{align}
for all $T\in(0,T_{\varepsilon,\beta,R}]$ and $\varepsilon\in(0,\varepsilon_0]$, where we used \eqref{eq_DC_ACND_DE} and $\varepsilon\leq\varepsilon_0$ for the terms that possess a higher $\varepsilon$-order. Now let $N\geq 5$. Then with \eqref{eq_DC_ACND_proof3} we estimate
\[
\|\overline{u}_\varepsilon(t)\|_{L^3(\Gamma_t(\delta))}^3 
\leq \|\overline{u}_\varepsilon(t)\|_{L^{\tilde{r}(N)}(\Gamma_t(\delta))}^{\tilde{r}(N)}
\]
for some $\tilde{r}(N)\in(2,3]$. We seek the maximal $\tilde{r}(N)$ such that similar calculations as above work. It will turn out that $\tilde{r}(N)=2+\frac{4}{N}$ is optimal. Note that the latter also was the best exponent for the estimate on $\Omega\setminus\Gamma_t(\delta)$ above in the case $N\geq 5$. But let us carry out the calculations with a general $\tilde{r}(N)$. The Gagliardo-Nirenberg Inequality on $\Sigma$ is applicable for $\theta=\theta(N)\in[0,1]$ with
\[
\theta\left(\frac{1}{2}-\frac{1}{N-1}\right)+\frac{1-\theta}{2}=\frac{1}{\tilde{r}(N)}\quad\Leftrightarrow\quad \theta=(N-1)\left(\frac{1}{2}-\frac{1}{\tilde{r}(N)}\right).
\]
This restricts $\tilde{r}(N)$ to be in $(2,2+\frac{4}{N-3}]$. Therefore for such $\tilde{r}(N)$ we obtain the estimate
\[
\int_0^T g_\beta^2\int_{-\delta}^\delta\|\overline{u}_\varepsilon|_{\overline{X}(r,.,t)}\|_{L^{\tilde{r}(N)}(\Sigma)}^{\tilde{r}(N)}\,dr\,dt
\leq \int_0^TCg_\beta^2\int_{-\delta}^\delta \|\overline{u}_\varepsilon|_{\overline{X}(r,.,t)}\|_{L^2(\Sigma)}^{q(N)}\|\overline{u}_\varepsilon|_{\overline{X}(r,.,t)}\|_{H^1(\Sigma)}^{p(N)}\,dr\,dt,
\]
where $p(N):=\theta\tilde{r}(N)=(N-1)\left(\frac{\tilde{r}(N)}{2}-1\right)$ and $q(N):=(1-\theta)\tilde{r}(N)=\tilde{r}(N)-p(N)$. The next step is to use the Hölder Inequality on $(-\delta,\delta)$. To this end we need $p(N)\leq 2$. This is equivalent to $\tilde{r}(N)\leq 2+\frac{4}{N-1}$. Hence for these $\tilde{r}(N)$ we can use the Hölder Inequality with exponents $\frac{2}{2-p(N)}$, $\frac{2}{p(N)}$ and obtain
\[
\int_{-\delta}^\delta\|\overline{u}_\varepsilon|_{\overline{X}(r,.,t)}\|_{L^{\tilde{r}(N)}(\Sigma)}^{\tilde{r}(N)}\,dr
\leq C \|\overline{u}_\varepsilon|_{\overline{X}(.,t)}\|_{L^{y(N)}(-\delta,\delta,L^2(\Sigma))}^{q(N)}\|\overline{u}_\varepsilon|_{\overline{X}(.,t)}\|_{L^2(-\delta,\delta,H^1(\Sigma))}^{p(N)},
\]
where we have set $y(N):=\frac{2q(N)}{2-p(N)}$. Note that $y(N)=\frac{2\tilde{r}(N)-2p(N)}{2-p(N)}>2$. Therefore we can use the Gagliardo-Nirenberg Inequality on $(-\delta,\delta)$. The condition for $\theta=\theta(N)\in[0,1]$ is
\[
\theta\left(\frac{1}{2}-1\right)+\frac{\theta}{2}=\frac{1}{y(N)}\quad\Leftrightarrow\quad \theta=\frac{1}{2}-\frac{1}{y(N)}=\frac{q(N)+p(N)-2}{2q(N)}=\frac{\tilde{r}(N)-2}{2q(N)}.
\]
Hence we obtain
\[
\|\overline{u}_\varepsilon|_{\overline{X}(.,t)}\|_{L^{y(N)}(-\delta,\delta,L^2(\Sigma))}^{q(N)}
\leq C\|\overline{u}_\varepsilon|_{\overline{X}(.,t)}\|_{L^2(-\delta,\delta,L^2(\Sigma))}^{q(N)+1-\frac{\tilde{r}(N)}{2}} \|\overline{u}_\varepsilon|_{\overline{X}(.,t)}\|_{H^1(-\delta,\delta,L^2(\Sigma))}^{\frac{\tilde{r}(N)}{2}-1}.
\]
Therefore \eqref{eq_DC_ACND_proof4} for $\Gamma_t(\delta)$ instead of $\Omega$ and $N\geq 5$ is estimated by
\begin{align}\label{eq_DC_ACND_proof_GN5}
\frac{C}{\varepsilon^2}\int_0^T g_\beta^2 \|\overline{u}_\varepsilon|_{\overline{X}(.,t)}\|_{L^2(-\delta,\delta,L^2(\Sigma))}^{1+q(N)-\frac{\tilde{r}(N)}{2}} \|\overline{u}_\varepsilon|_{\overline{X}(.,t)}\|_{H^1(-\delta,\delta,L^2(\Sigma))}^{\frac{\tilde{r}(N)}{2}-1} 
\|\overline{u}_\varepsilon|_{\overline{X}(.,t)}\|_{L^2(-\delta,\delta,H^1(\Sigma))}^{p(N)}\,dt.
\end{align}
We can estimate the terms using $\partial_n$ and $\nabla_\tau$, see below \eqref{eq_DC_ACND_proof_GN3}. The last step is to use the Hölder Inequality in time. For the product of the $\partial_n\overline{u}_\varepsilon$-term and the $\nabla_\tau\overline{u}_\varepsilon$-term we want to use the Hölder Inequality with exponents $2\frac{2}{\tilde{r}(N)-2}$, $\frac{2}{p(N)}$. This gives the following condition for $\tilde{r}(N)$:
\[
\frac{p(N)}{2}+\frac{\tilde{r}(N)-2}{2}\leq 2\quad\Leftrightarrow\quad 
\tilde{r}(N)\leq 2+\frac{4}{N}.
\]
For $\tilde{r}(N)=2+\frac{4}{N}$ all the calculations work and $p(N)=2-\frac{2}{N}=2\frac{N-1}{N}$, $q(N)=\frac{6}{N}$ as well as $1+q(N)-\frac{\tilde{r}(N)}{2}=\frac{4}{N}$ and $\frac{\tilde{r}(N)}{2}-1=\frac{2}{N}$. Altogether \eqref{eq_DC_ACND_proof4} for $\Gamma_t(\delta)$ instead of $\Omega$ and $N\geq 5$ is controlled by $CR^{2+\frac{4}{N}}$ times
\begin{align}\label{eq_DC_ACND_proof_GN6}
\varepsilon^{-2+(2+\frac{4}{N})M+1}\!
\left[\|g_\beta^{-\frac{4}{N}}\|_{L^1(0,T)}\!+\!\|g_\beta^{-\frac{4}{N}}\|_{L^N\!(0,T)}\!+\!\|g_\beta^{-\frac{4}{N}}\|_{L^{\frac{N}{N-1}}\!(0,T)}\!+\!\|g_\beta^{-\frac{4}{N}}\|_{L^\infty\!(0,T)}\right]
\end{align}
for all $T\in(0,T_{\varepsilon,\beta,R}]$ and $\varepsilon\in(0,\varepsilon_0]$, where we used \eqref{eq_DC_ACND_DE} and $\varepsilon\leq\varepsilon_0$.

Finally, we collect the above estimates. To reduce the number of $g_\beta$-terms we apply the embedding $L^p(0,t)\hookrightarrow L^q(0,t)$ for all $1\leq q\leq p\leq\infty$ and $0<t\leq T_0$ with embedding constant independent of $t$. Therefore if $N\in\{2,3,4\}$, then by \eqref{eq_DC_ACND_proof_GN1}, \eqref{eq_DC_ACND_proof_GN4}, and if $N\geq 5$, then by \eqref{eq_DC_ACND_proof_GN2}, \eqref{eq_DC_ACND_proof_GN6}, we get 
\begin{align*}
\left|\int_0^T\!g_\beta^2\int_\Omega r_\varepsilon(u_\varepsilon,u^A_\varepsilon)\overline{u}_\varepsilon\,dx\,dt\right|
\leq\begin{cases}
CR^3\varepsilon^{3M-1}\|g_\beta^{-1}\|_{L^{\frac{4}{4-N}}(0,T)} &\text{for } N\in\{2,3,4\},\\
CR^{2+\frac{4}{N}}\varepsilon^{(2+\frac{4}{N})M-1}\|g_\beta^{-\frac{4}{N}}\|_{L^\infty(0,T)} &\text{for }N\geq 4.
\end{cases}
\end{align*}
This shows Lemma \ref{th_DC_ACND_GN}.
\end{proof}

The proof of Theorem \ref{th_DC_ACND} is completed.

\subsubsection{Proof of Theorem \ref{th_AC_conv}}\label{sec_DC_ACND_conv}
Let $N\geq2$, $\Omega$, $Q_T$ and $\partial Q_T$ be as in Remark \ref{th_intro_coord},~1. Moreover, let $\Gamma=(\Gamma_t)_{t\in[0,T_0]}$ for some $T_0>0$ be a smooth solution to \eqref{MCF} with $90$°-contact angle condition parametrized as in Section \ref{sec_coord_surface_requ} and let $\delta>0$ be such that Theorem \ref{th_coordND} holds for $2\delta$ instead of $\delta$. We use the notation from Section \ref{sec_coord_surface_requ} and Section \ref{sec_coordND}. Let $M\in\N$ with $M\geq 2$ and denote with $(u^A_\varepsilon)_{\varepsilon>0}$ the approximate solution on $\overline{Q_{T_0}}$ defined in Section \ref{sec_asym_ACND_uA} (which we obtained from asymptotic expansions in Section \ref{sec_asym_ACND}) and let $\varepsilon_0>0$ be such that Lemma \ref{th_asym_ACND_uA} (\enquote{remainder estimate}) holds for $\varepsilon\in(0,\varepsilon_0]$. The property $\lim_{\varepsilon\rightarrow 0} u^A_\varepsilon=\pm 1$ uniformly on compact subsets of $Q_{T_0}^\pm$ follows from the construction in Section \ref{sec_asym_ACND}. 

Note that $g_\beta(t):=e^{-\beta t}$ for fixed $\beta$ is trapped between uniform positive constants for all $t\in[0,T_0]$. Therefore Theorem \ref{th_AC_conv} follows immediately from Theorem \ref{th_DC_ACND} if we show the conditions 1.-4.~in Theorem \ref{th_DC_ACND}. The requirement 1.~(\enquote{uniform boundedness}) is fulfilled due to Lemma \ref{th_asym_ACND_uA} for $u^A_\varepsilon$ and for $u_{0,\varepsilon}$ this is an assumption in Theorem \ref{th_AC_conv}. Condition 2.~(\enquote{spectral estimate}) is valid because of Theorem \ref{th_SE_ACND}. Requirement 4.~(\enquote{well prepared initial data}) is a condition on $u_{0,\varepsilon}$ and assumed in Theorem \ref{th_AC_conv}. It remains to prove 3.~(\enquote{approximate solution}). 

First we estimate the boundary term in \eqref{eq_DC_ACND_uA}. Lemma \ref{th_asym_ACND_uA} yields $s^A_\varepsilon=0$ on $\partial\Omega\setminus\Gamma_t(2\delta)$ and $|s^A_\varepsilon|\leq C\varepsilon^Me^{-c|\rho_\varepsilon|}$, where $\rho_\varepsilon$ is defined in \eqref{eq_asym_ACND_rho}. Therefore
\[
\left|\int_{\partial\Omega}s^A_\varepsilon\tr\,\overline{u}_\varepsilon(t)\,d\Hc^{N-1}\right|\leq \|s^A_\varepsilon\|_{L^2(\partial\Omega\cap\Gamma_t(2\delta))}
\|\tr\,\overline{u}_\varepsilon(t)\|_{L^2(\partial\Omega\cap\Gamma_t(2\delta))}.
\]
Here by the substitution rule in Theorem \ref{th_Leb_trafo_mfd} it holds
\[
\|s^A_\varepsilon\|_{L^2(\partial\Omega\cap\Gamma_t(2\delta))}^2=\int_{\partial\Sigma}\int_{-2\delta}^{2\delta}|s^A_\varepsilon|^2|_{\overline{X}(r,Y(\sigma,0),t)}|\!\det d_{(r,\sigma)}[X(.,Y(.,0),t)]|\,dr\,d\Hc^{N-2}(\sigma).
\]
With a scaling argument this is estimated by $C\varepsilon^{2M+1}$, see Lemma \ref{th_SE_1Dtrafo_remainder}. Moreover, one can prove
\[
\|\tr\,\overline{u}_\varepsilon(t)\|_{L^2(\partial\Omega\cap\Gamma_t(2\delta))}
\leq C
(\|\overline{u}_\varepsilon(t)\|_{L^2(\Gamma_t(2\delta))}
+\|\nabla_\tau\overline{u}_\varepsilon(t)\|_{L^2(\Gamma_t(2\delta))}).
\]
This can be shown with a similar idea as in the proof of Lemma \ref{th_SE_ACND_intpol_tr}. Here one uses $\vec{w}$ in the proof of the latter with $w_1:=0$ there and then Corollary \ref{th_coordND_nabla_tau_n} to estimate $|\nabla_\Sigma(\overline{u}_\varepsilon|_{\overline{X}})|\leq C|\nabla_\tau\overline{u}_\varepsilon|_{\overline{X}}|$. Moreover, $|\nabla_\tau\overline{u}_\varepsilon|\leq C|\nabla\overline{u}_\varepsilon|$ by Corollary \ref{th_coordND_nabla_tau_n} and the estimate for the $s^A_\varepsilon$-term in \eqref{eq_DC_ACND_uA} follows. 

Finally, we estimate the $r^A_\varepsilon$-term in \eqref{eq_DC_ACND_uA}. Lemma \ref{th_asym_ACND_uA} yields $r^A_\varepsilon=0$ in $\Omega\setminus\Gamma_t(2\delta)$ and
\begin{alignat*}{2}
|r^A_\varepsilon|&\leq C(\varepsilon^M e^{-c|\rho_\varepsilon|}+\varepsilon^{M+1})&\quad &\text{ in }\Gamma(2\delta,\mu_1),\\
|r^A_\varepsilon|&
\leq 
C(\varepsilon^{M-1} e^{-c(|\rho_\varepsilon|+H_\varepsilon)}
+\varepsilon^M e^{-c|\rho_\varepsilon|}+\varepsilon^{M+1})
&\quad &\text{ in }\Gamma^C(2\delta,2\mu_1).
\end{alignat*}
The substitution rule in Theorem \ref{th_Leb_trafo_mfd} implies
\[
\left|\int_\Omega r^A_\varepsilon\overline{u}_\varepsilon(t)\,dx\right|
\leq\int_{\Gamma_t(2\delta)}|r^A_\varepsilon\overline{u}_\varepsilon(t)|\,dx=\int_\Sigma\int_{-2\delta}^{2\delta}|r^A_\varepsilon\overline{u}_\varepsilon|_{\overline{X}(r,s,t)}| J_t(r,s)\,dr\,d\Hc^{N-1}(s),
\]
where $J_t$ is uniformly bounded in $t\in[0,T_0]$ by Remark \ref{th_coordND_rem}, 3. We split the integral over $\Sigma$ in integrals over $\Sigma\setminus Y(\partial\Sigma\times[0,2\mu_1])$ and $Y(\partial\Sigma\times[0,2\mu_1])$. For both we use the Hölder Inequality with exponents $2$, $2$ for the inner integral. With a scaling argument, cf.~Lemma \ref{th_SE_1Dtrafo_remainder}, and Hölder's inequality the integral over $\Sigma\setminus Y(\partial\Sigma\times[0,2\mu_1])$ is estimated by
\[
C\varepsilon^{M+\frac{1}{2}}\int_{\Sigma\setminus Y(\partial\Sigma\times[0,2\mu_1])} \|\overline{u}_\varepsilon|_{\overline{X}(.,s,t)}\|_{L^2(-2\delta,2\delta)}\,d\Hc^{N-1}(s)\leq C\varepsilon^{M+\frac{1}{2}}\|\overline{u}_\varepsilon(t)\|_{L^2(\Gamma_t(2\delta,2\mu_1))}.
\]
Moreover, by the substitution rule in Theorem \ref{th_Leb_trafo_mfd} the integral over $Y(\partial\Sigma\times[0,2\mu_1])$ is controlled via
\[
C\varepsilon^{M-\frac{1}{2}}\int_{\partial\Sigma}\int_0^{2\mu_1} \|\overline{u}_\varepsilon|_{\overline{X}(.,Y(\sigma,b),t)}\|_{L^2(-2\delta,2\delta)}(e^{-c\frac{b}{\varepsilon}}+\varepsilon)\,db\,d\Hc^{N-2}(\sigma).
\]
For the $\overline{u}_\varepsilon$-term we use $H^1(0,2\mu_1,L^2(-2\delta,2\delta))\hookrightarrow L^\infty(0,2\mu_1,L^2(-2\delta,2\delta))$. Moreover, a scaling argument yields $\int_0^{2\mu_1}e^{-cb/\varepsilon}\,db\leq C\varepsilon$. Hence the above term is estimated by
\[
C\varepsilon^{M+\frac{1}{2}}\int_{\partial\Sigma}\|\overline{u}_\varepsilon|_{\overline{X}(.,Y(\sigma,.),t)}\|_{L^2(-2\delta,2\delta,H^1(0,2\mu_1))}\,d\Hc^{N-2}(\sigma).
\]
Note that due to Lemma \ref{th_SobDom_prod_set} the expression $\|\overline{u}_\varepsilon|_{\overline{X}(.,Y(\sigma,.),t)}\|_{L^2(-2\delta,2\delta,H^1(0,2\mu_1))}$ equals 
\[
\|\overline{u}_\varepsilon|_{\overline{X}(.,Y(\sigma,.),t)}\|_{L^2((-2\delta,2\delta)\times(0,2\mu_1))}
+\|\partial_b\overline{u}_\varepsilon|_{\overline{X}(.,Y(\sigma,.),t)}\|_{L^2((-2\delta,2\delta)\times(0,2\mu_1))}
\] 
and Corollary \ref{th_coordND_nabla_tau_n} yields
\[
\|\partial_b\overline{u}_\varepsilon|_{\overline{X}(.,Y(\sigma,.),t)}\|_{L^2((-2\delta,2\delta)\times(0,2\mu_1))}\leq C\|\nabla_\tau\overline{u}_\varepsilon|_{\overline{X}(.,Y(\sigma,.),t)}\|_{L^2((-2\delta,2\delta)\times(0,2\mu_1))}.
\]
Finally, by Theorem \ref{th_Leb_trafo_mfd}, Lemma \ref{th_SobDom_prod_set} and Hölder's inequality we obtain
\[
\left|\int_\Omega r^A_\varepsilon\overline{u}_\varepsilon(t)\,dx\right|\leq C\varepsilon^{M+\frac{1}{2}}(\|\overline{u}_\varepsilon(t)\|_{L^2(\Gamma_t(2\delta))}+\|\nabla_\tau\overline{u}_\varepsilon(t)\|_{L^2(\Gamma_t(2\delta))}).
\]
The estimate $|\nabla_\tau\overline{u}_\varepsilon|\leq C|\nabla\overline{u}_\varepsilon|$ due to Corollary \ref{th_coordND_nabla_tau_n} yields \eqref{eq_DC_ACND_uA}. Therefore Theorem \ref{th_AC_conv} follows from the difference estimates in Theorem \ref{th_DC_ACND}.\hfill$\square$

\subsection[Difference Estimate and Proof of the Convergence Thm. for (vAC) in ND]{Difference Estimate and Proof of the Convergence Theorem for (vAC) in ND}\label{sec_DC_vAC}
We show in Section \ref{sec_DC_vAC_DE} the difference estimate for exact and suitable approximate solutions for the vector-valued Allen-Cahn equation \eqref{eq_vAC1}-\eqref{eq_vAC3}. Then in Section \ref{sec_DC_vAC_conv} we prove the Theorem \ref{th_vAC_conv} about convergence by checking the requirements for the difference estimate applied to the approximate solution from Section \ref{sec_asym_vAC_uA}. All computations are analogous to the scalar case in the last Section \ref{sec_DC_ACND}.

\subsubsection{Difference Estimate}\label{sec_DC_vAC_DE}

\begin{Theorem}[\textbf{Difference Estimate for (vAC)}]\label{th_DC_vAC}
	Let $N\geq2$, $\Omega$, $Q_T$ and $\partial Q_T$ be as in Remark \ref{th_intro_coord},~1. Moreover, let $\Gamma=(\Gamma_t)_{t\in[0,T_0]}$ for some $T_0>0$ be as in Section \ref{sec_coordND} and $\delta>0$ be such that Theorem \ref{th_coordND} holds for $2\delta$ instead of $\delta$. We use the notation for $\Gamma_t(\delta)$, $\Gamma(\delta)$, $\nabla_\tau$ and $\partial_n$ from Remark \ref{th_coordND_rem}. Additionally, let $W:\R^m\rightarrow\R$ be as in Definition \ref{th_vAC_W}.
	
	Moreover, let $\check{\varepsilon}_0>0$, $\vec{u}^A_\varepsilon\in C^2(\overline{Q_{T_0}})^m$, $\vec{u}_{0,\varepsilon}\in C^2(\overline{\Omega})^m$ with $\partial_{N_{\partial\Omega}} \vec{u}_{0,\varepsilon}=0$ on $\partial\Omega$ and let $\vec{u}_\varepsilon\in C^2(\overline{Q_{T_0}})^m$ be exact solutions to \eqref{eq_vAC1}-\eqref{eq_vAC3} with $\vec{u}_{0,\varepsilon}$ in \eqref{eq_vAC3} for $\varepsilon\in(0,\check{\varepsilon}_0]$. 
	
	For some $R>0$ and $M\in\N, M\geq k(N):=\max\{2,\frac{N}{2}\}$ we impose the following conditions:
	\begin{enumerate}
		\item \textup{Uniform Boundedness:} $\sup_{\varepsilon\in(0,\check{\varepsilon}_0]}\|\vec{u}^A_\varepsilon\|_{L^\infty(Q_{T_0})^m}+\|\vec{u}_{0,\varepsilon}\|_{L^\infty(\Omega)^m}<\infty$.
		\item \textup{Spectral Estimate:} There are $\check{c}_0,\check{C}>0$ such that
		\begin{align*}
		\int_\Omega|\nabla\vec{\psi}|^2+\frac{1}{\varepsilon^2}&(\vec{\psi},D^2W(\vec{u}^A_\varepsilon(.,t))\vec{\psi})_{\R^m}\,dx\\
		&\geq -\check{C}\|\vec{\psi}\|_{L^2(\Omega)^m}^2+\|\nabla\vec{\psi}\|_{L^2(\Omega\setminus\Gamma_t(\delta))^{N\times m}}^2+\check{c}_0\|\nabla_\tau\vec{\psi}\|_{L^2(\Gamma_t(\delta))^{N\times m}}^2
		\end{align*}
		for all $\vec{\psi}\in H^1(\Omega)^m$ and $\varepsilon\in(0,\check{\varepsilon}_0],t\in[0,T_0]$.
		\item \textup{Approximate Solution:} For the remainders 
		\[
		\vec{r}^A_\varepsilon:=\partial_t \vec{u}^A_\varepsilon-\Delta \vec{u}^A_\varepsilon+\frac{1}{\varepsilon^2}\nabla W(\vec{u}^A_\varepsilon)\quad\text{ and }\quad \vec{s}^A_\varepsilon:=\partial_{N_{\partial\Omega}}\vec{u}^A_\varepsilon
		\]
		in \eqref{eq_vAC1}-\eqref{eq_vAC2} for $\vec{u}^A_\varepsilon$ and the difference $\underline{u}_\varepsilon:=\vec{u}_\varepsilon-\vec{u}^A_\varepsilon$ it holds
		\begin{align}\begin{split}\label{eq_DC_vAC_uA}
		&\left|\int_{\partial\Omega}\vec{s}^A_\varepsilon\cdot\tr\,\underline{u}_\varepsilon(t)\,d\Hc^{N-1}+\int_\Omega \vec{r}^A_\varepsilon\cdot\underline{u}_\varepsilon(t)\,dx\right|\\
		&\leq\! C\varepsilon^{M+\frac{1}{2}}(\|\underline{u}_\varepsilon(t)\|_{L^2(\Omega)^m}\!
		+\!\|\nabla_\tau\underline{u}_\varepsilon(t)\|_{L^2(\Gamma_t(\delta))^{N\times m}}\!+\!\|\nabla\underline{u}_\varepsilon(t)\|_{L^2(\Omega\setminus\Gamma_t(\delta))^{N\times m}}\!)
		\end{split}
		\end{align}
		for all $\varepsilon\in(0,\check{\varepsilon}_0]$ and $T\in(0,T_0]$.
		\item \textup{Well-Prepared Initial Data:} For all $\varepsilon\in(0,\check{\varepsilon}_0]$ it holds
		\begin{align}\label{eq_DC_vAC_u0} \|\vec{u}_{0,\varepsilon}-\vec{u}^A_\varepsilon|_{t=0}\|_{L^2(\Omega)^m}\leq R\varepsilon^{M+\frac{1}{2}}.
		\end{align}
	\end{enumerate} 
	Then we obtain
	\begin{enumerate}
		\item Let $M>k(N)$. Then there are $\beta,\check{\varepsilon}_1>0$ such that for $g_\beta(t):=e^{-\beta t}$ it holds
		\begin{align}
		\begin{split}
		\sup_{t\in[0,T]}\|(g_\beta\underline{u}_\varepsilon)(t)\|_{L^2(\Omega)^m}^2+\|g_\beta\nabla\underline{u}_\varepsilon\|_{L^2(Q_T\setminus\Gamma(\delta))^{N\times m}}^2&\leq 2R^2\varepsilon^{2M+1},\\
		\check{c}_0\|g_\beta\nabla_\tau\underline{u}_\varepsilon\|^2_{L^2(Q_T\cap\Gamma(\delta))^{N\times m}}+\varepsilon^2\|g_\beta\partial_n\underline{u}_\varepsilon\|^2_{L^2(Q_T\cap\Gamma(\delta))^m}&\leq 2R^2\varepsilon^{2M+1}\label{eq_DC_vAC_DE}\end{split}
		\end{align}
		for all $\varepsilon\in(0,\check{\varepsilon}_1]$ and $T\in(0,T_0]$.
		\item Let $k(N)\in\N$ and $M=k(N)$. Let \eqref{eq_DC_vAC_uA} hold for some $\check{M}>M$ instead of $M$. Then there are $\beta,\check{R},\check{\varepsilon}_1>0$ such that, if \eqref{eq_DC_vAC_u0} holds for $\check{R}$ instead of $R$, then $\eqref{eq_DC_vAC_DE}$ for $\check{R}$ instead of $R$ is valid for all $\varepsilon\in(0,\check{\varepsilon}_1], T\in(0,T_0]$. 
		\item Let $N\in\{2,3\}$ and $M=2(=k(N))$. Then there are $\check{\varepsilon}_1,\check{T}_1>0$ such that $\eqref{eq_DC_vAC_DE}$ holds for $\beta=0$ and for all $\varepsilon\in(0,\check{\varepsilon}_1], T\in(0,\check{T}_1]$.
	\end{enumerate}
\end{Theorem}
\begin{Remark}\upshape\phantomsection{\label{th_DC_vAC_rem}}
	\begin{enumerate}
	\item The parameter $M$ corresponds to the order of the approximate solution constructed in Section \ref{sec_asym_vAC}.
	\item The comments for the scalar case in Remark \ref{th_DC_ACND_rem}, 2.-4., on the role of the parameters $\beta$, $k(N)$ and weaker requirements in the theorem, hold analogously for the vector-valued case. More precisely, see \eqref{eq_DC_vAC_DE2} below. The critical order $k(N)$ is the same as in the scalar case because we use the same Gagliardo-Nirenberg estimates that were used in the scalar case in the proof of Lemma \ref{th_DC_ACND_GN}. 
	\end{enumerate}
\end{Remark}
\begin{proof}[Proof of Theorem \ref{th_DC_vAC}]
	The continuity of the objects on the left hand side in \eqref{eq_DC_vAC_DE} yields that
	\begin{align}\label{eq_DC_vAC_T_epsR}
	\check{T}_{\varepsilon,\beta,R}:=\sup\,\{\tilde{T}\in(0,T_0]: \eqref{eq_DC_vAC_DE}\text{ holds for }\varepsilon, R\text{ and all }T\in(0,\tilde{T}]\}
	\end{align}
	is well-defined for all $\varepsilon\in(0,\check{\varepsilon}_0],\beta\geq 0$ and  $\check{T}_{\varepsilon,\beta,R}>0$. In the different cases we have to show:
	\begin{enumerate}
		\item If $M>k(N)$, then there are $\beta,\check{\varepsilon}_1>0$ such that $\check{T}_{\varepsilon,\beta,R}=T_0$ for all $\varepsilon\in(0,\check{\varepsilon}_1]$.
		\item If $M=k(N)\in\N$, then there are $\beta,\check{R},\check{\varepsilon}_1>0$ such that $\check{T}_{\varepsilon,\beta,\check{R}}=T_0$ provided that $\varepsilon\in(0,\check{\varepsilon}_1]$ and \eqref{eq_DC_vAC_uA} is true for some $\check{M}>M$ instead of $M$ and \eqref{eq_DC_vAC_u0} is valid with $R$ replaced by $\check{R}$.
		\item If $N\in\{2,3\}$, $M=2$, then there are $\check{T}_1,\check{\varepsilon}_1>0$ such that $\check{T}_{\varepsilon,0,R}\geq \check{T}_1$ for all $\varepsilon\in(0,\check{\varepsilon}_1]$.
	\end{enumerate}
	
	We do a general computation first and consider the specific cases later. The difference of the left hand sides in \eqref{eq_vAC1} for $\vec{u}_\varepsilon$ and $\vec{u}^A_\varepsilon$ yields
	\begin{align}\label{eq_DC_vAC_proof1}
	\left[\partial_t-\Delta+\frac{1}{\varepsilon^2}D^2W(\vec{u}^A_\varepsilon)\right]\underline{u}_\varepsilon=-\vec{r}^A_\varepsilon-\vec{r}_\varepsilon(\vec{u}_\varepsilon,\vec{u}^A_\varepsilon),
	\end{align} 
	where $\vec{r}_\varepsilon(\vec{u}_\varepsilon,\vec{u}^A_\varepsilon):=\frac{1}{\varepsilon^2}\left[\nabla W(\vec{u}_\varepsilon)-\nabla W(\vec{u}^A_\varepsilon)-D^2W(\vec{u}^A_\varepsilon)\underline{u}_\varepsilon\right]$. We multiply \eqref{eq_DC_vAC_proof1} by $g_\beta^2\underline{u}_\varepsilon$ and integrate over $Q_T$ for $T\in(0,\check{T}_{\varepsilon,\beta,R}]$, where $\varepsilon\in(0,\check{\varepsilon}_0]$ and $\beta\geq 0$ are fixed. This implies 
	\begin{align}\begin{split}\label{eq_DC_vAC_proof2}
	\int_0^T g _\beta^2\int_\Omega\underline{u}_\varepsilon \cdot &\left[\partial_t-\Delta+\frac{1}{\varepsilon^2}D^2W(\vec{u}^A_\varepsilon)\right]\underline{u}_\varepsilon\,dx\,dt\\
	&=-\int_0^T  g _\beta^2\int_\Omega [\vec{r}^A_\varepsilon+\vec{r}_\varepsilon(\vec{u}_\varepsilon,\vec{u}^A_\varepsilon)]\cdot\underline{u}_\varepsilon\,dx\,dt\end{split}
	\end{align} 
	for all $T\in(0,\check{T}_{\varepsilon,\beta,R}]$, $\varepsilon\in(0,\check{\varepsilon}_0]$ and $\beta\geq 0$.
	We estimate all terms. First, $\frac{1}{2}\partial_t|\underline{u}_\varepsilon|^2=\underline{u}_\varepsilon\cdot\partial_t\underline{u}_\varepsilon$, integration by parts in time and $\partial_tg_\beta=-\beta g_\beta$ yield
	\[
	\int_0^T\int_\Omega g_\beta^2\partial_t\underline{u}_\varepsilon\cdot\underline{u}_\varepsilon\,dx\,dt=\frac{1}{2}g_\beta(T)^2\|\underline{u}_\varepsilon(T)\|_{L^2(\Omega)^m}^2-\frac{1}{2}\|\underline{u}_\varepsilon(0)\|_{L^2(\Omega)^m}^2+\beta\int_0^T\! g_\beta^2 \|\underline{u}_\varepsilon\|_{L^2(\Omega)^m}^2\,dt,
	\]
	where $\|\underline{u}_\varepsilon(0)\|_{L^2(\Omega)^m}^2\leq R^2\varepsilon^{2M+1}$ because of \eqref{eq_DC_vAC_u0} (\enquote{well-prepared initial data}). For the other term on the left hand side in \eqref{eq_DC_vAC_proof2} we use integration by parts in space. This yields
	\begin{align*}
	&\int_0^Tg_\beta^2\int_\Omega \underline{u}_\varepsilon\cdot\left[-\Delta+\frac{1}{\varepsilon^2}D^2W(\vec{u}^A_\varepsilon)\right]\underline{u}_\varepsilon\,dx\,dt\\
	&=\int_0^Tg_\beta^2\int_\Omega|\nabla\underline{u}_\varepsilon|^2+\frac{1}{\varepsilon^2}(\underline{u}_\varepsilon,D^2W(\vec{u}^A_\varepsilon(.,t))\underline{u}_\varepsilon)_{\R^m}\,dx\,dt +\int_0^Tg_\beta^2\int_{\partial\Omega}\vec{s}^A_\varepsilon\cdot\tr\,\underline{u}_\varepsilon\,d\Hc^{N-1}\,dt.
	\end{align*}
	Using requirement 2.~(\enquote{spectral estimate}) in the theorem we obtain that the first integral on the right hand side of the latter equation is bounded from below by
	\[
	-\check{C}\|\underline{u}_\varepsilon\|_{L^2(\Omega)^m}^2
	+\|\nabla \underline{u}_\varepsilon\|_{L^2(\Omega\setminus\Gamma_t(\delta))^{N\times m}}^2
	+\check{c}_0\|\nabla_\tau\underline{u}_\varepsilon\|_{L^2(\Gamma_t(\delta))^{N\times m}}^2.
	\]
	For the remainder terms involving $\vec{r}^A_\varepsilon$ and $\vec{s}^A_\varepsilon$ we apply \eqref{eq_DC_vAC_uA} (\enquote{approximate solution}). Hence
	\[
	\left|\int_0^Tg_\beta^2\left[\int_{\partial\Omega}\vec{s}^A_\varepsilon\cdot\tr\,\underline{u}_\varepsilon(t)\,d\Hc^{N-1}
	+\int_\Omega \vec{r}^A_\varepsilon\cdot\underline{u}_\varepsilon(t)\,dx\right]dt\right|
	\leq \check{C}_1R\|g_\beta\|_{L^2(0,T)}\varepsilon^{2M+1}
	\]
	due to \eqref{eq_DC_vAC_DE} for all $T\in(0,\check{T}_{\varepsilon,\beta,R}]$, $\varepsilon\in(0,\check{\varepsilon}_0]$, where $\|g_\beta\|_{L^1(0,T)}\leq \sqrt{T_0}\|g_\beta\|_{L^2(0,T)}$ is used.
	
	Now we estimate the $\vec{r}_\varepsilon$-term in \eqref{eq_DC_vAC_proof2}. The requirement 1.~(\enquote{uniform boundedness}) in the theorem and Lemma \ref{th_DC_bdd_vect} yield
	\begin{align}\label{eq_DC_vAC_proof3}
	\sup_{\varepsilon\in(0,\check{\varepsilon}_0]}\left[\|\vec{u}_\varepsilon\|_{L^\infty(Q_{T_0})^m}+\|\vec{u}^A_\varepsilon\|_{L^\infty(Q_{T_0})^m}\right]<\infty.
	\end{align}
	Therefore the Taylor Theorem yields
	\begin{align}\label{eq_DC_vAC_proof4}
	\left|\int_0^Tg_\beta^2\int_\Omega \vec{r}_\varepsilon(\vec{u}_\varepsilon,\vec{u}^A_\varepsilon)\cdot\underline{u}_\varepsilon\,dx\,dt\right|
	\leq \frac{C}{\varepsilon^2}\int_0^Tg_\beta^2\|\underline{u}_\varepsilon\|_{L^3(\Omega)^m}^3\,dt.
	\end{align}
    This term can be estimated in the analogous way as in the scalar case with Gagliardo-Nirenberg inequalities on $\Omega\setminus\Gamma_t(\delta)$ and $\Gamma_t(\delta)$, cf.~the proof of Lemma \ref{th_DC_ACND_GN}. This yields
	\begin{align*}
	\left|\int_0^T\!\!g_\beta^2\int_\Omega \vec{r}_\varepsilon(\vec{u}_\varepsilon,\vec{u}^A_\varepsilon)\cdot\underline{u}_\varepsilon\,dx\,dt\right|
	\leq
	CR^{2+K(N)}\varepsilon^{2M+1} \varepsilon^{K(N)(M-k(N))}\|g_\beta^{-K(N)}\|_{L^{\frac{4}{4-\min\{4,N\}}}\!(0,T)}
	\end{align*}
	for all $T\in(0,\check{T}_{\varepsilon,\beta,R}]$ and $\varepsilon\in(0,\check{\varepsilon}_0]$, where $K(N):=\min\{1,\frac{4}{N}\}\in(0,1]$.
	
	In order to control $\partial_n\underline{u}_\varepsilon$ we use Corollary \ref{th_coordND_nabla_tau_n} and obtain
	\begin{align*}
	\varepsilon^2\|g_\beta\partial_n\underline{u}_\varepsilon\|_{L^2(Q_T\cap\Gamma(\delta))^m}^2
	&\leq C\varepsilon^2\int_0^T g_\beta^2\int_\Omega|\nabla\underline{u}_\varepsilon|^2+\frac{1}{\varepsilon^2}(\underline{u}_\varepsilon,D^2W(\vec{u}^A_\varepsilon(.,t))\underline{u}_\varepsilon)_{\R^m}\,dx\,dt\\
	&+C\sup_{\varepsilon\in(0,\check{\varepsilon}_0]}\|D^2W(\vec{u}^A_\varepsilon)\|_{L^\infty(Q_{T_0})^{m\times m}}\int_0^T g_\beta^2\|\underline{u}_\varepsilon(t)\|_{L^2(\Omega)^m}^2\,dt
	\end{align*}
	with a constant $C>0$ independent of $\varepsilon$, $T$ and $R$. The first term is absorbed with $\frac{1}{2}$ of the spectral term above if $\varepsilon\in(0,\check{\varepsilon}_1]$ and $\check{\varepsilon}_1>0$ is small (independent of $T$, $R$). Altogether we obtain
	\begin{align}\begin{split}
	&\frac{1}{2}g_\beta(T)\|\underline{u}_\varepsilon(T)\|_{L^2(\Omega)^m}^2
	+\frac{1}{2}\|g_\beta\nabla\underline{u}_\varepsilon\|_{L^2(Q_T\setminus\Gamma(\delta))^{N\times m}}^2\\
	&+\frac{\check{c}_0}{2}\|g_\beta\nabla_\tau\underline{u}_\varepsilon\|_{L^2(Q_T\cap\Gamma(\delta))^{N\times m}}^2
	+\frac{1}{2}\varepsilon^2\|g_\beta\partial_n\underline{u}_\varepsilon\|_{L^2(Q_T\cap\Gamma(\delta))^m}^2\\
	&\leq
	\frac{R^2}{2}\varepsilon^{2M+1}+\int_0^T(-\beta+\check{C}_0)g_\beta^2\|\underline{u}_\varepsilon(t)\|_{L^2(\Omega)^m}^2\,dt+\check{C}_1R\varepsilon^{2M+1}\|g_\beta\|_{L^2(0,T)}\\
	&+CR^{2+K(N)}\varepsilon^{2M+1} \varepsilon^{K(N)(M-k(N))}\|g_\beta^{-K(N)}\|_{L^{\frac{4}{4-\min\{4,N\}}}(0,T)}\label{eq_DC_vAC_DE2}\end{split}
	\end{align}
	for all $T\in(0,\check{T}_{\varepsilon,\beta,R}]$, $\varepsilon\in(0,\check{\varepsilon}_1]$ and constants $\check{C}_0,\check{C}_1,C>0$ independent of $\varepsilon,T,R$, where $k(N)=\max\{2,\frac{N}{2}\}$ and  $K(N)=\min\{1,\frac{4}{N}\}$. 
	
	Now we consider the different cases in the theorem.
	
	\begin{proof}[Ad 1] 
		If $M>k(N)$, then we choose $\beta\geq \check{C}_0$ large such that $\check{C}_1R\|g_\beta\|_{L^2(0,T_0)}\leq\frac{R^2}{8}$.
		Then \eqref{eq_DC_vAC_DE2} is estimated by $\frac{3}{4}R^2\varepsilon^{2M+1}$ for all $T\in (0,\check{T}_{\varepsilon,\beta,R}]$ and $\varepsilon\in(0,\check{\varepsilon}_1]$, if $\check{\varepsilon}_1>0$ is small. Via contradiction and continuity this proves $\check{T}_{\varepsilon,\beta,R}=T_0$ for all $\varepsilon\in(0,\check{\varepsilon}_1]$.\qedhere$_{1.}$
	\end{proof}
	
	\begin{proof}[Ad 2] 
		Let $M=k(N)\in\N$ and let \eqref{eq_DC_vAC_uA} hold for some $\check{M}>M$ instead of $M$. Then the term in \eqref{eq_DC_vAC_DE2} where $R$ enters linearly is improved by a factor $\varepsilon^{\check{M}-M}$. Let $\beta\geq\check{C}_0$ be fixed. Now we first choose $R>0$ small such that the $R^{2+K(N)}$-term in \eqref{eq_DC_vAC_DE2} is estimated by $\frac{1}{8}R^2\varepsilon^{2M+1}$. Then $\check{\varepsilon}_1>0$ can be taken small such that \eqref{eq_DC_vAC_DE2} is bounded by $\frac{3}{4}R^2\varepsilon^{2M+1}$ for all $T\in (0,\check{T}_{\varepsilon,\beta,R}]$ and $\varepsilon\in(0,\check{\varepsilon}_1]$. By contradiction and continuity we get $\check{T}_{\varepsilon,\beta,R}=T_0$ for all $\varepsilon\in(0,\check{\varepsilon}_1]$.\qedhere$_{2.}$
	\end{proof}
	
	\begin{proof}[Ad 3] 
		Finally, let $N\in\{2,3\}$, $M=2$ and $\beta=0$. Then \eqref{eq_DC_vAC_DE2} is estimated by
		\[
		\left[\frac{R^2}{2}+CR^2T+CRT^{\frac{1}{2}}+CR^3 T^{\frac{4-N}{4}}\right]\varepsilon^{2M+1}.
		\]
		Due to $\frac{4-N}{4}>0$ there are $\check{\varepsilon}_1,\check{T}_1>0$ such that the latter is bounded by $\frac{3}{4}R^2\varepsilon^{2M+1}$ for every $T\in(0,\min(T_{\varepsilon,\beta,R},\check{T}_1)]$ and $\varepsilon\in(0,\check{\varepsilon}_1]$. Therefore $\check{T}_{\varepsilon,0,R}\geq \check{T}_1$ for all $\varepsilon\in(0,\check{\varepsilon}_1]$.
		\qedhere$_{3.}$\end{proof}
	
The proof of Theorem \ref{th_DC_vAC} is completed.
\end{proof}

\subsubsection{Proof of Theorem \ref{th_vAC_conv}}\label{sec_DC_vAC_conv}
Let $N\geq2$, $\Omega$, $Q_T$ and $\partial Q_T$ be as in Remark \ref{th_intro_coord},~1. Let $W:\R^m\rightarrow\R$ be as in Definition \ref{th_vAC_W} and $\vec{u}_\pm$ be any distinct pair of minimizers of $W$. Moreover, let $\Gamma=(\Gamma_t)_{t\in[0,T_0]}$ for some $T_0>0$ be a smooth solution to \eqref{MCF} with $90$°-contact angle condition parametrized as in Section \ref{sec_coord_surface_requ} and let $\delta>0$ be such that Theorem \ref{th_coordND} holds for $2\delta$ instead of $\delta$. We use the notation from Section \ref{sec_coord_surface_requ} and Section \ref{sec_coordND}. Let $M\in\N$ with $M\geq 2$ and denote with $(\vec{u}^A_\varepsilon)_{\varepsilon>0}$ the approximate solution on $\overline{Q_{T_0}}$ from Section \ref{sec_asym_vAC_uA} (that was constructed with asymptotic expansions in Section \ref{sec_asym_vAC}) and let $\check{\varepsilon}_0>0$ be such that Lemma \ref{th_asym_vAC_uA} (\enquote{remainder estimate}) holds for $\varepsilon\in(0,\check{\varepsilon}_0]$. The property $\lim_{\varepsilon\rightarrow 0} \vec{u}^A_\varepsilon=\vec{u}_\pm$ uniformly on compact subsets of $Q_{T_0}^\pm$ follows from Section \ref{sec_asym_vAC}. 

Theorem \ref{th_vAC_conv} follows directly from Theorem \ref{th_DC_vAC} if we prove the conditions 1.-4.~in Theorem \ref{th_DC_vAC}. The requirement 1.~(\enquote{uniform boundedness}) is satisfied because of Lemma \ref{th_asym_vAC_uA} for $\vec{u}^A_\varepsilon$ and for $\vec{u}_{0,\varepsilon}$ this is an assumption in Theorem \ref{th_vAC_conv}. Condition 2.~(\enquote{spectral estimate}) is precisely the assertion in Theorem \ref{th_SE_vAC}. Requirement 4.~(\enquote{well prepared initial data}) is a condition on $\vec{u}_{0,\varepsilon}$ and assumed in Theorem \ref{th_vAC_conv}. It remains to prove 3.~(\enquote{approximate solution}). This can be done in the analogous way as in the scalar case, cf.~the proof of Theorem \ref{th_AC_conv} in Section \ref{sec_DC_ACND_conv}. Basically one uses suitable integral transformations, Hölder estimates, transformation arguments like in Lemma \ref{th_SE_1Dtrafo_remainder}, the properties of $\vec{r}^A_\varepsilon$, $\vec{s}^A_\varepsilon$ from Lemma \ref{th_asym_vAC_uA} as well as the comparison of several differential operators in Corollary \ref{th_coordND_nabla_tau_n}. Since the computations are completely analogous to the scalar case, we refrain from going into details.\hfill$\square$\\
\newline
\textit{Acknowledgments.} The author gratefully acknowledges support through DFG, GRK 1692 \enquote{Curvature, Cycles and Cohomology} during parts of the work.

\setcounter{secnumdepth}{0}

\makeatletter
\renewenvironment{thebibliography}[1]
{\section{\bibname}
	\@mkboth{\MakeUppercase\bibname}{\MakeUppercase\bibname}%
	\list{\@biblabel{\@arabic\c@enumiv}}%
	{\settowidth\labelwidth{\@biblabel{#1}}%
		\leftmargin\labelwidth
		\advance\leftmargin\labelsep
		\@openbib@code
		\usecounter{enumiv}%
		\let\p@enumiv\@empty
		\renewcommand\theenumiv{\@arabic\c@enumiv}}%
	\sloppy
	\clubpenalty4000
	\@clubpenalty \clubpenalty
	\widowpenalty4000%
	\sfcode`\.\@m}
{\def\@noitemerr
	{\@latex@warning{Empty `thebibliography' environment}}%
	\endlist}
\makeatother

\footnotesize

\bibliographystyle{siam}

\end{document}